\documentclass{amsart}

\usepackage[all]{xy}
\usepackage{pictexwd, amsthm, amssymb, amsmath, amsxtra, graphicx, psfrag, latexsym}
\usepackage{subfigure}

\usepackage{enumerate}        

\newcommand{\mQ}{\mathbb Q}
\newcommand{\mZ}{\mathbb Z}

\newcommand{\mF}{\mathbb F}

\DeclareMathOperator{\at}{Aut}
\DeclareMathOperator{\gal}{Gal} 
 \DeclareMathOperator{\z}{{\mathbb
Z}} 
\DeclareMathOperator{\g}{\Gamma}

\DeclareMathOperator{\im}{im} 
 \DeclareMathOperator{\iso}{Isom}

\DeclareMathOperator{\vc}{\mathcal V\mathcal C}

\DeclareMathOperator{\fin}{\mathcal F\mathcal I\mathcal N}

\newcommand{\mbf}[1]{\mathbf{#1}}

\newcommand {\bn}[1]{B_n}

\newcommand{\base}{\coprod_{\langle \ell \rangle} \mathbb{R}^{2}_{\ell}}
\newcommand{\baset}{\coprod_{\langle \ell \rangle \in \mathcal{T}} \mathbb{R}^{2}_{\ell}}
\newcommand{\borelt}{\coprod_{\langle \ell \rangle \in \mathcal{T}} \Gamma \times_{\Gamma(\ell)} \mathbb{R}^{2}_{\ell}}
\newcommand{\ybase}{\coprod_{\langle \ell \rangle} Y \times \mathbb{R}^{2}_{\ell}}
\newcommand{\yborelt}{\coprod_{\langle \ell \rangle \in \mathcal{T}} \Gamma \times_{\Gamma(\ell)} ( Y \times \mathbb{R}^{2}_{\ell})}
\newcommand{\topspace}{\Gamma \times \baset}

\newtheorem{theorem}{Theorem}[section]
\newtheorem{lemma}[theorem]{Lemma}

\newtheorem{proposition}[theorem]{Proposition}
\newtheorem{corollary}[theorem]{Corollary}

\theoremstyle{definition}
\newtheorem{procedure}[theorem]{Procedure}
\newtheorem{remark}[theorem]{Remark}
\newtheorem{definition}[theorem]{Definition}
\newtheorem{example}[theorem]{Example}

\begin{document}
\title
[$K$-theory of Crystallographic Groups]
{The Lower Algebraic $K$-theory of Split Three-Dimensional Crystallographic Groups}

\author[D. Farley]{Daniel Farley}
      \address{Department of Mathematics\\
               Miami University \\
               Oxford, OH 45056\\
               }
      \email{farleyds@muohio.edu}
\author[I. J. Ortiz]{Ivonne J. Ortiz}
      \address{Department of Mathematics\\
               Miami University\\
               Oxford, OH 45056\\
               }
      \email{ortizi@muohio.edu}

\dedicatory{To the memory of Almir Alves (1965 -- 2009)}

\begin{abstract}
We compute $K_{-1}$, $\widetilde{K}_{0}$, and the Whitehead groups of all three-dimensional
crystallographic groups $\Gamma$ which fit into a \emph{split} extension 
$$ \mathbb{Z}^{3} \rightarrowtail \Gamma \twoheadrightarrow H,$$
where $H$ is a finite group acting effectively on $\mathbb{Z}^{3}$.  Such groups $\Gamma$ account for $73$ isomorphism
types of three-dimensional crystallographic groups, out of $219$ types in all. 

We also prove a general splitting formula for the lower algebraic $K$-theory of all three-dimensional crystallographic groups,
which generalizes the one for the Whitehead group obtained by Alves and Ontaneda.
\end{abstract}

                                                                                

\maketitle

\section{Introduction}
The goal of this paper is to compute the lower algebraic $K$-theory 
of the \emph{split three-dimensional crystallographic groups}; i.e., the groups $\g \leq \mathrm{Isom}(\mathbb{R}^{3})$ that fit into a split short exact sequence
$$ 1 \rightarrow L \rightarrow \Gamma \rightarrow H \rightarrow 1,$$
where $L$ is a discrete cocompact additive subgroup of $\mathbb{R}^{3}$ (i.e., a \emph{lattice} in $\mathbb{R}^{3}$), 
$H$ is a finite subgroup of the orthogonal group $O(3)$, and the map $\Gamma \rightarrow H$ sends an isometry $\gamma \in \Gamma$
to its linear part. It is well known that 
every crystallographic group fits into a similar short exact sequence, although the sequence does not split in general. Note also that a group $\Gamma$ might admit
another type of splitting without being a split crystallographic group in our sense. One such example is the Klein bottle group, a two-dimensional crystallographic group that factors as $\mathbb{Z} \rtimes \mathbb{Z}$,
but admits no factorization of the above form since it is torsion-free. There are $73$ split three-dimensional crystallographic groups up to isomorphism, representing a third of the $219$ isomorphism classes of three-dimensional
crystallographic groups in all. The split crystallographic groups are also called \emph{splitting groups}, since every $n$-dimensional crystallographic group embeds in a split 
$n$-dimensional crystallographic group (its splitting group) as a subgroup of finite index (see \cite[pg. 312-313]{Ra94}). Thus, the $73$ split three-dimensional crystallographic groups contain 
the remaining three-dimensional crystallographic groups as subgroups of finite index.
 
Our computation of the $K$-groups uses the fundamental work of Farrell and Jones \cite{FJ93}, who established an isomorphism 
$$Wh_{n}(\g) \cong H_n^{\g}(E_{\vc}(\g); \mathbb{KZ}^{-\infty}), $$
where $E_{\vc}(\g)$ is a model for the classifying space of $\g$ with isotropy in the family of virtually cyclic subgroups, and $\Gamma$ is a cocompact discrete subgroup of a virtually connected Lie group.
In particular, their work applies to crystallographic groups of all dimensions. 
    
The first computations of the lower algebraic $K$-theory of crystallographic groups were made by Kimberly Pearson \cite{P98},
who completely handled the $2$-dimensional case. Alves and Ontaneda \cite{AO06} derived the following general formula for the 
Whitehead group of a three-dimensional crystallographic group: 
$$ Wh(\Gamma) \cong \bigoplus_{G \in \mathcal{I}} Wh(G),$$
where $\mathcal{I}$ is the set of conjugacy classes of maximal infinite virtually cyclic subgroups. They also prove that the above direct sum is finite.

 In this paper, we will extend the results of \cite{AO06} in two directions:
first, we will derive similar general formulas for $K_{-1}(\mathbb{Z}\Gamma)$, $\widetilde{K}_{0}(\mathbb{Z}\Gamma)$, and $Wh(\Gamma)$, where
$\Gamma$ is an arbitrary (not necessarily split) three-dimensional crystallographic group. 
These general results are presented in Section \ref{section:EVC}, 
which is largely self-contained. The main result is Theorem \ref{theorem:splitting2}, a splitting formula for the lower algebraic $K$-theory of three-dimensional crystallographic groups:
 \[
Wh_n(\Gamma)\cong H_{n}^{\g}(E_{\fin}(\g); \mathbb{KZ}^{-\infty}) \oplus  \bigoplus_{\widehat{\ell} \in \mathcal{T}''}  H_n^{\Gamma_{\widehat{\ell}}}(E_{\fin}(\Gamma_{\widehat{\ell}}) \rightarrow  \ast;\;  \mathbb{KZ}^{-\infty}).
 \]
 The indexing set $\mathcal{T}''$ consists of a selection of one line from each $\Gamma$-orbit of lines, and may be taken to be finite, since the groups in question are trivial in all but finitely many cases. 
 In the case $n=1$, the first summand in the above formula vanishes, and the second summand can be identified with the right
 side of the formula from \cite{AO06} (above). 
The second goal is 
to make explicit calculations of the lower algebraic $K$-groups of the split three-dimensional crystallographic groups; that is, to describe completely
the isomorphism types of $K_{-1}(\mathbb{Z}\g)$, $\widetilde{K}_{0}(\mathbb{Z}\g)$, and $Wh(\g)$ as abelian groups. In order to do this, it is necessary
to identify the split three-dimensional crystallographic groups as explicitly as possible. We do this by specifying all possible pairings $(L, H)$ (up to a certain type of 
equivalence), where $L$ and $H$ are as above, and the action of $H$ (a subgroup of $O(3)$) on $L$ is the obvious one. Our work in classifying these groups is contained 
in Sections \ref{section:3Dpg}, \ref{section:arithmetic}, and \ref{section:classification}; the results are summarized in Table \ref{splitcrystallographicgroups}.
 Sections \ref{section:fundamentaldomains}, \ref{section:contributionoffinites}, \ref{section:actionsonplanes}, and \ref{section:cokernels} contain parallel computations, for all $73$ of the split three-dimensional 
 crystallographic groups, of the first and second summands from
 the splitting formula; Sections \ref{section:fundamentaldomains} and \ref{section:contributionoffinites} describe the first summand, and Sections \ref{section:actionsonplanes} and \ref{section:cokernels}
 describe the second summand.  In Section \ref{section:summary}, we summarize the results of the calculations in Tables \ref{LoweralgebraicKtheoryofsplitthreedimensionalcrystallographicgroups1} and\ref{LoweralgebraicKtheoryofsplitthreedimensionalcrystallographicgroups2}, and give
 a pair of examples illustrating how to assemble the various pieces of the calculation from the preceding sections.
 
 Now let us give a more detailed, section-by-section description of this paper.
 
 Sections \ref{section:3Dpg}, \ref{section:arithmetic}, and \ref{section:classification} give a complete classification of the split three-dimensional crystallographic groups. The presentation is almost entirely
 self-contained,
 assumes no prior knowledge of crystallographic groups, and indeed involves little more than basic group theory and linear algebra. 
 
 Section \ref{section:3Dpg} describes the classification
 of \emph{three-dimensional point groups} (\emph{point groups} hereafter), 
 which are the subgroups of $O(3)$ that leave a lattice $L \leq \mathbb{R}^{3}$ invariant. Our treatment here is standard for the most part, and the basic elements 
 of our classification can be found in Chapter 4 of \cite{Sc80} (although we assume no prior familiarity with that source). One difference is that we need to find explicit groups of matrices for our computations
 in later sections, so we give a more detailed classification than any that we were able to find in the literature. We begin our analysis by proving that point groups are finite, and satisfy the \emph{crystallographic
 restriction} \cite[pg. 32]{Sc80}: every element of a point group has order $1$, $2$, $3$, $4$, or $6$. The next step is to classify the orientation-preserving point groups $H$. The classification heavily exploits the fact that every 
 $h \in H$ is a rotation about an axis (which we call a \emph{pole}, following \cite{Sc80}). Using a counting argument  (Proposition \ref{sums}; see also \cite[pg. 45]{Sc80}), one can enumerate all
 of the possible numbers of orbits of poles, and determine the possible orders of the elements $h$ that act on a given pole (Proposition \ref{OS}; also \cite[pgs. 46-49]{Sc80}). It turns out that the latter 
 numerical information
 completely determines the orientation-preserving point group $H$ up to conjugacy within $O(3)$; we carefully argue a particular case of this fact in our proof of Theorem \ref{class+}, which gives 
 simple descriptions of the orientation-preserving point groups. There are $11$ in all. It is then straightforward to classify the remaining point groups. If a point group $H$ contains the antipodal map $(-1)$, then
 it can be expressed as $\langle H^{+}, (-1) \rangle$, where $H^{+}$ is the orientation-preserving subgroup; thus there are also $11$ point groups that contain $(-1)$ (Theorem \ref{classinv}). The remaining point groups (which are not subgroups of $SO(3)$, but also do not contain the inversion $(-1)$)
 are all necessarily subgroups of index $2$ inside of the $11$ point groups containing the inversion, and may therefore be recovered as kernels of surjective homomorphisms $\phi: H \rightarrow \mathbb{Z}/2$, where $H$ contains
 the inversion $(-1)$. This scheme of classification is carried out in Subsection \ref{subsection:pointgroupremaining}; there are $10$ additional groups of this last type, making $32$ in all. The section
 concludes with an attempt at an intuitive description of the ``standard point groups", which are simply the preferred forms (up to conjugacy) of the point groups that are used throughout the
 rest of the paper. (We also describe a few non-standard point groups that arise naturally in our arguments.) The standard point groups are described by their generators in Tables \ref{orientationpreservingpointgroups} and \ref{otherpointgroups}.    
 
 Section \ref{section:arithmetic} contains a detailed classification of the possible arithmetic equivalence classes (Definition \ref{definition:arithmeticequivalence}, which is taken from {\cite[pg. 34]{Sc80}) 
 of pairs $(L, H)$, where $L$ is a lattice and $H$ is a point group such that $H \cdot L = L$.
 We will eventually show (Theorem \ref{theorem:arithmeticsplit}) that an arithmetic equivalence class uniquely determines a split crystallographic group up to isomorphism, so Section \ref{section:arithmetic}
 contains the heart of the classification of split crystallographic groups. Our general approach to the classification of the pairs $(L,H)$ is as follows. We begin with a standard point group $H$; for the sake
 of illustration, let us assume that $H = S_{4}^{+} \times (-1)$, the group of all signed permutation matrices. We can then deduce certain facts about the lattice $L$. For instance, $H$ 
 contains rotations about the coordinate axes, so an elementary argument (Lemma \ref{bigimp}(1)) shows that there are lattice points on each of the coordinate axes. Moreover, since $H$ acts transitively
 on the coordinate axes, the lattice points on the coordinate axes having minimal norm must all have the same norm, which we can assume is $1$ up to arithmetic equivalence (which permits rescaling of
 the lattice $L$). It follows directly that $\mathbb{Z}^{3} \leq L$, and that $c(1,0,0)$, $c(0,1,0)$, and $c(0,0,1)$ are not in $L$ if $0< c < 1$. (This conclusion is recorded in Proposition \ref{full}(3), but in 
 somewhat different language.) From here, it is straightforward to argue that there are only three possibilities for the lattice $L$ -- see Lemma \ref{2X} and Corollary \ref{LC}. It is then possible to show that
 all three of these possible lattices result in different arithmetic classes of pairs $(L, H)$ (Theorem \ref{noredundancy}). The arguments in all of Section \ref{section:arithmetic} follow the same pattern: 
 for a fixed standard point group $H$, we attempt to ``build" $L$ using properties of $H$, as in the above example. We are able to show that $L$ can always be chosen from a list of only $7$ lattices, which are
 described by generating sets in Corollaries \ref{LC} and \ref{LP}. The arguments from Section \ref{section:arithmetic}, while elementary, are significantly more detailed than what can be found 
 in \cite{Sc80}. We arrive at a total of $73$
 arithmetic classes, which agrees with the count from \cite[pg. 34]{Sc80}.
 
 Section \ref{section:classification} contains the crucial Theorem \ref{theorem:arithmeticsplit}, which shows that the arithmetic classes of pairs $(L,H)$ are in exact correspondence with 
 the isomorphism classes
 of split crystallographic groups. In the proof, we appeal to Ratcliffe \cite{Ra94}, which is the only place where the argument of Sections \ref{section:3Dpg}, \ref{section:arithmetic}, and \ref{section:classification}
 fails to be self-contained. The split three-dimensional crystallographic groups are classified up to isomorphism in Table \ref{splitcrystallographicgroups}. One interesting feature of our classification
 is that all $73$ groups are contained in $7$ basic maximal groups as subgroups of finite index. This fact will be heavily used in later sections.
  We denote these maximal groups $\Gamma_{i}$ ($i =1, \ldots, 7$). 
 
 Section \ref{section:EVC} contains a proof of the main splitting result, Theorem \ref{theorem:splitting2}. This section is independent of the previous sections, and the main results apply generally,
 to all three-dimensional crystallographic groups, not just to the split ones. We begin by giving an explicit description of a model for $E_{\mathcal{VC}}(\Gamma)$, for any crystallographic group 
 $\Gamma$ (our construction comes from \cite{Fa10}). Begin with a copy of $\mathbb{R}^{3}$, suitably cellulated to make it a $\Gamma$-CW complex. This is a model for $E_{\mathcal{FIN}}(\Gamma)$. 
 For each $\ell \in L \subseteq \mathbb{R}^{3}$ that generates a maximal cyclic subgroup of $L$, we define a ``space of lines" $\mathbb{R}^{2}_{\ell}$, consisting of the set of 
 all lines $\widehat{\ell} \subseteq \mathbb{R}^{3}$
 having $\ell$ as a tangent vector. Each space $\mathbb{R}^{2}_{\ell}$ is isometric to $\mathbb{R}^{2}$ with respect to a suitable metric (see Subsection \ref{subsection:constructEVC}). The space
 $\mathbb{R}^{3} \ast \coprod_{\langle \ell \rangle} \mathbb{R}^{2}_{\ell}$ is a model for $E_{\mathcal{VC}}(\Gamma)$ (Proposition \ref{prop:EVC}). The general form of this classifying space
 allows us to deduce a preliminary splitting result (Proposition \ref{proposition:splitting1}):
\[
\begin{split}
H_{\ast}^{\g}(&E_{\vc}(\g); \mathbb{KZ}^{-\infty}) \cong \\
&H_{\ast}^{\g}(E_{\fin}(\g); \mathbb{KZ}^{-\infty}) \oplus  \bigoplus_{\langle \ell \rangle \in \mathcal T} H_n^{\Gamma(\ell)}(E_{\fin}(\Gamma(\ell))\rightarrow
E_{\vc_{\langle \ell \rangle}}(\Gamma(\ell));  \mathbb{KZ}^{-\infty})).
\end{split}
\]
Here $\Gamma(\ell)$ is the (finite index) subgroup of $\Gamma$ that takes the line $\ell$ to a line parallel to $\ell$ (possibly reversing the direction), 
and $\mathcal{VC}_{\langle \ell \rangle}$ is the family of subgroups consisting of: i) finite subgroups of
$\Gamma(\ell)$ and ii) virtually cyclic subgroups of $\Gamma(\ell)$ that contain a translation $\widetilde{\ell} \in L$ parallel to $\ell$. The indexing set $\mathcal{T}$ consists of a single choice of
maximal cyclic subgroup $\langle \ell \rangle \leq L$ from each $H$-orbit. 
 
 The next step is to compute the sum of cokernels on the right side of the formula above. We are thus led to consider the classifying space $E_{\mathcal{VC}_{\langle \ell \rangle}}(\Gamma(\ell))$;
 we use the model $\mathbb{R}^{3} \ast \mathbb{R}^{2}_{\ell}$. 
 By making a detailed analysis of the possible 
 cell stabilizers in $E_{\mathcal{VC}_{\langle \ell \rangle}}(\Gamma(\ell))$ (Lemma \ref{lemma:negligiblesufficient} and Corollary \ref{corollary:cellulated}), we are able to conclude that the great majority of these stabilizers are ``negligible". Here the class of \emph{negligible} groups is carefully chosen
 so that $Wh_{n}(G) \cong 0$ when $G$ is negligible and $n \leq 1$, and so that the property of being negligible is closed under passage to subgroups. (See Definition \ref{definition:negligible1} and
 Lemma \ref{lemma:negligibleisomorphismtypes}.) As a result, we are able to argue (Propositions \ref{proposition:subcomplexfinvc} and \ref{proposition:borelfinvc}) that the only cells from the classifying space $E_{\mathcal{VC}_{\langle \ell \rangle}}(\Gamma(\ell))$ that make a contribution to
 the cokernels for $n \leq 1$ come from a subcomplex $E$ that is a disjoint union
 $$ E = \coprod_{\widehat{\ell}} E_{\mathcal{VC}}(\Gamma_{\widehat{\ell}}). $$
 The general splitting formula (Theorem \ref{theorem:splitting2}) now follows readily from Proposition \ref{proposition:splitting1}.
 
The remainder of the paper uses Theorem \ref{theorem:splitting2} and the Farrell-Jones isomorphism to compute the lower algebraic $K$-theory of the split crystallographic groups. It now suffices 
to make separate computations of 
 $$ H_{n}^{\g}(E_{\fin}(\g); \mathbb{KZ}^{-\infty}) \quad \text{and} \quad  \bigoplus_{\widehat{\ell} \in \mathcal{T}''}  H_n^{\Gamma_{\widehat{\ell}}}(E_{\fin}(\Gamma_{\widehat{\ell}}) \rightarrow  \ast;\;  \mathbb{KZ}^{-\infty}),$$
 where the latter sum is indexed over a choice from each $\Gamma$-orbit of a line $\widehat{\ell} \in \base$ with non-negligible stabilizer.  
 
 Our task in Section \ref{section:fundamentaldomains} is to describe explicit $\Gamma$-CW structures on $\mathbb{R}^{3}$, where $\Gamma$ ranges over the 
 groups $\Gamma_{i}$ ($i = 1, \ldots, 7$). This involves
 using Poincar\'{e}'s Fundamental Polyhedron Theorem (Theorem \ref{poincare}, adapted from \cite[pg. 711]{Ra94}) to produce a fundamental domain for the action of each $\Gamma_{i}$, which leads
  to the desired $\Gamma$-equivariant cellulation (Theorem \ref{cells}). We consider each group $\Gamma_{i}$ in great detail, producing the desired cellulation and recording the non-negligible cell stabilizers   
 (Theorems \ref{theorem:g1}, \ref{theorem:g2}, \ref{theorem:g3}, \ref{theorem:g4}, \ref{theorem:g5}, \ref{theorem:g6}, and \ref{theorem:g7}).
 
 In Section \ref{section:contributionoffinites}, we compute $H_{\ast}^{\g}(E_{\mathcal{FIN}}(\g); \mathbb{KZ}^{-\infty})$, for all $73$ split crystallographic groups $\Gamma$. 
 The section begins by summarizing the isomorphism types of the groups $Wh_{n}(G)$, where $n \leq 1$ and $G$ is a finite subgroup of a crystallographic group (see Table \ref{table:tableKtheoryEfin}). 
 Most of the groups $Wh_{n}(G)$ were known before (the reference \cite{LO09} collects the previously known results), but for three groups
 ($\mathbb{Z}/4 \times \mathbb{Z}/2$, $\mathbb{Z}/6 \times \mathbb{Z}/2$, and $A_{4} \times \mathbb{Z}/2$) we make original calculations of $Wh_{n}(G)$ ($n \leq 0$). These calculations
 are summarized in Theorems \ref{subsubsection:Z4Z2}, \ref{subsubsection:Z6Z2}, and \ref{subsubsection:A4Z2}, respectively. Once we have understood the above groups $Wh_{n}(G)$ ($n \leq 1$), 
 we are ready to use a spectral sequence due to Quinn \cite{Qu82} to compute the groups $H_{\ast}^{\g}(E_{\mathcal{FIN}}(\g); \mathbb{KZ}^{-\infty})$. The main additional ingredient that we will need
 is information about the cell stabilizers in a model for $E_{\mathcal{FIN}}(\Gamma)$, where $\Gamma$ is any of the $73$ split crystallographic groups. It is here that we use the fact that each 
 split crystallographic group $\Gamma$ is a finite index subgroup of some $\Gamma_{i}$, where $i \in \{ 1, \ldots, 7 \}$. The model for $E_{\mathcal{FIN}}(\Gamma_{i})$ (which was explicitly described
 in Section \ref{section:fundamentaldomains}) is also naturally a model for $E_{\mathcal{FIN}}(\Gamma)$; it is necessary only to recompute the cell stabilizer information for the action of the smaller group.
 This is done using a simple procedure (Procedure \ref{procedure:finitepart}) and the resulting cell stabilizer information is recorded in Tables \ref{finitesubgroupsingamma1},
 \ref{finitesubgroupsingamma2}, \ref{finitesubgroupsingamma34}, \ref{finitesubgroupsingamma67}, and \ref{finitesubgroupsingamma5}. From here, it is usually straightforward to compute
 $H_{\ast}^{\g}(E_{\mathcal{FIN}}(\g); \mathbb{KZ}^{-\infty})$: in almost all cases, the vertices make the only contribution to the calculation, so the group in question is a direct sum of
 $K$-groups of the vertex stabilizers. (A more precise statement is given in Lemma \ref{lemma:computational}.) We can therefore obtain a computation simply by referring to the relevant tables.
 There are five more difficult cases, in which there is a non-trivial contribution from the edges. We make detailed calculations of $H_{\ast}^{\g}(E_{\mathcal{FIN}}(\g); \mathbb{KZ}^{-\infty})$
 in each of these cases; the calculations are contained in Examples \ref{example:gamma5}, \ref{example:c6+-}, \ref{example:c6+}, \ref{example:d6+}, and \ref{example:d''6}. Examples
 \ref{example:c6+} and \ref{example:d''6} are especially notable, since they provide the first examples of infinite groups with torsion such that  $\widetilde{K}_{0}(\mathbb{Z}\Gamma)$ has elements of infinite order.  The calculations of Section \ref{section:contributionoffinites} are summarized in Table \ref{thefinitepart}. 
 
 The final step is to compute the second summand from Theorem \ref{theorem:splitting2}. This is done in Sections \ref{section:actionsonplanes} and \ref{section:cokernels}, which are organized like
 Sections \ref{section:fundamentaldomains} and \ref{section:contributionoffinites}. In Section \ref{section:actionsonplanes}, we must consider the action of each group $\Gamma_{i}$ ($i = 1, \ldots, 7$)
 on the associated space of lines $\base$, which (up to isometry) is a countably infinite disjoint union of planes. We are able to show that, for each $\Gamma_{i}$ ($i=1, \ldots, 5$), only two or three of these planes make a contribution to $K$-theory, and that no plane makes a contribution to $K$-theory when $i = 6$ or $7$
 (Proposition \ref{proposition:finiteT''}). It therefore becomes feasible to determine the required actions and their fundamental domains; this is done in Theorems \ref{theorem:gamma1T},
 \ref{theorem:gamma2T}, \ref{theorem:gamma3T}, \ref{theorem:gamma4T}, and \ref{theorem:gamma5T}, where we also compute (one possible choice of) the indexing set $\mathcal{T}''$ explicitly,
 for all of the groups $\Gamma_{i}$, $i=1, \ldots, 5$. We describe each of the lines $\widehat{\ell} \in \mathcal{T}''$ by an explicit parametrization, and compute the \emph{strict stabilizer group}
 $\overline{\Gamma}_{\widehat{\ell}} = \{ \gamma \in \Gamma \mid \gamma_{\mid \widehat{\ell}} = id_{\widehat{\ell}} \}$, where $\Gamma$ is any of the groups $\Gamma_{i}$, $i = 1, \ldots, 5$. 
 
 In the beginning of Section \ref{section:cokernels}, we describe how to compute the indexing sets $\mathcal{T}''$ and the strict stabilizers of lines $\widehat{\ell} \in \mathcal{T}''$ for all of the
 split crystallographic groups $\Gamma$. This involves a simple procedure (Procedure \ref{procedure:nilpart1}) that computes $\mathcal{T}''$ and the strict stabilizers of lines $\widehat{\ell} \in \mathcal{T}''$
 for a group $\Gamma'$, provided that the latter information is already known for a larger group $\Gamma$ such that $[\Gamma : \Gamma'] < \infty$. Procedure \ref{procedure:nilpart1} is directly analogous
 to Procedure \ref{procedure:finitepart}, and similarly exploits the fact that every split crystallographic group sits inside of one of the $\Gamma_{i}$, $i = 1, \ldots, 7$. The strict stabilizer information
 and $\mathcal{T}''$ is recorded in Tables \ref{table:ssgamma1}, \ref{table:ssgamma2}, \ref{table:ssgamma3}, \ref{table:ssgamma4}, and \ref{table:ssgamma5}. (In particular, it follows from our calculations
 that $\mathcal{T}''$ is always finite, as we claimed above.) We then describe a procedure (Procedure \ref{procedure:computethestabilizer}) that computes the stabilizer of a line, provided that its 
 strict stabilizer is given to us. Using this procedure, we compute the stabilizers of all of the lines in $\mathcal{T}''$ (for arbitrary $\Gamma$); the results are summarized in Tables \ref{table:vc1},
 \ref{table:vc2}, \ref{table:vc3}, \ref{table:vc4}, and \ref{table:vc5}.  It is then enough to determine the isomorphism types of the relevant cokernels from 
 Theorem \ref{theorem:splitting2}; these are summarized in Table \ref{table:cokernels}. (We note also that Subsection \ref{subsection:cokernels} contains original computations of a few of the cokernels from 
 Table \ref{table:cokernels}.) By the end of Section \ref{section:cokernels}, we have completely reduced the problem of computing the second summand from Theorem \ref{theorem:splitting2} to 
 a matter of consulting the relevant tables.
 
 Section \ref{section:summary} contains Tables \ref{LoweralgebraicKtheoryofsplitthreedimensionalcrystallographicgroups1} and \ref{LoweralgebraicKtheoryofsplitthreedimensionalcrystallographicgroups2}, 
 which complete summarize  the isomorphism types of the lower algebraic $K$-groups
 of the split crystallographic groups $\Gamma$. (A group is omitted from these table if all of its $K$-groups are trivial.) Section \ref{section:summary} also contains a few examples, which are intended to help
 the reader to assemble the calculations in this paper.
 
It is natural to ask if our arguments can be generalized to other classes of crystallographic groups. The first class to consider is that of the remaining three-dimensional crystallographic
groups -- there are 146 more up to isomorphism. We expect that all of the basic arguments of this paper will still be applicable. The splitting formula (Theorem \ref{theorem:splitting2})
applies to all three-dimensional crystallographic groups (as we have argued in Section \ref{section:EVC}). Moreover, each three-dimensional 
crystallographic group sits inside of one of the $\Gamma_{i}$ ($i=1, \ldots, 7$) as a subgroup of finite index, so some procedure for passing to subgroups of finite index should still work. It seems likely that
the classification itself, and the sheer number of cases, will be the biggest obstacles to a complete calculation of the lower algebraic $K$-theory for all $219$ three-dimensional crystallographic groups. 
The authors intend to complete such a calculation in later work.

The obstacles to extending our arguments to dimension four seem much more substantial. Aside from the greater difficulty of classifying four-dimensional crystallographic groups (and the large number
of such groups), it seems that some basic features of our approach are likely to fail. For instance, the splitting formula in Theorem \ref{theorem:splitting2} depends on having a large number of negligible cell
stabilizers, which in turn depends heavily on the low codimension of cells in $\mathbb{R}^{3}$ and $\mathbb{R}^{2}_{\ell}$ (see, especially, Lemma \ref{lemma:negligiblesufficient} and Corollary \ref{corollary:cellulated}). It is therefore not clear whether an analogous formula can be proved in
dimension four. In any event, higher-dimensional cells should make more contributions to the calculations in dimension four, which will make these calculations much more complicated.  
  
\vskip 20pt
\centerline{\bf Acknowledgments}

\vskip 10pt

The authors would like to thank Tom Farrell for originally suggesting this project to us, and for his invaluable encouragement and support
while it was being completed. We are also grateful to Jean-Fran\c{c}ois Lafont for his encouragement,
many helpful discussions, and for his comments on preliminary versions of this manuscript.
The graphics in
this paper were kindly produced by Dennis Burke.  

This work is dedicated to the memory of Almir Alves. Pedro Ontaneda has written a memorial essay describing Almir's remarkable life. 
It can be found (at the time
of this writing) at: 
\begin{center}
\emph{www.math.binghamton.edu/dept/alves/pedro/Almir2.html} 
\end{center}

Almir grew up in poverty, attending school for the first time when he was ten years old. He was a lifelong ``fighter", who flourished through
perseverance, and became a PhD mathematician
with a gift for explaining complicated ideas in a simple way. He is greatly missed by all of us.

This project was  supported in part by  the NSF, under  the second author's  grants DMS-0805605 and DMS-1207712. 

\tableofcontents

\section{Three-Dimensional Point Groups} \label{section:3Dpg}

In this section, we give a complete classification of the $32$ three-dimensional point groups. Our treatment is self-contained, but
the principles of the classification come from \cite[Chapter 4]{Sc80}. Our goal is to describe the point groups as explicit groups of
matrices.

\subsection{Preliminaries} \label{subsection:preliminary}

\begin{definition} \label{definition:first}
A \emph{lattice} in $\mathbb{R}^{n}$ is a discrete, cocompact subgroup of the additive group
$\mathbb{R}^{n}$. 

We say that $H \leq O(n)$ is a \emph{point group} (of dimension $n$) if there is
a lattice $L \leq \mathbb{R}^{n}$ such that $H$ leaves $L$ invariant, i.e.,
$H \cdot L = L$.  For the rest of the paper, we will be interested only in the case $n=3$, so
``point group" will refer to a three-dimensional point group, and ``lattice" will mean a lattice
in $\mathbb{R}^{3}$.
\end{definition}

Our first goal is to classify the point groups. We will prove some preliminary facts in this subsection, and complete the classification in 
Subsection \ref{subsection:pointgroupremaining}.

First, we set some conventions. 
If $\ell$ is a line defined by a system $\mathcal{S}$ of equations, we sometimes denote this line $\ell(\mathcal{S})$.
For instance, $\ell(y=z=0)$ denotes the $x$-axis.  Similarly, if a plane $P$ is defined by the equation $E$, we sometimes
denote this plane $P(E)$.  For instance, $P(z=0)$ denotes the $xy$-plane.

\begin{lemma} \label{basics} Let $H$ be a point group.
\begin{enumerate}
\item $|H| < \infty$;
\item Any $h \in H$ leaves some line $\ell$ invariant:  $h \cdot \ell = \ell$.  If $h \neq 1$
is orientation-preserving, then $h$ fixes a unique line $\ell$, and acts as a rotation about this line.
\item (The Crystallographic Restriction) If $h \in H$, then $|h| \in \{ 1, 2, 3, 4, 6 \}$. 
\item Let $\ell \subseteq \mathbb{R}^{3}$ be a one-dimensional subspace.  The group $H^{+}_{\ell} = \{ h \in H \cap SO(3) \mid
h_{\mid \ell} = id_{\ell} \}$ is a cyclic group of rotations about the axis $\ell$, of order $1$, $2$, $3$, $4$, or $6$.  
\end{enumerate}
\end{lemma}

\begin{proof}
\begin{enumerate}
\item 
Choose an element $v \in L - \{ 0 \}$. Since $L$ is discrete, $H \leq O(3)$, and $H \cdot L = L$, there
can be at most finitely many translates of $v$ under the action of $H$. It follows that the stabilizer group
$H_{v}$ satisfies $[ H : H_{v} ] < \infty$. 

We next choose some  
$v_{1} \in L$ such that
$\{ v, v_{1} \}$ is a linearly independent set, and conclude by the same reasoning that
$[H_{v}, H_{\{ v, v_{1} \}}] < \infty$, where $H_{\{ v, v_{1} \}}$ is the subgroup of $H$ that fixes
both $v$ and $v_{1}$. 

It follows that $H_{\{v, v_{1} \}}$ has finite index in $H$. But clearly $H_{\{v, v_{1} \}}$ has order at most
$2$ (since  it acts orthogonally on $\mathbb{R}^{3}$ and fixes a plane), so $H$ itself must be finite.

\item Let $h \in SO(3)$.  It follows from basic linear algebra that $1$ is an eigenvalue of $h$, so $h$ fixes a non-trivial subspace.
  If $h$ fixed a subspace of dimension greater than or equal to $2$, then $h$ would necessarily be the identity
since $h \in SO(3)$.  The second statement follows, since $h$ must act in an orientation-preserving fashion on the plane
perpendicular to $\ell$, which means that it is a rotation in this plane.   

The first statement follows easily from the second one, since any $h \in O(3) - SO(3)$ can be expressed as $h = (-1)\hat{h}$
where $(-1) \in O(3)$ is the antipodal map and $\hat{h}$ is orientation-preserving.

\item It is sufficient to prove the statement for the case in which $h \in SO(3) - \{ 1 \}$, 
since any $\bar{h} \in O(3) - SO(3)$ can be expressed
as a product of an orientation-preserving element and the antipodal map $(-1)$.  
By (2), $h$ fixes a line, which we can assume is 
the $z$-axis.  

Let $L$ be a lattice that is invariant under the action of $H$ on $\mathbb{R}^{3}$.
We claim that the $xy$-plane contains a non-zero vector $v \in L$.  Let $v_{1} \in L - \ell(x=y=0)$.  The vector
$v_{1} - hv_{1}$ is non-zero and lies in the $xy$-plane, proving the claim.  
Now suppose that $v_{2}$ is the smallest vector
in $L \cap P(z=0)$; we can assume that $|| v_{2} || = 1$.  

We consider the set $\{ v_{2}, hv_{2}, \ldots , h^{n-1} v_{2} \}$, where $|h| = n$.  No two of the elements in this set
are equal, and any two $h^{i} \cdot v_{2}$, $h^{j} \cdot v_{2}$, $(i \neq j)$ must satisfy
$$||h^{i} \cdot v_{2} - h^{j} \cdot v_{2} || \geq 1,$$
by the minimality of the norm of $v_{2}$.  Moreover,
$$ \{ v_{2}, \ldots, h^{n-1}v_{2} \} \subseteq \{ (x,y,0) \mid x^{2} + y^{2} = 1 \}.$$
Thus, we have $n$ points arranged on the unit circle in such a way that no two are closer
than $1$ unit.  It follows easily from this that $n \leq 6$, since the circumference of the circle is $2\pi$.

We now rule out the case $|h| = 5$.  If $|h|=5$, then $h$ is rotation about the $z$-axis through $2\pi / 5$ radians.
One checks that
$$ || v_{2} + h^{2} \cdot v_{2}|| = 2 \cos \left( \frac{2\pi}{5} \right) < 1.$$
This is a contradiction.
\item This is a straightforward observation based on (2) and (3).
\end{enumerate}
\end{proof} 

\subsection{Classification of Orientation-Preserving Point Groups} \label{subsection:classifypointgroups}
Now we turn to the classification of point groups, based on \cite{Sc80}.
We would like to point out that Propositions \ref{sums} and \ref{OS}, and Definition \ref{def:pole}, are based on \cite[pgs. 45-48]{Sc80}.
Our contribution here is to describe the point groups explicitly in terms of generating sets.

\begin{definition}\label{def:pole}
Let $H \leq SO(3)$ be a non-trivial point group, and let $h \in H - \{ 1 \}$.  
If $\ell$ is a one-dimensional subspace of $\mathbb{R}^{3}$ that is fixed by $h$, then we say that
$\ell$ is a \emph{pole} of $h$.  We also say that $\ell$ is a pole of $H$.  
\end{definition}

Suppose $H \leq SO(3)$ is a non-trivial point group.  Lemma \ref{basics}(2) implies that each $\gamma \in H - \{ 1 \}$
has a unique pole.  
We let $\mathcal{L}$ be the set of all poles of $H$.  This set must be finite by Lemma \ref{basics}(1,2).
For each $\ell \in \mathcal{L}$, choose unit vectors $v_{\ell}^{+}, v_{\ell}^{-} \in
\ell$ ($ v_{\ell}^{+} \neq v_{\ell}^{-}$).  Let $\mathcal{L}^{+} = \{ v^{+}_{\ell} \mid \ell \in \mathcal{L} \}$
and $\mathcal{L}^{-} = \{ v_{\ell}^{-} \mid \ell \in \mathcal{L} \}$.  It is not difficult to see that $H$ acts on the set
$\mathcal{L}$ (and, thus, on the set $\mathcal{L}^{+} \cup \mathcal{L}^{-}$).  We will sometimes
call the elements of $\mathcal{L}^{+} \cup \mathcal{L}^{-}$ \emph{pole vectors}.
 We let $\mathcal{T}$ denote a choice of orbit
representatives of $\mathcal{L}^{+} \cup \mathcal{L}^{-}$ under this action.  Finally, we let $H_{v} = \{ h \in H \mid hv = v \}$,
and let $\mathcal{O}_{H}(v)$ be the orbit of $v$ under the action of $H$.

\begin{proposition} \label{sums}
Let $H \leq SO(3)$ be a non-trivial point group.     
$$ 2 - \frac{2}{|H|} = \sum_{v \in \mathcal{T}} \left( 1 - \frac{1}{|H_{v}|} \right).$$
\end{proposition}

\begin{proof}
First, we note that
\begin{enumerate}
\item $\displaystyle |H| - 1 + |\mathcal{L}^{+}| = \sum_{v^{+} \in \mathcal{L}^{+}} |H_{v^{+}}|$
\item $\displaystyle |H| - 1 + |\mathcal{L}^{-}| = \sum_{v^{-} \in \mathcal{L}^{-}} |H_{v^{-}}|$
\end{enumerate}
The proofs that (1) and (2) hold are identical.  Formula (1) follows directly from the observations that
$H_{v_{1}^{+}} \cap H_{v_{2}^{+}} = \{1\}$ if $v_{1}^{+} \neq v_{2}^{+}$, and that
$H$ is the union of the $H_{v}$ as $v$ ranges over $\mathcal{L}^{+}$ (Lemma
\ref{basics}(2)).  
Thus, every element of $H$ is counted exactly once
on the right side of (1), except for the identity, which is counted $|\mathcal{L}^{+}|$ times.

It follows from (1) and (2) that
\begin{eqnarray*}
2|H| - 2 = \sum_{v \in \mathcal{L}^{+} \cup \mathcal{L}^{-}} \left( |H_{v}| - 1 \right)  & \Rightarrow &
2 - \frac{2}{|H|} = \sum_{v \in \mathcal{L}^{+} \cup \mathcal{L}^{-}} \left( \frac{|H_{v}|}{|H|} - \frac{1}{|H|} \right) \\
& \Rightarrow & 2 - \frac{2}{|H|} = \sum_{v \in \mathcal{T}} \left( 1 - \frac{|\mathcal{O}_{H}(v)|}{|H|} \right) \\
& \Rightarrow & 2 - \frac{2}{|H|} = \sum_{v \in \mathcal{T}} \left( 1 - \frac{1}{|H_{v}|} \right).
\end{eqnarray*}
\end{proof}
 
\begin{proposition} \label{OS}
Let $1 \neq H \leq SO(3)$ be a point group.
The number $|\mathcal{T}|$ of orbits under the action of $H$ on $\mathcal{L}^{+} \cup \mathcal{L}^{-}$
is either $2$ or $3$.  If $|\mathcal{T}| = 2$, then $H$ leaves a line invariant, and is therefore a cyclic
group of rotations.  Its order must be $2$, $3$, $4$, or $6$.  If $|\mathcal{T}| = 3$, then let
$v_{1}$, $v_{2}$, $v_{3}$ be orbit representatives  chosen so that $|H_{v_1}| \leq |H_{v_{2}}| \leq |H_{v_{3}}|$.
We let $\alpha = |H_{v_1}|$, $\beta = |H_{v_2}|$, and $\gamma = |H_{v_{3}}|$.  The only possibilities for the 
triple $( \alpha , \beta , \gamma )$ and the order of $H$ are as follows:
\begin{enumerate}
\item $(2,2,n)$ ($n \in \{ 2,3,4,6 \}$), and $|H| = 2n$;
\item $(2,3,3)$, and $|H| = 12$; 
\item $(2,3,4)$, and $|H| = 24$.
\end{enumerate}
\end{proposition}

\begin{proof}
Suppose that $|\mathcal{T}| = 1$.
 Proposition \ref{sums} implies that
$$ 2 - \frac{2}{|H|} = 1 - \frac{1}{|H_{v}|}.$$
This has no solutions, since the left side will always be at least $1$, and the right side
is less than $1$. Thus, $|\mathcal{T}|$ is never equal to $1$.

Suppose that $|\mathcal{T}| = 2$.  Let $\mathcal{T} = \{ v_{1} , v_{2} \}$.
By the previous Proposition
\begin{eqnarray*} 
2 - \frac{2}{|H|} = 2 - \frac{1}{|H_{v_1}|} - \frac{1}{|H_{v_2}|} & \Rightarrow & 
\frac{|H|}{|H_{v_1}|} + \frac{|H|}{|H_{v_2}|} =  2 \\
& \Rightarrow & |\mathcal{O}(v_1)| + |\mathcal{O}(v_2)| = 2.
\end{eqnarray*}  
It follows that $|\mathcal{O}(v_1)| = |\mathcal{O}(v_2)| = 1$, so $v_{1}$ and $v_{2}$ are both fixed by all of $H$.
If $v_{1}$ and $v_{2}$ formed a linearly independent set, it would follow that $H = 1$ (since every element of $H$ would then
act as the identity on the plane spanned by $v_{1}$, $v_{2}$).  We have ruled this out by our hypothesis.  The only possibility
is that $v_{1}$ and $v_{2}$ are in the same $1$-dimensional subspace.  
The rest of the picture is now clear:  every element of $H$ is a rotation about the
axis determined by $v_1$ and $v_2$.  The conclusion follows.

Suppose that $|\mathcal{T}| = 3$.  Let $\mathcal{T} = \{ v_{1}, v_{2}, v_{3} \}$.  We suppose that $v_{1}$, $v_{2}$, $v_{3}$
are chosen as in the statement of the proposition.   We apply Proposition \ref{sums}:
\begin{eqnarray*}
2 - \frac{2}{|H|} = 3 - \frac{1}{\alpha} - \frac{1}{\beta} - \frac{1}{\gamma} & \Rightarrow & 
\frac{1}{\alpha} + \frac{1}{\beta} + \frac{1}{\gamma} = 1 + \frac{2}{|H|}.  
\end{eqnarray*}
We are therefore led to consider solutions $( \alpha, \beta , \gamma)$ of the inequality:
$$ \frac{1}{\alpha} + \frac{1}{\beta} + \frac{1}{\gamma} > 1,$$
where $1 < \alpha \leq \beta \leq \gamma$.  The only such solutions are:
i) $(2,2,n)$ ($n$ arbitrary), ii) $(2,3,3)$, iii) $(2,3,4)$, and iv) $(2,3,5)$. 
If we keep in mind that each integer in these $3$-tuples is $|H_{v}|$, for some $v \in \mathbb{R}^{3}$,
where $H \leq SO(3)$ is a point group, then we see that each represents the order of a cyclic subgroup of $H$.
It follows from the crystallographic restriction that we can ignore the solution $(2,3,5)$ and all solutions
$(2,2,n)$ where $n \not \in \{ 1,2,3,4,6 \}$.  This leads us to the solutions described in (1) - (3) in the 
statement of the proposition. 

Suppose that $|\mathcal{T}| \geq 4$.  We consider the equation from the statement of Proposition \ref{sums}.
Note that each term in the sum on the right is at least $1/2$, since $|H_{v}| \geq 2$.  It follows that
the right side is at least $2$.  The left side is clearly less than $2$, so it is impossible that
$|\mathcal{T}| \geq 4$.  This completes the proof.
\end{proof}
 
\begin{theorem} \label{class+} (Classification of Orientation-Preserving Point Groups) 
Suppose that $H \leq SO(3)$ acts on a lattice $L$ in $\mathbb{R}^{3}$.  The group $H$ is conjugate within $O(3)$ to
one of the groups in Table \ref{orientationpreservingpointgroups}.

\begin{table}
\vskip 5pt
\renewcommand{\arraystretch}{2.5}
\[
\begin{array} {| l | l |}  \hline   
C_{1}^{+} = \left \langle\left(\begin{smallmatrix} 1 & 0 & 0 \\ 0 & 1 & 0 \\ 0 & 0 & 1 \end{smallmatrix} \right)\right \rangle   &
D_{2}^{+} = \left\langle \left( \begin{smallmatrix} -1 & 0 & 0 \\ 0 & 1 & 0 \\ 0 & 0 & -1 \end{smallmatrix} \right), 
\left( \begin{smallmatrix} -1 & 0 & 0 \\ 0 & -1 & 0 \\ 0 & 0 & 1 \end{smallmatrix} \right) \right\rangle  \\ 
\hline
C_{2}^{+} = \left\langle \left( \begin{smallmatrix} -1 & 0 & 0 \\ 0 & -1 & 0 \\ 0 & 0 & 1 \end{smallmatrix} \right) \right\rangle  &
D_{3}^{+} = \left\langle \left( \begin{smallmatrix} 0 & 1 & 0 \\ 0 & 0 & 1 \\ 1 & 0 & 0 \end{smallmatrix} \right),   
 \left( \begin{smallmatrix} 0 & -1 & 0 \\ -1 & 0 & 0 \\ 0 & 0 & -1 \end{smallmatrix} \right) \right\rangle \\ 
 \hline
C_{3}^{+} = \left\langle \left( \begin{smallmatrix} 0 & 1 & 0 \\ 0 & 0 & 1 \\ 1 & 0 & 0 \end{smallmatrix} \right) \right\rangle &  
D_{4}^{+} = \left\langle \left( \begin{smallmatrix} 0 & 1 & 0 \\ -1 & 0 & 0 \\ 0 & 0 & 1  \end{smallmatrix} \right), 
\left( \begin{smallmatrix} 1 & 0 & 0 \\  0 & -1 & 0 \\ 0 & 0 & -1 \end{smallmatrix} \right) \right\rangle \\
\hline
C_{4}^{+} = \left\langle \left( \begin{smallmatrix} 0 & -1 & 0 \\ 1 & 0 & 0 \\ 0 & 0 & 1 \end{smallmatrix} \right) \right\rangle &
A_{4}^{+} = \left\langle \left( \begin{smallmatrix} 1 & 0 & 0 \\  0 & -1 & 0 \\ 0 & 0 & -1 \end{smallmatrix} \right), 
\left( \begin{smallmatrix} 0 & 1 & 0 \\ 0 & 0 & 1 \\ 1 & 0 & 0 \end{smallmatrix} \right) \right\rangle \\
\hline
C_{6}^{+} = \left\langle \displaystyle \frac{1}{3} \left( \begin{smallmatrix}   2 & 2 & -1 \\ -1 & 2 & 2 \\ 
2 & -1 & 2 \end{smallmatrix} \right) \right\rangle &      
S_{4}^{+} = \left\langle \left( \begin{smallmatrix} 1 & 0 & 0 \\  0 & 0 & 1 \\ 0 & -1 & 0 \end{smallmatrix} \right), 
\left( \begin{smallmatrix} 0 & 1 & 0 \\ 0 & 0 & 1 \\ 1 & 0 & 0 \end{smallmatrix} \right) \right\rangle \\ \hline     
 & D_{6}^{+} = \left\langle \displaystyle 
\frac{1}{3}\left( \begin{smallmatrix} 2 & 2 & -1 \\ -1 & 2 & 2 \\ 2 & -1 & 2  \end{smallmatrix} \right), 
\left( \begin{smallmatrix} 0 & -1 & 0 \\  -1 & 0 & 0 \\ 0 & 0 & -1 \end{smallmatrix} \right) \right\rangle \\ \hline  
\end{array}  
\]
\caption{Orientation-Preserving Point Groups}
\label{orientationpreservingpointgroups}
\end{table}

\end{theorem}

\vskip 5pt
\begin{proof}
We will show that the triple $(2,3,4)$ from Proposition \ref{OS} uniquely determines the
group $S_{4}^{+}$ up to conjugacy within $O(3)$. 

Suppose that $H \leq SO(3)$ is a point group, $\mathcal{T} = \{ v_{1} , v_{2} , v_{3} \}$, $|H_{v_{1}}| =2$,
$|H_{v_{2}}| = 3$, and $|H_{v_{3}}| = 4$.  Proposition \ref{OS} says that $|H| = 24$.  It follows that
$|\mathcal{O}_{H}(v_{3})| = 6$.

By Lemma \ref{basics}(4), the groups $H_{v_{i}}$ are cyclic, for $i =1,2,3$.
Let $h \in H_{v_{3}}$ be a generator of $H_{v_{3}}$, i.e., $\langle h \rangle = H_{v_{3}}$, and $|h| = 4$.  We consider
the action of $h$ on the set $\mathcal{O}_{H}(v_{3})$.  Certainly $h$ fixes $v_{3}$ (by definition) and $h$ can fix at most
one other vector  in $\mathcal{L}^{+} \cup \mathcal{L}^{-}$, namely $-v_{3}$ (for, otherwise, $h$ would fix a $2$-element linearly independent set,
and would necessarily be the identity, since $h \in SO(3)$).  It follows that $h$ acts on $\mathcal{O}_{H}(v_{3}) - \{ v_{3} \}$
with at most one fixed point.  After considering the possible cycle types in 
the latter action, we easily conclude that $h$ must be
a $4$-cycle. 
Therefore $h$ must have
another fixed point in $\mathcal{O}_{H}(v_{3}) - \{ v_{3} \}$, and this fixed point must be $-v_{3}$. It follows, in particular, that $-v_{3} \in \mathcal{O}_{H}(v_{3})$.

Let $\mathcal{O}_{H}(v_{3}) = \{ v_{3}, -v_{3} , v_{4} , hv_{4} , h^{2}v_{4}, h^{3}v_{4} \}$ (without loss of generality).  Since
$\mathcal{O}_{H}(v_{3})$ is $H$-invariant, we have
$$v_{3} - v_{3} + v_{4} + hv_{4} + h^{2}v_{4} + h^{3}v_{4} = (1 + h + h^{2} + h^{3})v_{4} = 0,$$
for otherwise the sum on the left would be a non-zero vector that is held invariant by all of $H$, and
this would force $H$ to be cyclic, by Lemma \ref{basics}(4).  The group $H$ is, however, not cyclic 
by the crystallographic restriction (Lemma \ref{basics}(2)). Now
$$ 0 = v_{3} \cdot ( 1 + h + h^{2} + h^{3} ) v_{4} = v_{3} \cdot v_{4} + hv_{3} \cdot hv_{4} + h^{2}v_{3} \cdot h^{2}v_{4}
+ h^{3}v_{3} \cdot h^{3}v_{4} = 4 (v_{3} \cdot v_{4}).$$
It follows that $v_{3}$ is perpendicular to $v_{4} , hv_{4}, h^{2}v_{4}, h^{3}v_{4}$.

We can repeat the above argument for each element of $\mathcal{O}_{H}(v_{3})$ to conclude that for any 
$v \in \mathcal{O}_{H}(v_{3})$:
i) $-v \in \mathcal{O}_{H}(v_{3})$, and
ii) $v \perp \hat{v}$ for any $\hat{v} \in \mathcal{O}_{H}(v_{3}) - \{ v , -v \}$.

Now choose $h_{1} \in H_{v_{2}}$ such that $\langle h_{1} \rangle = H_{v_{2}}$ (and so $|h_{1}| = 3$).  We consider the action
of $h_{1}$ on $\mathcal{O}_{H}(v_{3})$.  We claim that the cycle type of $h_{1}$ as a permutation can only be $(\ast \ast \ast) (\ast \ast \ast )$.
Indeed, if any power $n$ of $h_{1}$ fixes $\widetilde{v} \in \mathcal{O}_{H}(v_{3})$, then the 
vectors $v_{2}$ and $\widetilde{v}$ are linearly independent, and therefore span a plane that is fixed by $h^{n}_{1}$. Since $h_{1}^{n}$
is also orientation-preserving, we must have $h_{1}^{n} = 1$. Thus the action of $\langle h_{1} \rangle$ on $\mathcal{O}_{H}(v_{3})$
is free, from which the claim easily follows. 

We consider the set 
$B = \{ v_{3} , h_{1} v_{3} , h_{1}^{2} v_{3} \}$.  It is easy to check that $B \cap -B = \emptyset$, since the assumption
$B \cap -B \neq \emptyset$ quickly forces $h_{1}$ to have the wrong cycle type.  
It follows that $B$ is an orthonormal basis for
$\mathbb{R}^{3}$.  The element $h_{1}$ has the matrix
$$ \left( \begin{smallmatrix} 0 & 0 & 1 \\ 1 & 0 & 0 \\ 0 & 1 & 0 \end{smallmatrix} \right) $$
with respect to the ordered basis $[v_{3} , h_{1}v_{3} , h_{1}^{2}v_{3}]$.

We return to the element $h$, which permutes $B \cup -B = \mathcal{O}_{H}(v_{3})$.  Since $|h| = 4$, and
$h v_{3} = v_{3}$, we must have
$$ h = \left( \begin{smallmatrix} 1 & 0 & 0 \\ 0 & 0 & \mp 1 \\ 0 & \pm 1 & 0 \end{smallmatrix} \right).$$
It is now clear that the change of basis matrix sending
$\mbf{x}$ to $v_{3}$, $\mbf{y}$ to $h_{1}v_{3}$, and $\mbf{z}$ to $h_{1}^{2} v_{3}$ is orthogonal
and conjugates $\langle h, h_{1} \rangle$ to $S_{4}^{+}$ (as described in Table \ref{orientationpreservingpointgroups}). 
Since $|H| = 24$, the equality $H = \langle h, h_{1} \rangle$ is forced.
 Moreover, it is not difficult
to argue that the matrices in the statement of the theorem generate a group isomorphic to $S_{4}$.  One possible approach is
to examine the action
of the group on the diagonals of the cube $[-1, 1]^{3}$.

We have now argued that the $3$-tuple $(2,3,4)$ from Proposition \ref{OS} determines the group $S^{+}_{4}$
uniquely up to conjugacy. One can argue that the cardinality of $\mathcal{T}$, the order of $H$, and
the $3$-tuple $(\alpha, \beta, \gamma)$ (if applicable) always determine the group $H$ in the remaining
cases as well. (There are eleven cases in all, including the one above 
and the case of the trivial group.) We omit the details,
but note that the triple $(2,2,n)$ (for $n = 2, 3,4,6$) determines $D_{n}^{+}$, and $(2,3,3)$
determines $A_{4}^{+}$, while the various cyclic groups of rotations from the statement
of Proposition \ref{OS} are accounted for by the groups $C_{n}^{+}$, for $n=1,2,3,4,6$.
\end{proof}


\subsection{Classification of Point Groups with Central Inversion} \label{subsection:pointgroupwithinversion}

Suppose $H \leq O(3)$ is a point group, and $H$ contains the central inversion $(-1)$ (i.e., the antipodal map).  Let $H^{+}$ be the orientation-preserving subgroup of $H$.  We note that $[H: H^{+}] = 2$,
and $H = \langle H^{+}, (-1) \rangle$. The following observation has an obvious proof:

\begin{proposition} \label{conj}
Let $H_{1}$, $H_{2}$ be point groups containing $(-1)$.  The groups $H_{1}$, $H_{2}$ are conjugate in $O(3)$ if and only if
$H_{1}^{+}$ and $H_{2}^{+}$ are conjugate in $O(3)$. \qed
\end{proposition}

\begin{theorem} \label{classinv}
Let $H$ be a point group containing the central inversion $(-1)$.  The group $H$ is conjugate to one of the eleven groups
$\langle H^{+} , (-1) \rangle$, where $H^{+}$ is one of the orientation-preserving point groups from Theorem \ref{class+}.
\end{theorem}

\begin{proof}
This follows easily from Proposition \ref{conj} and Theorem \ref{class+}.
\end{proof}

\subsection{Classification of the Remaining Point Groups and Summary} \label{subsection:pointgroupremaining}

Let $H \leq O(3)$ be a point group such that:  i) $H \not \leq SO(3)$, and ii) $(-1) \not \in H$.  Consider the group
$\widehat{H} = \langle H,  (-1) \rangle$.  This group contains the central inversion $(-1)$, 
so it is conjugate to one of the eleven from the previous
subsection. (Note that $L$ is closed under additive inverses, so the group $\langle H, (-1) \rangle$
is still a point group.) We note that $[\widehat{H}:H] = 2$, so $H \unlhd \widehat{H}$, and therefore it
is the kernel of some surjective 
homomorphism $\phi: \langle \widehat{H}^{+}, (-1) \rangle \rightarrow \mathbb{Z}/2\mathbb{Z}$, where $\widehat{H}^{+}$
denotes the orientation-preserving subgroup of $\widehat{H}$. Thus, to find all possibilities for
$H$ up to conjugacy, we can examine the kernels of all such homomorphisms $\phi$ as 
$\langle \widehat{H}^{+}, (-1) \rangle$ 
ranges over all eleven possibilities.    
We need only consider homomorphisms
$\phi :  \langle \widehat{H}^{+}, (-1) \rangle \rightarrow \mathbb{Z} /2\mathbb{Z}$ such that: i)  $\phi ((-1)) = 1$, and
ii) $\phi (\widehat{H}^{+}) \not \leq \mathrm{Ker} \, \phi$. 

\begin{lemma} \label{a}
Let $G$ be a group generated by the set $S$, and let $(-1)$ be a central element of order two in $G$.  Let
$\phi : G \rightarrow \mathbb{Z}/ 2\mathbb{Z}$ be a homomorphism satisfying $\phi ((-1)) = 1$.  The kernel of
$\phi$ is generated by 
$$ \{ s \in S \mid \phi (s) = 0 \} \cup \{ (-1)s \in S \mid \phi (s) = 1 \}.$$
\end{lemma}
           
\begin{proof}
Let $g$ be in the kernel of $\phi$.  Since $S$ is a generating set for $G$, we must have
$g = s_{1} \ldots s_{k}$, for appropriate $s_{i} \in S \cup S^{-1}$.  Let $S_{0} = \{ s \in S \cup S^{-1} \mid \phi(s) = 0 \}$ and
$S_{1} = \{ s \in S \cup S^{-1} \mid \phi(s) = 1 \}$.  Since $g$ is in the kernel of $\phi$, it must be that 
elements of $S_{1}$ occur an even number of times in the string $s_{1} \ldots s_{k}$.  We can replace each such element $s_{i}$
with $(-1)s_{i}$ without changing the product of the string.    
\end{proof}

\begin{theorem} \label{class-}
Let $H \leq O(3) - SO(3)$ be a point group such that $(-1) \not \in H$.
The group $H$ is conjugate within $O(3)$ to one of the groups listed in Table \ref{otherpointgroups}.

\begin{table}[!h]
\vskip 5pt
\renewcommand{\arraystretch}{2.5}
\[
 \begin{array} {| l | l |} \hline    
C'_{2} = \left\langle \left( \begin{smallmatrix}1 & 0 & 0 \\ 0 & 1 & 0 \\ 0 & 0 & -1 \end{smallmatrix} \right) \right\rangle &
D'_{4} = \left\langle \left( \begin{smallmatrix} 0 & -1 & 0 \\ -1 & 0 & 0 \\ 0 & 0 & 1  \end{smallmatrix} \right), 
\left( \begin{smallmatrix} 1 & 0 & 0 \\  0 & -1 & 0 \\ 0 & 0 & -1 \end{smallmatrix} \right) \right\rangle \\ 
\hline
C'_{4} = \left\langle \left( \begin{smallmatrix} 0 & 1 & 0 \\ -1 & 0 & 0 \\ 0 & 0 & -1 \end{smallmatrix} \right) \right\rangle &
D''_{4} = \left\langle \left( \begin{smallmatrix} 0 & -1 & 0 \\ -1 & 0 & 0 \\ 0 & 0 & 1  \end{smallmatrix} \right), 
\left( \begin{smallmatrix} -1 & 0 & 0 \\  0 & 1 & 0 \\ 0 & 0 & 1 \end{smallmatrix} \right) \right\rangle \\ 
\hline
C'_{6} = 
\left\langle \displaystyle \frac{1}{3}\left( 
\begin{smallmatrix}  -2 & -2 & 1 \\ 1 & -2 & -2 \\ -2 & 1 & -2 \end{smallmatrix} \right) \right\rangle &
D'_{6} = \left\langle \displaystyle \frac{1}{3}\left( \begin{smallmatrix} 1 & -2 & -2 \\ -2 & -2 & 1 \\ -2 & 1 & -2  \end{smallmatrix} \right), 
\left( \begin{smallmatrix} 0 & 1 & 0 \\  1 & 0 & 0 \\ 0 & 0 & 1 \end{smallmatrix} \right) \right\rangle \\ 
\hline
D'_{2} = \left\langle \left( \begin{smallmatrix} -1 & 0 & 0 \\  0 & 1 & 0 \\ 0 & 0 & 1 \end{smallmatrix} \right), 
\left( \begin{smallmatrix} 1 & 0 & 0 \\ 0 & -1 & 0 \\ 0 & 0 & 1 \end{smallmatrix} \right) \right\rangle &      
D''_{6} = \left\langle \displaystyle \frac{1}{3}\left( \begin{smallmatrix} -1 & 2 & 2 \\ 2 & 2 & -1 \\ 2 & -1 & 2  \end{smallmatrix} \right), 
\left( \begin{smallmatrix} 0 & 1 & 0 \\  1 & 0 & 0 \\ 0 & 0 & 1 \end{smallmatrix} \right) \right\rangle \\ 
\hline
D'_{3} = \left\langle \left( \begin{smallmatrix} 0 & 1 & 0 \\  0 & 0 & 1 \\ 1 & 0 & 0 \end{smallmatrix} \right), 
\left( \begin{smallmatrix} 0 & 1 & 0 \\ 1 & 0 & 0 \\ 0 & 0 & 1 \end{smallmatrix} \right) \right\rangle &     
S'_{4}     = \left\langle \left( \begin{smallmatrix} 0 & 1 & 0 \\  0 & 0 & 1 \\ 1 & 0 & 0 \end{smallmatrix} \right), 
\left( \begin{smallmatrix} 0 & 0 & -1 \\ 0 & 1 & 0 \\ -1 & 0 & 0 \end{smallmatrix} \right) \right\rangle \\
\hline
\end{array}
\]
\caption{The Remaining Point Groups}
\label{otherpointgroups}
\end{table}
\end{theorem}

\begin{proof}
Most of the proof is a straightforward application of Lemma \ref{a}.  A few cases are worth some additional remarks.

Consider the case in which $[ \langle D_{4}^{+}, (-1) \rangle : H ] = 2$.  We use the equality 
$$ D_{4}^{+} = \left\langle \left( \begin{smallmatrix} 0 & 1 & 0 \\ 1 & 0 & 0 \\ 0 & 0 & -1 \end{smallmatrix} \right),
\left( \begin{smallmatrix} 1 & 0 & 0 \\ 0 & -1 & 0 \\ 0 & 0 & -1 \end{smallmatrix}\right) \right\rangle = \langle A, B \rangle,$$
where $A$ and $B$ (respectively) are the matrices in the term between the equal signs.  Note that $A$ is the $180$-degree
rotation about the line $\ell$ defined by the equations $x=y$ and $z=0$, 
and $B$ is the $180$-degree rotation about the $x$-axis.  It is easy to see that
these are generators, as claimed.  There are three homomorphisms $\phi_{1}$, 
$\phi_{2}$, $\phi_{3}$ to consider (all of which send $(-1)$ to $1$):  
i) $\phi_{1} (A) = 1$, $\phi_{1} (B)=1$; ii) $\phi_{2} (A) = 1$, $\phi_{2}(B) = 0$; iii) $\phi_{3} (A) = 0$, $\phi_{3}(B) = 1$.
We note that the kernels of $\phi_{2}$ and $\phi_{3}$ are conjugate in $O(3)$, since there is an element $\lambda \in O(3)$
conjugating $A$ to $B$ and $B$ to $A$.  The first two homomorphisms have the kernels $D''_{4}$ and $D'_{4}$, respectively, by
Lemma \ref{a}.  We note that $C_{4}^{+} \leq D''_{4}$ and $D_{2}^{+} \leq D'_{4}$, and this distinguishes the groups up to conjugacy
within $O(3)$. Thus, if $H$ is as in Theorem \ref{class-}, and 
$\langle H, (-1) \rangle$ is conjugate to 
$\langle D_{4}^{+}, (-1) \rangle$, then $H$ is conjugate either to $D_{4}'$ or to $D_{4}''$.

The case in which $[ \langle D_{6}^{+}, (-1) \rangle : H ] = 2$ is analogous to the previous one.  We use the generating
elements
$$ A = \frac{1}{3}\left( \begin{smallmatrix} 1 & -2 & -2 \\ -2 & -2 & 1 \\ -2 & 1 & -2 \end{smallmatrix} \right),
\quad \quad 
B = \left(\begin{smallmatrix} 0 & -1 & 0 \\ -1 & 0 & 0 \\ 0 & 0 & -1 \end{smallmatrix} \right)$$
  and proceed as before.

For the case in which $[ \langle S_{4}^{+}, (-1) \rangle : H ] = 2$ we use a generating set $T$
for $S_{4}^{+}$ different from the one appearing
in Theorem \ref{class+}; here 
$$ T = \left\{ \left( \begin{smallmatrix} 0 & 1 & 0 \\  0 & 0 & 1 \\ 1 & 0 & 0 
\end{smallmatrix} \right), 
\left( \begin{smallmatrix} 0 & 0 & 1 \\ 0 & -1 & 0 \\ 1 & 0 & 0 
\end{smallmatrix} \right) \right\}.$$
We note that the first matrix has order $3$, and must therefore be sent to the identity in any homomorphism
$\phi: \langle S^{+}_{4}, (-1) \rangle \rightarrow \mathbb{Z}/2\mathbb{Z}$. It follows that there is just one homomorphism such
that $\phi(-1) = 1$ and $S^{+}_{4} \not \leq \mathrm{Ker} \phi$; the kernel of this homomorphism is $S'_{4}$ by Lemma \ref{a}.   
\end{proof}

\begin{definition} \label{standard}
We say that a point group $H \leq O(3)$ is \emph{standard} if
it is one of the $32$ described in Theorems \ref{class+}, \ref{classinv}, and \ref{class-}.
\end{definition}

\subsection{Descriptions of selected point groups}

In this subsection, we will attempt to give simple descriptions of the $32$ standard point
groups. Our goal is to help the reader develop a working knowledge of these groups, which
will be essential in subsequent sections. We will also introduce certain non-standard
point groups that will arise naturally later.

\subsubsection{The orientation-preserving standard point groups} \label{subsubsection:description1}
There are $11$ of these in all, as listed in Table \ref{orientationpreservingpointgroups}.

\begin{itemize}
\item Five are cyclic: $C_{i}^{+}$ ($i= 1,2,3,4,6$). If $i \in \{ 1,2,4 \}$, then $C_{i}^{+}$ 
is generated by a rotation of order $i$ about the $z$-axis. If $i \in \{ 3, 6 \}$, then $C_{i}^{+}$
is generated by a rotation of order $i$ about the line $x=y=z$. 

\item The dihedral group $D_{2}^{+}$ consists of all of the $180$-degree rotations about the coordinate
axes, and the identity. We can also describe $D_{2}^{+}$ algebraically: it is the group
of $3 \times 3$ diagonal matrices 
with an even number $-1$s down the diagonal, where all other entries on the 
diagonal are $1$.

\item The dihedral group $D_{4}^{+}$ is generated by $180$-degree rotations about the lines
$\ell(y=z=0)$ and $\ell(x=y; z=0)$. (Note that the group elements in question are 
not the generators listed in Table \ref{orientationpreservingpointgroups}.) The
 $xy$-plane is invariant, and $D_{4}^{+}$ acts as the group of symmetries of the square
$[-1,1]^{2}$ in that plane.
We can also describe $D_{4}^{+}$ algebraically: $D_{4}^{+}$ consists of all matrices
having determinant $1$ and the form $SP$, where $S$ is a \emph{sign matrix} (i.e., a matrix whose 
off-diagonal entries are $0$, and whose diagonal entries are $\pm 1$) and $P$ is either
the identity matrix or the permutation matrix that interchanges the first two columns.

\item The dihedral group $D_{3}^{+}$ is generated by $180$-degree rotations about the lines
$\ell(x+y=0; z=0)$ and $\ell(x=0; y+z=0)$. The plane $P(x+y+z=0)$ is invariant; the elements
of order $3$ are rotations through $120$ degrees about the axis $\ell(x=y=z)$. Algebraically,
$D_{3}^{+}$ is the group of matrices having the form $SP$, where $P$
is any permutation matrix and $S$ is either the identity matrix (if $\mathrm{det}(P)=1$)
or the antipodal map (if $\mathrm{det}(P)=-1$).

\item The dihedral group $D_{6}^{+}$ is generated by $180$-degree rotations about the lines
$\ell(x+y+z=0; 2x+y =0)$ and $\ell(x+y=0; z=0)$. It leaves the plane $P(x+y+z=0)$
invariant. The additive group $\langle (1,-1,0), (0,-1,1) \rangle$ is a lattice in $P(x+y+z=0)$.
The six lattice points of smallest norm describe a regular hexagon. The group $D_{6}^{+}$
acts as the group of symmetries of this hexagon.

\item The group $A_{4}^{+}$ consists of all matrices of the form $SP$, where $P$ is 
a permutation matrix that permutes the coordinate axes cyclically and $S$ is a signed
matrix with an even number of $-1$s on the diagonal. 

\item The group $S_{4}^{+}$ is the group of signed permutation matrices with determinant 
equal to $1$.
\end{itemize}

\subsubsection{The standard point groups with inversion} 
\label{subsubsection:pointgroupsinversion}
The $11$ standard point groups $H$ that contain the inversion $(-1)$ all have the form $\langle H^{+}, (-1) \rangle$,
where $H^{+}$ is one of the $11$ orientation-preserving standard point groups. Thus, the groups $H$ have descriptions
similar to the ones that were given in \ref{subsubsection:description1}; we briefly give details for the $H = \langle
H^{+}, (-1) \rangle$ when $H^{+}$ is neither cyclic nor $D_{6}^{+}$.

\begin{itemize}
\item If $H^{+} = D_{2}^{+}$, then $H$ is the group of sign matrices.

\item If $H^{+} = D_{4}^{+}$, then $H$ is the set of all matrices expressible in the form $SP$, where $S$ is 
an arbitrary sign matrix and $P$ is either the identity or the permutation matrix which interchanges
the first two coordinates.

\item If $H^{+} = D_{3}^{+}$, then $H$ is the set of all matrices of the form $SP$, where $S$ is either the identity
or $(-1)$, and $P$ is an arbitrary permutation matrix.

\item If $H^{+} = A_{4}^{+}$, then $H$ is the set of all matrices of the form $SP$, where $S$ is an arbitrary sign
matrix, and $P$ is a cyclic permutation matrix.

\item If $H^{+} = S_{4}^{+}$, then $H$ is the full group of signed permutation matrices.
\end{itemize}

\subsubsection{The remaining standard point groups} \label{subsubsection:pointgroupsmixed}
We now briefly describe the remaining point groups (as listed in Theorem \ref{class-}).
\begin{itemize}
\item The group $C'_{2}$ is generated by reflection
across the $xy$-plane. The group $C'_{4}$ is generated by a rotation about the $z$-axis of $90$ degrees, followed by reflection in the
$xy$-plane. 

The group $C'_{6}$ is generated by a rotation through $120$ degrees about the line $x=y=z$, followed by reflection
across the plane $P(x+y+z=0)$ (the complementary subspace to $\ell(x=y=z)$). 

\item The group $D'_{2}$ consists of the sign matrices having a $1$ in the lower right corner.

\item The group $D'_{3}$ is simply the group of permutation matrices.

\item The group $D'_{4}$ (like $D_{4}^{+}$) leaves the $xy$-plane invariant, and acts as the group of symmetries
of the square $[-1,1]^{2}$ in that plane. However, the elements that behave like reflections in the coordinate axes
of $P(z=0)$ (when we consider the restriction of the action to $P(z=0)$) are actually rotations in the ambient $\mathbb{R}^{3}$. (In other words, $D_{2}^{+} \leq D'_{4}$.)
The elements that behave like reflections in the lines $\ell(x=y; z=0)$ and $\ell(x=-y; z=0)$ are reflections in the
planes $P(x=y)$ and $P(x=-y)$, respectively.

We can also give a simple algebraic description of $D'_{4}$: it is the group of all matrices having the form
$SP$, where $S$ is a sign matrix with an even number of negative (i.e., $-1$) entries, and $P$ is either the
identity matrix or the permutation matrix that interchanges the first two coordinates.

We also note that $C'_{4} \leq D'_{4}$.

\item The group $D''_{4}$ (like $D_{4}^{+}$ and $D'_{4}$)  leaves the $xy$-plane invariant, and acts as the 
group of symmetries of the square $[-1,1]^{2}$. Each element of $D''_{4}$ that acts as a reflection in the (restricted)
action of $D''_{4}$ on the $xy$-plane is also a reflection of the ambient $\mathbb{R}^{3}$ across a plane. 

The algebraic description of $D''_{4}$ is also easy: it is the group of all signed permutation matrices having a $1$
in the lower right corner. 

It is clear that $D'_{2} \leq D''_{4}$ and $C_{4}^{+} \leq D''_{4}$.

\item The group $D'_{6}$ can be generated in the following way:
$$D'_{6} \cong 
 D'_{3} \times \left\langle \frac{1}{3} \left( \begin{smallmatrix} 1 & -2 & -2 \\ -2 & 1 & -2 \\ -2 & -2 & 1 \end{smallmatrix} \right)
\right\rangle.$$
The latter matrix is reflection across the plane $P(x+y+z=0)$. In fact, we can factor $\mathbb{R}^{3}$ orthogonally as 
$P(x+y+z=0) \times \ell(x=y=z)$. With respect to this factorization, $D'_{3}$ acts on the first factor (leaving it invariant),
and acts trivially on the second factor; the above reflection acts trivially on the first factor and as inversion on the 
second factor. We note also that $C'_{6} \leq D'_{6}$.

\item The group $D''_{6}$ is analogous to $D''_{4}$. In the orthogonal 
factorization $\mathbb{R}^{3} = P(x+y+z=0) \times 
\ell(x=y=z)$, $D''_{6}$ acts as the full group of symmetries of a regular hexagon in the first factor, and trivially in the
second. We note also that $C^{+}_{6} \leq D''_{6}$.

\item The group $S'_{4}$ consists of all matrices of the form $SP$, where $P$ is an arbitrary permutation matrix
and $S$ is a sign matrix with an even number
of negative entries. In particular, we have the inclusion $D'_{4} \leq S'_{4}$.
\end{itemize}

\subsubsection{Some non-standard point groups} \label{subsubsection:nonstandard}
The property of being a point group is (clearly) inherited under passage to subgroups,
but the property of being a \emph{standard} point group is not. The arguments of subsequent 
sections will frequently involve passing to subgroups, which will mean that we cannot always
consider only the standard point groups. Here we describe a few of the non-standard
point groups that will arise in practice.

\begin{itemize}
\item First, let
$$ \widehat{D}'_{4} = \left\langle \left( \begin{smallmatrix} 0 & 1 & 0 \\ 1 & 0 & 0 \\ 0 & 0 & -1 \end{smallmatrix} \right), \left( \begin{smallmatrix} -1 & 0 & 0 \\ 0 & 1 & 0 \\ 0 & 0 & 1 \end{smallmatrix} \right) \right\rangle.$$
This group is conjugate within $O(3)$ to the standard point group $D'_{4}$, and its
action on $\mathbb{R}^{3}$ is similar to that of $D'_{4}$ in most respects: the group
 $\widehat{D}'_{4}$  leaves the $xy$-plane invariant, and acts as the group of symmetries
of the square $[-1,1]^{2}$ in that plane. The elements that behave like reflections in the coordinate axes
of $P(z=0)$ (when we consider the restriction of the action to $P(z=0)$) are also
reflections in the ambient $\mathbb{R}^{3}$ (i.e., they are the reflections in the $xz$- and
$yz$-planes, respectively).
The elements that behave like reflections in the lines $\ell(x=y; z=0)$ and $\ell(x=-y; z=0)$ 
(the diagonals of the square) are
actually rotations through $180$ degrees about those lines.
\item The group 
$$\widehat{D}'_{6}  =  \left\langle \frac{1}{3}\left( \begin{smallmatrix} -1 & 2 & 2 \\ 2 & 2 & -1 \\ 2 & -1 & 2 \end{smallmatrix} \right),
\left( \begin{smallmatrix} 0 & -1 & 0 \\ -1 & 0 & 0 \\ 0 & 0 & -1 \end{smallmatrix} \right) \right\rangle$$
is conjugate within $O(3)$ to $D'_{6}$. 
 We can
factor $\widehat{D}'_{6}$ as follows:
$$ \widehat{D}'_{6} \cong D_{3}^{+} \times \left\langle \frac{1}{3}\left( \begin{smallmatrix} 1 & -2 & -2 \\ -2 & 1 & -2
\\ -2 & -2 & 1 \end{smallmatrix} \right) \right\rangle. $$
Each part of the factorization leaves the subspaces $P(x+y+z=0)$ and $\ell(x=y=z)$ invariant. (As noted in the
description of $D'_{6}$, the latter matrix is reflection in the plane $P(x+y+z=0)$, and so acts trivially on the first factor.)
We note that $C'_{6} \leq \widehat{D}'_{6}$.
\item Finally, we describe some non-standard variations of point groups that are isomorphic
to $D_{2}$ or $D_{4}$. Each standard point group of this type (with the exception of 
$D^{+}_{2}$) has a distinguished coordinate axis, namely the $z$-axis, which is invariant under the action of the group. For instance, 
each of the standard point groups $D^{+}_{4}$, $D'_{4}$, and $D''_{4}$ leaves the $z$-axis
invariant. We will denote the non-standard point group with a different distinguished
axis by a subscript of $1$ or $2$, where $1$ indicates that the $x$-axis is distinguished, and
$2$ indicates that the $y$-axis is distinguished. For instance, $D^{+}_{4_{1}}$ denotes 
the group of all matrices having the form $SP$, where $\mathrm{det}(SP) = 1$, 
$S$ is a sign matrix, and $P$ is either the identity or the permutation matrix that interchanges
the $y$- and $z$-coordinates. The group $D_{4_{2}}^{+}$ has the same description, except
that the matrix $P$ may be either the identity or the permutation matrix interchanging the 
$x$- and $z$-coordinates. The groups $D'_{4_{i}}$ and $D''_{4_{i}}$ for $i=1,2$ have 
analogous descriptions.

Similarly, for $i=1,2$, we let $D'_{2_{i}}$ denote the group of sign matrices having 
a $1$ in the $i$th position on the diagonal. 
\end{itemize}

\section{Arithmetic Classification of Pairs $(L, H)$} \label{section:arithmetic}

Let $L$ be a lattice in $\mathbb{R}^{3}$, and let $H \leq O(3)$
be a point group such that $H \cdot L = L$.
In this section, we classify pairs $(L,H)$ up to arithmetic equivalence
(defined below).  The equivalence classes of pairs $(L,H)$ are in one-to-one
correspondence with isomorphism classes of split crystallographic
groups (see Section \ref{section:classification}).

\subsection{Definition of arithmetic equivalence and a lemma}

Now we introduce one of the central tools in the classification of split three-dimensional crystallographic groups. The definition below
is due to Schwarzenberger \cite[pg. 34]{Sc80}.

\begin{definition} \label{definition:arithmeticequivalence}
Let $H \leq O(3)$ be a point group, and let $L$ be a lattice
in $\mathbb{R}^{3}$ satisfying $H \cdot L = L$.  We say that
two pairs $(L', H')$, $(L, H)$ are \emph{arithmetically equivalent},
and write $(L', H') \sim ( L, H)$,
if there is an invertible linear transformation $\lambda \in GL_{3}(\mathbb{R})$
such that:
\begin{enumerate}
\item $\lambda L' = L$, and 
\item $\lambda H' \lambda^{-1} = H$.
\end{enumerate}
\end{definition}

The following lemma will be used heavily in our classification of pairs $(L,H)$ up to arithmetic equivalence.

\begin{lemma} \label{bigimp}
Let $H$ be a point group acting on the lattice $L$.
\begin{enumerate}
\item If $h \in H \cap SO(3)$, $h \neq 1$, and $\ell$ is the (unique) line fixed by 
$h$, then $\ell$ contains a non-zero element of $L$.
\item If $h \in H \cap SO(3)$, $h \neq 1$, and $\ell$ is the unique line
 fixed by $h$, then $P = \{ v \in \mathbb{R}^{3} \mid v \perp \ell \}$
contains a non-zero element of $L$.
\item If $h \in H$ is reflection in the plane $P$, then $L \cap P$ is free abelian of rank two.
\end{enumerate}
\end{lemma}

\begin{proof}  
\begin{enumerate}
\item Let $v \in L$ be such that $v \not \perp \ell$ and $v \notin \ell$. Thus, we can write $v = v_{1} + v_{2}$, where $v_{1} \perp \ell$,
$v_{2} \in \ell$, and neither $v_{1}$ nor $v_{2}$ is $0$. Assume that $|h| =n$; we can write
$$(1+h+h^{2}+ \ldots + h^{n-1}) v = nv_{2} + (1+h+\ldots+h^{n-1})v_{1},$$
since $v_{2}$ is fixed by $h$. We note that $nv_{2}$ and $(1+h+ \ldots  + h^{n-1})v_{1}$
are perpendicular, essentially by orthogonality of $h$, and that $(1+h+h^{2}+ \ldots + h^{n-1})v_{1}$
is $h$-invariant. It follows that $(1+h+h^{2}+\ldots + h^{n-1})v_{1} \in \ell$, and therefore can only be $0$ (since it is perpendicular to 
$v_{2} \in \ell - \{ 0 \}$). Now clearly $(1+h+h^{2}+ \ldots +h^{n-1})v = nv_{2} \in \ell - \{ 0 \}$ and $(1+h+h^{2}+ \ldots +h^{n-1})v \in L$. 



\item Let $x \in L - \ell$.  It follows that $x - hx \neq 0$.  Now we show that $x - hx \in P$.  Let
$v \in \ell - \{ 0 \}$.
$$ v \cdot x = hv \cdot hx = v \cdot hx,$$
so $v \cdot (x - hx) = 0$, and $x - hx \in P$.

\item Our assumptions imply that we can choose an ordered basis for $\mathbb{R}^{3}$ in such a way
that $h$ is represented by the matrix
$$\left( \begin{smallmatrix} 1 & 0 & 0 \\ 0 & 1 & 0 \\ 0 & 0 &-1 \end{smallmatrix} \right)$$
over that basis. It is then clear that the transformation $1+h$ has a $1$-dimensional null space. 

Let $v_{1}$, $v_{2}$, $v_{3}$ be a linearly independent subset of $L$. It follows that the 
vectors $(1+h)v_{1}$, $(1+h)v_{2}$, and $(1+h)v_{3}$ span a $2$-dimensional subspace of $\mathbb{R}^{3}$. 
It follows, in particular, that the group $G = \langle (1+h)v_{1}, (1+h)v_{2}, (1+h)v_{3} \rangle$ has rank $2$. (The rank can be no larger,
since $G$ is a discrete subgroup of a real subspace of dimension $2$; the rank can be no smaller, since the generators span a real vector subspace
of dimension $2$.) Since $h$ fixes each element of the generating set, $G \subseteq P$.

Finally, we note that $G \leq P \cap L$, so the latter group must have rank at least $2$. It cannot have rank more than $2$ since
it is a discrete additive subgroup of a $2$-dimensional real vector space. 

\end{enumerate}
\end{proof}
  
\subsection{Full Sublattices in Pairs $(L, H)$, where $H$ contains $(-1)$}

In this subsection, we take our first steps toward classifying the pairs $(L,H)$ up to arithmetic equivalence
(i.e., classifying the split crystallographic groups  -- see Theorem \ref{theorem:arithmeticsplit}).
Our strategy is to build the lattice $L$ around the point group $H$.  The lattice $L$ (for any choice of $H$)
will be made to contain one of two specific lattices in a certain way (related to the definition of ``fullness" below). 
Our technical results here will be important in the subsequent classification of pairs $(L,H)$. 
   
\begin{definition} \label{lattice}
We let
$$ \mbf{x} = \left( \begin{smallmatrix} 1 \\ 0 \\ 0 \end{smallmatrix} \right),  \, \, 
\mbf{y} = \left( \begin{smallmatrix} 0 \\ 1 \\ 0 \end{smallmatrix} \right), \, \,  
\mbf{z} = \left( \begin{smallmatrix} 0 \\ 0 \\ 1 \end{smallmatrix} \right), \, \,  
\mbf{v}_{1} = \left( \begin{smallmatrix} 1 \\ 1 \\ 1 \end{smallmatrix} \right), \, \,  
\mbf{v}_{2} = \left( \begin{smallmatrix} 1 \\ -1 \\ 0 \end{smallmatrix} \right), \, \,  
\mbf{v}_{3} = \left( \begin{smallmatrix} 0 \\ -1 \\ 1 \end{smallmatrix} \right).$$ 
The lattices $\langle \mbf{x} , \mbf{y} , \mbf{z} \rangle$
and $\langle \mbf{v}_{1} , \mbf{v}_{2} , \mbf{v}_{3} \rangle$ 
are, respectively, the \emph{cubic lattice} $L_{\mathcal{C}}$ and the \emph{prismatic lattice} $L_{\mathcal{P}}$.   
\end{definition}

\begin{definition} \label{definition:full}
If $L$ is any lattice, then a subgroup $\widehat{L}\leq L$ is \emph{full} in $L$ if $\widehat{L}$ is the maximal subgroup of $L$
that is contained in the span of $\widehat{L}$ as a vector space.
\end{definition}

\begin{proposition} \label{full}
Let $H$ be a standard point group acting on the lattice $L$; suppose $(-1) \in H$.  
We let $H^{+}$ denote the orientation-preserving subgroup of $H$.
\begin{enumerate}
\item If $H^{+} = C_{1}^{+}$, then $(L, H) \sim   
( L_{\mathcal{C}}, H)$.
\item If $H^{+} =  C_{2}^{+}, C_{4}^{+},$ or $D_{4}^{+}$, then $(L, H) \sim 
(L', H)$, where $L_{\mathcal{C}} \leq L'$ and each of the subgroups
$\langle \mbf{x} , \mbf{y} \rangle$, $\langle \mbf{z} \rangle$
is full in $L'$.
\item If $H^{+} = D_{2}^{+}, A_{4}^{+},$ or $S_{4}^{+}$, then $(L, H) \sim  
( L' , H)$ where $L_{\mathcal{C}} \leq L'$ and each of the subgroups
$\langle \mbf{x} \rangle$, $\langle \mbf{y} \rangle$, $\langle \mbf{z} \rangle$ 
is full in $L'$.
\item If $H^{+} = C_{3}^{+}, C_{6}^{+},$ or $D_{6}^{+}$, then $(L, H) \sim  
(L', H)$, where $L_{\mathcal{P}} \leq L'$ and each of the subgroups
$\langle \mbf{v}_{2} , \mbf{v}_{3} \rangle$, $\langle \mbf{v}_{1} \rangle$
is full in $L'$.
\item If $H^{+} = D_{3}^{+}$, then $(L, H) \sim  
( L' , H)$ where $L_{\mathcal{P}} \leq L'$ and each of the subgroups
$\langle \mbf{v}_{1} \rangle$, $\langle \mbf{v}_{2} \rangle$, $\langle \mbf{v}_{3} \rangle$
is full in $L'$.
\end{enumerate}
\end{proposition}  

\begin{proof}
\begin{enumerate}
\item This is easy.  If $H^{+} = C_{1}^{+} = 1$, then $H = \langle (-1) \rangle$.  There is 
$\lambda \in GL_{3}(\mathbb{R})$ such that $\lambda L = L_{\mathcal{C}}$.  It is obvious that
$\lambda (-1) \lambda^{-1} = (-1)$, so $(L, H) \sim (\lambda L, \lambda H \lambda^{-1}) = ( L_{\mathcal{C}}, H)$.

\item Suppose $H^{+} = C_{2}^{+}$.  The group $\langle C_{2}^{+} , (-1) \rangle$ contains the reflection across
the $xy$-plane, so $L \cap P(z=0)$ is free abelian of rank two, by Lemma \ref{bigimp}(3).  It follows that there is some
$$ \lambda = \left( \begin{smallmatrix} \ast & \ast & 0 \\ \ast & \ast & 0 \\ 0 & 0 & \ast \end{smallmatrix} \right) 
\in GL_{3}(\mathbb{R})$$
such that $\lambda ( L \cap P (z=0) ) = L_{\mathcal{C}} \cap P(z=0) = \langle \mbf{x}, \mbf{y} \rangle$.  Any such $\lambda$ commutes with
$\langle C_{2}^{+} , (-1) \rangle$.  It follows that $(L, H)$ is arithmetically equivalent to a pair $(\widehat{L}, H)$
such that $\langle \mbf{x}, \mbf{y} \rangle$ is a full subgroup of $\widehat{L}$.  By Lemma \ref{bigimp}(1), there is some
$\alpha \mathbf{z} \in \widehat{L}$, where $\alpha \neq 0$ and $|\alpha|$ is minimal.  We multiply $\widehat{L}$
by
$$ \lambda' = \left( \begin{smallmatrix} 1 & 0 & 0 \\ 0 & 1 & 0 \\ 0 & 0 & \alpha^{-1} \end{smallmatrix} \right).$$
The matrix $\lambda'$ commutes with $H$ and $\lambda' \widehat{L} = L'$ has the desired properties.

Suppose that $H^{+} = C_{4}^{+}$.  By Lemma \ref{bigimp}(2), there is some non-zero $v \in L \cap P(z=0)$, since the generator
of $C_{4}^{+}$ acts as a rotation about the $z$-axis.  We choose $v$ to have
the minimal norm of all non-zero vectors in $L \cap P(z=0)$. After multiplying by a scalar matrix, we can assume
that $||v||=1$. After applying a suitable rotation $\lambda$ 
(which, as in the case
of $C_{2}^{+}$, is a block matrix with a $2$ by $2$ block in the upper left corner), we can assume $v = \mathbf{x}$, i.e.,
$\mathbf{x} \in \lambda L = \widehat{L}$, $\mathbf{x}$ is the non-zero vector of minimal norm in $\widehat{L} \cap P(z=0)$,
and $\lambda H \lambda^{-1} = H$ (since any two rotations in the $xy-$plane commute).  We conclude that $\mathbf{y} \in \widehat{L}$ since $H^{+}$ contains a rotation
through $90$ degrees about the $z$-axis.  It easily follows from the minimality of the norm of
$\mbf{x}$ in $\widehat{L} \cap P(z=0)$ that $\langle \mbf{x}, \mbf{y} \rangle$ is a full subgroup of 
$\widehat{L}$.  We can then continue as before (that is, rescale along the $z$-axis while leaving the
$xy-$plane alone) to get a new lattice $L'$ having the additional property that
$\langle \mathbf{z} \rangle$ is a full subgroup of $L'$,    
and $(L, H) \sim (L' , H)$. 

Suppose that $H^{+} = D_{4}^{+}$.  The group $D_{4}^{+}$ contains $180$-degree rotations
about the axes $\ell(y=z=0)$, $\ell(x=z=0)$, $\ell(z=0=x-y)$, and
$\ell (z=0=x+y)$.  We claim that a smallest non-zero vector in $L \cap P(z=0)$ lies on one of these
axes.  (A non-zero vector in $L \cap P(z=0)$ exists by Lemma \ref{bigimp}(2).)  
We can assume first that $\mathbf{x}$ is the smallest non-zero vector in $L \cap \ell (y=z=0)$ (after multiplying
by a suitable scalar matrix if necessary).  (Such a non-zero vector in $L \cap \ell(y=z=0)$ exists,
by Lemma \ref{bigimp}(1).)  It follows that $\mathbf{y}$ is likewise the smallest non-zero vector in
$L \cap \ell(x=z=0)$, since there is a rotation in $D_{4}^{+}$ taking the $x$-axis to the $y$-axis.  Now let $v = \alpha \mathbf{x} + \beta \mathbf{y}$ be the smallest non-zero vector in
$L \cap P(z=0)$.  Apply the $180$-degree rotation about the $x$-axis:
$$ \left( \begin{smallmatrix} 1 & 0 & 0 \\ 0 & -1 & 0 \\ 0 & 0 & -1 \end{smallmatrix} \right) 
\left( \begin{smallmatrix} \alpha \\ \beta \\ 0  \end{smallmatrix} \right) = 
\left( \begin{smallmatrix} \alpha \\ -\beta \\ 0 \end{smallmatrix} \right) \Rightarrow 
\left( \begin{smallmatrix} \alpha \\ \beta \\ 0 \end{smallmatrix} \right)
+ \left( \begin{smallmatrix} \alpha \\ -\beta \\ 0 \end{smallmatrix} \right)  = 
\left( \begin{smallmatrix} 2\alpha \\ 0 \\ 0 \end{smallmatrix} \right) \in L.$$
It follows that $2\alpha \in \mathbb{Z}$, since $\langle \mathbf{x} \rangle$ is a full subgroup of $L$.  By similar
reasoning (applying the rotation about the $y$-axis), $2\beta \in \mathbb{Z}$.  This means that either:  i) $\mathbf{x}$ is a smallest non-zero vector
in $L \cap P(z=0)$, or ii) $\frac{1}{2} ( \mathbf{x} + \mathbf{y} )$ is a smallest non-zero vector in 
$L \cap P(z=0)$.  This proves the claim, since both of these lattice points lie on axes of rotation in $D_{4}^{+}$.

Thus a smallest vector in $L \cap P(z=0)$ is either $\mathbf{x}$ or we can assume that this smallest vector in 
$L \cap P(z=0)$ is $\mathbf{x}$ after multiplying by a scalar and rotating $45$ degrees about the $z$-axis.  (The latter
rotation normalizes $H$.)  It now follows directly that $(L, H) \sim ( \widehat{L}, H)$, where
$\langle \mathbf{x}, \mathbf{y} \rangle$ is a full subgroup of $\widehat{L}$.  One then easily produces $L'$ in which
$\mathbf{z}$ is full as well (as in previous cases), and $(\widehat{L}, H) \sim (L', H)$.

\item Suppose $H^{+} = D_{2}^{+}$.  Since $D_{2}^{+}$ contains rotations about each of the coordinate axes, Lemma
\ref{bigimp}(1) implies that there
are vectors $\alpha \mathbf{x}$, $\beta \mathbf{y}$, and $\gamma \mathbf{z}$ ($\alpha, \beta, \gamma \neq 0$) such
that each generates a full subgroup.  We can scale each vector independently to arrive at the desired $L'$, and the suggested
matrix commutes with $H$, so $(L, H) \sim (L', H)$.

If $H^{+} = A_{4}^{+}$ or $S_{4}^{+}$, then $D_{2}^{+} \leq H^{+}$ and there are again vectors
$\alpha \mathbf{x}$, $\beta \mathbf{y}$, and $\gamma \mathbf{z}$  ($\alpha, \beta, \gamma > 0$), each
generating a full subgroup of $L$.  This time $\alpha = \beta = \gamma$ since $H^{+}$ permutes the coordinate axes.  After multiplying
by $\alpha^{-1}$, we get the desired $L'$.

\item Suppose $H^{+} = C_{3}^{+}$ or $C_{6}^{+}$.  Let $v$ be a smallest non-zero vector in $L \cap P(x+y+z=0)$.  Such
a vector exists by Lemma \ref{bigimp}(2), since $P(x+y+z=0)$ is perpendicular to the axis of rotation $\ell (x=y=z)$.  
After applying an appropriate
rotation about $\ell (x=y=z)$ and multiplying in the plane $P(x+y+z=0)$ by an appropriate scalar, 
we can assume that $v = \mathbf{x} - \mathbf{y} \in L$.
(All of the suggested matrices commute with $H$.) 
It follows easily that $- \mathbf{y} + \mathbf{z} \in L$, as well, since $H$ contains a rotation through
$120$ degrees about the line $\ell (x=y=z)$.  Since $\mathbf{x} - \mathbf{y}$ has the minimal 
possible norm of
all non-zero vectors in $L \cap P(x+y+z=0)$, $ \langle \mathbf{x} - \mathbf{y} , - \mathbf{y} + \mathbf{z} \rangle$ must 
be full. (Here the lattice points $\langle \mbf{x} - \mbf{y}, -\mbf{y} + \mbf{z} \rangle$ describe a grid
in the plane $P(x+y+z=0)$ made up of equilateral triangles. If $\langle \mbf{x} - \mbf{y}, -\mbf{y} + \mbf{z} \rangle
\subsetneq L \cap P(x+y+z=0)$, then we could choose a point $v' \in L \cap P(x+y+z=0)$ outside of the grid,
and then the difference between $v'$ and the nearest member of $\langle \mbf{x} - \mbf{y}, -\mbf{y} + \mbf{z}
\rangle$ would violate the minimality of $|| \mbf{x} - \mbf{y} ||$.)

Since $L \cap \ell(x=y=z) \neq 0$ by Lemma \ref{bigimp}(1), there is $\alpha \neq 0$ such that 
$\langle \alpha (\mathbf{x} + \mathbf{y} + \mathbf{z}) \rangle$ is a full subgroup of $L'$.  We
scale this vector to obtain the desired conclusion.  This scaling can be done while leaving the perpendicular
plane $P(x+y+z=0)$ fixed.

Now we suppose that $H^{+} = D_{6}^{+}$.  We claim that a non-zero vector of minimal norm in $L \cap P(x+y+z=0)$
must lie on one of the axes $\ell \subseteq P(x+y+z=0)$ of rotation for $D_{6}^{+}$.
(There are $6$ in all, and each makes a $30$-degree angle
with the axes closest to it.)  One of the axes is the line $\ell (x+y=0=z)$, and it follows from
Lemma \ref{bigimp}(1) that
we can assume
that $\mathbf{x} - \mathbf{y} \in L$ after scaling (if necessary), and that $\mathbf{x} - \mathbf{y}$ generates a full
subgroup of $L$.  It then follows quickly that $\langle -\mathbf{y} + \mathbf{z} \rangle$ is also a full subgroup
of $L$, since there is an element of $D_{6}^{+}$ that carries the subspace 
$\langle \mbf{x} - \mbf{y} \rangle$ to $\langle -\mbf{y} + \mbf{z} \rangle$.  

Let $v$ be a smallest non-zero vector in $L \cap P(x+y+z=0)$.  Suppose $v = \alpha \mathbf{x} + \beta \mathbf{y}
+ (-\alpha - \beta) \mathbf{z}$.  We can assume that $\alpha, \beta \geq 0$. (Indeed, either two or more of the numbers
$\alpha$, $\beta$, $-\alpha - \beta$ are nonnegative, or two or more are nonpositive. If two or more are nonnegative, then we
can apply a suitable cyclic permutation matrix to arrange that the first two entries are nonnegative, yielding the desired result.
It two or more entries of $v$ are nonpositive, then we apply a suitable cyclic permutation matrix to arrange that the first two entries
are nonpositive, and then we apply the antipodal map. All of the matrices in question lie in $\langle D_{6}^{+}, (-1) \rangle$.)
\begin{eqnarray*}
\left( \begin{smallmatrix} 0 & -1 & 0 \\ -1 & 0 & 0 \\ 0 & 0 & -1 \end{smallmatrix} \right)  
\left( \begin{smallmatrix} \alpha \\ \beta \\ -\alpha - \beta \end{smallmatrix} \right) & = &
\left( \begin{smallmatrix} -\beta \\ -\alpha \\ \alpha + \beta \end{smallmatrix} \right). 
\end{eqnarray*} 
Adding the two column vectors in the above expression, we arrive at $(\alpha - \beta)(\mathbf{x} - \mathbf{y})$,
which is in $L$ since each of the above column vectors is in $L$.
It follows that $\alpha - \beta \in \mathbb{Z}$, since $\langle \mathbf{x} - \mathbf{y} \rangle$ is a full subgroup
of $L$.   
$$ \left( \begin{smallmatrix} -1 & 0 & 0 \\ 0 & 0 & -1 \\ 0 & -1 & 0  \end{smallmatrix} \right)  
\left( \begin{smallmatrix} \alpha \\ \beta \\ -\alpha - \beta \end{smallmatrix} \right) = 
\left( \begin{smallmatrix} -\alpha \\ \alpha + \beta \\ -\beta \end{smallmatrix} \right).$$
Adding the two column vectors in the above expression, we arrive at $(-\alpha - 2\beta)(-\mathbf{y} + \mathbf{z})$,
which is in $L$ since each of the above column vectors is in $L$.
It follows that $-\alpha - 2\beta \in \mathbb{Z}$, since $\langle -\mathbf{y} + \mathbf{z} \rangle$ is a full subgroup
of $L$.  This now implies that $3\alpha, 3\beta \in \mathbb{Z}$. Let's suppose that $\alpha = m/3$ and $\beta = n/3$, where
$m$ and $n$ are non-negative integers.  The minimality of the norm of $v$ in $L \cap P(x+y+z=0)$ implies that
$$ m^{2} + mn + n^{2} \leq 9,$$
and $m$ and $n$ are congruent modulo $3$ by the condition $\alpha - \beta \in \mathbb{Z}$.  
It is routine to check that either:  i) $m$ and $n$ are both divisible by $3$ (and so $\alpha$, $\beta$ are integers, one
$0$ and the other $1$),
or ii) $m=n=1$.  In the first case, it is clear that $v$ lies on one of the axes 
$\ell(y+z=0=x)$, or $\ell(x+z=0=y)$, and these are axes for rotations
in $D_{6}^{+}$.  In the second case, we have 
$v = \alpha (\mathbf{x} + \mathbf{y} -2\mathbf{z})$.  Since $v$ makes an angle of $30$ degrees with
the vector $\mathbf{x} - \mathbf{z}$ (which lies on an axis of rotation), it follows that $v$ itself lies on an axis
of rotation.  This proves the claim.

We summarize.  We can assume that a smallest non-zero vector in $L \cap P(x+y+z=0)$ is $\mathbf{x} - \mathbf{y}$ after:
i) scaling in the plane $P(x+y+z=0)$, 
if a smallest vector in $L \cap P(x+y+z=0)$ lies on $\ell (x+y = 0; z=0)$ (or on one of its orbits under the
action of $D_{6}^{+}$), or  ii) rotating by $30$ degrees about the axis $\ell (x=y=z)$ otherwise. 

The remainder of the argument is easy, and follows the lines of the cases $H = C_{3}^{+}$ and $H = C_{6}^{+}$.

\item By Lemma \ref{bigimp}(1), there is some $v \in L \cap \ell(x+y=0=z)$, where $v \neq 0$. 
After multiplying by a suitable scalar matrix, we can assume that $\langle \mbf{x} - \mbf{y} \rangle$ is a full
subgroup of $L$. Since there are elements in $D^{+}_{3}$ that move the subspace 
$\langle \mbf{x} - \mbf{y} \rangle$ to the subspace $\langle -\mbf{y} + \mbf{z} \rangle$, it follows that
$\langle -\mbf{y} + \mbf{z} \rangle$ is also full in $L$. Thus, $\langle \mbf{v}_{i} \rangle$ is full in $L$, for $i=2,3$.
We can rescale along the line $\ell(x=y=z)$ while holding the plane $P(x+y+z=0)$ fixed; the suggested
transformation $\lambda$ commutes with $D_{3}^{+}$ for any scaling factor. There is some non-zero
$v \in \ell(x=y=z) \cap L$. If we assume that $v$ is chosen so that  $||v||$ is minimal, then scaling along
$\ell(x=y=z)$ by a factor of $1/||v||$ yields the desired lattice $L'$.   
\end{enumerate}
\end{proof}

\subsection{Description of Possible Lattices $L$} 
 In this subsection, we will show that each pair $(L,H)$ is equivalent to a pair
$(L',H)$, where $L'$ is one of seven lattices.

\begin{lemma} \label{2X}
Suppose that $H$ is a standard point group which stabilizes 
the lattice $L$, suppose $L_{\mathcal{C}} \leq L$, and each of 
$\langle \mathbf{x} \rangle$, $\langle \mathbf{y} \rangle$, $\langle \mathbf{z} \rangle$
is a full subgroup of $L$.  
\begin{enumerate}
\item If $H$ contains the $180$-degree rotation 
through the $x$-, $y$-, or $z$-axis, then, for any $v \in L$,
$v = \alpha \mathbf{x} + \beta \mathbf{y} + \gamma \mathbf{z}$, we have
that $2\alpha$, $2\beta$, or $2\gamma \in \mathbb{Z}$ (respectively).
\item If $H$ contains the reflection across the plane $P(z=0)$, $\langle \mathbf{x} , \mathbf{y} \rangle$
is a full subgroup of $L$, and $v = \alpha \mathbf{x} + \beta \mathbf{y} + \gamma \mathbf{z} \in L$, then
$2\alpha, 2\beta, 2\gamma \in \mathbb{Z}$.
\end{enumerate}
\end{lemma}

\begin{proof}
\begin{enumerate}
\item Let $v \in L$ have the form indicated in the lemma.  
We suppose that $H$ contains the rotation $\lambda$ through $180$ degrees about the $x$-axis.  We have
$$ \lambda \cdot v = \left( \begin{smallmatrix} 1 & 0 & 0 \\ 0 & -1 & 0 \\ 0 & 0 & -1 \end{smallmatrix} \right) 
\left( \begin{smallmatrix} \alpha \\ \beta \\ \gamma \end{smallmatrix} \right) = 
\left( \begin{smallmatrix} \alpha \\ -\beta \\ -\gamma \end{smallmatrix} \right) \in L.$$
We conclude that $2\alpha \mathbf{x} \in L$, so $2\alpha \in \mathbb{Z}$ by the fullness
of the subgroup $\langle \mathbf{x} \rangle$. 
\item Let $v \in L$ once again have the form indicated in the lemma.  We get
$$ \lambda \cdot v = \left( \begin{smallmatrix} 1 & 0 & 0 \\ 0 & 1 & 0 \\ 0 & 0 & -1 \end{smallmatrix} \right) 
\left( \begin{smallmatrix} \alpha \\ \beta \\ \gamma \end{smallmatrix} \right) = 
\left( \begin{smallmatrix} \alpha \\ \beta \\ -\gamma \end{smallmatrix} \right) \in L.$$
It follows directly that both $2\alpha \mathbf{x}  + 2\beta \mathbf{y}$ and $2\gamma \mathbf{z}$ are in $L$.
But now it follows that $2\alpha , 2\beta \in \mathbb{Z}$ and $2\gamma \in \mathbb{Z}$, by the fullness of the
subgroups $\langle \mathbf{x} , \mathbf{y} \rangle$ and $\langle \mathbf{z} \rangle$, respectively.
\end{enumerate}
\end{proof}
 
\begin{corollary} \label{LC}
Suppose that the standard point group $H$ contains the involution $(-1)$, and suppose that the 
orientation-preserving subgroup $H^{+}$ of $H$ is one of  
$C_{1}^{+}, C_{2}^{+}, D_{2}^{+}, C_{4}^{+}, D_{4}^{+}, A_{4}^{+}$, or $S_{4}^{+}$.  
Suppose that $L$ is a lattice such that $H \cdot L = L$.
The pair $(L,H)$ is arithmetically equivalent to $(L',H)$, where
$L'$ is equal to one of the following lattices (or the image of one of these under a permutation of coordinate
axes):
$$\left\langle \mathbf{x}, \mathbf{y}, \mathbf{z} \right\rangle;
\left\langle \mathbf{x}, \mathbf{y}, \frac{\mathbf{x} + \mathbf{y} + \mathbf{z}}{2}  \right\rangle;
\left\langle \mathbf{x} , \mathbf{y} , \frac{\mathbf{x} + \mathbf{z}}{2} \right\rangle;
\left\langle \frac{\mathbf{x} + \mathbf{y}}{2}, \frac{\mathbf{x} + \mathbf{z}}{2} , \frac{\mathbf{y} + \mathbf{z}}{2}
\right\rangle.$$
Moreover, in case $H^{+} = C_{2}^{+}$, $C_{4}^{+}$, or $D_{4}^{+}$, 
we can arrange that $L'$ contains $\langle \mathbf{x}, \mathbf{y} \rangle$ as a full subgroup.
In particular, for $H^{+} = C_{2}^{+}$, $C_{4}^{+}$, or $D_{4}^{+}$, $L'$ is not the last of these four lattices.
\end{corollary} 

\begin{proof}
If $H^{+} = C_{1}^{+}$, then $(L,H) \sim (L_{\mathcal{C}}, H)$ by Proposition \ref{full}(1).  
 
We now show that any pair $(L,H)$ with $H^{+} \neq 1$ is equivalent to $(L',H)$, where $2L' \leq L_{\mathcal{C}}$.
If $H^{+} = C_{2}^{+}$, $C_{4}^{+}$,
or $D_{4}^{+}$, then, by Proposition \ref{full}(2), $(L,H) \sim (L',H)$, where $L_{\mathcal{C}} \leq L'$ and 
$\langle \mathbf{x}, \mathbf{y} \rangle$, $\langle \mathbf{z} \rangle \subseteq L'$
are full subgroups.
Since, in each case, $H$ contains the reflection across the plane $P(z=0)$, we get the desired conclusion
from Lemma \ref{2X}(2).  
If $H^{+}$ is any of the remaining groups, then $(L,H) \sim (L',H)$, where $L_{\mathcal{C}} \leq L'$ and each of 
$\langle \mathbf{x} \rangle$, $\langle \mathbf{y} \rangle$, and $\langle \mathbf{z} \rangle$ are full subgroups
of $L'$, by Proposition \ref{full}(3).
In these cases, $H$ (indeed, $H^{+}$) contains
the rotations through $180$ degrees about each of the coordinate axes, and the first part of Lemma \ref{2X} directly implies the 
desired conclusion.

Note that $L' \cap [0,1]^{3}$ generates $L'$ (since $L_{\mathcal{C}} \leq L'$).  The first
part of the corollary (and fullness of the subgroups $\langle \mathbf{x} \rangle$, $\langle \mathbf{y} \rangle$, and
$\langle \mathbf{z} \rangle$) implies $L' \cap [0,1]^{3}$ consists of the corners of the cube $[0,1]^{3}$, and possibly
some subcollection of seven other points:  the center of the cube $\frac{1}{2}(\mathbf{x} + \mathbf{y} + \mathbf{z})$,
and the centers of the two-dimensional faces. 

Now we run through the cases.  Of course, if $L' \cap [0,1]^{3}$ consists only of the set of the corners of $[0,1]^{3}$,
then $L' = \langle \mathbf{x} , \mathbf{y}, \mathbf{z} \rangle$.  If $L' \cap [0,1]^{3}$ contains the center
of the cube, then no other points can be in $L' \cap [0,1]^{3}$ (other than the corners) without violating the fullness
of one of the subgroups $\langle \mathbf{x} \rangle$, $\langle \mathbf{y} \rangle$, or $\langle \mathbf{z} \rangle$.  This case
thus yields the second lattice mentioned in the conclusion.  The only cases left to consider are those in which any
additional lattice points occur in the middle of two-dimensional faces of $[0,1]^{3}$.  It is obvious that these lattice
points must appear in pairs (on opposing pairs of faces), and it is an elementary exercise to show that it is impossible
for exactly $4$ of these centers to be lattice points.  It follows that either $2$ opposing center points are in $L'$,
or all $6$ are in $L'$.  These yield (respectively) the last two lattices mentioned in the conclusion.

The final statement is a consequence of the fact that the $L'$ constructed in this proof 
has $\langle \mathbf{x}, \mathbf{y} \rangle$ as a full
subgroup in the specified cases.

We note finally that all of the lattices in the corollary 
are invariant under the permutation of coordinate axes, with the exception 
of the second-to-last.
\end{proof}
  
\begin{lemma} \label{3X} Suppose that $H$ is a  standard point group which stabilizes the lattice $L$, suppose
$L_{\mathcal{P}} \leq L$, and each of $\langle \mathbf{v}_{1} \rangle$,
$\langle \mbf{v}_{2} \rangle$, $\langle \mbf{v}_{3} \rangle$ is a full subgroup of $L$.
\begin{enumerate}
\item If $H$ contains the $180$-degree rotation through the vector space spanned by $\mbf{v}_{2}$, and
a rotation through $120$ degrees about the line $\ell(x=y=z)$, 
then, for any $\mbf{v} \in L$, $\mbf{v} = \alpha \mbf{v}_{1} + \beta \mbf{v}_{2} + \gamma \mbf{v}_{3}$, we
have $3 \beta$, $3 \gamma \in \mathbb{Z}$ and $\beta - \gamma \in \mathbb{Z}$.
\item If $H$ contains 
a rotation through $120$ degrees about the line $\ell(x=y=z)$ and $\langle \mbf{v}_{2}, \mbf{v}_{3} \rangle$
is a full subgroup of $L$,  
then, for any $\mbf{v} \in L$, $\mbf{v} = \alpha \mbf{v}_{1} + \beta \mbf{v}_{2} + \gamma \mbf{v}_{3}$, we
have $3 \beta$, $3 \gamma \in \mathbb{Z}$ and $\beta - \gamma \in \mathbb{Z}$.
\item If $H$ contains 
a rotation through $120$ degrees about the line $\ell(x=y=z)$ 
then, for any $\mbf{v} \in L$, $\mbf{v} = \alpha \mbf{v}_{1} + \beta \mbf{v}_{2} + \gamma \mbf{v}_{3}$, we
have $3 \alpha \in \mathbb{Z}$. 
\item If $H$ contains both the reflection across the plane $P(x+y+z=0)$ and 
a rotation through $120$ degrees about the line $\ell (x=y=z)$, and $\langle \mbf{v}_{2} , \mbf{v}_{3} \rangle$
is a full subgroup of $L$, then
$L = \langle \mbf{v}_{1} , \mbf{v}_{2}, \mbf{v}_{3} \rangle$.  
\end{enumerate}
\end{lemma}

\begin{proof}
\begin{enumerate}
\item Let $\mbf{v}$ be as in the statement of the lemma.  We first apply the rotation $R_{1}$ through $180$ degrees
about $\ell(x+y=0 ; z=0)$:
$$ \left( \begin{smallmatrix} 0 & -1 & 0 \\ -1 & 0 & 0 \\ 0 & 0 & -1 \end{smallmatrix} \right) 
\left( \begin{smallmatrix} \alpha + \beta \\ \alpha - \beta - \gamma \\ \alpha + \gamma \end{smallmatrix} \right)
= \left( \begin{smallmatrix} -\alpha + \beta + \gamma \\ -\alpha - \beta \\ -\alpha - \gamma \end{smallmatrix} \right) \in L.$$
It follows that $\mbf{v} + R_{1} \mbf{v} = (2 \beta + \gamma) \mbf{v}_{2} \in L$.  Since $\langle \mbf{v}_{2} \rangle$
is a full subgroup of $L$, it follows that $2 \beta + \gamma \in \mathbb{Z}$.  Now apply the rotation $R_{2}$ through $120$
degrees about $\ell(x=y=z)$:
$$ \left( \begin{smallmatrix} 0 & 1 & 0 \\  0 & 0 & 1 \\ 1 & 0 & 0  \end{smallmatrix} \right) 
\left( \begin{smallmatrix} \alpha + \beta \\ \alpha - \beta - \gamma \\ \alpha + \gamma \end{smallmatrix} \right)
= \left( \begin{smallmatrix} \alpha - \beta - \gamma \\  \alpha + \gamma \\ \alpha + \beta \end{smallmatrix} \right) \in L.$$
It follows that 
$$ \mbf{v} - R_{2}\mbf{v} = \left( 2 \beta + \gamma \right) \mbf{v}_{2} + \left( \gamma - \beta \right) \mbf{v}_{3} \in L.$$
This means that $( \gamma - \beta ) \mbf{v}_{3} \in L$, since $(2\beta + \gamma) \mbf{v}_{2} \in L$ by the previous calculation.
Since $\langle \mbf{v}_{3} \rangle$ is full in $L$, it must be that $\gamma - \beta \in \mathbb{Z}$.  The desired conclusions
follow easily.
\item This is easier than the proof of (1).  We need consider only 
the last displayed equation, and the desired conclusions follow
from the fullness of $\langle \mbf{v}_{2}, \mbf{v}_{3} \rangle$ in $L$.  
\item We note that the rotation $R_{2}$ (from (1)) simply permutes the coordinates of any vector 
$$\mbf{v} = \alpha \mbf{v}_{1} + \beta \mbf{v}_{2} + \gamma \mbf{v}_{3} = ( \alpha + \beta ) \mbf{x} + 
( \alpha - \beta - \gamma ) \mbf{y} + ( \alpha + \gamma ) \mbf{z}$$ 
cyclically.  It is not difficult to see that
$$ \left( 1 + R_{2} + R_{2}^{2} \right) \cdot \mbf{v} = 3 \alpha ( \mbf{x} + \mbf{y} + \mbf{z} ) = 3 \alpha \mbf{v}_{1}.$$
Since $\langle \mbf{v}_{1} \rangle$ is full in $L$, we have $3 \alpha \in \mathbb{Z}$.
\item Let $R_{3}$ denote the reflection in question.  Let 
$\mbf{v} = \alpha \mbf{v}_{1} + \beta \mbf{v}_{2} + \gamma \mbf{v}_{3} \in L$
be arbitrary.  It follows that
$$ (1 - R_{3} ) \mbf{v} = 
\left( \alpha \mbf{v}_{1} + \beta \mbf{v}_{2} + \gamma \mbf{v}_{3} \right)  
- \left( - \alpha \mbf{v}_{1} + \beta \mbf{v}_{2} + \gamma \mbf{v}_{3} \right) = 2 \alpha \mbf{v}_{1} \in L.$$
The fullness of $\langle \mbf{v}_{1} \rangle$ in $L$ implies that $2\alpha \in \mathbb{Z}$.  Since $2\alpha \in \mathbb{Z}$ 
and $3\alpha \in \mathbb{Z}$ (by (3)), $\alpha \in \mathbb{Z}$.  It follows that
$\beta \mbf{v}_{2} + \gamma \mbf{v}_{3} \in L$.  Fullness of $\langle \mbf{v}_{2}, \mbf{v}_{3} \rangle$ in $L$ implies
that $\beta, \gamma \in \mathbb{Z}$.
\end{enumerate}
\end{proof}

\begin{corollary} \label{LP}
Suppose that $H$ contains the involution $(-1)$, and the orientation-preserving subgroup
$H^{+}$ of $H$ is one of  
$C_{3}^{+}, D_{3}^{+}, C_{6}^{+}$, or $D_{6}^{+}$.
Let $L$ be a lattice such that $H \cdot L = L$.  
The pair $(L, H)$ is equivalent to a pair $(L' , H)$ where
$L'$ is one of the following:
$$ \left\langle \mbf{v}_{1}, \mbf{v}_{2} , \mbf{v}_{3} \right\rangle,     
\left\langle \frac{1}{3} \left( \mbf{v}_{1} + \mbf{v}_{2} + \mbf{v}_{3} \right), \mbf{v}_{2} , \mbf{v}_{3} \right\rangle,     
\left\langle \mbf{v}_{1}, \frac{1}{3}\left( \mbf{v}_{2} + \mbf{v}_{3} \right) , \mbf{v}_{3} \right\rangle.$$
Indeed, we can assume $L'$ is the first lattice if $H^{+} = C_{6}^{+}$ or $D_{6}^{+}$, 
or that it is one of the first two lattices if $H^{+} = C_{3}^{+}$.  
\end{corollary}
     
\begin{proof}
Suppose first that $H^{+} = C_{6}^{+}$ or $D_{6}^{+}$.  We can apply Proposition \ref{full}(4), and conclude
that $(L,H) \sim (L',H)$, where $L_{\mathcal{P}} \leq L'$ and each of the subgroups
$\langle \mathbf{v}_{2}, \mathbf{v}_{3} \rangle$, $\langle \mathbf{v}_{1} \rangle$ is full in $L'$.
Since $H$ contains both the reflection across
the plane $P(x+y+z=0)$ and the rotation through $120$ degrees about $\ell(x=y=z)$, Lemma \ref{3X}(4) shows
that $L' = \langle \mbf{v}_{1}, \mbf{v}_{2} , \mbf{v}_{3} \rangle$.

Next suppose $H^{+} = C_{3}^{+}$.  We can again apply Proposition \ref{full}(4), and conclude
that $(L,H) \sim (L',H)$, where $L_{\mathcal{P}} \leq L'$ and each of the subgroups
$\langle \mathbf{v}_{2}, \mathbf{v}_{3} \rangle$, $\langle \mathbf{v}_{1} \rangle$ is full in $L'$.
One possibility is that $L' = \langle \mbf{v}_{1}, \mbf{v}_{2}, \mbf{v}_{3} \rangle$;
we suppose otherwise.  Let us consider a typical $\mbf{v} = \alpha \mbf{v}_{1} + \beta \mbf{v}_{2} + \gamma \mbf{v}_{3} \in L'$.  By
Lemma \ref{3X}((2) and (3)) and Proposition \ref{full}(4), we have that $3\alpha, 3\beta, 3\gamma, \gamma - \beta \in \mathbb{Z}$.  Since 
$L' \neq \langle \mbf{v}_{1} , \mbf{v}_{2} , \mbf{v}_{3} \rangle$, we can find $\mbf{v}$ so that not all of
$\alpha , \beta, \gamma$ are integers.  Indeed, fullness of $\langle \mbf{v}_{1} \rangle$ and 
$\langle \mbf{v}_{2}, \mbf{v}_{3} \rangle$ in $L'$ (and the inclusions $3\alpha, 3\beta, 3\gamma, \gamma - \beta \in \mathbb{Z}$) 
imply that none of $\alpha, \beta, \gamma$ are integers if one is not.
It is now routine to show that one of the vectors   
$$ \frac{1}{3} \left( \mbf{v}_{1} - \mbf{v}_{2} - \mbf{v}_{3} \right) \quad \mathrm{or} \quad 
\frac{1}{3} \left( \mbf{v}_{1} + \mbf{v}_{2} + \mbf{v}_{3} \right)$$
is in $L'$.  After applying the $180$ degree rotation about $\ell(x=y=z)$, we can assume that $\frac{1}{3} \left( \mbf{v}_{1} + \mbf{v}_{2} + \mbf{v}_{3} \right)
\in L'$, since the rotation in question preserves $H$.  

We claim that $L' = \langle \frac{1}{3}\left( \mbf{v}_{1} + \mbf{v}_{2} + \mbf{v}_{3} \right) , \mbf{v}_{2} , \mbf{v}_{3} \rangle$.
Suppose $\mbf{v} = \alpha \mbf{v}_{1} + \beta \mbf{v}_{2} + \gamma \mbf{v}_{3} \in L'$.  We can conclude, as before,
that $3\alpha , 3\beta , 3\gamma , \gamma - \beta \in \mathbb{Z}$.
$$ \mbf{v} = \alpha \mbf{v}_{1} + \beta \mbf{v}_{2} + \gamma \mbf{v}_{3}
= (\alpha - \beta) \mbf{v}_{1} + (3\beta) \frac{1}{3} \left( \mbf{v}_{1} + \mbf{v}_{2} + \mbf{v}_{3} \right) + 
( \gamma - \beta) \mbf{v}_{3}.$$
Since $\mbf{v}$ and every other term of the right-most sum is in $L'$, so is $( \alpha - \beta ) \mbf{v}_{1}$.  It follows
that $\alpha - \beta \in \mathbb{Z}$.  The equation above displays $\mbf{v}$ as an integral combination of elements
in $\langle \frac{1}{3} \left( \mbf{v}_{1} + \mbf{v}_{2} + \mbf{v}_{3} \right) , \mbf{v}_{2} , \mbf{v}_{3} \rangle$.
This completes the proof in the case of $H^{+} = C^{+}_{3}$.

Suppose that $H^{+} = D_{3}^{+}$.  We apply Proposition \ref{full}(5):  $(L,H) \sim (L',H)$, where
$L_{\mathcal{P}} \leq L'$ and each of the subgroups $\langle \mathbf{v}_{1} \rangle$, $\langle \mathbf{v}_{2} \rangle$
and $\langle \mathbf{v}_{3} \rangle$ is full in $L'$.  The lattice $L'$ could be  
$\langle \mbf{v}_{1} , \mbf{v}_{2} , \mbf{v}_{3} \rangle$
or $\langle \frac{1}{3} \left( \mbf{v}_{1} + \mbf{v}_{2} + \mbf{v}_{3} \right), \mbf{v}_{2} , \mbf{v}_{3} \rangle$.  Indeed,
an argument essentially identical to the one for the case $H^{+} = C_{3}^{+}$ shows that these are the only possibilities
if $\langle \mbf{v}_{2} , \mbf{v}_{3} \rangle$ is full in $L'$.  Suppose that $\langle \mbf{v}_{2} , \mbf{v}_{3} \rangle$ is
not full in $L'$.  Let $\beta \mbf{v}_{2} + \gamma \mbf{v}_{3} \in L'$, where, by Lemma \ref{3X}(1),
$3\beta , 3\gamma, \gamma - \beta \in \mathbb{Z}$.  We assume that one of $\beta, \gamma$ (equivalently, both of $\beta, \gamma$)
are not integers.  It follows quickly that $\frac{1}{3} ( \mbf{v}_{2} + \mbf{v}_{3}) \in L'$.

We claim that $L' = \langle \mbf{v}_{1} , \frac{1}{3} ( \mbf{v}_{2} + \mbf{v}_{3} ), \mbf{v}_{3} \rangle$.  
Let $\mbf{v} = \alpha \mbf{v}_{1} + \beta \mbf{v}_{2} + \gamma \mbf{v}_{3} \in L'$.
By (1) and (3) of
Lemma \ref{3X}, we  have $3 \alpha , 3\beta, 3\gamma, \gamma - \beta \in \mathbb{Z}$.
$$ \mbf{v} = \alpha \mbf{v}_{1} + (3 \beta) \left( \frac{1}{3} \mbf{v}_{2} + \frac{1}{3} \mbf{v}_{3} \right)
+ \left( \gamma - \beta \right) \mbf{v}_{3} \in L'.$$
It follows that $\alpha \mbf{v}_{1} \in L'$, which implies that $\alpha \in \mathbb{Z}$, by fullness of $\langle \mbf{v}_{1} \rangle$
in $L'$.  This completes the proof.
\end{proof} 

\subsection{Classification of $(L, H)$ up to arithmetic equivalence, where $(-1) \in H$}

We now sort the pairs $(L,H)$ up to arithmetic equivalence, assuming that $(-1) \in H$. The following theorem lists $24$ such pairs as distinct
possibilities; the theorem leaves open the possibility that some of these pairs will be equivalent. We will see in Theorem \ref{noredundancy} 
that indeed all of them are different.

\begin{theorem} \label{apair(-1)}
Let $L \leq \mathbb{R}^{3}$ be a lattice, and let $H \leq O(3)$ be a standard point group acting on $L$, such that
$(-1) \in H$.  The pair
$(L, H)$ is equivalent to one of the $24$ on the following list.
\begin{enumerate}
\item If $H^{+} = A_{4}^{+},$ or $S_{4}^{+}$, then $(L, H) \sim (L', H)$, where
$$ L' = L_{\mathcal{C}}, 
\left\langle \frac{1}{2} \left( \mbf{x} + \mbf{y} + \mbf{z} \right),
\mbf{y} , \mbf{z} \right\rangle, \, \mathrm{or} \, \left\langle \frac{1}{2} \left( \mbf{x} + \mbf{y} \right),
\frac{1}{2} \left( \mbf{x} + \mbf{z} \right) , \frac{1}{2} \left( \mbf{y} + \mbf{z} \right) \right\rangle.$$
(There are six possibilities in all.)
\item If $H^{+} = D_{2}^{+}$, then $(L, H) \sim (L', H)$, where $L'$ is any of the lattices mentioned in Corollary \ref{LC}.
(There are four possibilities.)
\item If $H^{+} = C_{2}^{+}$, $C_{4}^{+}$, or $D_{4}^{+}$, then $(L, H) \sim (L', H)$, where
$$ L' = \left\langle \mbf{x} , \mbf{y} , \mbf{z} \right\rangle \, \mathrm{or} \, 
\left\langle \mbf{x}, \mbf{y}, \frac{1}{2} \left( \mbf{x} +
\mbf{y} + \mbf{z} \right) \right\rangle.$$
(There are six possibilities.)
\item If $H^{+} = C_{1}^{+}$, then $(L, H) \sim (L', H)$, where $L' = L_{\mathcal{C}}$.
(There is only one possibility.)
\item If $H^{+} = C_{6}^{+}$, or $D_{6}^{+}$, then $(L, H) \sim (L', H)$, where
$L' = L_{\mathcal{P}}$.
(There are two possibilities.)
\item If $H^{+} = C_{3}^{+}$, then $(L, H) \sim (L', H)$, where
$$ L' = L_{\mathcal{P}} \, \mathrm{or} \, \left\langle \frac{1}{3}  \left( \mbf{v}_{1} + \mbf{v}_{2} + 
\mbf{v}_{3}\right),
\mbf{v}_{2} , \mbf{v}_{3} \right\rangle.$$
(There are two possibilities.)
\item If $H^{+} = D_{3}^{+}$, then $(L, H) \sim (L', H)$, where $L'$ can be any of the lattices listed in Corollary \ref{LP}.
(There are three possibilities.)
\end{enumerate}
\end{theorem}
 
\begin{proof}
We note first that all of the pairs $(L',H)$ mentioned as possibilities above truly occur
(i.e., $H \cdot L' =L'$). 
\begin{enumerate}
\item Let $H^{+} = A_{4}^{+}$ or $S_{4}^{+}$.   Corollary \ref{LC} gives us four possibilities.  We rule out the
third possibility, $\langle \mbf{x}, \mbf{y}, \frac{1}{2}\left( \mbf{x} + \mbf{z} \right) \rangle$ (and its images
under permutations of the coordinate axes), since it fails to be invariant under the action of $H$.
\item Let $H^{+} = D_{2}^{+}$.  In this case, all four lattices from Corollary \ref{LC} are possible.  We note
that $H$ is normalized by any permutation of the coordinate axes, so the lattices
$$ \langle \mbf{x}, \mbf{y}, \frac{1}{2}\left( \mbf{x} + \mbf{z} \right) \rangle, 
\langle \mbf{x}, \frac{1}{2} \left( \mbf{x} + \mbf{y} \right), \mbf{z} \rangle, \, \mathrm{and} \,
\langle \mbf{x}, \mbf{y}, \frac{1}{2}\left( \mbf{y} + \mbf{z} \right) \rangle$$
all lead to arithmetically equivalent pairs.
\item Let $H^{+} = C_{4}^{+}$ or $D_{4}^{+}$.  Corollary \ref{LC} gives us five possibilities.
The third lattice mentioned in Corollary \ref{LC} is really a collection of three distinct lattices.  
We can dispose with these three lattices, either because:  i) they are not invariant under the action of $H$
(and so cannot occur as part of the pair $(L,H)$), or ii) they do not contain $\langle \mathbf{x}, \mathbf{y} \rangle$
as a full subgroup (which we can arrange by Corollary \ref{LC}).  It follows that the lattices in the statement 
of the theorem are the only possibilities up to arithmetic equivalence.  

Now suppose $H^{+} = C_{2}^{+}$.  We conclude, exactly as in the previous paragraph, that there are
five possibilities for $L'$, by Corollary \ref{LC}.  We can rule out 
$$\left\langle \mathbf{x}, \mathbf{z}, \frac{\mathbf{x} + \mathbf{y}}{2} \right\rangle$$
since we can assume that $\langle \mathbf{x}, \mathbf{y} \rangle$ is full in $L$.
Now we note that the lattices
$$ \left\langle \mathbf{x}, \mathbf{y}, \frac{\mathbf{x} + \mathbf{z}}{2} \right\rangle,
\left\langle \mathbf{x}, \mathbf{y}, \frac{\mathbf{y} + \mathbf{z}}{2} \right\rangle$$
are both equivalent to the second lattice from Corollary \ref{LC}, 
by the
matrices
$$ \lambda_{1} = \left( \begin{smallmatrix} 1 & 0 & 0 \\ -1 & 1 & 0 \\ 0 & 0 & 1 \end{smallmatrix} \right)
\quad \mathrm{and} \quad 
\lambda_{2} = \left( \begin{smallmatrix} 1 & -1 & 0 \\ 0 & 1 & 0 \\ 0 & 0 & 1 \end{smallmatrix} \right),$$
respectively. 
\item is clear.
\item follows immediately from Corollary \ref{LP}, as do (6) and (7).
\end{enumerate}
\end{proof} 
The following lemma uses only basic group theory
and linear algebra. The proof is left to the reader.

\begin{lemma} \label{silly}
Let $\lambda \in GL_{3}(\mathbb{R})$ normalize the point group $H$.
\begin{enumerate}
\item If $\ell \subseteq \mathbb{R}^{3}$ is a pole such that the stabilizer
group $H^{+}_{\ell} = \{ h \in H^{+} \mid h_{\mid \ell} = id_{\ell} \}$ has order $n$,
then $\lambda \ell$ is also a pole of $H^{+}$, and $H^{+}_{\lambda \cdot \ell}$ has order $n$.
\item If two vectors $v_{1}, v_{2} \in \mathcal{L}^{+} \cup \mathcal{L}^{-}$ are in
the same orbit under the action of $H$, then $|| \lambda v_{1} || = || \lambda v_{2} ||$, and
$\lambda v_{1}$, $\lambda v_{2}$ are in the same orbit under the action of $H$.
\item If $V \subseteq \mathbb{R}^{3}$ is an $H$-invariant vector subspace,
then $\lambda V$ is also an $H$-invariant vector subspace.
\end{enumerate}
\qed
\end{lemma}

\begin{theorem} \label{noredundancy}
No two of the $24$ arithmetic classes of pairs $(L, H)$ from Theorem \ref{apair(-1)} are the same.
\end{theorem}

\begin{proof}
We first note that if two pairs $(L_{1}, H_{1})$, $(L_{2}, H_{2})$ are arithmetically equivalent, then $H_1$ and $H_{2}$ 
are isomorphic groups (in fact, conjugate in $GL_{3}(\mathbb{R})$). It follows that if two pairs from the list
in Theorem \ref{apair(-1)} are the same, then their point groups must be the same (since different point groups of the
given type have different isomorphism types). Thus, we can assume that $H_{1} = H_{2}$ in the following arguments. 
\begin{enumerate}
\item Suppose $H^{+} = A_{4}^{+}$ or $S_{4}^{+}$.  Let $(L_{1}, H)$, $(L_{2}, H)$ be distinct pairs, where
$L_{1}$ and $L_{2}$ are as described in Theorem \ref{apair(-1)}(1).  Suppose $(L_{1}, H) \sim (L_{2}, H)$.  
This means that there is $\lambda$ such that
$\lambda L_{1} = L_{2}$ and $\lambda H \lambda^{-1} = H$.  By Lemma \ref{silly}(1), $\lambda$ must permute
the coordinate axes.  Since $\langle \mbf{x} \rangle$, $\langle \mbf{y} \rangle$, $\langle \mbf{z} \rangle$ are
full subgroups of both $L_{1}$ and $L_{2}$, it must be that $\lambda$ is a signed permutation matrix.  Such
a matrix fixes each of the lattices from Theorem \ref{apair(-1)}(1), which is a contradiction.
\item follows the exact pattern of (1).
\item Suppose $H^{+} = C_{2}^{+}$, $C_{4}^{+}$, or $D_{4}^{+}$.  Let $(L_{1}, H)$, $(L_{2}, H)$ be distinct pairs,
where $L_{1}$ and $L_{2}$ are chosen from the possibilities in Theorem \ref{apair(-1)}(3). 
We can assume, without loss of generality, that $L_{1} = \langle \mbf{x}, \mbf{y}, \mbf{z} \rangle$ and
$L_{2} = \langle \mbf{x}, \mbf{y}, \frac{1}{2}(\mbf{x} + \mbf{y} + \mbf{z} ) \rangle$.  Suppose 
$(L_{1}, H) \sim (L_{2}, H)$; let $\lambda \in GL_{3}(\mathbb{R})$ satisfy $\lambda L_{1} = L_{2}$
and $\lambda H \lambda^{-1} = H$.  We claim that $\lambda$ has the form
$$\lambda = \left( \begin{smallmatrix} \ast & \ast & 0 \\ \ast & \ast & 0 \\ 0 & 0 & \ast \end{smallmatrix} \right).$$
Moreover, if we assume that $\lambda$ has the latter form, then $\lambda$ must have integral entries by the fullness of $\langle \mbf{x}, \mbf{y} \rangle$ and 
$\langle \mbf{z} \rangle$ in both $L_{1}$ and $L_{2}$.  It will then follow that $\lambda L_{1}$ is a sublattice of
$L_{\mathcal{C}}$, a contradiction.

We turn to a proof of the claim. First, assume that $H^{+} = C_{2}^{+}$ or $ C_{4}^{+}$. Since $\lambda$ normalizes $H$, it must be that $\lambda$ actually
commutes with the generator of $C_{2}^{+}$, which is the unique element of $H$ having positive determinant and order $2$. It now follows from a straightforward calculation
that $\lambda$ has the required block form. If $H^{+} = D_{4}^{+}$,  then we appeal to parts (3) and (1) (respectively) 
of Lemma \ref{silly}: since the $xy$-plane is the unique $2$-dimensional
$H$-invariant subspace, it must be preserved by $\lambda$; since the $z$-axis is the unique $1$-dimensional subspace to be an axis of rotation for an element of order $4$ in
$H^{+}$, it must be preserved. It follows directly that $\lambda$ has the required form in this case as well. This proves the claim.
\item is trivial.
\item is also trivial.
\item Suppose 
$(L_{\mathcal{P}}, \langle C_{3}^{+}, (-1) \rangle) \sim ( \frac{1}{3}(\mbf{v}_{1} + \mbf{v}_{2} + \mbf{v}_{3}), 
\mbf{v}_{2}, \mbf{v}_{3} \rangle, \langle C_{3}^{+}, (-1) \rangle )$.  
Let $\lambda \in GL_{3}(\mathbb{R})$
satisfy
$\lambda L_{\mathcal{P}} = \langle \frac{1}{3}(\mbf{v}_{1} + \mbf{v}_{2} + \mbf{v}_{3}), 
\mbf{v}_{2}, \mbf{v}_{3} \rangle$   
and $\lambda H \lambda^{-1} = H$.
We claim that $\lambda$ has the form
$$ \left( \begin{smallmatrix} \ast & 0 & 0 \\ 0 & \ast & \ast \\ 0 & \ast & \ast \end{smallmatrix} \right)$$
as a matrix over the ordered basis $(\mbf{v}_{1}, \mbf{v}_{2}, \mbf{v}_{3})$. Indeed, if this is the case, then all of the entries must
be integers by the fullness of $\langle \mbf{v}_{1} \rangle$ and $\langle \mbf{v}_{2}, \mbf{v}_{3} \rangle$ in both
lattices.
It will then follow that $\lambda L_{\mathcal{P}} \leq L_{\mathcal{C}}$, a contradiction. 

We prove the claim. Note that $\ell( x= y=z)$ is the unique $1$-dimensional $H$-invariant subspace, and so must be invariant under
$\lambda$ by Lemma \ref{silly}(3). Similarly, the plane $P(x+y+z=0)$ is the unique $2$-dimensional $H$-invariant subspace, so it
is also invariant under $\lambda$. The claim now follows directly, since $\{ \mbf{v}_{2}, \mbf{v}_{3} \}$ spans $P(x+y+z=0)$
and $\{ \mbf{v}_{1} \}$ spans $\ell(x=y=z)$.

\item Let $H = \langle D_{3}^{+}, (-1) \rangle$, $L_{1} = L_{\mathcal{P}}$, 
$L_{2} = \langle \frac{1}{3}(\mbf{v}_{1} + \mbf{v}_{2} + \mbf{v}_{3} ), \mbf{v}_{2}, \mbf{v}_{3} \rangle$, and
$L_{3} = \langle \mbf{v}_{1}, \frac{1}{3}(\mbf{v}_{2} + \mbf{v}_{3}), \mbf{v}_{3} \rangle$.  We can conclude that
$(L_{1}, H) \not \sim (L_{2}, H)$ exactly as in (6).  If $(L_{2}, H) \sim (L_{3}, H)$, where 
$\lambda L_{2} = L_{3}$ and $\lambda H \lambda^{-1} = H$, then we conclude as in (6) that $\lambda$ has
the same block form (as a matrix over the same ordered basis), although, in the current case, we can conclude only that
the upper left entry is $\pm 1$, by fullness of $\langle \mbf{v}_{1} \rangle$ in both $L_{2}$ and $L_{3}$.  This leads to 
a contradiction, since $\lambda ( \frac{1}{3} ( \mbf{v}_{1} + \mbf{v}_{2} + \mbf{v}_{3} )) \not \in L_{3}$.

Finally, suppose $(L_{1}, H) \sim (L_{3}, H)$.  Let $\lambda \in GL_{3}(\mathbb{R})$ satisfy $\lambda L_{1} = L_{3}$ and
$\lambda H \lambda^{-1} = H$.  By Lemma \ref{silly}(1), the vector subspaces
$\langle \mbf{v}_{2} \rangle$, $\langle \mbf{v}_{3} \rangle$, and $\langle \mbf{v}_{2} - \mbf{v}_{3} \rangle$ must be permuted
by $\lambda$.  Since $H$ acts transitively on these lines, by Lemma \ref{silly}(2),
 $|| \lambda \mbf{v}_{2} || = || \lambda \mbf{v}_{3} || = \alpha \sqrt{2}$, say.  By Lemma \ref{silly}(1), the vector subspace
spanned by $\mathbf{v}_{1}$ must be $\lambda$-invariant as well, since it is the unique $1$-dimensional vector subspace that
is the axis for an element of order $3$ in $H^{+}$. 
It follows that $\lambda$ has the block form from (6) over the ordered basis 
$( \mbf{v}_{1}, \mbf{v}_{2}, \mbf{v}_{3})$.  In addition, we know that
the upper left entry is an integer, by fullness of $\langle \mbf{v}_{1} \rangle$ in $L_{1}$ and $L_{3}$.  The final two
column vectors of
$\lambda$ must be linearly independent elements of the set
$$ \{ \pm \alpha \mbf{v}_{2}, \pm \alpha \mbf{v}_{3}, \pm \alpha (\mbf{v}_{2} - \mbf{v}_{3}) \}.$$
It follows that at least one has the form $\pm \alpha \mbf{v}_{i}$ (for $i \in \{ 2,3 \}$).  Since $\langle \mbf{v}_{2} \rangle$
and $\langle \mbf{v}_{3} \rangle$ are full subgroups of both $L_{1}$ and $L_{3}$, it follows that $\alpha$ must be an integer.
Thus $\lambda$ must have integral entries.  It follows that $\lambda L_{1} \subseteq L_{\mathcal{C}}$, a contradiction.   
\end{enumerate}
\end{proof} 
 
\subsection{The classification of the remaining pairs $(L, H)$}

In this subsection, we conclude the classification of pairs $(L,H)$ up to arithmetic
equivalence. Our approach is to reduce the problem of classifying the remaining pairs
to the (previously solved)
 problem of classifying pairs in which the group $H$ contains the inversion.

\begin{theorem} \label{remainingpairsclassification}
Let $L \leq \mathbb{R}^{3}$ be a lattice, and let $H \leq O(3)$ be a standard point group acting on $L$; 
suppose $(-1) \not \in H$.
The pair $(L, H)$ is equivalent to exactly one on the following list: 
\begin{enumerate}
\item If $H \leq SO(3)$, then the classification of pairs $(L, H)$ is exactly the same as that for the
group $\langle H, (-1) \rangle$, as described in Theorem \ref{apair(-1)}.  (This case accounts for $24$ different
possibilities.)
\item If $H = C'_{2}$, $C'_{4}$, $C'_{6}$, $D'_{3}$, $D''_{4}$, $D''_{6}$, or $S'_{4}$, then any pair
$(L, H)$ is equivalent to one of the $(L', H)$, where $L'$ is one of the lattices listed in Theorem \ref{apair(-1)}
for $\langle H, (-1) \rangle$.  Moreover, any two of the resulting pairs are distinct.  (There
are a total of $14$ possibilities.) 
\item Suppose $H = D'_{2}$, $D'_{4}$, or $D'_{6}$.  
\begin{enumerate}
\item If $H = D'_{4}$, then $(L, H) \sim (L', H')$ where $L' = L_{\mathcal{C}}$ or
$\langle \mbf{x}, \mbf{y}, \frac{1}{2}( \mbf{x} + \mbf{y} + \mbf{z} ) \rangle$ and
$H' = D'_{4}$ or $\widehat{D}'_{4}$. 
\item If $H = D'_{6}$, then $(L, H) \sim (L', H')$ where
$L' = L_{\mathcal{P}}$ and $H' = D'_{6}$ or $\widehat{D}'_{6}$.  
\item If $H = D'_{2}$, then $(L,H) \sim
(L', H)$, where $L'$ is any of the lattices mentioned in Corollary \ref{LC},
or $(L,H) \sim (L',D'_{2_{2}})$, where 
$$ L' = \left\langle \mathbf{x}, \mathbf{y}, \frac{\mathbf{x}+\mathbf{z}}{2} \right\rangle.$$
\end{enumerate}
(There are $11$ possibilities.)
\end{enumerate}
\end{theorem}

\begin{proof}
We note first that it is impossible for a pair $(L,H)$ to be counted twice in the different cases (1), (2), and (3), since the point groups
in question are distinguished either by isomorphism type or by their orientation-preserving subgroups. (Note, in particular, that the
groups $D_{n}'$ and $D_{n}''$  ($n=4,6$) cannot be conjugate even in $GL_{3}(\mathbb{R})$ by the descriptions in \ref{subsubsection:pointgroupsmixed}.) It is therefore enough
to consider each of the cases (1), (2), and (3) individually.
\begin{enumerate}
\item Let $H$ be an orientation-preserving standard point group, and let $L$ be a lattice satisfying $H \cdot L = L$. There is some pair $(L', \langle H, (-1) \rangle)$ from the statement of Theorem \ref{apair(-1)} such that
$(L, \langle H, (-1) \rangle) \sim (L', \langle H, (-1) \rangle)$; that is, 
we can find $\lambda \in GL_{3}(\mathbb{R})$ such that $\lambda \langle H, (-1) \rangle \lambda^{-1} 
= \langle H, (-1) \rangle$, and $\lambda L = L'$. We note that $\lambda H \lambda^{-1} = H$ since the latter groups
are the orientation-preserving subgroups of $\lambda \langle H, (-1) \rangle \lambda^{-1}$ and $\langle H, (-1) \rangle$, respectively.
It follows that $(L,H) \sim (L',H)$. Thus, any pair $(L,H)$ corresponds to one of the $24$ listed in Theorem \ref{apair(-1)}.

We now need to show that there are no repetitions on the given list of $24$ pairs $(L,H)$. Suppose that $(L_{1}, H_{1}) \sim
(L_{2}, H_{2})$, where each of $H_{1}$, $H_{2}$ is an orientation-preserving standard point group, and $L_{1}$, $L_{2}$
are lattices chosen from the statement of Theorem \ref{apair(-1)}. We first note that $H_{1}$ and $H_{2}$ must be isomorphic by 
the definition of arithmetic equivalence, and therefore equal since no two groups from the list in Table \ref{orientationpreservingpointgroups}
are isomorphic.

Thus, we assume that $(L_{1}, H) \sim (L_{2}, H)$, where $L_{1}$, $L_{2}$, and $H = H_{1} = H_{2}$ are all still as above. 
Let $\lambda \in GL_{3}(\mathbb{R})$ be such that $\lambda L_{1} = L_{2}$ and $\lambda H \lambda^{-1} = H$. This $\lambda$ shows that
$(L_{1}, \langle H, (-1) \rangle) \sim (L_{2}, \langle H, (-1) \rangle)$. It follows that $L_{1} = L_{2}$, by Theorem \ref{apair(-1)},
completing the proof.

\item  Let $H$ be one of the point groups from (2), and let $L$ be a lattice satisfying $H \cdot L = L$. There is some pair $(L', \langle H, (-1) \rangle)$ from the statement of Theorem \ref{apair(-1)} such that
$(L, \langle H, (-1) \rangle) \sim (L', \langle H, (-1) \rangle)$; that is, 
we can find $\lambda \in GL_{3}(\mathbb{R})$ such that $\lambda \langle H, (-1) \rangle \lambda^{-1} 
= \langle H, (-1) \rangle$, and $\lambda L = L'$. 

 The group $H$ is unique in the following sense. If $K \leq \langle H, (-1) \rangle$ satisfies:
i) $[\langle H, (-1) \rangle : K] = 2$; ii) $K$ does not contain the inversion; iii) $K^{+} \cong H^{+}$
(where $H^{+}$ and $K^{+}$ denote the orientation-preserving subgroups), and iv) $K \cong H$, then $K = H$.
(This can be proved by enumerating the homomorphisms $\phi: \langle H, (-1) \rangle \rightarrow
\mathbb{Z}/2\mathbb{Z}$ such that $\phi (-1) = 1$ and $\phi(H) = \mathbb{Z}/2\mathbb{Z}$. The 
group $K$ must occur as the kernel of some such $\phi$, and the given conditions force $K =H$.
Note that we have already seen this method of argument in the proof of Theorem \ref{class-}.)

Now note that $\lambda H \lambda^{-1}$ is a subgroup of $\langle H, (-1) \rangle$ satisfying i)-iv). It follows
that $\lambda H \lambda^{-1} = H$. We have now shown that $(L,H) \sim (L',H)$, where $L'$ is one
of the lattices that is paired with $\langle H, (-1) \rangle$ in Theorem \ref{apair(-1)}. 

We should next show that there are no repetitions in our list, but the proof of the latter
 fact follows the pattern from the 
final paragraph of the proof of (1).

\item Suppose that $H = D'_{4}$.  We consider the pair $(L, H)$.  Add the element $(-1)$ to the point group
to get $(L, \langle H, (-1) \rangle )$.  By the arithmetic classification of pairs with central inversion 
(Theorem \ref{apair(-1)}(3)), we know that $(L, \langle H, (-1) \rangle ) \sim (L', \langle H, (-1) \rangle)$, where
$$ L' = \langle \mbf{x}, \mbf{y}, \mbf{z} \rangle \quad \mathrm{or} \quad \left\langle \mbf{x}, \mbf{y}, 
\frac{1}{2}(\mbf{x} + \mbf{y} + \mbf{z}) \right\rangle.$$
Let us suppose that $\lambda L = L'$ and $\lambda \langle H, (-1) \rangle \lambda^{-1} = \langle H, (-1) \rangle$.
There are two possibilities for $\lambda H \lambda^{-1}$:  $D'_{4}$ and $\widehat{D}'_{4}$.  (One again sees this by enumerating the homomorphisms from  $\langle D_{4}^{+}, (-1) \rangle$
to $\mathbb{Z}/2\mathbb{Z}$.)  This leads to four possibilities for 
$(L', H')$, where $\lambda H \lambda^{-1} = H'$; we shall see that all are different.

Completely analogous reasoning shows that if $H = D'_{6}$ then there are two possibilities:
$( L_{\mathcal{P}}, D'_{6} )$ and $( L_{\mathcal{P}}, \widehat{D}'_{6})$.

Suppose $H = D'_{2}$.  Consider the pair $(L, H)$.  We add the element $(-1)$ to the point group
to get the pair $(L, \langle H, (-1) \rangle)$.  By Theorem \ref{apair(-1)}(2), $(L, \langle H, (-1) \rangle) \sim
(L', \langle H, (-1) \rangle)$, where $L'$ is any of the lattices listed in Corollary \ref{LC}.  
Let $\lambda \in GL_{3}(\mathbb{R})$ satisfy:  i) $\lambda L = L'$, and ii) $\lambda \langle H, (-1) \rangle \lambda^{-1}
= \langle H, (-1) \rangle$.  It is not difficult to check that $\lambda H \lambda^{-1} = H'$ is one of the following groups:
$$ \langle R_{xy}, R_{yz} \rangle, \quad \langle R_{xy}, R_{xz} \rangle, \quad \langle R_{yz}, R_{xz} \rangle$$
where $R_{xy}$ (for instance) is the reflection across the $xy$-plane.  We note that $\langle R_{yz}, R_{xz} \rangle = D'_{2}$. 
Thus, to summarize: we've shown that $(L, D'_{2}) \sim (L', H')$, where $L'$ is one of the standard
lattices from Corollary \ref{LC} and $H'$ is one of the groups above. 
All three of the latter groups are clearly conjugate to $D'_{2}$, by an element $\hat{\lambda} \in GL_{3}(\mathbb{R})$ that
simply permutes the coordinate axes (i.e., $\hat{\lambda} H' \hat{\lambda}^{-1} = D'_{2}$). We conclude
that $(L, D'_{2}) \sim (\hat{\lambda} L', D'_{2})$, where $\hat{\lambda}$ is a permutation matrix
and $L'$ is as above.

If $L' = \langle \mbf{x}, \mbf{y}, \mbf{z} \rangle$, $\langle \mbf{x}, \mbf{y}, \frac{1}{2}(\mbf{x} + \mbf{y} + \mbf{z}) \rangle$,
or $\langle \frac{1}{2}(\mbf{x} + \mbf{y}), \frac{1}{2}(\mbf{x} + \mbf{z}), \frac{1}{2}(\mbf{y} + \mbf{z}) \rangle$, then
$\hat{\lambda}L' = L'$.  It follows that these three possibilities give rise to three arithmetic classes of the form
$(L', D'_{2})$ (and all three are different, as we'll see).  If $L' = \langle \mbf{x}, \mbf{y}, \frac{1}{2}(\mbf{x} + \mbf{z}) \rangle$,
then $L'$ is not necessarily invariant under $\hat{\lambda}$ and there are two essentially different pairs of the form
$( \hat{\lambda} L', D'_{2} )$:
$$ \left\langle \mbf{x}, \mbf{z}, \frac{1}{2}(\mbf{x} + \mbf{y}) \right\rangle \quad \mathrm{and} \quad 
\left\langle \mbf{x}, \mbf{y}, \frac{1}{2}(\mbf{x} + \mbf{z}) \right\rangle.$$
(The case in which $L' = \langle \mbf{x}, \mbf{y}, \frac{1}{2}(\mbf{y} + \mbf{z}) \rangle$ is identical with that in which  
$L' = \langle \mbf{x}, \mbf{y}, \frac{1}{2}(\mbf{x} + \mbf{z}) \rangle$ up to arithmetic equivalence, since
the transposition that flips the $x$- and $y$- coordinates normalizes $D'_{2}$.)

Now we need to show that all $11$ of the above pairs are arithmetically distinct. As always, it is enough
to consider the subcases (a), (b), and (c) separately.

Consider first the case in which $H = D'_{2}$.  There are five such pairs; the only two that might be equal are
$( \langle \mbf{x}, \mbf{z}, \frac{1}{2}(\mbf{x} + \mbf{y}) \rangle, D'_{2})$ and 
$( \langle \mbf{x}, \mbf{y}, \frac{1}{2}(\mbf{x} + \mbf{z}) \rangle, D'_{2})$.  (Any other choice of
pairs is distinct by Theorem \ref{apair(-1)}(2): here we consider the usual reduction to point groups containing inversion.)  
If $\lambda$ normalizes $D'_{2}$ and sends
one lattice to the other, then Lemma \ref{silly}(1) and fullness of the subgroups $\langle \mbf{x} \rangle$, $\langle \mbf{y} \rangle$,
$\langle \mbf{z} \rangle$ in both lattices  imply that  $\lambda$ factors as $AB$, where $A$ is a diagonal matrix with $1$s and $-1$s on the
diagonal, and $B$ is a permutation
matrix which leaves the vector subspace $\langle \mbf{z} \rangle$ invariant.  No such matrix can send the first lattice
to the second one.  It follows that all five pairs with point group $D'_{2}$ are distinct.  We apply a transposition
$\bar{\lambda} \in GL_{3}(\mathbb{R})$ of the $y$- and $z$-coordinates to the first of these pairs,
$( \langle \mbf{x}, \mbf{z}, \frac{1}{2}(\mbf{x} + \mbf{y}) \rangle, D'_{2})$, to arrive at the pair
$(L', D'_{2_{2}})$ from the statement of the theorem.

Now we consider the case in which $H= D'_{4}$.  The possible arithmetic classes are represented by four pairs
$(L_{i}, D'_{4})$, $(L_{i}, \widehat{D}'_{4})$, where $L_{i}$ $(i \in \{ 1, 2 \})$ is one of two lattices.  We first
note that two such pairs $(L', H')$, $(L'', H'')$ will represent different classes    
if $L' \neq L''$ by Theorem \ref{apair(-1)}(3).

Thus, suppose $\widehat{L}$ is one of the two possible lattices from (3).  Suppose $( \widehat{L}, D'_{4} ) \sim ( \widehat{L},
\widehat{D}'_{4} )$; suppose $\lambda \in GL_{3}(\mathbb{R})$ satisfies $\lambda \widehat{L} = \widehat{L}$ and
$\lambda D'_{4} \lambda^{-1} = \widehat{D}'_{4}$.  The condition $\lambda D'_{4} \lambda^{-1} = \widehat{D}'_{4}$ implies that
$\lambda$ leaves the $xy$-plane invariant (here we can apply Lemma \ref{silly}(3) with $H = \langle
D_{4}^{+}, (-1) \rangle$).  Now $\lambda$ must send the poles of $D'_{4}$ to those of $\widehat{D}'_{4}$.
It follows that $\lambda$ sends the groups $\langle \mbf{x} \rangle$, $\langle \mbf{y} \rangle$ to $\langle \mbf{x} + \mbf{y} \rangle$,
$\langle \mbf{x} - \mbf{y} \rangle$ (not necessarily in that order), and all groups in question are full in $\widehat{L}$.  It follows
that $\lambda$ restricts to a similarity on $P(z=0)$.  This leads to a contradiction, in the following way.  In the pair $(\widehat{L}, D'_{4})$,
the smallest non-zero lattice point in $P(z=0)$ lies on a pole, but in the pair $(\widehat{L}, \widehat{D}'_{4})$ the smallest non-zero
lattice point in $P(z=0)$ does not (instead it lies on a plane of reflection for $\widehat{D}'_{4}$, either $P(x=0)$ or $P(y=0)$).  The fact that $\lambda$ maps
$P(z=0)$ to itself by a similarity implies that a smallest lattice point in $\widehat{L} \cap P(z=0)$
must be sent to another such.  This is the
contradiction.

The proof for the case $H = D'_{6}$ is similar.
\end{enumerate}
\end{proof}

\vspace{.5cm}
\section{Classification of Split Three-Dimensional Crystallographic Groups} \label{section:classification} 

\begin{definition}
An \emph{$n$-dimensional crystallographic group} $\Gamma$ is a discrete, cocompact subgroup of the
group of isometries of Euclidean $n$-space.  Each $\gamma \in \Gamma$ can be written in the
form $v_{\gamma} + A_{\gamma}$, where $v_{\gamma} \in R^{n}$ is a translation and $A_{\gamma} \in
O(n)$.  There is a natural map $\pi : \Gamma \rightarrow O(n)$ sending $v_{\gamma} + A_{\gamma}$
to $A_{\gamma}$, and this map is easily seen to be a homomorphism.  We get a short exact sequence
as follows:
$$ L \rightarrowtail \Gamma \twoheadrightarrow H,$$
where $H = \pi(\Gamma) \leq O(n)$ and $L$ is the kernel.  (by \cite[Theorem 7.4.2]{Ra94}, $L$ is a lattice 
in $\mathbb{R}^{n}$, and so necessarily
isomorphic to $\mathbb{Z}^{n}$.)  We note that $H$ acts naturally on $L$, which makes $H$ a point group in the sense of 
Definition \ref{definition:first}. We say that $H$ is the \emph{point group} of $\Gamma$. 
The group $\Gamma$ is a \emph{split $n$-dimensional crystallographic
group} if the above sequence splits, i.e., if there is a homomorphism $s: H \rightarrow \Gamma$ such that
$\pi s = id_{H}$.

In the remainder of the paper, all of our crystallographic groups will be $3$-dimensional.
\end{definition}
 
\begin{definition} 
Suppose that $L$ is a lattice in $\mathbb{R}^{3}$ and $H \leq O(3)$ satisfies $H \cdot L = L$.  We
let $\Gamma(L,H)$ denote the group $\langle L, H \rangle$.
\end{definition}

\begin{remark}
It is straightforward to verify that every $\Gamma(L,H)$ is a split crystallographic group.
\end{remark}

\begin{theorem} \label{theorem:arithmeticsplit}
Any split crystallographic group $\widehat{\Gamma}$ is isomorphic to $\Gamma(L,H)$, for some
lattice $L \leq \mathbb{R}^{3}$ and $H \leq O(3)$ satisfying $H \cdot L = L$.  The groups
$\Gamma(L,H)$ and $\Gamma(L',H')$ are isomorphic if and only if the pairs $(L,H)$ and $(L',H')$
are arithmetically equivalent.
\end{theorem}
\begin{proof}
We prove the first statement.  Let $\widehat{\Gamma}$ be a split crystallographic group, $\widehat{L}$ denote the lattice of $\widehat{\Gamma}$, and 
$\widehat{H}$ denote the point group of $\widehat{\Gamma}$.  Since $\widehat{\Gamma}$ is split,
it follows that there is a finite subgroup $J$ of $\widehat{\Gamma}$ such that $\pi: \widehat{\Gamma} \rightarrow
\widehat{H}$ satisfies $\pi(J) = \widehat{H}$. It is routine to check that $\pi_{\mid J} : J \rightarrow \widehat{H}$ must also be injective.
Since $J$ is a finite group of isometries of $\mathbb{R}^{3}$, it must be that the entire group $J$
fixes a point $v \in \mathbb{R}^{3}$. We consider the isometry $T_{v} \in \mathrm{Isom}(\mathbb{R}^{3})$, which is simply translation by the
vector $v$. It follows that $T_{v}^{-1} J T_{v}$ fixes the origin, so the map $\pi: T_{v}^{-1} J T_{v} \rightarrow \widehat{H}$ is the identity. We can therefore
write
$$ 1 \rightarrow \widehat{L} \rightarrow T_{v}^{-1} \widehat{\Gamma}T_{v} \rightarrow \widehat{H} \rightarrow 1,$$
where $\widehat{H} \leq T_{v}^{-1} \widehat{\Gamma}T_{v}$. It follows directly that $\widehat{\Gamma} \cong T_{v}^{-1} \widehat{\Gamma} T_{v} = \langle \widehat{L}, \widehat{H} \rangle$,
proving the first statement.


Now we prove the second statement.  Assume that $\Gamma(L,H)$ and $\Gamma(L',H')$ are isomorphic.
Ratcliffe \cite[Theorem 7.4.4]{Ra94} says that there is an affine bijection $\alpha$ of $\mathbb{R}^{3}$
such that $\alpha \Gamma(L,H) \alpha^{-1} = \Gamma(L',H')$.  We write $\alpha = v_{\alpha} + A_{\alpha}$,
where $v_{\alpha} \in \mathbb{R}^{3}$ and $A_{\alpha} \in GL_{3}(\mathbb{R})$.  We note:
\begin{eqnarray*}
\alpha L \alpha^{-1} & = & A_{\alpha} \cdot L; \\
\alpha H \alpha^{-1} & = & \left( v_{\alpha} - A_{\alpha} H A^{-1}_{\alpha}(v_{\alpha})\right) + 
A_{\alpha}HA^{-1}_{\alpha}.  
\end{eqnarray*}  
Now $\alpha \Gamma(L,H) \alpha^{-1}$ and $\Gamma(L',H')$ must have the same kernel and image under the
canonical projection $\pi: \mathrm{Isom}(\mathbb{R}^{3}) \rightarrow O(3)$, so $A_{\alpha} \cdot L = L'$
and $A_{\alpha}HA^{-1}_{\alpha} = H'$.  (These last two equations are between kernels and images, respectively.)
It follows that $(L,H)$ and $(L',H')$ are arithmetically equivalent.  

If two pairs $(L,H)$ and $(L',H')$ are arithmetically equivalent, then there is $\lambda \in GL_{3}(\mathbb{R})$
such that $\lambda L = L'$ and $\lambda H \lambda^{-1} = H'$.  It follows easily that 
$\lambda \Gamma(L,H) \lambda^{-1} =\Gamma(L',H')$, so $\Gamma(L,H)$ and $\Gamma(L',H')$ are 
isomorphic.
\end{proof}

\begin{theorem}[List of Split Three-Dimensional Crystallographic Groups]  \label{biglist}
Let $\mathbf{x}$, $\mathbf{y}$, and $\mathbf{z}$ denote the standard coordinate vectors, and let
$$ \mathbf{v}_{1} = \left  (\begin{smallmatrix} 1 \\  1 \\ 1 \end{smallmatrix} \right), \quad
\mathbf{v}_{2} = \left(\begin{smallmatrix} 1 \\ -1 \\ 0 \end{smallmatrix} \right), \quad  
\mathbf{v}_{3} = \left( \begin{smallmatrix} 0 \\ -1 \\ 1 \end{smallmatrix} \right).$$
A complete list of the split three-dimensional crystallographic groups $\langle L, H \rangle$
appears in Table \ref{splitcrystallographicgroups}.

\begin{table} [!h]
\renewcommand{\arraystretch}{1.3}
\begin{equation*}
\begin{array}{ | c | c | c | c | c | c | c |} \hline
L & \multicolumn{6}{c |}{H} \\ \hline
 &  S^{+}_{4} \times (-1) & S^{+}_{4} & S'_{4} & A^{+}_{4} \times (-1)  & A^{+}_{4} & D''_{4} \\  \cline{2-7}
\langle \mathbf{x}, \mathbf{y}, \mathbf{z} \rangle &  D^{+}_{4} \times (-1) &  D^{+}_{4} & C'_{2} &  D^{+}_{2} \times (-1) &  D^{+}_{2} & C'_{4} \\  \cline{2-7}
 & C^{+}_{4} \times (-1) & C^{+}_{4}  & D'_{2} & C^{+}_{2} \times (-1) & C^{+}_{2}  & D'_{4} \\ \cline{2-7}
 & C^{+}_{1} \times (-1) & C^{+}_{1} &  \widehat{D}'_{4} & & & \\ \hline
  &  S^{+}_{4} \times (-1) & S^{+}_{4} & S'_{4} & A^{+}_{4} \times (-1)  & A^{+}_{4} & D''_{4} \\  \cline{2-7}
\langle \frac{1}{2} \left( \mathbf{x} + \mathbf{y} + \mathbf{z} \right), \mathbf{y}, \mathbf{z} \rangle &  D^{+}_{4} \times (-1) &  D^{+}_{4} & C'_{2} &  D^{+}_{2} \times (-1) &  D^{+}_{2} & C'_{4} \\  \cline{2-7}
 & C^{+}_{4} \times (-1) & C^{+}_{4}  & D'_{2} & C^{+}_{2} \times (-1) & C^{+}_{2}  & D'_{4} \\ \cline{2-7}
 & \widehat{D}'_{4}   &  & & & & \\ \hline
 & S^{+}_{4} \times (-1) & S^{+}_{4} & S'_{4} & A^{+}_{4} \times (-1)  & A^{+}_{4} & D'_{2} \\  \cline{2-7}
\frac{1}{2}\langle( \mathbf{x} + \mathbf{y} ), (\mathbf{x} + \mathbf{z}), (\mathbf{y} + \mathbf{z}) \rangle & 
D^{+}_{2} \times (-1) & D^{+}_{2} & & & & \\ \hline
 \langle \frac{1}{2}( \mathbf{x} + \mathbf{z} ), \mathbf{y}, \mathbf{z} \rangle & D^{+}_{2} \times (-1) & D^{+}_{2} & D'_{2} & D'_{2_{2}} & & \\ \hline
&  D^{+}_{6} \times (-1) & D^{+}_{6} &  C'_{6}  & C^{+}_{6} \times (-1) & D'_6 & C^{+}_{6} \\ \cline{2-7}
\langle \mathbf{v}_{1}, \mathbf{v}_{2}, \mathbf{v}_{3} \rangle &    D^{+}_{3} \times (-1) & \widehat{D}'_{6}  & C^{+}_{3}   &C^{+}_{3} \times (-1)&  D'_{3} & D^{+}_{3}  \\ \cline{2-7}
&D''_6 & & & & & \\ \hline
 \langle \frac{1}{3}( \mathbf{v}_{1} + \mathbf{v}_{2} + \mathbf{v}_{3} ), \mathbf{v}_{2}, \mathbf{v}_{3} \rangle & D^{+}_{3} \times (-1) & D^{+}_{3} & D'_{3} & C^{+}_{3} \times (-1) & C^{+}_{3} &  \\ \hline
\langle \mathbf{v}_{1}, \frac{1}{3}(\mathbf{v}_{2} + \mathbf{v}_{3}), \mathbf{v}_{3} \rangle & D^{+}_{3} \times (-1) & D^{+}_{3} & D'_{3} & & &  \\ \hline
\end{array}
 \end{equation*}
 
 \caption{The Split Three-Dimensional Crystallographic Groups} 
\label{splitcrystallographicgroups} 
\end{table}
 
\end{theorem}

\begin{proof}
Table \ref{splitcrystallographicgroups} lists all pairings of lattices and point groups from Theorems \ref{apair(-1)} and \ref{remainingpairsclassification}.
We have already shown that no two of the pairs from Theorem \ref{apair(-1)} determine the same arithmetic equivalence class (Theorem \ref{noredundancy}).
Also, no two of the pairs from Theorem \ref{remainingpairsclassification} determine the same arithmetic class. If the pair $(L_{1},H_{1})$ is chosen
from the pairs listed in Theorem \ref{apair(-1)}, and $(L_{2},H_{2})$ is chosen from the pairs listed in Theorem \ref{remainingpairsclassification}, then
$(L_{1}, H_{1}) \not \sim (L_{2}, H_{2})$, since $H_{1}$ contains the inversion $(-1)$ and $H_{2}$ does not, and containing the inversion will
be preserved by arithmetic equivalence. Thus, all $73$ pairs in Table \ref{splitcrystallographicgroups} represent distinct equivalence classes. 

It is clear from Theorems \ref{apair(-1)} and \ref{remainingpairsclassification} that any pair $(L,H)$ is equivalent to one on the list, since $H$ must be conjugate to one of the standard point groups. It follows that
there are exactly $73$ classes of such pairs.

The theorem now follows from Theorem \ref{theorem:arithmeticsplit}.
\end{proof}

\begin{remark} \label{remark:labelconvention}
Let $\widehat{\Gamma}$ be a point group. For the sake of brevity, we will sometimes let $\widehat{\Gamma}_{i}$ denote the 
split crystallographic group $\langle L_{i}, \widehat{\Gamma} \rangle$, where
$L_{i}$ denotes the $i$th lattice (in the order that they are listed in Table \ref{splitcrystallographicgroups}).
Thus, $(D_{2}^{+})_{1}$ denotes the split crystallographic group generated by the point group 
$D_{2}^{+}$ and the standard cubical lattice.

We will let $\Gamma_{i}$ denote the $i$th maximal split crystallographic group; i.e., the pairing of the $i$th lattice with the largest point group
from Table \ref{splitcrystallographicgroups}. Thus, for instance, $\Gamma_{1}$ denotes the
group $\langle \mathbf{x}, \mathbf{y}, \mathbf{z} \rangle \rtimes (S_{4}^{+} \times (-1))$.
\end{remark}

\section{A model for $E_{\vc}(\Gamma)$ and a formula for the algebraic $K$-theory} \label{section:EVC}

 Let $\Gamma$ be a three-dimensional crystallographic group with lattice $L$ and point group $H$. (We do not assume that $\Gamma$
is a split crystallographic group.)  
In this section, we describe a simple construction of $E_{\mathcal{VC}}(\Gamma)$ and derive a splitting formula for the lower algebraic $K$-theory
of any three-dimensional crystallographic group.

\subsection{A construction of $E_{\mathcal{FIN}}(\Gamma)$ for crystallographic groups}
We will need to have a specific model of $E_{\mathcal{FIN}}(\Gamma)$ for our crystallographic groups $\Gamma$.

\begin{proposition} \label{proposition:efin}
If $\Gamma$ is a three-dimensional crystallographic group, then there is an equivariant cell structure on 
$\mathbb{R}^{3}$ making it a model for $E_{\mathcal{FIN}}(\Gamma)$.
\end{proposition}

\begin{proof}
For every crystallographic group $\Gamma$, there is a crystallographic group $\Gamma'$ of the same dimension, called
the \emph{splitting group} of $\Gamma$ (\cite[pgs. 312-313]{Ra94}), and an embedding $\phi: \Gamma \rightarrow \Gamma'$. The group $\Gamma'$
is a split crystallographic group in our sense, by Lemma 7 on page 313 of \cite{Ra94}. It is therefore sufficient to prove the proposition for every 
split three-dimensional crystallographic group. Table \ref{splitcrystallographicgroups} shows that all of the split crystallographic groups are
subgroups of seven maximal ones (consider the pairing of the maximal point group with each of the seven lattices). We will show
in Section \ref{section:fundamentaldomains} (without circularity) that
each of these maximal groups has the required model. The proposition now follows easily.
\end{proof}

\subsection{A construction of $E_{\mathcal{VC}}(\Gamma)$ for crystallographic groups} \label{subsection:constructEVC}

Let $\Gamma$ be a three-dimensional crystallographic group.
We begin with a copy of $E_{\mathcal{FIN}}(\Gamma)$, which we can identify with a suitably cellulated copy of $\mathbb{R}^{3}$ by Proposition \ref{proposition:efin}. 
 For each $\ell \in L$ such that $\ell$ 
generates a 
maximal cyclic subgroup of $L$, we define
$$ \mathbb{R}_{\ell}^{2} = \{ \widehat{\ell} \subseteq \mathbb{R}^{3} \mid \widehat{\ell} 
\text{ is a line parallel to } \langle \ell \rangle \},$$
where $\langle \ell \rangle$ denotes the $1$-dimensional vector subspace spanned by $\ell$.
Consider $\base$, where the disjoint union is over all maximal cyclic subgroups $\langle \ell \rangle$ of $L$. We define
a metric on $\base$ as follows. If $\ell_{1}, \ell_{2} \in \base$, we set $d(\ell_{1}, \ell_{2}) = \infty$ if $\ell_{1}$ and $\ell_{2}$ are not
parallel, and $d(\ell_{1}, \ell_{2}) = K$ if $\ell_{1}$ is parallel to $\ell_{2}$ and $K = \mathrm{min} \{ d_{\mathbb{R}^{3}}(x,y) \mid x \in \ell_{1},
y \in \ell_{2} \}$. One readily checks that $d$ is a metric on $\base$, and that each $\mathbb{R}^{2}_{\ell}$ is isometric to $\mathbb{R}^{2}$. We will therefore freely refer to the $\mathbb{R}^{2}_{\ell}$ as ``planes" in what follows.
Moreover, $\Gamma$ acts by isometries on $\base$.

Next we would like to introduce an equivariant cell structure on $\base$. Choose a plane $\mathbb{R}^{2}_{\ell}$.

\begin{definition}
Let $\pi: \Gamma \rightarrow H$ be the usual projection into the point group. We let 
$H_{\langle \ell \rangle} = \{ h \in H \mid h \cdot \langle \ell \rangle = \langle \ell \rangle \}$ and 
$\Gamma(\ell) = \pi^{-1}(H_{\langle \ell \rangle})$.
\end{definition}

 It is straightforward to check that $\Gamma(\ell)$ acts on $\mathbb{R}^{2}_{\ell}$.
Since $\langle \ell \rangle$ is a maximal cyclic subgroup of $L$, we can choose a basis $\{ \ell_{1}, \ell_{2}, \ell_{3} \}$ of $L$, with $\ell_{3} = \ell$. Each of the
$\ell_{i}$ can be written $\ell_{i} = \alpha_{i} \ell + \widehat{\ell}_{i}$, where
$\alpha_{i} \in \mathbb{R}$ and $\widehat{\ell}_{i}$ is perpendicular to $\ell$. Since $\ell_{1}$, $\ell_{2}$, and $\ell_{3}$ are linearly
independent over $\mathbb{R}$, the same must be true of $\widehat{\ell}_{1}$ and $\widehat{\ell}_{2}$. The translation $\ell$ acts trivially on 
$\mathbb{R}^{2}_{\ell}$, so the action of $L$ on $\mathbb{R}^{2}_{\ell}$ is the same as the action of $\langle \widehat{\ell}_{1}, \widehat{\ell}_{2} \rangle$.
In particular, the action of $L$ has discrete orbits, from which it follows readily that the action of $\Gamma(\ell)$ on $\mathbb{R}^{2}_{\ell}$
has discrete orbits. We can therefore find a $\Gamma(\ell)$-equivariant cell structure on $\mathbb{R}^{2}_{\ell}$ making it
a $\Gamma(\ell)$-CW complex.

Now we choose a (finite) left transversal $T \subseteq H$ of $\Gamma(\ell)$ in $\Gamma$. For each $t \in T$, we cellulate $\mathbb{R}^{2}_{t \cdot \ell}$ using
the equality $\mathbb{R}^{2}_{\ell} = t \cdot \mathbb{R}^{2}_{\ell}$ (that is, for each cell $\sigma \subseteq \mathbb{R}^{2}_{\ell}$, we let
$t \cdot \sigma$ be a cell in the cellulation of $\mathbb{R}^{2}_{t \cdot \ell}$). The result is an equivariant cellulation of all of 
$\Gamma \cdot \mathbb{R}^{2}_{\ell}$, which is a disjoint union of finitely many planes. We can continue in the same way, choosing a new plane
$\mathbb{R}^{2}_{\ell'}$ and applying the same procedure, until we have cellulated all of $\base$. The space $\base$ is a $\Gamma$-CW complex with respect to the resulting cellulation.

\begin{proposition} \label{prop:EVC}
Let $Y = E_{\mathcal{FIN}}(\Gamma)$ and $Z = \base$.
The space $X = Y \ast Z$ is a model for $E_{\mathcal{VC}}(\Gamma)$.
\end{proposition}

\begin{proof}
Since $Y$ and $Z$ are $\Gamma$-CW complexes, the join $X$ inherits a natural $\Gamma$-CW complex structure.
For $G \leq \Gamma$ and $W \in \{ X, Y, Z \}$, we let $\mathrm{Fix}_{W}(G) = \{ w \in W \mid g \cdot w = w \text{ for all } g \in G \}$.
We note that $\mathrm{Fix}_{W}(G)$ is a subcomplex of $W$, and $\mathrm{Fix}_{X}(G) = \mathrm{Fix}_{Y}(G) \ast \mathrm{Fix}_{Z}(G)$, for all
$G \leq \Gamma$.

Let $G \in \mathcal{VC}(\Gamma)$. There are two cases. Assume first that $G$ is finite. In this case,
$\mathrm{Fix}_{X}(G) = \mathrm{Fix}_{Y}(G) \ast \mathrm{Fix}_{Z}(G)$, where $\mathrm{Fix}_{Y}(G)$ is contractible by our assumptions. It follows
that $\mathrm{Fix}_{X}(G)$ is contractible.

If $G$ is infinite and virtually cyclic, then there is a cyclic subgroup $\langle g \rangle \leq L$ having finite index in $G$ such that 
$\langle g \rangle \unlhd G$. This group $\langle g \rangle$ is contained in a maximal cyclic subgroup $\langle \ell \rangle$ of $L$. The kernel of  the
action of $G$ on $\mathbb{R}^{2}_{\ell}$ therefore contains $\langle g \rangle$. It follows that the fixed set of the action of $G$ on $\mathbb{R}^{2}_{\ell}$
is the same as the fixed set of the action of $G/\langle g \rangle$ on $\mathbb{R}^{2}_{\ell}$. The latter group is a finite group acting by isometries, so
the fixed set is contractible. In particular, $\mathrm{Fix}_{X}(G) \cap \mathbb{R}^{2}_{\ell}$ is contractible. Now we claim that 
$\mathrm{Fix}_{X}(G) = \mathrm{Fix}_{X}(G) \cap \mathbb{R}^{2}_{\ell}$ (i.e., that $\mathrm{Fix}_{X}(G) \subseteq \mathbb{R}^{2}_{\ell}$). Indeed, it is 
enough to check that $\mathrm{Fix}_{X}(G) \cap \mathbb{R}^{2}_{\ell'} = \emptyset$ for $\langle \ell' \rangle \neq \langle \ell \rangle$ and
$\mathrm{Fix}_{X}(G) \cap Y = \emptyset$. The latter equality follows directly from the definition of $Y$.
If $\langle \ell' \rangle \neq \langle \ell \rangle$, then $g$ acts as $\widehat{g}$ on $\mathbb{R}^{2}_{\ell'}$, where $\widehat{g}$ is the component
of $g$ perpendicular to $\ell'$. The claim follows directly.

Now suppose that $G \notin \mathcal{VC}(\Gamma)$. It follows that $rk( G \cap L) \geq 2$. One easily sees that $G \cap L$ cannot have any global fixed point in
any $\mathbb{R}^{2}_{\ell}$ and $\mathrm{Fix}_{Y}(G \cap L) = \emptyset$ by definition. It follows that $\mathrm{Fix}_{X}(G) = \emptyset$, as required.
\end{proof}

\subsection{A splitting formula for the lower algebraic $K$-theory}\label{subsection:splitting formula}
In the next sections, we will use the following theorem to compute the lower algebraic $K$-theory of the integral group ring of all $73$ split three-dimensional crystallographic groups.
Our goal in this subsection is to provide a proof.

\begin{theorem} \label{theorem:splitting2}
Let $\g$ be a three-dimensional crystallographic group. We have a splitting
\[
\begin{split}
H_{\ast}^{\g}(&E_{\vc}(\g); \mathbb{KZ}^{-\infty}) \cong \\
&H_{\ast}^{\g}(E_{\fin}(\g); \mathbb{KZ}^{-\infty}) \oplus  \bigoplus_{ \widehat{\ell}  \in \mathcal{T}''}  H_n^{\Gamma_{\widehat{\ell}}}(E_{\fin}(\Gamma_{\widehat{\ell}}) \rightarrow  \ast;\;  \mathbb{KZ}^{-\infty}),
\end{split}
\]
The indexing set $\mathcal{T}''$ consists of a selection of one vertex
$v \in \coprod_{ \langle \ell \rangle} \mathbb{R}^{2}_{\ell}$ from each $\Gamma$-orbit
of such non-negligible vertices.
\end{theorem}
\noindent We refer the reader to Definition \ref{definition:negligible1} for a definition of 
negligible. We will eventually see that $\mathcal{T}''$ is finite. Note that Theorem \ref{theorem:splitting2} 
used the following definition.

\begin{definition}
Let $G$ be a group acting on a set $X$. If $A \subseteq X$, we let 
$G_{A} = \{g \in G \mid g \cdot A = A \}$.
\end{definition}

For each maximal cyclic subgroup $\langle \ell \rangle \leq L$,  we set
$$\mathcal{VC}_{\langle \ell \rangle} = \mathcal{VC} \cap \{ G \leq \Gamma(\ell) \mid |G \cap \langle \ell \rangle| = \infty \text{ if }
|G| = \infty \}.$$
In words, $\mathcal{VC}_{\langle \ell \rangle}$ is the collection consisting of finite subgroups of $\Gamma(\ell)$ 
and the infinite virtually cyclic subgroups of $\Gamma(\ell)$ that contain some non-zero multiple of the translation
$\ell$.
It is easy to check that $\mathcal{VC}_{\langle \ell \rangle}$ is a family of subgroups in $\Gamma(\ell)$.

The point group $H$ acts on the set of maximal cyclic subgroups  in $L$ by the rule $h \cdot \langle \ell \rangle
= \langle h(\ell) \rangle$. We choose a single maximal cyclic subgroup from each orbit and call the resulting collection $\mathcal{T}$. 

\begin{proposition} \label{prop:Ghomeo}
Assume that $\Gamma$ is given the discrete topology. We continue to write $Y = E_{\mathcal{FIN}}(\Gamma)$
and $Z = \base$ when convenient.
\begin{enumerate}
\item $\base$ is homeomorphic to $\borelt$ by a homeomorphism that is compatible with the $\Gamma$-action. (Here
the action on the latter space is given by the rule $\gamma \cdot (\gamma', \ell') = (\gamma \gamma', \ell')$.)
Each $\mathbb{R}^{2}_{\ell}$
is a model for $E_{\mathcal{VC}_{\langle \ell \rangle}}( \Gamma(\ell))$.
\item $\ybase$ is homeomorphic to $\yborelt$ by a homeomorphism that is compatible with the $\Gamma$-action. Each $Y \times \mathbb{R}^{2}_{\ell}$
is a model for $E_{\mathcal{FIN}}(\Gamma(\ell))$.
\end{enumerate}
The space $\ybase$ can be identified with $Y \times Z \times \{ 1/2 \} \subseteq Y \ast Z$ and $\base$ can be 
identified with the bottom of the join $Y \ast Z$.
\end{proposition}

\begin{proof}
We prove (1), the proof of (2) being similar. Consider the $\Gamma$-space $\topspace$, where $\Gamma$ acts by left multiplication on the first coordinate and trivially on the second coordinate.
We let $\borelt$ be  the usual Borel construction (so that $\Gamma$ acts only on the first coordinate). We regard $\base$ as a $\Gamma$-space with respect
to its usual action.

Define maps $\pi_{1}: \topspace \rightarrow \borelt$ and $\pi_{2}: \topspace \rightarrow \base$ by the rules
$\pi_{1}( \gamma, x) = (\gamma, x)$ and $\pi_{2}( \gamma, x) = \gamma \cdot x$. Both of these are quotient maps and commute with the $\Gamma$-action.

We claim that $\pi_{1}$ is constant on point inverses of $\pi_{2}$ and $\pi_{2}$ is constant on point inverses of $\pi_{1}$. It will then follow from a
well-known principle (see \cite[Theorem 22.2]{Mu00}) that there is a $\Gamma$-homeomorphism $f: \borelt \rightarrow \base$ such that $f \circ \pi_{1} = \pi_{2}$. 
\[
\xymatrix{
\topspace \ar[d]^{\pi_1}   \ar[dr]^{\pi_2} &\\ 
\borelt  \ar[r]^{\hskip 15pt f}  & \base}
\]      

Let $\gamma \in \Gamma$ and $x \in \mathbb{R}^{2}_{\ell}$ for some $\langle \ell \rangle \in \mathcal{T}$. One easily checks that
$$ \pi_{1}^{-1}( \gamma, x) = \{ ( \gamma \widetilde{\gamma}, \widetilde{\gamma}^{-1} \cdot x) \mid \widetilde{\gamma} \in \Gamma(\ell) \}.$$
It follows directly that $\pi_{2}(\pi_{1}^{-1}(\gamma, x)) = \{ \gamma \cdot x \}$ (a singleton), as required.

If $x \in \base$, then 
$$\pi_{2}^{-1}(x) = \{ ( \widehat{\gamma}, \widehat{x}) \mid \widehat{\gamma} \cdot \widehat{x} = x \}.$$
Now we suppose that $(\gamma_{1}, x_{1})$ and $(\gamma_{2}, x_{2})$ are in $\pi_{2}^{-1}(x)$. It follows that $\gamma_{1} \cdot x_{1} = \gamma_{2} \cdot x_{2}$,
so $\gamma_{2}^{-1} \gamma_{1} \cdot x_{1} = x_{2}$. Since $x_{1}$ and $x_{2}$ are in the same $\Gamma$-orbit, it must be that both are in $\mathbb{R}^{2}_{\ell}$,
for some $\langle \ell \rangle \in \mathcal{T}$. It then follows that $\gamma_{2}^{-1}\gamma_{1} \in \Gamma(\ell)$. Now we apply $\pi_{1}$:
\begin{align*}
\pi_{1}(\gamma_{1}, x_{1}) &= (\gamma_{1}, x_{1}) \\
&= (\gamma_{2} (\gamma_{2}^{-1} \gamma_{1}), x_{1}) \\
&\sim (\gamma_{2}, \gamma_{2}^{-1}\gamma_{1} \cdot x_{1}) \\
&= (\gamma_{2}, x_{2}) \\
&= \pi_{1}(\gamma_{2}, x_{2})
\end{align*}
It follows that $\pi_{1}(\pi_{2}^{-1}(x))$ is a singleton, as required.

We have now demonstrated the existence of $f$. The remaining statements are straightforward to check.
\end{proof}  

\begin{proposition} \label{proposition:splitting1}
Let $\g$ be a three-dimensional crystallographic group. We have a splitting

\[
\begin{split}
H_{\ast}^{\g}(&E_{\vc}(\g); \mathbb{KZ}^{-\infty}) \cong \\
&H_{\ast}^{\g}(E_{\fin}(\g); \mathbb{KZ}^{-\infty}) \oplus  \bigoplus_{\langle \ell \rangle \in \mathcal T} H_n^{\Gamma(\ell)}(E_{\fin}(\Gamma(\ell))\rightarrow
E_{\vc_{\langle \ell \rangle}}(\Gamma(\ell));  \mathbb{KZ}^{-\infty})),
\end{split}
\]
where   $H_n^{\Gamma(\ell)}(E_{\fin}(\Gamma(\ell))\rightarrow
E_{\vc_{\langle \ell \rangle}}(\Gamma(\ell));  \mathbb{KZ}^{-\infty})$ denotes the cokernel of the relative assembly map 
 \[
 H_n^{\Gamma(\ell)}(E_{\fin}(\Gamma(\ell));  \mathbb{KZ}^{-\infty}) \longrightarrow
 H_n^{\Gamma(\ell)}(E_{\vc_{\langle \ell \rangle}}(\Gamma(\ell));  \mathbb{KZ}^{-\infty}).
\]
\end{proposition}

The proof of this proposition resembles others that have appeared in \cite{J-PL06}  and \cite{LO09}.

\begin{proof}
Let us work with the explicit model $X$ for $E_{\vc}(\g)$ constructed in Propositions \ref{prop:EVC} and 
\ref{prop:Ghomeo}.  Since $X$ is obtained as
a join, there exists an obvious map $\rho:
X\rightarrow [0,1]$, which further has the property that every point
pre-image is $\g$-invariant.  In particular, corresponding to the
splitting of $[0,1]$ into $[0,2/3)\cup (1/3,1]$, we get a
$\g$-invariant splitting of $X$.  If we let $A=\rho^{-1}[0,2/3)$,
$B=\rho^{-1}(1/3,1]$, then from the Mayer-Vietoris sequence in
equivariant homology (and omitting coefficients in order to simplify our
notation), we have that:

\[
\ldots \rightarrow H_n^{\g}(A\cap B)\rightarrow H_n^{\g}(A)\oplus
H_n^{\g}(B) \rightarrow H_n^{\g}(X)\rightarrow \ldots
\]

Next, observe that we have obvious $\g$-equivariant homotopy equivalences:

\begin{itemize}
\item $A=\rho^{-1}[0,2/3)\simeq \rho^{-1}(0) =E_{\fin}(\g)$

\item $B=\rho^{-1}(1/3,1] \simeq \rho^{-1}(1) =  \coprod_{\langle \ell \rangle  \in \mathcal T} \g \times_{\Gamma(\ell)} E_{\vc_{\langle \ell \rangle}}( \Gamma(\ell))$

\item $A\cap B = \rho^{-1}(1/3,2/3) \simeq  \rho^{-1}(1/2)=
\coprod_{\langle \ell \rangle  \in \mathcal T} \g \times_{\Gamma(\ell)}  E_{\fin}(\Gamma(\ell))$
\end{itemize}

Now, using the induction structure and the fact that  our equivariant generalized homology
theory turns disjoint unions into direct sums, we can evaluate the terms in the
Mayer-Vietoris sequence as follows:

\[
\begin{split}
\ldots \rightarrow \bigoplus_{\langle \ell \rangle  \in \mathcal T} H_n^{\Gamma(\ell)}(E_{\fin}(\Gamma(\ell)))
\rightarrow H_n^{\g}(E_{\fin}(\g)) \oplus  &\bigoplus _{\langle \ell \rangle  \in \mathcal T}
H_n^{\Gamma(\ell)}(E_{\vc_{\langle \ell \rangle}}(\Gamma(\ell)) \rightarrow \\
&\rightarrow H_n^{\g}(E_{\vc}(\g))\rightarrow \ldots
\end{split}
\]

Next, we study the relative assembly  map $$\varPhi_{\langle \ell \rangle}: H_n^{\Gamma(\ell)}(E_{\fin}(\Gamma(\ell))) \rightarrow H_n^{\Gamma(\ell)}(E_{\vc_{\langle \ell \rangle}}(\Gamma(\ell)).$$ We claim  
$\varPhi_{\langle \ell \rangle}$ is split  injective.  This can be seen as follows. Consider the following commutative diagram:
\[
\xymatrix{
H_n^{\Gamma(\ell)}(E_{\fin}(\Gamma(\ell)))  \ar[ddr]_{\beta}  \ar[r]^{\varPhi_{\langle \ell \rangle}}  &  H_n^{\Gamma(\ell)}(E_{\vc_{\langle \ell \rangle}}(\Gamma(\ell)) 
\ar[dd]^{\alpha} \\ 
&\\
  &   H_n^{\Gamma(\ell)}(E_{\vc}(\Gamma(\ell))}
\]
where $\alpha$ and $\beta$ are the relative assembly  maps induced by the inclusions  $\vc_{\langle \ell \rangle} \subset \vc$ and $\fin \subset \vc$.
Recall that Bartels \cite{Bar03} has established that for \emph{any} group $G$, the relative assembly map:
$$H_*^{G}(E_{\fin}(G);\mathbb{KZ}^{-\infty})
\rightarrow H_*^{G}(E_{\vc}(G);\mathbb{KZ}^{-\infty})$$
is split injective.  Using this result from Bartels,  it follows that 
$$\beta: H_n^{\Gamma(\ell)}(E_{\fin}(\Gamma(\ell))) \longrightarrow H_n^{\Gamma(\ell)}(E_{\vc}(\Gamma(\ell))$$
is split injective. Therefore  $\varPhi_{\langle \ell \rangle}$ is also split injective.

Now, for each integer $n$, the above
portion of the Mayer-Vietoris long exact sequence breaks off as a
short exact sequence (since the initial term injects). Since the map
from the $H_n^{\g}(E_{\fin}\g) \rightarrow H_n^{\g}(E_{\vc}\g)$ is
also split injective (from the Bartels result), we obtain an
identification of the cokernel of the latter  map with the cokernel of the
map

\begin{equation}
\bigoplus_{\langle \ell \rangle  \in \mathcal T} H_n^{\Gamma(\ell)}(E_{\fin}(\Gamma(\ell))) \longrightarrow \bigoplus _{\langle \ell \rangle  \in \mathcal T}
H_n^{\Gamma(\ell)}(E_{\vc_{\langle \ell \rangle}}(\Gamma(\ell))
\end{equation}

Next, since the  inclusion map 

\[\coprod_{\langle \ell \rangle  \in \mathcal T} \g \times_{\Gamma(\ell)}  E_{\fin}(\Gamma(\ell)) \longrightarrow \coprod_{\langle \ell \rangle  \in \mathcal T} \g \times_{\Gamma(\ell)} E_{\vc_{\langle \ell \rangle}}( \Gamma(\ell))
\]
 is the disjoint union of cellular  $\Gamma(\ell)$-maps (for all $\langle \ell \rangle \in \mathcal T$), we see that
the  maps given in (1)  split as a direct sum (over   $\langle \ell \rangle \in \mathcal T$) of the
relative assembly maps  $H_n^{\Gamma(\ell)}(E_{\fin}(\Gamma(\ell))) \rightarrow H_n^{\Gamma(\ell)}(E_{\vc_{\langle \ell \rangle}}(\Gamma(\ell))).$  This immediately yields a corresponding splitting of the cokernel, completing the
proof of the proposition.
\end{proof}

The next step is to analyze the summands
$$ H_n^{\Gamma(\ell)}(E_{\fin}(\Gamma(\ell))\rightarrow
E_{\vc_{\langle \ell \rangle}}(\Gamma(\ell));  \mathbb{KZ}^{-\infty})$$
from Proposition \ref{proposition:splitting1}.

We fix a maximal cyclic subgroup $\langle \ell \rangle \leq L$ for the remainder of this section. We note that the 
space $\mathbb{R}^{3} \ast \mathbb{R}^{2}_{\ell}$ is a model for $E_{\mathcal{VC}_{\langle \ell \rangle}}(\Gamma(\ell))$, where 
both factors are given $\Gamma(\ell)$-equivariant cell structures and the action of $\Gamma(\ell)$ is the usual one.

Next we will need to describe the class of negligible groups.

\begin{remark} \label{remark:VCgroups}
If $G \in \mathcal{VC}$, then $G$ has one of three possible forms:
\begin{enumerate}
\item $G$ is finite, or
\item $G$ is \emph{infinite virtually cyclic of type I}; that is, $G$ admits a surjective homomorphism onto 
$\mathbb{Z}$ with finite kernel. Such a group will necessarily have the form $G \cong F \rtimes \mathbb{Z}$,
where $F$ is the kernel of the surjection onto $\mathbb{Z}$, or
\item $G$ is \emph{infinite virtually cyclic of type II}; that is, $G$ admits a surjective homomorphism onto 
$D_{\infty}$ with finite kernel. In this case, $G \cong A \ast_{F} B$, where $A$, $B$, and $F$ are finite groups,
and $F$ has index two in both $A$ and $B$.
\end{enumerate}
\end{remark}

\begin{definition} \label{definition:negligible1}
A group $G \in \mathcal{VC}$ is \emph{negligible} if
\begin{enumerate}
\item for each finite subgroup $H \leq G$, $H$ is isomorphic
to a subgroup of $S_{4}$ (the symmetric
group on four symbols), and 
\item if $G \in \mathcal{VC}_{\infty}$, then the finite group $F$ from Remark \ref{remark:VCgroups} has square-free
order. 
\end{enumerate}

(Thus, a finite group $G$ is negligible if it is isomorphic to a subgroup of $S_{4}$. An infinite virtually cyclic group
of type I is negligible if $F$ is of square-free order and isomorphic to a subgroup of $S_{4}$. An infinite virtually
cyclic group of type II is negligible if the factors $A$ and $B$ are isomorphic to subgroups of $S_{4}$, and
$F$ has square-free order.)

We will also say that
a cell $\sigma$ is negligible if its stabilizer group is negligible.
\end{definition}

\begin{remark}
This is the first of two different definitions of ``negligible" that we will
use in this paper.  We will need a different definition in Sections \ref{section:actionsonplanes} 
and \ref{section:cokernels}.

Definition \ref{definition:negligible1} allows us to describe classes of cells that make no contribution to $K$-theory (see Lemma \ref{lemma:negligibleisomorphismtypes} below),
which will let us ignore them in our work.
\end{remark}

\begin{lemma} \label{lemma:negligibleisomorphismtypes}
Let $G$ be a negligible group. The groups $Wh_{q}(G)$ are trivial for $q \leq 1$, and the same is true for all subgroups 
of $G$.
\end{lemma} 

\begin{proof}
We first note that subgroups of negligible groups are negligible.

If $G$ is finite and negligible, then $G$ is isomorphic to a subgroup of $S_{4}$; i.e., $G \cong \{1\}$, $\mathbb{Z}/2$, $\mathbb{Z}/3$, $\mathbb{Z}/4$, $D_{2}$, $D_{3}$, $D_{4}$,
$A_{4}$, or $S_{4}$. The lemma then follows from Table \ref{table:tableKtheoryEfin} and the accompanying discussion. (See also \cite{LO09}.)

Now we assume that $G$ is negligible and infinite virtually cyclic of type I. Therefore, $G \cong F \rtimes_{\alpha} \mathbb{Z}$, where $F$ is a subgroup of $S_{4}$ with square-free order. By results of
Farrell and Hsiang \cite{FH68} and Farrell and Jones \cite{FJ95},
$$ Wh_{q}(F \rtimes_{\alpha} \mathbb{Z}) \cong C \oplus NK_{q}(\mathbb{Z}F, \alpha) \oplus NK_{q}(\mathbb{Z}F, \alpha^{-1}),$$
 where $C$ is a suitable quotient of the group $Wh_{q-1}(F) \oplus Wh_{q}(F)$ and $q \leq 1$. Since $F$ is finite and negligible, $C$ is trivial. Therefore,
$$ Wh_{q}(F \rtimes_{\alpha} \mathbb{Z}) \cong  2NK_{q}(\mathbb{Z}F, \alpha),$$  
since Farrell and Hsiang also show that $NK_{q}(\mathbb{Z}F, \alpha) \cong NK_{q}(\mathbb{Z}F, \alpha^{-1})$. Since $F$ has square-free order, $NK_{q}(\mathbb{Z}F, \alpha)$ is trivial for $q \leq 1$ by results
of \cite{Ha87} and \cite{J-PR09}. (The case in which $\alpha = \mathrm{id}$ was established by Harmon \cite{Ha87}, and the general case is due to \cite{J-PR09}.) 
This proves the lemma in the case that $G$ is negligible and infinite virtually cyclic of type I.

Finally, we assume that $G$ is negligible and infinite virtually cyclic of type II. Therefore, we can write $G \cong G_{1} \ast_{F} G_{2}$, where $F$ has square-free order and index two in both factors, and both
$G_{1}$ and $G_{2}$ are isomorphic to subgroups of $S_{4}$. By results of \cite{Wal78} (see also \cite{CP02}), 
$$ Wh_{q}(G) \cong X \oplus NK_{q}(\mathbb{Z}F; \mathbb{Z}[G_{1} - F], \mathbb{Z}[G_{2} - F]),$$
for all $q \leq 1$, where $X$ is a suitable quotient of $Wh_{q}(G_{1}) \oplus Wh_{q}(G_{2})$. Since $G_{1}$ and $G_{2}$ are negligible, both factors in the latter direct sum are trivial, so $X$ is trivial. It follows that
$$ Wh_{q}(G) \cong NK_{q}(\mathbb{Z}F; \mathbb{Z}[G_{1} - F], \mathbb{Z}[G_{2} - F]).$$
Let $F \rtimes_{\alpha} \mathbb{Z}$ be the canonical index two subgroup of $G$. Since $NK_{q}(\mathbb{Z}F, \alpha)$ is trivial for $q \leq 1$ by the previous case, 
it follows that $NK_{q}(\mathbb{Z}F; \mathbb{Z}[G_{1} - F], \mathbb{Z}[G_{2} - F])$
is also trivial for $q \leq 1$, by results of Lafont and Ortiz \cite{LO08} (see also  \cite{DQR11} and \cite{DKR11}). It follows that $Wh_{q}(G)$ is trivial for $q \leq 1$ in this case as well.
\end{proof}

\begin{lemma} \label{lemma:negligiblesufficient}
Let $G \leq \Gamma(\ell)$.
\begin{enumerate}
\item If there is a line $\widehat{\ell}$ and a point $p \notin \widehat{\ell}$ such that $G$ fixes $p$ and leaves $\widehat{\ell}$
invariant, then $G$ is negligible.
\item If $G$ leaves a line $\widehat{\ell}$ invariant and $\widehat{\ell} \cap c \neq \emptyset$ for some open $2$-cell $c \subseteq \mathbb{R}^{3}$,
then $G$ is negligible.
\item If $G$ fixes two points $p_{1}, p_{2} \in \mathbb{R}^{3}$ and the line $\overleftrightarrow{p_{1}p_{2}}$ is not parallel to the line $\ell$, then
$G$ is negligible.
\item If $G$ leaves a strip $\widehat{\ell} \times [0, K]$ invariant and acts trivially on the second factor, then $G$ is negligible.
\end{enumerate}
\end{lemma}

\begin{proof}
\begin{enumerate}
\item Let $\widehat{p}$ be the point on $\widehat{\ell}$ that is closest to $p$. We let $v_{1}$ be the vector originating at $\widehat{p}$
and terminating at $p$. Let $v_{2}$ be a tangent vector to $\widehat{\ell}$ at $\widehat{p}$, and let $v_{3}$ be a vector that is perpendicular
to $v_{1}$ and $v_{2}$. We note that the vectors $v_{1}$, $v_{2}$, and $v_{3}$ are pairwise orthogonal.

The point $\widehat{p}$ must be fixed by $G$, and so $G$ must act by orthogonal matrices with respect to the 
ordered basis $[v_{1}, v_{2}, v_{3}]$.
By our assumptions, $G$ fixes $v_{1}$. The inclusion $G \cdot v_{2} \subseteq \{ v_{2}, -v_{2} \}$ holds, since
$\widehat{\ell}$ is $G$-invariant. It follows from orthogonality that $G \cdot v_{3} \subseteq \{ v_{3}, -v_{3} \}$ as well. We conclude that
$G$ is isomorphic to a subgroup of $( \mathbb{Z}/2)^{2}$, and therefore negligible.

\item Suppose that $G$ leaves $\widehat{\ell}$ invariant, and $\widehat{\ell} \cap c \neq \emptyset$ for some open $2$-cell $c \subseteq \mathbb{R}^{3}$.
We consider the restriction homomorphism $r: G \rightarrow \mathrm{Isom}(\widehat{\ell})$. The kernel of this map is a subgroup of the stabilizer group
of $c$. It follows that $| \mathrm{Ker} \, r | = 1$ or $2$. Thus, $G$ maps into $\mathbb{Z}$ or $D_{\infty}$ with $1$ or $\mathbb{Z}/2$ as
kernel, so $G$ is negligible.

\item Let $p_{1}$, $p_{2}$ be fixed by $G$. We consider the line $\widehat{\ell}$ through $p_{2}$ that is parallel to $\ell$. 
Since $\overleftrightarrow{p_{1}p_{2}}$ is not parallel to $\ell$, $p_{1} \notin \widehat{\ell}$. Since $G \subseteq \Gamma(\ell)$, $G$ leaves 
$\widehat{\ell}$ invariant. As a result, $G$ is negligible by (1).

\item If $G$ leaves a strip $\widehat{\ell} \times [0,K]$ invariant and acts trivially on the second coordinate, then $G$ acts
on each line $\widehat{\ell} \times \{ k \} \subseteq \widehat{\ell} \times [0,K]$. At least one of these lines must meet an open $2$-cell
in $\mathbb{R}^{3}$, so $G$ is negligible by (2).
\end{enumerate}
\end{proof}

\begin{corollary} \label{corollary:cellulated}
Let $\widehat{\ell} \in \mathbb{R}^{2}_{\ell}$ have a non-negligible stabilizer group. The point $\widehat{\ell} \in \mathbb{R}^{2}_{\ell}$ must be a
vertex, and the line $\widehat{\ell} \subseteq \mathbb{R}^{3}$ occurs as a cellulated subcomplex in $E_{\mathcal{FIN}}(\Gamma(\ell))$.
\end{corollary}

\begin{proof}
If $\widehat{\ell}$ is not a vertex of $\mathbb{R}^{2}_{\ell}$, then the stabilizer group of $\widehat{\ell}$ leaves a strip invariant in $\mathbb{R}^{3}$
and 
acts trivially on the bounded factor, so the stabilizer group of $\widehat{\ell}$ is negligible by Lemma \ref{lemma:negligiblesufficient}(4).

The second statement follows from the fact that $\widehat{\ell} \subseteq (\mathbb{R}^{3})^{1}$,
by Lemma \ref{lemma:negligiblesufficient}(2).
\end{proof}

\begin{definition}
For each vertex $\widehat{\ell} \in \mathbb{R}^{2}_{\ell}$ with non-negligible stabilizer, set
$$F(\widehat{\ell}) = \widehat{\ell} \subseteq \mathbb{R}^{3}.$$
Note that $\widehat{\ell}$ is a cellulated line in $\mathbb{R}^{3}$ by Corollary \ref{corollary:cellulated}, 
so $F(\widehat{\ell})$ is a subcomplex of our model for 
$E_{\mathcal{FIN}}(\Gamma(\ell))$.

Under the same assumptions on $\widehat{\ell}$, we also set
$$E(\widehat{\ell}) = F(\widehat{\ell}) \ast \widehat{\ell}.$$
We note that $E(\widehat{\ell})$ is a subcomplex of $E_{\mathcal{VC}_{\langle \ell \rangle}}(\Gamma(\ell)) = \mathbb{R}^{3} \ast 
\mathbb{R}^{2}_{\ell}$. 
\end{definition}

\begin{proposition} \label{proposition:subcomplexfinvc}
The subcomplexes
$$ F=\coprod_{ \widehat{\ell} } F(\widehat{\ell}),  \quad \text{and} 
\quad E=\coprod_{ \widehat{\ell} } E(\widehat{\ell})$$

of $E_{\mathcal{FIN}}(\Gamma(\ell))$ and $E_{\mathcal{VC}_{\langle \ell \rangle}}(\Gamma(\ell))$ (respectively) are
$\Gamma(\ell)$-equivariant. (The disjoint unions are indexed over all lines $\widehat{\ell} \in \mathbb{R}^{2}_{\ell}$
with non-negligible stabilizers.)

These subcomplexes also contain the only non-negligible cells in $E_{\mathcal{FIN}}(\Gamma(\ell))$ and 
$E_{\mathcal{VC}_{\langle \ell \rangle}}(\Gamma(\ell))$ (respectively). In particular, the natural inclusions
induce  isomorphisms
$$H_n^{\Gamma(\ell)}(F; \mathbb{KZ}^{-\infty}) \cong H_n^{\Gamma(\ell)}(E_{\fin}(\Gamma(\ell)); \mathbb{KZ}^{-\infty}),$$

$$H_n^{\Gamma(\ell)}(E; \mathbb{KZ}^{-\infty}) \cong H_n^{\Gamma(\ell)}(E_{\vc_{\langle \ell \rangle}}(\Gamma(\ell));  \mathbb{KZ}^{-\infty}).$$\end{proposition}

\begin{proof}
The statement that the given subcomplexes are $\Gamma(\ell)$-equivariant follows from the $\Gamma(\ell)$-equivariance
of the indexing sets.

Now we would like to show that the given subcomplexes contain the only non-negligible cells. It is good enough
to do this for the subcomplex $E = \coprod_{\widehat{\ell}} E(\widehat{\ell})$, since
$$F= \coprod_{\widehat{\ell}} F(\widehat{\ell}) =  E_{\mathcal{FIN}}(\Gamma(\ell))
\cap \coprod_{ \widehat{\ell}} E(\widehat{\ell}).$$
We will consider cells in the top of the join $\mathbb{R}^{3} \ast \mathbb{R}^{2}_{\ell}$ (i.e., in $\mathbb{R}^{3}$), then
cells in $\mathbb{R}^{2}_{\ell}$, and finally the cells that can be described as joins of cells from $\mathbb{R}^{3}$ and $\mathbb{R}^{2}_{\ell}$.

We first describe the collection of all cells $c \subseteq \mathbb{R}^{3}$ with non-negligible stabilizers. It is clear that
all such cells are $0$- or $1$-cells, since the stabilizer of a $2$-dimensional cell $\sigma \subseteq \mathbb{R}^{3}$ has order at most $2$, and
the stabilizer of a $3$-dimensional cell is necessarily trivial.

Suppose $c$ is a $0$-cell in $\mathbb{R}^{3}$ and $c$ has non-negligible stabilizer. 
We consider the line $\widehat{\ell}$ that 
is parallel to $\ell$ and passes through $c$. Thus $\widehat{\ell} \in \mathbb{R}^{2}_{\ell}$ has a non-negligible
stabilizer (since the stabilizer of $c$ is contained in the stabilizer of $\widehat{\ell}$), so it must be a vertex. 
It now follows that
$c \in E(\widehat{\ell})$.

Now suppose that $c \subseteq \mathbb{R}^{3}$ is an open $1$-cell with non-negligible stabilizer. We choose two points $p_{1}, p_{2} \in c$. The stabilizer group of $c$ fixes
both $p_{1}$ and $p_{2}$. Since the latter group is non-negligible, it must be that $\overleftrightarrow{p_{1}p_{2}} = \widehat{\ell}$
is parallel to $\ell$ by Lemma \ref{lemma:negligiblesufficient}(3), so $\widehat{\ell} \in \mathbb{R}^{2}_{\ell}$. The stabilizer group of
$\widehat{\ell}$ is non-negligible since it contains the stabilizer group of $c$. Thus, $\widehat{\ell}$ is a vertex in $\mathbb{R}^{2}_{\ell}$, 
and $c \subseteq E(\widehat{\ell})$. This concludes our analysis of cells in $\mathbb{R}^{3}$; we have shown that all non-negligible cells in $\mathbb{R}^{3}$
are contained in $E$.

We next consider the cells of $\mathbb{R}^{2}_{\ell}$.
Corollary \ref{corollary:cellulated} shows that each open cell $c \subseteq \mathbb{R}^{2}_{\ell}$ of dimension greater 
than $0$ has negligible stabilizer. 
Thus, if $c \subseteq \mathbb{R}^{2}_{\ell}$ has non-negligible stabilizer, then $c$ is a vertex, so $c \in E(c)$.

Finally, we consider the cells
$c = c_{1} \ast c_{2}$ having non-negligible stabilizer, where $c_{1} \subseteq \mathbb{R}^{3}$, $c_{2} \subseteq \mathbb{R}^{2}_{\ell}$
are open cells. Since $G_{c} = G_{c_{1}} \cap G_{c_{2}}$, both $c_{1}$ and $c_{2}$ have non-negligible stabilizer groups. It follows
from Corollary \ref{corollary:cellulated}  that
$c_{2}$ is a vertex (in $\mathbb{R}^{2}_{\ell})$ and a line in $\mathbb{R}^{3}$. Since $G_{c}$ is non-negligible, we must have $c_{1} \subseteq c_{2}$
by Lemma \ref{lemma:negligiblesufficient}(1). Thus, $c \subseteq E(c_{2})$. We have now shown that all of the non-negligible cells are in $E$.

The final statement now follows from Lemma \ref{lemma:negligibleisomorphismtypes}. Indeed, consider the inclusion of $F$ into $E_{\fin}(\Gamma(\ell))$. 
The only cells to have non-zero $K$-groups are contained in the image (by the above argument and Lemma \ref{lemma:negligibleisomorphismtypes}), proving
that the inclusion induces an isomorphism. The other case is similar. 
 \end{proof}
 
\begin{proposition} \label{proposition:borelfinvc}
We have a $\Gamma(\ell)$-homeomorphism 
$$h: \coprod_{ \widehat{\ell}  \in \mathcal{T}'}
\Gamma(\ell) \times_{\Gamma(\ell)_{\widehat{\ell}}} X(\widehat{\ell})
\rightarrow 
\coprod_{ \widehat{\ell} } X(\widehat{\ell}) 
,$$
where $X \in \{ E, F \}$ and $\mathcal{T}'$ is a selection of one vertex $ \widehat{\ell}$ from each $\Gamma(\ell)$-orbit of
non-negligible vertices in $\mathbb{R}^{2}_{\ell}$. The domain is a $\Gamma(\ell)$-space
relative to the action that is trivial on the second coordinate, and left multiplication on the first.

Moreover, each $F(\widehat{\ell})$ is a model for $E_{\mathcal{FIN}}(\Gamma_{\widehat{\ell}})$, and
each
$E(\widehat{\ell})$ is a model for $E_{\mathcal{VC}}(\Gamma_{\widehat{\ell}})$.
\end{proposition}

\begin{proof}
The argument is similar to the proof of Proposition \ref{prop:Ghomeo}. We will prove the proposition
in the case $X =E$, the other case being similar. Consider the commutative
diagram:

\[
\xymatrix{
\Gamma(\ell) \times \coprod_{ \widehat{\ell} \in \mathcal{T}'} E(\widehat{\ell}) 
 \ar[d]^{\pi_1}   \ar[dr]^{\pi_2} &\\ 
\coprod_{ \widehat{\ell} \in \mathcal{T}'}
\Gamma(\ell) \times_{\Gamma(\ell)_{\widehat{\ell}}} E(\widehat{\ell})
\ar[r]^{\hspace{28pt} h}  &
\coprod_{ \widehat{\ell}} E(\widehat{\ell})}
\] 
The space on top is a $\Gamma(\ell)$-space relative to the action $\gamma' \cdot (\gamma, x)
= (\gamma' \gamma, x)$.
We set $\pi_{1}(\gamma, x) = (\gamma, x)$ and $\pi_{2}(\gamma, x) = \gamma \cdot x$.     
As in the proof of Proposition \ref{prop:Ghomeo}, we will show that $\pi_1$ is constant on point inverses
of $\pi_2$, and $\pi_{2}$ is constant on point inverses of $\pi_{1}$. This will establish the existence of
the desired $\Gamma(\ell)$-homeomorphism $h$, since both $\pi_{1}$ and $\pi_{2}$ are quotient maps
that commute with the $\Gamma(\ell)$-action.

Choose $(\gamma, x) \in \coprod_{ \widehat{\ell} \in \mathcal{T}'}
\Gamma(\ell) \times_{\Gamma(\ell)_{\widehat{\ell}}} E(\widehat{\ell})$. We have
the equality 
$$\pi_{1}^{-1}(\gamma, x) = \{ (\gamma \gamma_{1}, \gamma_{1}^{-1} x) \mid \gamma_{1} \in 
\Gamma(\ell)_{\widehat{\ell}} \}.$$
It follows directly that $\pi_{2}(\pi_{1}^{-1}(\gamma, x)) = \{ \gamma \cdot x \}$, as required.

Now we choose an arbitrary $x \in \coprod_{ \widehat{\ell}} E(\widehat{\ell})$. 
We have
$$ \pi_{2}^{-1}(x) = \{ (\gamma, z) \mid \gamma \cdot z = x\}.$$
We choose two elements of the latter set, $(\gamma_{1}, z_{1})$ and $(\gamma_{2}, z_{2})$. It follows
directly that $\gamma_{2}^{-1} \gamma_{1} \cdot z_{1} = z_{2}$, so $z_{1}$ and $z_{2}$ are in
the same $\Gamma(\ell)$-orbit. Given the nature of the indexing set $\mathcal{T}'$, it must be
that both $z_{1}$ and $z_{2}$ are in $E(\widehat{\ell})$, for some non-negligible vertex $\widehat{\ell}
\in \mathbb{R}^{2}_{\ell}$. It now follows that 
$\gamma_{2}^{-1} \gamma_{1} \in \Gamma(\ell)_{\widehat{\ell}}$. Thus,
\begin{align*}
\pi_{1}(\gamma_{1}, z_{1}) &= (\gamma_{2} \gamma_{2}^{-1} \gamma_{1}, z_{1}) \\
&\sim (\gamma_{2}, \gamma_{2}^{-1} \gamma_{1} z_{1}) \\
&= (\gamma_{2}, z_{2}) \\
&= \pi_{1} (\gamma_{2}, z_{2})
\end{align*}
It follows that $\pi_{1}$ is constant on point inverses of $\pi_{2}$, as required. The existence of
the homeomorphism $h$ follows directly.

Finally, we note that $\Gamma(\ell)_{\widehat{\ell}} = \Gamma_{\widehat{\ell}}$ is an
infinite virtually cyclic group. Since $F(\widehat{\ell})$
is simply a cellulated line, and $E(\widehat{\ell})$ is the join of $F(\widehat{\ell})$ with a point, both
are well-known models for $E_{\mathcal{FIN}}(\Gamma_{\widehat{\ell}})$ and
$E_{\mathcal{VC}}(\Gamma_{\widehat{\ell}})$, respectively.

\end{proof}

\begin{remark}
We note that $\Gamma_{\widehat{\ell}}$ denotes the same subgroup of $\Gamma$
no matter whether we view $\widehat{\ell}$ as a vertex in $\mathbb{R}^{2}_{\ell}$
or as a line in $\mathbb{R}^{3}$. 
\end{remark}

\begin{proof}[Proof of Theorem \ref{theorem:splitting2}] Combining Proposition \ref{proposition:subcomplexfinvc}, Proposition \ref{proposition:borelfinvc}, and the fact  
that our equivariant generalized homology theory turns disjoint unions into direct sums, we obtain the following isomorphisms
\[
H_n^{\Gamma(\ell)}(E_{\fin}(\Gamma(\ell)); \mathbb{KZ}^{-\infty}) \cong \bigoplus _{\widehat{\ell} \in \mathcal{T}'} H_n^{\Gamma_{\widehat{\ell}}}(E_{\mathcal{FIN}}(\Gamma_{\widehat{\ell}});  \mathbb{KZ}^{-\infty}), \;\; \text{and}
\]
  \[
H_n^{\Gamma(\ell)}(E_{\vc_{\langle \ell \rangle}}(\Gamma(\ell));  \mathbb{KZ}^{-\infty}) \cong \bigoplus _{ \widehat{\ell}  \in \mathcal{T}'} H_n^{\Gamma_{\widehat{\ell}}}(E_{\mathcal{VC}}(\Gamma_{\widehat{\ell}});  \mathbb{KZ}^{-\infty}).
\]
We immediately get an identification of the cokernel of the relative assembly map $H_n^{\Gamma(\ell)}(E_{\fin}(\Gamma(\ell)); \mathbb{KZ}^{-\infty}) \rightarrow H_n^{\Gamma(\ell)}(E_{\vc_{\langle \ell \rangle}}(\Gamma(\ell));  \mathbb{KZ}^{-\infty})$ with the direct sum of the cokernels of the relative assembly maps 

$$H_n^{\Gamma_{\widehat{\ell}}}(E_{\mathcal{FIN}}(\Gamma_{\widehat{\ell}});  \mathbb{KZ}^{-\infty}) \rightarrow H_n^{\Gamma_{\widehat{\ell}}}(E_{\mathcal{VC}}(\Gamma_{\widehat{\ell}});  \mathbb{KZ}^{-\infty}).$$  
Since $\Gamma_{\widehat{\ell}} \in \vc$, then $E_{\mathcal{VC}}(\Gamma_{\widehat{\ell}})= \{\ast\}$, and the summands in Proposition \ref{proposition:splitting1} are
\[
H_n^{\Gamma(\ell)}(E_{\fin}(\Gamma(\ell))\rightarrow
E_{\vc_{\langle \ell \rangle}}(\Gamma(\ell))) \cong \bigoplus _{ \widehat{\ell} \in \mathcal{T}'} H_n^{\Gamma_{\widehat{\ell}}}(E_{\mathcal{FIN}}(\Gamma_{\widehat{\ell}}) \rightarrow \ast).
\]

Finally, combining these observations with Proposition \ref{proposition:splitting1} 
completes the proof.
\end{proof}

\vskip 20 pt
\section{Fundamental Domains for the Maximal Split Crystallographic Groups} \label{section:fundamentaldomains}

The classification of split three-dimensional crystallographic groups from Theorem \ref{biglist}
shows that seven of the groups contain all of the others as subgroups. For $i=1, \ldots, 7$, we let
$\Gamma_{i} = \langle L_{i}, H_{i} \rangle$, where $L_{i}$ is the $i$th lattice (in the order that the lattices are listed
in Table \ref{splitcrystallographicgroups}) and $H_{i}$ is the maximal point group to be paired with $L_{i}$. For instance,
$$ \Gamma_{4} = \left\langle \left\langle \frac{1}{2}(\mbf{x} + \mbf{z}), \mbf{y}, \mbf{z} 
\right\rangle, D_{2}^{+} \times (-1) \right\rangle.$$ 

For the computations of $K$-groups in subsequent sections, we 
will need to find fundamental polyhedra for  the groups $\Gamma_{i}$.  
We review a case of Poincare's fundamental polyhedron theorem in Subsection \ref{specialcase}, describe equivariant
cell structures and cell stabilizers in Subsection \ref{cell}, 
and then describe the seven fundamental domains in the remaining subsections.

\subsection{Special Case of Poincare's Fundamental Polyhedron Theorem for $\mathbb{R}^{n}$}\label{specialcase}

In this subsection, we collect a number of results from \cite{Ra94}.  We state each result
only for $\mathbb{R}^{n}$ and only for convex compact polyhedra, although usually  
the corresponding theorem (or definition) in \cite{Ra94}
is more general.  All page citations in the current subsection are from \cite{Ra94}, unless otherwise noted.

\begin{definition}
If $P \subseteq \mathbb{R}^{n}$ is the set of solutions to a system of finitely many linear inequalities, and $P$ is
compact, then we say that $P$ is a \emph{convex compact polyhedron}.  Suppose that $P$ is $m$-dimensional.  
We let $\partial P$ denote the topological
boundary of $P$ in the unique $m$-plane $\langle P \rangle$ containing $P$. The
\emph{interior} of $P$ is $P - \partial P$.  A \emph{side}
of a convex compact polyhedron is a non-empty maximal convex subset of $\partial P$.  If $P$ has dimension $m>0$, then
each side of $P$ is a convex compact polyhedron of dimension $m-1$.  A \emph{ridge} of $P$ is a side of a side of $P$.       
 
(Notes:  our definition of convex compact polyhedra combines Ratcliffe's definition of convex polyhedra (pg. 205) with
his characterization of convex compact polyhedra (Theorem 6.3.7, pg. 209).  The definitions of $\partial P$, side, and ridge
occur on pages 199, 202, and 207, respectively.)
\end{definition}

\begin{definition}
Let $P$ be a convex compact $n$-dimensional 
polyhedron in $\mathbb{R}^{n}$.  A \emph{side-pairing} (cf. $G$-side-pairing, pg. 694) for
$P$ is a set
$$ \Phi = \{ \phi_{S} \in \iso(\mathbb{R}^{n}) \mid S \in \mathcal{S} \}$$
indexed by the collection $\mathcal{S}$ of all sides of $P$ such that, for each $S \in \mathcal{S}$,
\begin{enumerate}
\item there is a side $S'$ of $P$ such that $\phi_{S}(S') = S$;
\item if $\phi_{S}(S') = S$, then the isometries $\phi_{S}$ and $\phi_{S'}$
satisfy $\phi_{S'} = \phi_{S}^{-1}$;
\item the polyhedra $P$ and $\phi_{S}(P)$ satisfy $P \cap \phi_{S}(P) = S$.
\end{enumerate}
The pairing of side points by elements of $\Phi$ generates an equivalence relation on $\partial P$.
The equivalence classes are called \emph{cycles} of $\Phi$ (pg. 694).  We let $[x]$ denote the cycle
containing $x$.  We say that the cycle $[x]$ is a \emph{ridge cycle} (pg. 694) if some (equivalently, any)
representative of $[x]$ lies in the interior of a ridge of $P$.  

Let $[x] = \{ x_{1}, \ldots, x_{m} \}$ be a finite ridge cycle of $\Phi$.  
Each $x_{i}$ is contained in exactly two sides of $P$ (Theorem 6.3.6, pg. 207), so
$x_{i}$ is paired to at most two other points of $[x]$ for each $i$.  Therefore,
we can reindex the set $\{ x_{1}, \ldots, x_{m} \}$ such that
$$ x_{1} \simeq x_{2} \ldots \simeq x_{m},$$
where, for $i=1, \ldots, m-1$, $x_{i} \simeq x_{i+1}$ if and only if some element $\phi (x_{i}) = x_{i+1}$ for
some $\phi \in \Phi$.
The ridge cycle $[x]$ is said to be \emph{dihedral} (pg. 695) if the  sides $S_{1}$ and $S_{m}$
 of $P$ are 
such that $x_{i} \in S_{i}$ and $\phi_{S_{i}}$ is the reflection of $\mathbb{R}^{n}$ in
$\langle S_{i} \rangle$ for $i=1,m$. (Note that if $\phi_{S_{i}}$ is a reflection
in $S_{i}$, then $x_{i}$ is paired with at most one other point by the relation
$\simeq$, so such points $x_{i}$ can only appear at the beginning or
end of the sequence $x_{1}, \ldots, x_{m}$.)
Otherwise, $[x]$ is said to be \emph{cyclic} (pg. 695). 

If $S_{1}$ and $S_{2}$ are two sides of a convex polyhedron $P$, and the vectors
$N_{1}$ and $N_{2}$ are the outward-pointing unit normal vectors, then the \emph{dihedral angle}
between $S_{1}$ and $S_{2}$ is
$$ \theta( S_{1}, S_{2} ) = \pi - \arccos (N_{1} \cdot N_{2}).$$
 The \emph{dihedral angle sum} (pg. 695) of the ridge cycle $[x]$
is 
$$ \theta[x] = \theta_{1} + \ldots + \theta_{m},$$
where $\theta_{i}$ is the dihedral angle between the two sides containing $x_{i}$.

The side-pairing $\Phi$ is \emph{subproper} (pg. 695) if and only if each cycle of $\Phi$ is finite, each
dihedral ridge cycle of $\Phi$ has dihedral angle sum an integral submultiple of $\pi$, and each
cyclic ridge cycle of $\Phi$ has dihedral angle sum an integral submultiple of $2\pi$.
\end{definition}

\begin{definition} \label{definition:fundamentaldomain}
A subset $R$ of $\mathbb{R}^{n}$ is a \emph{fundamental domain} (pg. 234) for a group $\Gamma \leq Isom(\mathbb{R}^{n})$
if and only if
\begin{enumerate}
\item the set $R$ is open in $\mathbb{R}^{n}$;
\item the members of $\{ gR \mid g \in \Gamma \}$ are pairwise disjoint;
\item $\mathbb{R}^{n} = \bigcup \{ g \overline{R} \mid g \in \Gamma \}$, and
\item $R$ is connected.
\end{enumerate}
We say that $R$ is \emph{locally finite} (pg. 237) if the collection $\{ g \overline{R} \mid g \in \Gamma \}$
is locally finite, i.e., if, for every $x \in \mathbb{R}^{n}$, there is an open ball around $x$ meeting
only finitely many members of $\{ g \overline{R} \mid g \in \Gamma \}$.

A \emph{convex compact fundamental polyhedron} (pg. 247) for a discrete group $\Gamma \leq \iso(\mathbb{R}^{n})$ is
a convex compact polyhedron $P$ such that the interior of $P$ is a locally finite fundamental domain for $\Gamma$.
We say that such a $P$ is \emph{exact} (pg. 250) if and only if for each side $S$ of $P$ there is an element $g \in \Gamma$
such that $S = P \cap gP$.  (If $P$ is exact, then the element $g$ is unique (Theorem 6.6.3, pg. 251).)   
\end{definition}

We can now state the relevant special case of Poincar\'{e}'s fundamental polyhedron theorem.

\begin{theorem}\label{poincare}
(\cite[pg. 711]{Ra94})
If $\Phi$ is a subproper side-pairing of a convex compact polyhedron $P \subseteq \mathbb{R}^{n}$, then the group $\Gamma$
generated by $\Phi$ is discrete, and $P$ is an exact convex compact fundamental polyhedron for $\Gamma$.
\end{theorem}
\begin{proof}
The general statement in \cite{Ra94} has an additional condition in the hypothesis --
namely, that the $(\mathbb{R}^{n}, \iso(\mathbb{R}^{n}))-$orbifold $M$ obtained from $P$ by gluing
together the sides of $P$ by $\Phi$ is complete.  For convex compact Euclidean polyhedra $P$, this condition
is automatically satisfied under the given hypothesis (Theorem 13.4.2, pg. 704).
\end{proof}

\begin{corollary} \label{coxeter}
If $P$ is a convex compact $n$-dimensional polyhedron in $\mathbb{R}^{n}$, and the dihedral angle between any pair
of adjacent sides of $P$ is a submultiple of $\pi$, then the group $\Gamma$ generated by the reflections
in all of the sides of $P$ is discrete, and $P$ is an exact convex compact fundamental polyhedron
for $\Gamma$.
\end{corollary}
\begin{proof}
One easily checks that the collection of reflections $\Phi$ is a side-pairing.  Each cycle in this side-pairing
consists of a single point $x \in P$, so the dihedral angle sum of any ridge cycle is a submultiple of
$\pi$ by the assumption about $P$.  It follows that $\Phi$ is a subproper side-pairing, and one can apply
Theorem \ref{poincare}.
\end{proof}

\subsection{Cell Structures and Stabilizers} \label{cell}

\subsubsection{Standard Cellulations and Equivariant Cell Structures}

\begin{definition} \label{cellulation}
An \emph{open cell of dimension $n$} is a
space homeomorphic to $(0,1)^{n}$, or, if $n=0$, to a singleton set.
If $X$ is a Hausdorff space, then a \emph{CW structure} on $X$
\cite[Proposition A.2]{Hatcher}
is a collection $\mathcal{C}$ of open cells in $X$ such that
\begin{enumerate}
\item the elements of $\mathcal{C}$ are pairwise disjoint, and their union is $X$;
\item the boundary $\partial (e) = \bar{e} - e$ of an element $e \in \mathcal{C}$
is contained in the union of elements from $\mathcal{C}$ of lower dimension, and
\item a subset $U$ of $X$ is closed if and only if $U \cap e$ is closed as a subset of $e$,
for all $e \in \mathcal{C}$.
\end{enumerate}
\end{definition}   

\begin{definition} \label{subdivision}
Let $P$ be a convex compact three-dimensional polyhedron in $\mathbb{R}^{3}$, and let $\Phi$ be a subproper
side-pairing of $P$.  If all cycles $[x]$ of $\Phi$ meet the interiors of ridges and sides of $P$ in at most one point, 
then the
\emph{standard cellulation} of $P$ is the set whose members are vertices of $P$, interiors of ridges of $P$,
interiors of sides of $P$, and the interior of $P$ itself.

If a cycle $[x]$ meets the interior of a ridge or a side in two points, then we call this ridge or side \emph{bad}. (Note that, under the current hypotheses, 
it is impossible for a cycle $[x]$ to meet the interior of a ridge or side in more than two points.) 
If $P$ has bad ridges or bad sides, then we divide each bad ridge or side exactly in half to arrive at the standard
cellulation of $P$.  The operation of subdivision is self-explanatory in the case of bad ridges.  If a side $S$
is bad, then we choose two points $x, y \in S$ such that both are in the same
cycle, and both lie in the interior of $S$. We let $\ell$ denote the perpendicular 
bisector of $[x,y]$ in $\langle S \rangle$. The isometry 
$\phi_{S} \in \mathrm{Isom}(\mathbb{R}^{3})$ maps the side $S$
to itself, and interchanges $x$ and $y$. It follows that $\phi_{S}(\ell) = \ell$.
Now either $\phi_{S}$ fixes $\ell$ pointwise
or $\phi_{S}$ reverses the orientation of $\ell$. In the former case,
we can simply subdivide the side $S$ along $\ell$. In the latter case,
we must subdivide $S$ along $\ell$ and
also introduce a vertex at the midpoint of $\ell \cap S$.
\end{definition}

\begin{theorem} \label{cells}
(Equivariant Cell Structures) 
If $\Phi$ is a subproper side-pairing of a convex compact three-dimensional polyhedron $P \subseteq \mathbb{R}^{3}$, and 
$\Gamma = \langle \Phi \rangle$, then the standard cellulation $\mathcal{C}$ of $P$ extends to a
$\Gamma$-equivariant CW structure $\widehat{\mathcal{C}}$ on all of $\mathbb{R}^{3}$.  

If $g \in \Gamma$ leaves a cell $e \in \widehat{\mathcal{C}}$ invariant, then $g$ fixes $e$ pointwise.   
\end{theorem}
 
\begin{proof}
Theorem \ref{poincare} shows that $P$ is an exact convex compact fundamental polyhedron for the action
of $\Gamma$ on $\mathbb{R}^{n}$, and that $\Gamma$ is discrete. Ratcliffe
 \cite[Theorem 6.7.1]{Ra94} implies that
$$ \mathcal{P} = \{ gP \mid g \in \Gamma \}$$
is an exact tessellation of $\mathbb{R}^{n}$.  This means \cite[page 251]{Ra94}
that $\mathcal{P}$ satisfies the following conditions:
\begin{enumerate}
\item if $g_{1}, g_{2} \in \Gamma$, and $g_{1} \neq g_{2}$, then the interiors of
$g_{1}P$ and $g_{2}P$ are disjoint;
\item the union of the polyhedra in $\mathcal{P}$ is $\mathbb{R}^{n}$;
\item the collection $\mathcal{P}$ is locally finite, and
\item each side of a polyhedron in $\mathcal{P}$ is a side of exactly
two polyhedra $P$ and $Q$ in $\mathcal{P}$.  (This last condition is the definition
of exactness.)
\end{enumerate}  

It is not difficult to see that $\Gamma \cdot \mathcal{C}$ will be an equivariant CW complex structure
on all of $\mathbb{R}^{n}$ if and only if the elements of $\Gamma \cdot \mathcal{C}$ are pairwise
disjoint.  We therefore suppose that two cells $g_{1} \cdot e_{1}$, $g_{2} \cdot e_{2}$ have a point in common, for
some $e_{1}, e_{2} \in \mathcal{C}$.  We wish to show that $g_{1} \cdot e_{1} = g_{2} \cdot e_{2}$.  
We can clearly assume, without loss of generality, that $g_{2} = 1$ and that $\mathrm{dim} \, e_{1} \leq \mathrm{dim} \, e_{2}$.  

If $e_{2}$ is a three-dimensional open cell (i.e., the interior of $P$), 
then $g_{1} \cdot e_{1}$ is contained in the closure
of $g_{1} \cdot e_{2}$.  It follows that $g_{1} \cdot e_{2}$ and $e_{2}$ must have a point in common, so
$g_{1} = 1$ by property (1) of tessellations.  This implies that $e_{1}$ and $e_{2}$ have a point in common,
which can only mean that $e_{1} = e_{2}$, since $\mathcal{C}$ is a CW structure.

If $e_{2}$ is a two-dimensional open cell or a one-dimensional open cell contained in the interior of a bad side,
then there is a unique side $S$ of $P$ containing $e_{2}$.  We have
$e_{2} \subseteq \mathrm{int} \left( P \cup \phi_{S}(P)\right)$.  Since distinct translates of $P$ can meet only
in their boundaries, and $(g_{1} P) \cap \, \mathrm{int}(S) \neq \emptyset$ by our assumption, 
we must have $g_{1} = \phi_{S}$ or $g_{1} = 1$, and we can assume that $g_{1} = \phi_{S}$. The
isometry $\phi_{S}$ restricts to a bijection from $S'$ to $S$, where $S'$ is a side of $P$, and we allow the possibility
that $S' = S$.  Since, by the construction of $\mathcal{C}$,
the map $\phi_{S}: S' \rightarrow S$ is a bijection mapping cells of $\mathcal{C}$ to cells
of $\mathcal{C}$, the only possibility is that $e_{1} = \phi_{S}^{-1}(e_{2})$, so $g_{1} \cdot e_{1} = e_{2}$.

Now suppose $e_{2}$ is a one-dimensional open cell contained in a unique ridge $R$ of $P$. 
Ratcliffe \cite[Theorem 6.7.6]{Ra94} proves 
that if $R$ is a ridge of a polyhedron $P$ in an exact tessellation $\mathcal{P}$ 
of $\mathbb{R}^{3}$, then the set of all polyhedra in $\mathcal{P}$ containing $R$ forms a cycle whose intersection
is $R$.  We avoid reproducing the definition of a cycle of polyhedra here (see
\cite[pg. 256]{Ra94}, but this theorem 
implies that
$R$ is a ridge of each of the polyhedra in the cycle, and that the interior of $R$ lies in the interior of the union
of the polyhedra in the cycle.  It follows that there is some ridge $R_{1}$ of $P$ such that $e_{1} \subseteq R_{1}$ and
$g_{1} \cdot R_{1} = R$.

We can conclude that $g_{1} \cdot e_{1} = e_{2}$ provided that $R_{1}$ is a bad ridge if and only if $R$ is bad as well.
Suppose that $R_{1}$ is a bad ridge.  This means that there is $x \in R_{1}$ such that $[x]$ meets the interior of $R_{1}$
in two points, say $x$ and $x'$.  It follows that $g_{1} \cdot [x]$ meets the interior of $R$ in two points, $g_{1} \cdot x$
and $g_{1} \cdot x'$.  Theorem 6.7.5 from \cite{Ra94} says that $[x] = P \cap \Gamma x$, so 
$g_{1} \cdot x, g_{1} \cdot x' \in [x]$.  It follows that $R$ is bad.  The converse is proved in the same way.  It follows
that $g_{1} \cdot e_{1} = e_{2}$.

Finally, there is nothing to prove if $e_{2}$ is $0$-dimensional.

The final statement is an easy consequence of the fact that each cycle $[x]$ meets a given cell of $\mathcal{C}$ in at most
one point, and the fact that $[x] = P \cap \Gamma \cdot x$.
\end{proof}

\subsubsection{Computation of Cell Stabilizers and Negligible Groups}

We will soon want to compute the stabilizer groups $\Gamma_{v}$, where
$\Gamma$ is a split crystallographic group and $v \in \mathbb{R}^{3}$. The 
following lemma affords an effective procedure for computing such groups. The
proof is elementary and will be omitted.

\begin{lemma} \label{pointstab}
Assume that $\Gamma$ is a split crystallographic group: 
$\Gamma = \langle L, H \rangle$.  Let
$\pi: \Gamma \rightarrow H$ be the natural projection into the point group $H$.
\begin{enumerate}
\item $\pi ( \Gamma_{v} ) = \{ h \in H \mid v - hv \in L \}$, and
\item the map $\pi: \Gamma_{v} \rightarrow H$ is injective.  Thus, we can uniquely
recover $\Gamma_{v}$ from $\pi (\Gamma_{v})$.
\end{enumerate} \qed
\end{lemma}

Let $\sigma \subseteq \mathbb{R}^{3}$ be a cell. If the stabilizer
group $\Gamma_{\sigma}$ and all of its subgroups $K$ satisfy
$Wh_{q}(K) = 0$ for $q \leq 1$, then $\sigma$ makes no contribution
in the calculation of $Wh_{q}(\Gamma)$ ($q \leq 1$). The following definition
will give us a systematic way to ignore such cells, based on the isomorphism types
of their stabilizer groups.

\begin{definition} \label{definition:negligible}
A group is \emph{negligible} if it is isomorphic to a subgroup of
$S_{4}$. We will also say that a cell is negligible if its stabilizer group is negligible in the above sense.
\end{definition}

\begin{remark} \label{remark:newneg}
We note that Definition \ref{definition:negligible} is equivalent to Definition \ref{definition:negligible1} for
finite groups $G$.
\end{remark}



\begin{proposition} \label{negproposition}
If $\Gamma = \langle L, H \rangle$, where $H \leq S_{4}^{+} \times (-1)$,
and $v \in \mathbb{R}^{3}$ satisfies $2v \not \in L$, then the stabilizer group $\Gamma_{v}$ is negligible.
\end{proposition}

\begin{proof}
Note that $v - (-1)v = 2v \notin L$, so $(-1) \notin \pi(\Gamma_{v})$ by Lemma \ref{pointstab}(1).
The homomorphism $\pi: \Gamma_{v} \rightarrow H$ is injective, so $\Gamma_{v}$
is isomorphic to a subgroup of $S_{4}^{+} \times (-1)$ that does not contain
$(-1)$. All such groups are isomorphic to subgroups of $S_{4}$, so $\Gamma_{v}$
is negligible.
\end{proof}

\subsection{A Fundamental Polyhedron for $\Gamma_{1}$} \label{symlc}

Recall that $\Gamma_{1} = \langle \mathbf{x}, \mathbf{y}, \mathbf{z} \rangle
\rtimes (S_{4}^{+} \times (-1))$.
We consider the convex compact polyhedron 
$$ P = \left\{ (x,y,z) \in \mathbb{R}^{3} \mid 0 \leq z \leq y \leq x \leq \frac{1}{2} \right\}.$$
The set $P$ is the tetrahedron pictured in Figure \ref{pictureofgamma1}. (Note
that the shape of the given tetrahedron is not intended to be accurate.)

\begin{figure}[!h]
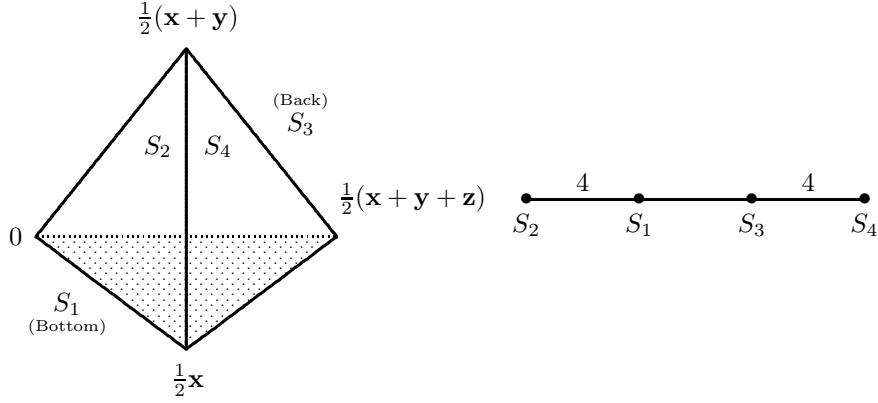
   
    
\begin{center}
\hbox{
\vbox{\beginpicture
\setcoordinatesystem units <1cm,1cm> point at -.2 3
\setplotarea x from -.5 to 9.2, y from -2 to 3
\linethickness=.7pt
%
\def\myarrow{\arrow <4pt> [.2, 1]} 
\setplotsymbol ({\circle*{.4}})
\plotsymbolspacing=.3pt        
\plot 0 0  2 2.5  4 0  2 -1.5  0 0 /
\plot 2 2.5  2 -1.5 /
\setdots <2pt>  
\putrule from 0 0 to 4 0
\put {$0$} [r] at -.2 0
\put {${\frac{1}{2}({\mathbf x}+{\mathbf y})}$} [b] at 2 2.7
\put {${\frac{1}{2}({\mathbf x}+{\mathbf y}+{\mathbf z})}$} [bl] at 4 .3
\put {${\frac{1}{2}{\mathbf x}}$} [t] at 2 -1.7
\put {$S_1$} at .4 -.9
\put {\tiny (Bottom)} at .4 -1.2
\put {$S_2$} at 1.6 1.2 
\put {$S_3$} at 3.5 1.5 
\put {\tiny (Back)} at 3.5 1.8
\put {$S_4$} at 2.4 1.2 
\setshadegrid span <2pt>
\hshade -1.5 2 2  0 0 4 /       

\setsolid
\putrule from 6.5 .5 to 11 .5
\put {$\bullet$} at  6.5 .5
\put {$S_2$} [t] at  6.5 .3
\put {$4$} [b]   at  7.25 .6
\put {$\bullet$} at  8 .5
\put {$S_1$} [t] at  8 .3
\put {$\bullet$} at  9.5 .5
\put {$S_3$} [t] at  9.5 .3
\put {$\bullet$} at  11 .5 
\put {$S_4$} [t] at  11 .3
\put {$4$} [b]   at  10.25 .6
\endpicture}
}
\end{center}

\hfill

\caption{The tetrahedron on the left is a fundamental domain for the action
  of $\Gamma_{1}$ on ${\mathbb R}^3$. The group
  $\Gamma_{1}$ is a Coxeter group.  Its Coxeter diagram
  appears on the right. }
\label{pictureofgamma1}
\end{figure}

The sides $\{ S_{1}, S_{2}, S_{3}, S_{4}, \}$ are contained in the planes $y=z$, $z=0$, $x=y$, and $x= 1/2$, respectively.
The outward unit normal vectors $N_{1}, N_{2}, N_{3}, N_{4}$ are, respectively,
$$ \frac{1}{\sqrt{2}} \left( -\mathbf{y} + \mathbf{z} \right), \quad -\mathbf{z}, \quad 
\frac{1}{\sqrt{2}} \left( -\mathbf{x} + \mathbf{y} \right), \quad \mathbf{x}.$$ 
The dihedral angles are easy to compute:
\begin{eqnarray*}
\theta ( S_{1}, S_{2}) = \pi /4, & \quad & \theta ( S_{1}, S_{3}) = \pi /3, \\
\theta ( S_{1}, S_{4}) = \pi /2, & \quad & \theta ( S_{2}, S_{3}) = \pi /2, \\
\theta ( S_{2}, S_{4}) = \pi /2, & \quad & \theta( S_{3}, S_{4}) = \pi /4.
\end{eqnarray*}
Since all of these are submultiples of $\pi$, it follows from Corollary \ref{coxeter}
that the group $\Gamma_{P}$ generated by reflections in the sides of $P$
is discrete, and that $P$ is an exact, convex compact fundamental polyhedron for $\Gamma_{P}$.  Furthermore,
$$\Gamma_{P}= \left\langle 
\left( \begin{smallmatrix} 1 & 0 & 0 \\ 0 & 0 & 1 \\ 0 & 1 & 0 \\ \end{smallmatrix} \right),
\left( \begin{smallmatrix} 1 & 0 & 0 \\ 0 & 1 & 0 \\ 0 & 0 & -1 \\ \end{smallmatrix} \right), 
\left( \begin{smallmatrix} 0 & 1 & 0 \\ 1 & 0 & 0 \\ 0 & 0 & 1 \\ \end{smallmatrix} \right),
\left( \begin{smallmatrix} 1 \\ 0 \\ 0 \end{smallmatrix} \right) + 
\left( \begin{smallmatrix} -1 & 0 & 0 \\ 0 & 1 & 0 \\ 0 & 0 & 1 \\ \end{smallmatrix} \right) \right\rangle,$$
where the isometries listed between the brackets are the reflections in $S_{1}$, $S_{2}$, $S_{3}$, and $S_{4}$,
respectively.

\vspace{.2cm}
It is fairly easy to check that $\Gamma_{P} = \Gamma_{1}$, so $P$ is an exact convex compact fundamental
polyhedron for the action of $\Gamma_{1}$ on $\mathbb{R}^{3}$.

\begin{theorem} \label{theorem:g1} 
Let $\widehat{\mathcal{C}}$ denote the $\Gamma_{1}$-equivariant cell structure on $\mathbb{R}^{3}$ determined by the standard
cellulation $\mathcal{C}$ of $P$.  The quotient $\Gamma_{1} \backslash \mathbb{R}^{3}$ is $P$ itself (see Figure \ref{pictureofgamma1}), endowed with the standard
cellulation $\mathcal{C}$.  The vertex stabilizer groups are determined by the following equalities (we write $\Gamma$ in place of $\Gamma_{1}$):
\begin{eqnarray*}
\pi(\Gamma_{(0,0,0)}) & = & S_{4}^{+} \times (-1). \\
\pi(\Gamma_{(1/2,0,0)}) & = & D_{4_{1}}^{+} \times (-1). \\
\pi(\Gamma_{(1/2,1/2,0)}) & = & D_{4}^{+} \times (-1). \\
\pi(\Gamma_{(1/2,1/2,1/2)}) & = & S_{4}^{+} \times (-1).
\end{eqnarray*}
The stabilizer groups of all other cells are negligible.
\end{theorem}

\begin{proof}
The first statement follows from the fact that each cycle $[x]$ is a singleton, and $[x] = P \cap (\Gamma \cdot x)$.

We now describe the vertex stabilizer groups.  The first equality is clear.  By Lemma \ref{pointstab}(1), an 
element $h \in S_{4}^{+} \times (-1)$
is in $\pi(\Gamma_{(1/2,0,0)})$ if and only if $(1/2,0,0) - h \cdot (1/2,0,0) \in L$.  It is clear that the latter condition
is satisfied exactly when the upper left corner of the matrix $h$ is $\pm 1$.  This proves the second equality.  One proves
the third equality in a similar way:  Lemma \ref{pointstab}(1) implies that $h \in S_{4}^{+} \times (-1)$ is in
$\pi(\Gamma_{(1/2,1/2,0)})$ if and only if the bottom right entry in $h$ is $\pm 1$.  The fourth equality is straightforward:
if $h \in S_{4}^{+} \times (-1)$, then $h \cdot (1/2,1/2,1/2) = (\pm 1/2, \pm 1/2, \pm 1/2)$, where the signs may be chosen
independently.  It is clear then that $(1/2,1/2,1/2) - h \cdot (1/2,1/2,1/2)$ has integral entries, regardless of the choice
of $h$.             

Proposition \ref{negproposition} shows that each edge stabilizer is negligible.  Indeed, the stabilizer of an edge is the same as
the stabilizer of its midpoint, and each of these midpoints has at least one entry which is $1/4$.  Since the edge stabilizers
are negligible, the stabilizers of all higher-dimensional cells are negligible as well.
\end{proof}
 
\subsection{A Fundamental Polyhedron for $\Gamma_{2}$}
We recall that 
$$\Gamma_{2} = \left\langle \frac{1}{2}( \mathbf{x} + \mathbf{y} + \mathbf{z}), \mathbf{y}, 
\mathbf{z} \right\rangle \rtimes (S_{4}^{+} \times (-1)).$$


\begin{figure}[!h]
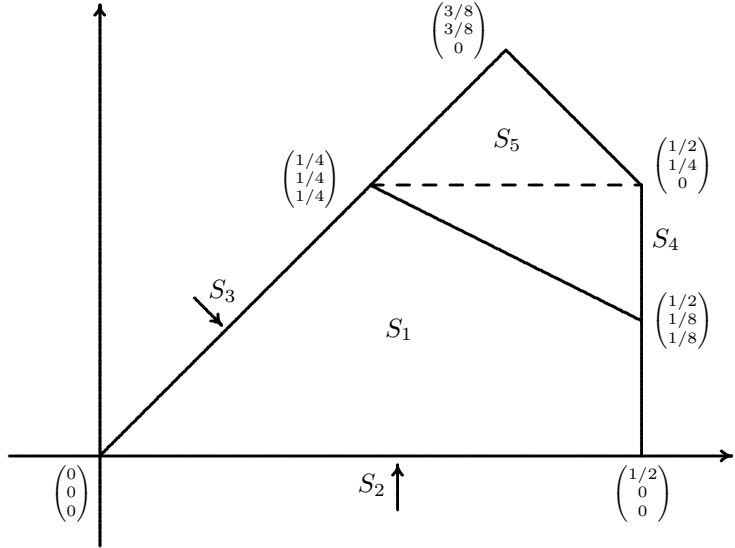

    
\begin{center}
\hbox{
\vbox{\beginpicture
\setcoordinatesystem units <1.2cm,1.2cm> point at -.2 4
\setplotarea x from -2 to 7, y from -.6 to 5.5
\linethickness=.7pt
%

 \def\myarrow{\arrow <4pt> [.2, 1]} 
%
\setplotsymbol ({\circle*{.4}})
\plotsymbolspacing=.3pt        
\myarrow from -1 0 to 7 0 
\myarrow from   0 -1  to  0 5 
\plot 0 0  3 3   4.5 4.5  6 3  6 0 /
\plot 3 3   6 1.5  /
\setplotsymbol ({\circle*{.2}})
\setdots <2pt>
\setdashes
\plot 3 3   6 3 /
\setsolid 
\put {\tiny $\left(\begin{matrix} 0 \\ 0 \\ 0\end{matrix}\right)$} [tr] at   -.1 -.1
\put {\tiny $\left(\begin{matrix}1/2 \\ 0 \\ 0\end{matrix}\right)$} [t] at 6 -.1
\put {\tiny $\left(\begin{matrix}1/4 \\ 1/4 \\ 1/4 \end{matrix}\right)$} [br] at 2.65 2.75
\put {\tiny $\left(\begin{matrix}3/8 \\ 3/8 \\ 0 \end{matrix}\right)$} [br]
at 4.3 4.4
\put {\tiny $\left(\begin{matrix} 1/2 \\ 1/4 \\ 0\end{matrix}\right)$} [bl] at 6.1 2.9
\put {\tiny $\left(\begin{matrix} 1/2 \\ 1/8 \\ 1/8\end{matrix}\right)$} [l]
at 6.1 1.5

\put {$S_2$} [t] at 3 -.2
\myarrow from  3.3 -.6 to 3.3 -.1

\put {$S_1$} at 3.3  1.4

\put {$S_3$} [br] at 1.5 1.7  
\myarrow from  1.05 1.75  to 1.35 1.45

\put {$S_4$} [l] at 6.1  2.4

\put {$S_5$} at 4.5  3.5

\endpicture}
}
\end{center}

\hfill

\caption{This polyhedron (viewed from the
  positive $z$-direction) is a fundamental domain for $\Gamma_{2}$. The dashed segment represents an axis of rotation,
  not a division between faces. The sides $S_{3}$ and $S_{4}$ are triangular and 
perpendicular to the $xy$-plane. }
\label{pictureofgamma2}
\end{figure}

Throughout this subsection, we will write $\Gamma$ in place of $\Gamma_{2}$.  Consider the convex compact polyhedron
$$ P = \left\{ (x,y,z) \in \mathbb{R}^{3} \mid 0 \leq z \leq y \leq x \leq 1/2; \, \, x+y+z \leq 3/4 \right\}.$$
It is possible to check that $P$ is the five-sided polyhedron depicted in Figure \ref{pictureofgamma2}.
The sides $S_{1}$, $S_{2}$, $S_{3}$, $S_{4}$, and $S_{5}$ are contained in the planes $y=z$, $z=0$, $x=y$,
$x=1/2$, and $x+y+z=3/4$, respectively.  (Indeed, the sides $S_{i}$, for $i \in \{1,2,3,4\}$, are contained in
the same planes as the corresponding sides from Subsection \ref{symlc}.)

The outward unit normal vectors $N_{1}$, $N_{2}$, $N_{3}$, $N_{4}$, and $N_{5}$ are (respectively),
$$ \frac{1}{\sqrt{2}} \left( -\mathbf{y} + \mathbf{z} \right), \quad   -\mathbf{z}, \quad
\frac{1}{\sqrt{2}} \left( -\mathbf{x} + \mathbf{y} \right), \quad \mathbf{x}, \quad \mathrm{and}
\quad \frac{1}{\sqrt{3}} \left( \mathbf{x} + \mathbf{y} + \mathbf{z} \right).$$
The dihedral angles are as follows:
\begin{eqnarray*}
\theta ( S_{1}, S_{2}) = \pi /4, & \quad & \theta ( S_{1}, S_{3}) = \pi /3, \\
\theta ( S_{1}, S_{4}) = \pi /2, & \quad & \theta ( S_{2}, S_{3}) = \pi /2, \\
\theta ( S_{2}, S_{4}) = \pi /2, & \quad & \theta( S_{3}, S_{4}) = \pi /4, \\
\theta( S_{1}, S_{5}) = \pi/2, & \quad & \theta( S_{3}, S_{5}) = \pi/2, \\
\theta( S_{2}, S_{5}) & = & \arccos \left( 1/ \sqrt{3}\right), \\
\theta( S_{4}, S_{5}) & = & \arccos \left(-1/\sqrt{3} \right) 
\end{eqnarray*}
We consider the collection $\Phi = \{ \phi_{S_{1}}, \phi_{S_{2}}, \phi_{S_{3}}, \phi_{S_{4}}, \phi_{S_{5}} \}$, 
where $\phi_{S_{i}}$ is the reflection 
in the side $S_{i}$, for $i \in \{1, 2, 3, 4 \}$, and $\phi_{S_{5}}$ is the rotation $180$ degrees
about the line through $(1/4,1/4,1/4)$ and $(1/2,1/4,0)$ (which is dashed in Figure \ref{pictureofgamma2}).  
It is rather clear that $\Phi$ is a side-pairing.  We must show that $\Phi$ is a subproper side-pairing.

If $x$ is a point in the interior of some ridge that is not a face of $S_{5}$, 
then it is easy to check that the dihedral
angle sum of $[x]$ is a submultiple of $\pi$.
There are two more kinds of ridge cycles to consider.  Each has the form $\{ x, \phi_{S_{5}}(x) \}$,
where $\phi_{S_{5}} \in \Phi$ is the rotation about the dashed line in Figure \ref{pictureofgamma2},
and $x$ is either on the ridge between the sides $S_{1}$ and $S_{5}$, or on the
ridge between $S_{4}$ and $S_{5}$.  The dihedral angle sum of $[x]$ is either
$$ \theta(S_{1},S_{5}) + \theta(S_{3},S_{5}) \quad \mathrm{or} \quad
\theta(S_{4},S_{5}) + \theta(S_{2},S_{5}),$$
both of which are equal to $\pi$.  It follows that $\Phi$ is a subproper side-pairing,
so Theorem \ref{poincare} applies.  The polyhedron $P$ is therefore an exact convex
compact fundamental polyhedron for the action of $\langle \Phi \rangle$ 
on $\mathbb{R}^{3}$. 

We need to show that $\Gamma = \langle \Phi \rangle$.  We note that $\{ \phi_{S_{1}}, \phi_{S_{2}},
\phi_{S_{3}}, \phi_{S_{4}} \}$ is a set that  generates $\Gamma_{1}$ by the
argument of Subsection \ref{symlc}.  Furthermore, we have
$$ \phi_{S_{5}} = \left(\begin{smallmatrix} 1/2 \\ 1/2 \\ 1/2 \end{smallmatrix} \right) + 
\left( \begin{smallmatrix} 0 & 0 & -1 \\ 0 & -1 & 0 \\ -1 & 0 & 0 \end{smallmatrix} \right).$$
It easily follows that $\langle \Phi \rangle = \Gamma$, so $P$ is an exact convex compact fundamental
polyhedron for the action of $\Gamma$ on $\mathbb{R}^{3}$.  
        
\begin{theorem} \label{theorem:g2}
Let $\widehat{\mathcal{C}}$ denote the $\Gamma$-equivariant cell structure on $\mathbb{R}^{3}$
determined by the standard cellulation $\mathcal{C}$ of $P$.  The quotient $\Gamma \backslash \mathbb{R}^{3}$
is obtained from $P$ by identifying the two halves of the side $S_{5}$ from Figure \ref{pictureofgamma2}.  The
set consisting of the vertices $(0,0,0)$, $(1/2,0,0)$, and $(1/4,1/4,1/4)$ maps injectively into the
quotient.  The stabilizers of these vertices are determined by the following equalities:
\begin{eqnarray*}
\pi(\Gamma_{(0,0,0)}) & = & S_{4}^{+} \times (-1). \\
\pi(\Gamma_{(1/2,0,0)}) & = & D_{4_{1}}^{+} \times (-1). \\
\pi(\Gamma_{(1/4,1/4,1/4)}) & = & D_{3}^{+} \times (-1).
\end{eqnarray*}
All of the other stabilizers of cells in the quotient are negligible.
\end{theorem}
\begin{proof}
The first statement follows easily from a description of the cycles $[x]$, and from the fact that
$[x] = P \cap (\Gamma \cdot x)$.  The second statement, that
the given vertices map injectively into the quotient, follows from the fact that the cycle generated by
each vertex is a singleton.  

We turn now to a consideration of the cell stabilizers.  First, note that each of the vertices
$$ (3/8, 3/8, 0), \quad (1/2, 1/4, 0), \quad (1/2, 1/8, 1/8)$$
has a negligible stabilizer group, by Proposition \ref{negproposition}.  This directly implies that all of the edges
and faces incident with these vertices must also have negligible stabilizer groups.  This leaves three vertices
and two edges to consider.  It is clear that the first equality in the Theorem holds.  The second equality
holds for reasons similar to those used in establishing the second equality in Theorem \ref{theorem:g1}.  The
third equality follows from Lemma \ref{pointstab}(1) and the 
fact that $h \in S_{4}^{+} \times (-1)$ satisfies the condition
$(1/4,1/4,1/4) - h \cdot (1/4,1/4,1/4)$ if and only if $h \cdot (1/4,1/4,1/4) = \pm (1/4,1/4,1/4)$.

Finally, we note that the remaining edges are negligible by Proposition \ref{negproposition}.
\end{proof}

\subsection{A Fundamental Polyhedron for 
$\Gamma_{3}$}
Note that
$$\Gamma_{3} = \left\langle 
\frac{1}{2}( \mathbf{x} + \mathbf{y}), \frac{1}{2}( \mathbf{y} + \mathbf{z}),
\frac{1}{2}( \mathbf{x} + \mathbf{z}) \right\rangle \rtimes (S_{4}^{+} \times (-1)).
$$
We set $\Gamma = \Gamma_{3}$ in this subsection.
Consider the convex compact polyedron 
$$P = \left\{ (x,y,z) \in \mathbb{R}^{3} 
\mid 0 \leq z \leq y \leq x, \, \,  x+y \leq 1/2 \right\}.$$ 
 A straightforward check shows that $P$ is the tetrahedron depicted in Figure \ref{pictureofgamma3}.  The
sides $S_{1}$, $S_{2}$, $S_{3}$, and $S_{4}$ are contained in the planes $y=z$, $z=0$, $x=y$,
and $x+y=1/2$, respectively.  (The sides $S_{i}$ for $i \in \{ 1,2,3 \}$ are contained
in the same planes as the corresponding sides from Subsection \ref{symlc}.)

\begin{figure}[h]
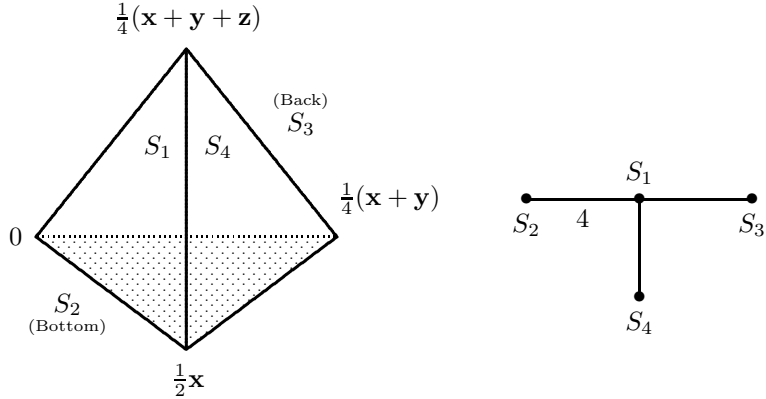
   
    
\begin{center}
\hbox{
\vbox{\beginpicture
\setcoordinatesystem units <1cm,1cm> point at -.2 3
\setplotarea x from -1.3 to 9.2, y from -2 to 3
\linethickness=.7pt
%
\def\myarrow{\arrow <4pt> [.2, 1]} 
\setplotsymbol ({\circle*{.4}})
\plotsymbolspacing=.3pt        
\plot 0 0  2 2.5  4 0  2 -1.5  0 0 /
\plot 2 2.5  2 -1.5 /
\setdots <2pt>  
\putrule from 0 0 to 4 0
\put {$0$} [r] at -.2 0
\put {${\frac{1}{4}({\mathbf x}+{\mathbf y}+{\mathbf z})}$} [b] at 2 2.7
\put {${\frac{1}{4}({\mathbf x}+{\mathbf y})}$} [bl] at 4 .3
\put {${\frac{1}{2}{\mathbf x}}$} [t] at 2 -1.7
\put {$S_2$} at .4 -.9
\put {\tiny (Bottom)} at .4 -1.2
\put {$S_1$} at 1.6 1.2 
\put {$S_3$} at 3.5 1.5 
\put {\tiny (Back)} at 3.5 1.8
\put {$S_4$} at 2.4 1.2 
\setshadegrid span <2pt>
\hshade -1.5 2 2  0 0 4 /       

\setsolid
\putrule from 6.5 .5 to 9.5 .5
\putrule from 8 .5 to 8 -.8
\put {$\bullet$} at  6.5 .5
\put {$S_2$} [t] at  6.5 .3
\put {$4$} [t]   at  7.25 .36
\put {$\bullet$} at  9.5 .5
\put {$S_3$} [t] at  9.5 .3
\put {$\bullet$} at  8 .5
\put {$S_1$} [b] at  8 .7
\put {$\bullet$} at  8 -.8 
\put {$S_4$} [t] at  8 -1
\endpicture}
}
\end{center}

\hfill

\caption{On the left, we have a fundamental domain for the action of
  $\Gamma_{3}$ on ${\mathbb R}^3$. The group $\Gamma_{3}$ is a Coxeter group, and
  its Coxeter diagram appears on the right. }
\label{pictureofgamma3}
\end{figure}

The outward unit normal vectors $N_{1}$, $N_{2}$, $N_{3}$, and $N_{4}$ are (respectively)
$$ \frac{1}{\sqrt{2}} \left( -\mathbf{y} + \mathbf{z} \right), \quad -\mathbf{z}, \quad 
\frac{1}{\sqrt{2}} \left( -\mathbf{x} + \mathbf{y} \right), \quad \mathrm{and} \quad 
\frac{1}{\sqrt{2}}\left( \mathbf{x} + \mathbf{y}\right).$$ 
A routine check shows that the dihedral angles are 
\begin{eqnarray*}
\theta(S_{1}, S_{2}) = \pi/4, & \quad & \theta(S_{1}, S_{3}) = \pi/3, \\ 
\theta(S_{1}, S_{4}) = \pi /3, & \quad & \theta(S_{2}, S_{3}) = \pi/2, \\
\theta(S_{2}, S_{4}) = \pi/2, & \quad & \theta(S_{3}, S_{4}) = \pi/2.  
\end{eqnarray*}
It follows from Corollary \ref{coxeter} that $\Gamma_{P}$, the group generated by the reflections in the sides of the tetrahedron
from Figure \ref{pictureofgamma3}, is a discrete group, and that $P$ is an exact convex compact fundamental polyhedron
for the action of $\Gamma_{P}$ on $\mathbb{R}^{3}$.  We have 
$$\Gamma_{P} = \left\langle 
\left( \begin{smallmatrix} 1 & 0 & 0 \\ 0 & 0 & 1 \\ 0 & 1 & 0 \\ \end{smallmatrix} \right),
\left( \begin{smallmatrix} 1 & 0 & 0 \\ 0 & 1 & 0 \\ 0 & 0 & -1 \\ \end{smallmatrix} \right), 
\left( \begin{smallmatrix} 0 & 1 & 0 \\ 1 & 0 & 0 \\ 0 & 0 & 1 \\ \end{smallmatrix} \right),
\left( \begin{smallmatrix} \frac{1}{2} \\ \frac{1}{2} \\ 0 \end{smallmatrix} \right) + \left( \begin{smallmatrix} 0 & -1 & 0 \\
-1 & 0 & 0 \\ 0 & 0 & 1 \end{smallmatrix} \right) \right\rangle,$$
where the generators are reflections in the sides $S_{1}$, $S_{2}$, $S_{3}$, and $S_{4}$, respectively.  It
is easy to check that $\Gamma_{P} = \Gamma$.

It follows that $P$ is an exact convex compact fundamental polyhedron for the action of $\Gamma$ on $\mathbb{R}^{3}$.

\begin{theorem} \label{theorem:g3}
Let $\widehat{\mathcal{C}}$ denote the $\Gamma$-equivariant cell structure on $\mathbb{R}^{3}$ determined by the standard
cellulation $\mathcal{C}$ of $P$.  The quotient $\Gamma \backslash \mathbb{R}^{3}$ is $P$ itself, endowed with the standard
cellulation $\mathcal{C}$.  The only non-negligible cell stabilizer groups are as follows:
$$ \pi(\Gamma_{(0,0,0)}) = \pi(\Gamma_{(1/2,0,0)}) = S_{4}^{+} \times (-1),  \quad  \pi(\Gamma_{(1/4,1/4,0)}) \cong D_2 \times \mathbb Z/2.$$
\end{theorem}

\begin{proof}
The statement about the quotient follows directly from the fact that each cycle $[x]$ is a singleton.  

The vertex $(1/4,1/4,1/4)$ has a negligible stabilizer group by Proposition \ref{negproposition}.  To find the vertex stabilizer of 
$(1/4,1/4,0)$ note that the group $S_{4}^{+} \times (-1)$ translates
$(1/4,1/4,0)$ to a total of $12$ different points in $\mathbb{R}^{3}$.  (It is easy to describe these points explicitly
using the fact that $S_{4}^{+} \times (-1)$ is the group of $3 \times 3$ signed permutation matrices.)  By 
Lemma \ref{pointstab}(1), $h \in S_{4}^{+} \times (-1)$ is an element of $\pi(\Gamma_{(1/4,1/4,0)})$ if and only
if $h \cdot (1/4,1/4,0) = \pm (1/4,1/4,0)$.  It follows from this that $\pi(\Gamma_{(1/4,1/4,0)})$ has index $6$
in $S_{4}^{+} \times (-1)$, i.e., $\pi(\Gamma_{(1/4,1/4,0)})$ has order $8$.  It is easy to check that the following
set is contained in $\pi(\Gamma_{(1/4,1/4,0)})$, and generates a group of order $8$:
$$ \left\{ \left( \begin{smallmatrix} 0 & 1 & 0 \\ 1 & 0 & 0 \\ 0 & 0 & 1 \end{smallmatrix} \right),
\left( \begin{smallmatrix} -1 & 0 & 0 \\ 0 & -1 & 0 \\ 0 & 0 & 1 \end{smallmatrix} \right),
\left( \begin{smallmatrix} -1 & 0 & 0 \\ 0 & -1 & 0 \\ 0 & 0 & -1 \end{smallmatrix} \right) \right\}.$$
Therefore, $\pi(\Gamma_{(1/4,1/4,0)})$ is generated by the above matrices.  One can easily show from this that
$\pi(\Gamma_{(1/4,1/4,0)}) \cong D_{2} \times \mathbb Z/2$.

 A routine check using Lemma \ref{pointstab}(1)
shows that $\pi( \Gamma_{(0,0,0)}) = \pi( \Gamma_{(1/2,0,0)}) = S_{4}^{+} \times (-1)$.  
One proves that the edge stabilizer groups are negligible by applying Proposition \ref{negproposition} to the midpoints of the edges.
\end{proof}

\subsection{A Fundamental Polyhedron for $\Gamma_{4}$}
Note that
$$\Gamma_{4}= \left\langle \frac{1}{2}(\mathbf{x} + \mathbf{z}), \mathbf{y}, \mathbf{z}
\right\rangle \rtimes (D_{2}^{+} \times (-1)).$$
We consider the convex compact polyhedron  $P$ given in Figure \ref{pictureofgamma4}. It is easy to see that
\[
P = \{ (x,y,z) \in \mathbb{R}^{3} \mid 0 \leq x , y , z \leq 1/2,  \, \, \, x + z \leq1/2 \}.
\]


\begin{figure}[h]
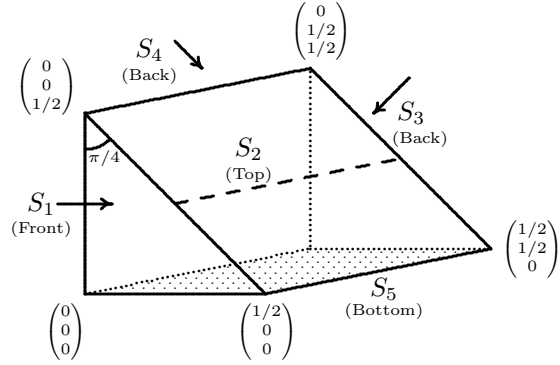

    
\begin{center}
\hbox{
\vbox{\beginpicture
\setcoordinatesystem units <1.2cm,1.2cm> point at 0 2
\setplotarea x from 0 to 6, y from -.5 to 3
\linethickness=.7pt
%
 \def\myarrow{\arrow <4pt> [.2, 1]} 
%
\setplotsymbol ({\circle*{.4}})
\plotsymbolspacing=.3pt        
\plot 2 0  2 2  4.5 2.5  /
\plot 2 0  4 0  2 2 /
\plot 4 0   6.5 .5  /
\plot 6.5 .5   4.5 2.5 /
\setplotsymbol ({\circle*{.2}})
\setdots <2pt> 
\plot 2 0  4.5 .5   4.5 2.5 /
\plot 4.5 .5   6.5 .5  /
\setdashes
\plot 3 1   5.5 1.5 /
\setsolid  
\circulararc 44 degrees from 2 1.6 center at 2 2
\put {\tiny$\pi/4$} at 2.2 1.5
\put {\tiny $\left(\begin{matrix}0 \\ 0 \\ 0\end{matrix}\right)$} [tr] at  2 -.1
\put {\tiny $\left(\begin{matrix}0 \\ 0 \\ 1/2\end{matrix}\right)$} [r] at
1.9 2.3
\put {\tiny $\left(\begin{matrix}0 \\ 1/2 \\ 1/2\end{matrix}\right)$} [b] at
4.6 2.6

\put {\tiny $\left(\begin{matrix}1/2 \\ 1/2 \\ 0\end{matrix}\right)$} [l] at
6.6 .5
\put {\tiny $\left(\begin{matrix}1/2 \\ 0 \\ 0\end{matrix}\right)$} [t] at
4 -.1

\put {$S_1$} at 1.5 1
\put {\tiny (Front)} [t] at 1.5 .8
\myarrow from 1.7 1 to   2.3 1

\put {$S_2$} at 3.8  1.6
\put {\tiny (Top)} [t] at 3.8 1.4
\put {$S_3$} at 5.6 2 
\put {\tiny (Back)} [t] at 5.7 1.8
\myarrow from  5.6 2.4  to  5.2 2
\put {$S_4$} at 2.7 2.7
\put {\tiny (Back)} [t] at 2.7 2.5
\myarrow from  3 2.8  to  3.3 2.5
\put {$S_5$} at 5.3 .05
\put {\tiny (Bottom)} [t] at 5.3  -.1
%
\setshadegrid span <2pt>
\hshade 0 2 4   .5 4.5 6.5 /  
\endpicture}
}
\end{center}

\hfill

\caption{This is a fundamental domain for the action of $\Gamma_{4}$ on ${\mathbb R}^3$. The dashed line
  indicates an axis of rotation.}
\label{pictureofgamma4}
\end{figure}


The sides $\{ S_{1}, S_{2}, S_{3}, S_{4}, S_5 \}$ are contained in the planes $y=0$, $x+z=1/2$, $y=1/2$, $x=0$,  and $z=0$ respectively.  The outward unit normal vectors $N_{1}, N_{2}, N_{3}, N_{4}$ $N_5$ are, respectively,
$$ -\mathbf{y}, \quad \frac{1}{\sqrt{2}} \left( \mathbf{x} + \mathbf{z} \right), \quad  \mathbf{y}, \quad -\mathbf{x}, \quad - \mathbf{z}.$$ 
The dihedral angles are as follows:
\begin{eqnarray*}
\theta ( S_{1}, S_{2}) = \pi /2, & \quad  \theta ( S_{1}, S_{4}) = \pi /2, & \quad  \theta ( S_{1}, S_{5}) = \pi /2,\\
\theta ( S_{2}, S_{3}) = \pi /2, & \quad  \theta ( S_{2}, S_{4}) = \pi /4, & \quad  \theta( S_{2}, S_{5}) = \pi /4,\\
\theta ( S_{3}, S_{4}) = \pi /2, & \quad  \theta( S_{3}, S_{5}) = \pi /2, & \quad  \theta ( S_{4}, S_{5}) = \pi /2.
\end{eqnarray*}

We consider the collection $\Phi = \{ \phi_{S_{1}}, \phi_{S_{2}}, \phi_{S_{3}}, \phi_{S_{4}}, \phi_{S_{5}} \}$, 
where $\phi_{S_{i}}$ is the reflection 
in the side $S_{i}$, for $i \in \{1, 3, 4, 5 \}$, and $\phi_{S_{2}}$ is the rotation $180$ degrees
about the line through $(1/4, 0,1/4)$ and $(1/4, 1/2, 1/4)$ (which is  the line dashed in Figure \ref{pictureofgamma4}).  
It is rather clear that $\Phi$ is a side-pairing.  We must show that $\Phi$ is a subproper side-pairing.

If $x$ is a point in the interior of some ridge that is not a face of $S_{2}$,  then it is easy to check that the dihedral
angle sum of $[x]$ is a submultiple of $\pi$. If $x$ lies on the axis of
rotation, then $[x]$ is again a singleton, and it follows easily that the dihedral
angle sum is a submultiple of $2\pi$. There are 3 more kinds of ridge cycles to consider.  Each has the form $\{ x, \phi_{S_{2}}(x) \}$, where $\phi_{S_{2}} \in \Phi$ is the rotation about the dashed line in Figure \ref{pictureofgamma4}, 
and $x$ is either on the ridge between the sides $S_{1}$ and $S_{2}$,  $S_{3}$ and $S_{2}$,  or $S_{4}$ and $S_{2}$.  The dihedral angle sum of $[x]$ is
$$2 \theta(S_{1},S_{2}), \quad  2\theta(S_{3},S_{2}), \quad \mathrm{or} \quad
\theta(S_{4},S_{2}) + \theta(S_{5},S_{2}),$$
respectively. All of the latter sums are $\pi$ or $\pi/2$.  It follows that $\Phi$ is a subproper side-pairing,
so Theorem \ref{poincare} applies.  The polyhedron $P$ is therefore an exact convex
compact fundamental polyhedron for the action of $\langle \Phi \rangle$ 
on $\mathbb{R}^{3}$.  Furthermore, we have
$$\Gamma_{P}=
\left\langle \left( \begin{smallmatrix} 1 & 0 & 0 \\ 0 & -1 & 0 \\ 0 & 0 & 1 \end{smallmatrix} \right),
\left( \begin{smallmatrix} 1/2 \\ 0 \\1/2 \end{smallmatrix}\right) + 
\left( \begin{smallmatrix} -1 & 0 & 0 \\ 0 & 1 & 0 \\ 0 & 0 & -1 \end{smallmatrix}\right),
\left( \begin{smallmatrix} 0 \\ 1 \\ 0 \end{smallmatrix}\right) + 
\left( \begin{smallmatrix}  1 & 0 & 0 \\ 0 & -1 & 0 \\ 0 & 0 & 1 \end{smallmatrix}\right),
\left( \begin{smallmatrix}  -1 & 0 & 0 \\ 0 & 1 & 0 \\ 0 & 0 & 1 \end{smallmatrix}\right),
\left( \begin{smallmatrix}  1 & 0 & 0 \\ 0 & 1 & 0 \\ 0 & 0 & -1 \end{smallmatrix}\right)\right\rangle,$$
where the isometries in question are $\phi_{S_{i}}$, for $i=1, \ldots, 5$, respectively.
It is straightforward to verify that $\Gamma_{P} = \Gamma$.

\begin{theorem} \label{theorem:g4}
Let $\widehat{\mathcal{C}}$ denote the $\Gamma$-equivariant cell structure on $\mathbb{R}^{3}$
determined by the standard cellulation $\mathcal{C}$ of $P$.  The quotient $\Gamma 
\backslash \mathbb{R}^{3}$
is obtained from $P$ by identifying the two halves of the side $S_{2}$ from Figure \ref{pictureofgamma4}.  The
set consisting of the vertices $(0,0,0)$, $(0,0, 1/2)$, $(0,1/2,1/2)$  and $(0, 1/2, 0)$ maps injectively into the
quotient.  The stabilizers of these vertices are determined by the following equalities:
\[
\pi(\Gamma_{(0,0,0)})=  \pi(\Gamma_{(0,1/2,0)}) = \pi(\Gamma_{(0,0,1/2)})= \pi(\Gamma_{(0,1/2,1/2)})= D_{2}^{+} \times (-1). \\
\]
All of the other stabilizers of cells in the quotient are negligible.
\end{theorem}

\begin{proof}
First, we note that there are eight vertices in the cellulation $\mathcal{C}$: the four listed
in the statement of the theorem, and
$$ (1/2,0,0), (1/2,1/2,0), (1/4,0,1/4), \text{and} \, \,  (1/4,1/2,1/4).$$
We note that the first two of these last vertices occupy the same orbits
as vertices in the statement of the theorem. The final two vertices $v$ are negligible
by Lemma \ref{pointstab}(1)
since $v - \phi_{S_{4}}(v) \notin L$, and so $|\Gamma_{v}| \leq 4$.

It is easy to check the equalities in the theorem using Lemma \ref{pointstab}(1).
The edges (and, therefore all higher-dimensional cells) have negligible stabilizers,
either because they are incident with vertices having negligible stabilizers, or because
of Proposition \ref{negproposition} (applied to the midpoints of the edges).

To verify the description of the quotient is straightforward.
\end{proof}

\subsection{A Fundamental Polyhedron for $\Gamma_{5}$}
Recall that $\Gamma_{5} = \langle \mathbf{v}_{1}, \mathbf{v}_{2}, \mathbf{v}_{3} \rangle
\rtimes (D_{6}^{+} \times (-1))$. We write $\Gamma$ in place of $\Gamma_{5}$.
Consider the convex polyhedron $P$ depicted in Figure \ref{pictureofgamma5}.  
One checks that 
$$P = \{ (x,y,z) \in \mathbb{R}^{3} \mid 0 \leq x + y + z \leq 3/2, -2x + y + z \leq 0, x \leq z \leq y+1 \}.$$
The sides $S_{1}$, $S_{2}$, $S_{3}$, $S_{4}$, and $S_{5}$ are convex subsets of the planes
$x+y+z=3/2$, $z=y+1$, $x+y+z=0$, $-2x+y+z=0$, and $x=z$, respectively.  The outward unit normal vectors $N_{1}$, $N_{2}$,
$N_{3}$, $N_{4}$, and $N_{5}$ are 
$$ \frac{1}{\sqrt{3}} \mathbf{v}_{1}, \quad \frac{1}{\sqrt{2}} \mathbf{v}_{3}, \quad \frac{-1}{\sqrt{3}} \mathbf{v}_{1},
\quad \frac{1}{\sqrt{6}} \left( -2\mathbf{v}_{2} + \mathbf{v}_{3} \right), \quad \mathrm{and} \quad 
\frac{1}{\sqrt{2}} \left( \mathbf{v}_{2} - \mathbf{v}_{3} \right),$$
respectively.

\vspace{-1.9cm}
\begin{figure}[!h]
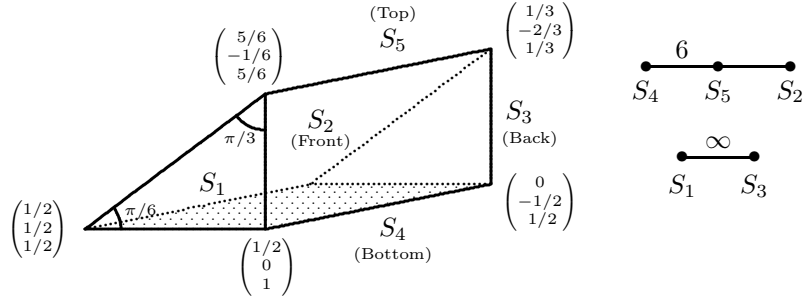

\begin{center}
\hbox{
\vbox{\beginpicture
\setcoordinatesystem units <1.2cm,1.2cm> point at -.2 4
\setplotarea x from -1.3 to 9.2, y from -.5 to 4
\linethickness=.7pt
%
%
\setplotsymbol ({\circle*{.4}})
\plotsymbolspacing=.3pt        
\plot 0 0  2 0   4.5 .5   4.5 2  2 1.5  0 0 /
\plot 2 1.5  2 0 /
\setplotsymbol ({\circle*{.2}})
\setdots <2pt>
\plot 0 0  2.5 .5   4.5 2 /
\plot 2.5 .5   4.5 .5 /
\setsolid  
\circulararc 37 degrees from .4 0 center at 0 0
\circulararc -53 degrees from 2 1.1 center at 2 1.5
\put {\tiny $\left(\begin{matrix}1/2 \\ 1/2 \\ 1/2\end{matrix}\right)$} [r] at -.2 0
\put {\tiny $\left(\begin{matrix}1/2 \\ 0 \\ 1\end{matrix}\right)$} [t] at 2 -.1
\put {\tiny $\left(\begin{matrix}0 \\ -1/2 \\ \;1/2\end{matrix}\right)$} [l] at 4.6 .3
\put {\tiny $\left(\begin{matrix}1/3 \\ -2/3 \\ 1/3\end{matrix}\right)$} [l] at 4.6 2.2
\put {\tiny $\left(\begin{matrix}\;5/6 \\ -1/6 \\ \;5/6\end{matrix}\right)$} [b] at 1.8 1.6
\put {\tiny$\pi/6$} at .6 .2
\put {\tiny$\pi/3$} at 1.7 1
\put {$S_1$} at 1.4 .5
\put {$S_2$} at 2.6  1.2
\put {\tiny (Front)} [t] at 2.6 1.05
\put {$S_3$} at 4.8 1.3 
\put {\tiny (Back)} [l] at 4.6 1
\put {$S_4$} at 3.4 0
\put {\tiny (Bottom)} [t] at 3.4 -.2
\put {$S_5$} at 3.4  2.1
\put {\tiny (Top)} [b] at 3.4  2.3
\setshadegrid span <2pt>
\hshade 0 0 2  .5 2.5 4.5 /       

\setsolid
\putrule from 6.2 1.8 to 7.8 1.8
\put {$\bullet$} at  6.2 1.8
\put {$S_4$} [t] at  6.2 1.65
\put {$6$} [b]   at  6.6 1.9
\put {$\bullet$} at  7 1.8
\put {$S_5$} [t] at  7 1.65
\put {$\bullet$} at  7.8 1.8
\put {$S_2$} [t] at  7.8 1.65
\putrule from 6.6 .8 to 7.4 .8
\put {$\bullet$} at  6.6 .8
\put {$\bullet$} at  7.4 .8
\put {$S_1$} [t] at  6.6 .6
\put {$S_3$} [t] at  7.4 .6
\put {$\infty$} [b] at  7 .9
\endpicture}
}
\end{center}

\hfill

\caption{On the left is a fundamental domain for the action of
 $\Gamma_{5}$ on ${\mathbb R}^3$. The unlabelled vertex at the intersection of
 three dotted lines is the origin. On the right is the Coxeter diagram for
 $\Gamma_{5}$.}
\label{pictureofgamma5}
\end{figure}

It is straightforward to check that 
\begin{eqnarray*}
\theta(S_{1}, S_{2}) = \pi / 2, & \quad & \theta(S_{1}, S_{4}) = \pi / 2, \\
\theta(S_{1}, S_{5}) = \pi/2, & \quad & \theta(S_{2}, S_{3}) = \pi / 2, \\
\theta(S_{2}, S_{4}) = \pi/2, & \quad & \theta(S_{2}, S_{5}) = \pi/3, \\ 
\theta(S_{3}, S_{4}) = \pi / 2, & \quad & \theta(S_{3}, S_{5}) = \pi/2, \\
\theta(S_{4}, S_{5}) = \pi/6.
\end{eqnarray*}
Since all of these dihedral angles are submultiples of $\pi$, by Corollary \ref{coxeter}
the group $\Gamma_{P}$ generated by reflections
in the sides of $P$ is discrete, and $P$ is an exact convex compact fundamental polyhedron for the action
of $\Gamma_{P}$ on $\mathbb{R}^{3}$.  We have $\Gamma_{P}=$
$$ \left\langle \left( \begin{smallmatrix} 1 \\ 1 \\ 1 \end{smallmatrix} \right) + 
\frac{1}{3}\left( \begin{smallmatrix} 1 & -2 & -2 \\ -2 & 1 & -2 \\ -2 & -2 & 1 \end{smallmatrix}\right),
\left( \begin{smallmatrix}  0 \\ -1 \\ 1 \end{smallmatrix} \right) + 
\left( \begin{smallmatrix}  1 & 0 & 0 \\ 0 & 0 & 1 \\ 0 & 1 & 0 \end{smallmatrix}\right),
 \frac{1}{3}\left( \begin{smallmatrix}  1 & -2 & -2 \\ -2 & 1 & -2 \\ -2 & -2 & 1 \end{smallmatrix}\right),
\frac{1}{3}\left( \begin{smallmatrix}  -1 & 2 & 2 \\ 2 & 2 & -1 \\ 2 & -1 & 2 \end{smallmatrix}\right),
\left( \begin{smallmatrix} 0 & 0 & 1 \\ 0 & 1 & 0 \\ 1 & 0 & 0 \end{smallmatrix}\right) \right\rangle.$$
It is not difficult to check that $\Gamma_{P} = \Gamma$, so $P$ is an exact convex compact
fundamental polyhedron for the action of $\Gamma$ on $\mathbb{R}^{3}$.

\begin{theorem} \label{theorem:g5}
Let $\widehat{\mathcal{C}}$ denote the $\Gamma$-equivariant cell structure on $\mathbb{R}^{3}$ determined by the standard
cellulation $\mathcal{C}$ of $P$.  The quotient $\Gamma \backslash \mathbb{R}^{3}$ is $P$ itself, endowed with the standard
cellulation $\mathcal{C}$.  The non-negligible stabilizer groups are determined by the following equalities:
$$\pi(\Gamma_{(0,0,0)}) =  \pi(\Gamma_{(1/2,1/2,1/2)}) = D_{6}^{+} \times (-1);$$
$$\pi(\Gamma_{(1/2,0,1)}) =  \pi(\Gamma_{(0,-1/2,1/2)}) \cong D_2 \times \mathbb Z/2;$$
$$\pi(\Gamma_{(1/3,-2/3,1/3)})  =  \pi(\Gamma_{(5/6,-1/6,5/6)})  =  D_{6}';$$
$$\pi(\Gamma_{(1/4,1/4,1/4)}) = D_{6}''.$$
(Note that $\Gamma_{(1/4,1/4,1/4)}$ is the stabilizer group of the edge connecting $(0,0,0)$ to
$(1/2,1/2,1/2)$.)
The stabilizer groups of all other cells are negligible.
\end{theorem}

\begin{proof}
The statement about the quotient follows because each cycle $[x]$ is a singleton.

We now consider vertex stabilizers.  The equality $\pi(\Gamma_{(0,0,0)}) = D_{6}^{+} \times (-1)$ is trivial.
The equality $\pi(\Gamma_{(1/2,1/2,1/2)}) = D_{6}^{+} \times (-1)$ follows from Lemma \ref{pointstab}(1) and
the fact that $h \cdot (1/2,1/2,1/2) = \pm (1/2,1/2,1/2)$, for any $h \in D_{6}^{+} \times (-1)$.  Since 
$2v \not \in \langle \mathbf{v}_{1}, \mathbf{v}_{2}, \mathbf{v}_{3} \rangle$, for 
$v \in \{ (1/3,-2/3,1/3), (5/6,-1/6,5/6) \}$, it follows from Lemma \ref{pointstab}(1) that $(-1)$ (the antipodal map)
is in neither $\pi(\Gamma_{(1/3,-2/3,1/3)})$ nor $\pi(\Gamma_{(5/6,-1/6,5/6)})$.  This means that each of these groups
has order at most $12$.  It is routine to check, using Lemma \ref{pointstab}(1), that $D_{6}'$ is a 
subgroup of each of these two groups.  This shows that $\pi(\Gamma_{(1/3,-2/3,1/3)}) = \pi(\Gamma_{(5/6,-1/6,5/6)}) =
D_{6}'$.  Finally, we consider $(1/2,0,1)$ and $(0,-1/2,1/2)$.  Using the theory of Coxeter groups, we note that these vertices are stabilized by the groups $\langle \phi_{S_{1}}, \phi_{S_{2}},   \phi_{S_{4}} \rangle$, and $\langle \phi_{S_{2}}, \phi_{S_{3}},   \phi_{S_{4}} \rangle$, both of which are isomorphic to $D_2 \times \mathbb Z/2$.

Next,  consider the stabilizer of the edge connecting $(0,0,0)$ to $(1/2,1/2,1/2)$, which is the same as $\Gamma_{(1/4,1/4,1/4)}$.  We note
that, by Lemma \ref{pointstab}(1), 
$(-1) \not \in \pi(\Gamma_{(1/4,1/4,1/4)})$ since $(1/2,1/2,1/2) \not \in \langle \mathbf{v}_{1},
\mathbf{v}_{2}, \mathbf{v}_{3} \rangle$.  It is not difficult to see that
$D_{6}'' \leq \pi(\Gamma_{(1/4,1/4,1/4)})$, and this inclusion must be an equality by order considerations.

We need to show that the remaining edges have negligible stabilizers.  Once again we use the theory of Coxeter groups
to simplify our work.  We  consider the subdiagrams of the Coxeter diagram that are determined by pairs of distinct vertices other than $\{ S_{1}, S_{3} \}$ (which does not determine an edge) and $\{ S_{4}, S_{5} \}$ (which is accounted for above).
These subdiagrams determine subgroups that are isomorphic
to $D_{2}$ or $D_{3}$, so all are negligible.

It follows easily that all remaining cells have negligible stabilizer groups.
\end{proof}  

\subsection{A Fundamental Polyhedron for $\Gamma_{6}$}

We note that 
$$\Gamma_{6} = \left\langle \frac{1}{3}(\mbf{v}_{1} + \mbf{v}_{2} + \mbf{v}_{3}), \mbf{v}_{2}, \mbf{v}_{3} 
\right\rangle \rtimes (D_{3}^{+} \times (-1)).$$ We will write $\Gamma$ in place of $\Gamma_{6}$ throughout this subsection.

  Let $P$ denote the polyhedron in
Figure \ref{pictureofgamma6}.  The sides $S_{1}$, $S_{2}$, $S_{3}$, $S_{4}$, and $S_{5}$ are contained in the
planes $x=z$, $5x+2y+5z = 3$, $x-y=1$, $y=z$, and $x+y+z=0$, respectively.  The outward-pointing
unit normal vectors $N_{1}$, $N_{2}$, $N_{3}$, $N_{4}$, and $N_{5}$ are, respectively,
$$ \frac{1}{\sqrt{2}}\left( -\mathbf{x} + \mathbf{z} \right),
 \frac{1}{3\sqrt{6}} \left( 5 \mathbf{x} + 2\mathbf{y} + 5 \mathbf{z} \right), 
\frac{1}{\sqrt{2}} \left( \mathbf{x} - \mathbf{y} \right),
\frac{1}{\sqrt{2}} \left( \mathbf{y} - \mathbf{z} \right),
\frac{-1}{\sqrt{3}} \left( \mathbf{x} + \mathbf{y} + \mathbf{z} \right).$$


\begin{figure}[h]
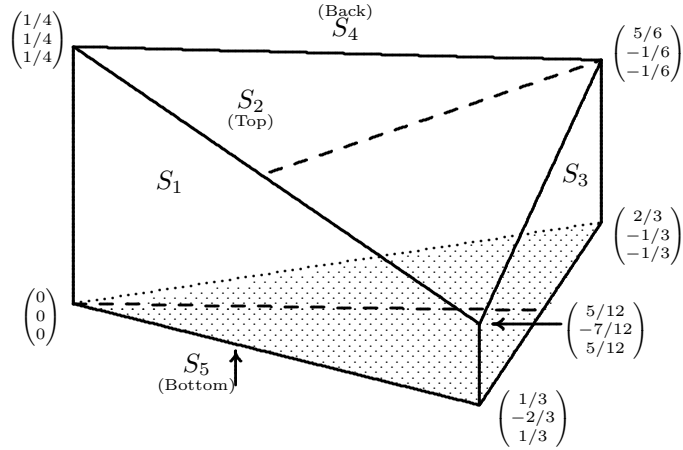
 
    
\begin{center}\hfil
\hbox{
\vbox{
  \beginpicture
\setcoordinatesystem units <.9cm,.9cm> point at -.2  4.2
\setplotarea x from -.5 to 8.2, y from -2 to 4.2
\linethickness=.7pt
%
 \def\myarrow{\arrow <4pt> [.2, 1]} 
%
\setplotsymbol ({\circle*{.4}})
\plotsymbolspacing=.3pt        
\plot 0 0  6 -1.5   7.8 1.2   7.8 3.6   0 3.8  0 0 /
\plot 0 3.8  6 -.3  7.8 3.6 /
\plot 6 -1.5  6 -.3 /
\setplotsymbol ({\circle*{.2}})
\setdashes
\plot 2.9 1.94   7.8 3.6 /
\plot 0 0  6.9 -.1 /
\setdots <3pt>   
\plot 0 0 7.8 1.2 /
\put {\tiny $\left(\begin{matrix}0 \\ 0 \\ 0\end{matrix}\right)$} [r] at -.2 -.2
\put {\tiny $\left(\begin{matrix}5/12\\-7/12\\5/12\end{matrix}\right)$} 
[l] at 7.2 -.4
\put {\tiny $\left(\begin{matrix}1/3 \\ -2/3 \\ 1/3\end{matrix}\right)$} [l]
at 6.2 -1.7
\put {\tiny $\left(\begin{matrix}2/3 \\ -1/3 \\ -1/3\end{matrix}\right)$}
[l] at 7.9 1
\put {\tiny $\left(\begin{matrix}5/6 \\ -1/6 \\ -1/6\end{matrix}\right)$}
[l] at 7.9 3.7
\put {\tiny $\left(\begin{matrix}1/4 \\ 1/4 \\ 1/4\end{matrix}\right)$} [r]
at -.1 3.9
\put {$S_1$} at 1.4 1.8
\put {$S_2$} at 2.6 3
\put {\tiny (Top)} [t] at 2.6 2.8
\put {$S_3$} at 7.4 1.9 
\put {$S_4$} [b] at 4 3.9
\put {\tiny (Back)} [b] at 4 4.2
\put {$S_5$} at 1.8 -.9
\put {\tiny (Bottom)} [t] at 1.8 -1.1
\setshadegrid span <2pt>
\hshade -1.5 6 6   0 0 7 /  
\hshade -.1 0 7   1.2 7.8 7.8 /  

\setsolid
\myarrow from  7.2 -.3 to 6.2 -.3 
\myarrow from 2.4 -1.2 to 2.4 -.7
\endpicture}
}
\hfil

\end{center}


\caption{The polyhedron pictured here is an exact convex compact fundamental
  polyhedron for the action of $\Gamma_{6}$ on ${\mathbb R}^3$. The dashed lines
  represent axes of rotation (through 180 degrees) for certain elements of
  $\Gamma_{6}$.  Note that the base of the figure is an equilateral triangle,
  but the top is isosceles.}
\label{pictureofgamma6}
\end{figure}

The dihedral angles between faces are:
\begin{eqnarray*}
\theta(S_{1},S_{2}) = \pi/2, &  & \theta(S_{1},S_{3}) = \pi/3, \\
\theta(S_{1},S_{4}) = \pi/3,&  & \theta(S_{1},S_{5}) = \pi/2, \\
\theta(S_{2},S_{3}) & = & \arccos(\frac{-1}{\sqrt{12}}), \\
\theta(S_{2},S_{4}) & = & \arccos(\frac{1}{\sqrt{12}}), \\
\theta(S_{3},S_{4}) = \pi/3, &  & \theta(S_{3},S_{5}) = \pi/2,\\
\theta(S_{4},S_{5}) = \pi/2.
\end{eqnarray*}

We define $\Phi = \{ \phi_{S_{1}}, \phi_{S_{2}}, \phi_{S_{3}}, \phi_{S_{4}}, \phi_{S_{5}} \}$ as follows.
Each of $\phi_{S_{1}}$, $\phi_{S_{3}}$, and $\phi_{S_{4}}$ is reflection in the corresponding face of $P$.
In particular, we note that the isometries 
$$ \left( \begin{smallmatrix} 0 & 0 & 1 \\ 0 & 1 & 0 \\ 1 & 0 & 0 \end{smallmatrix} \right),
\left( \begin{smallmatrix} 1 \\ -1 \\ 0 \end{smallmatrix} \right) + \left( \begin{smallmatrix} 0 & 1 & 0 \\
1 & 0 & 0 \\ 0 & 0 & 1 \end{smallmatrix} \right), \left( \begin{smallmatrix} 1 & 0 & 0 \\ 0 & 0 & 1 \\ 0 & 1 & 0
\end{smallmatrix} \right),$$
are $\phi_{S_{1}}$, $\phi_{S_{3}}$, and $\phi_{S_{4}}$, respectively.  The isometries $\phi_{S_{2}}$ and
$\phi_{S_{5}}$ are rotations about the dashed lines in the faces $S_{2}$ and $S_{5}$ (respectively).  
The isometries 
$$ \left( \begin{smallmatrix} 2/3 \\ -1/3 \\ 2/3 \end{smallmatrix} \right) + \left( \begin{smallmatrix} 0 & 0 & -1 \\
0 & -1 & 0 \\ -1 & 0 & 0 \end{smallmatrix} \right), 
\left( \begin{smallmatrix} 0 & -1 & 0 \\ -1 & 0 & 0 \\ 0 & 0 & -1 \end{smallmatrix} \right) $$
are $\phi_{S_{2}}$ and $\phi_{S_{5}}$ (respectively).

The set $\Phi$ is a subproper side-pairing of $P$.  We leave the checking to the reader, but 
note that no ridge cycle has more than two elements and the dihedral angle sum of a ridge cycle
is always a submultiple of $\pi$.
It follows from Theorem \ref{poincare} that $P$ is an exact convex compact fundamental polyhedron for the action of
$\langle \Phi \rangle$ on $\mathbb{R}^{3}$.

Finally, we briefly argue that $\Gamma = \langle \Phi \rangle$.  We notice
that $\langle \phi_{S_{1}}, \phi_{S_{4}}, \phi_{S_{5}} \rangle = D_{3}^{+} \times (-1)$.
From this, it follows easily that
$\langle \phi_{S_{1}}, \phi_{S_{3}}, \phi_{S_{4}}, \phi_{S_{5}} \rangle$ contains all
of $\langle \mathbf{v}_{2}, \mathbf{v}_{3} \rangle$, as well.  It is easy to see that 
the group $\langle \Phi \rangle$ must contain the translation $(2/3, -1/3, 2/3)$
(since $D_{3}^{+} \times (-1) \leq \langle \Phi \rangle$, and $\phi_{S_{2}} \in \langle \Phi \rangle$),
so $\Gamma \leq \langle \Phi \rangle$.  The reverse inclusion is clear.

It follows that $P$ is an exact convex compact fundamental polyhedron for the action of $\Gamma$
on $\mathrm{R}^{3}$.  

\begin{theorem} \label{theorem:g6}
Let $\widehat{\mathcal{C}}$ denote the $\Gamma$-equivariant cell structure on $\mathbb{R}^{3}$
determined by the standard cellulation $\mathcal{C}$ of $P$.  The quotient $\Gamma \backslash \mathbb{R}^{3}$
is obtained from $P$ by identifying the two halves of each of the sides $S_{2}$ and $S_{5}$ from Figure 6.  
The set consisting of the vertices $(0,0,0)$ and $(5/6,-1/6,-1/6)$ maps injectively into the
quotient.  The stabilizers of these vertices are determined by the following equalities:
$$ \pi(\Gamma_{(0,0,0)}) = \pi(\Gamma_{(5/6,-1/6,-1/6)}) =  D_{3}^{+} \times (-1).$$
All of the other stabilizers of cells in the quotient are negligible.
\end{theorem}

\begin{proof}
The statement about the quotient follows easily from a straightforward description of the cycles $[x]$, and
from the fact that $[x] = P \cap (\Gamma \cdot x)$.

We turn to a description of the vertex stabilizers.  Note that there are a total of $8$ vertices to consider.
We must describe the stabilizers of $(1/3,-1/6,1/3)$ and $(1/2,-1/2,0)$ since these points are endpoints of dashed
lines from Figure \ref{pictureofgamma6}, and will therefore be vertices in the standard cellulation of $P$.  

First, we note the stabilizers of the vertices  $(1/4,1/4,1/4)$, $(5/12,-7/12,5/12)$, $(1/3,-2/3,1/3)$, and $(2/3,-1/3,-1/3)$
are all negligible by Proposition \ref{negproposition}.  Second, we note that $\pi(\Gamma_{(1/3,-1/6,1/3)})$ and 
$\pi(\Gamma_{(1/2,-1/2,0)})$ do not contain the matrix
$$ \left( \begin{smallmatrix} 0 & 1 & 0 \\ 0 & 0 & 1 \\ 1 & 0 & 0 \end{smallmatrix} \right),$$
by Lemma \ref{pointstab}(1).  It follows that both of these groups have order $1$, $2$, or $4$; this forces both of these groups to be negligible. Third, we leave the easy verification of the equalities from the statement of the theorem 
to the reader.

It follows easily that all of the remaining cells must have negligible stabilizers, since each such stabilizer
is a subgroup of a negligible group.
\end{proof}

\subsection{A Fundamental Polyhedron for $\Gamma_{7}$}

We note that 
$$ \Gamma_{7} = 
\left\langle \mbf{v}_{1}, \frac{1}{3}\left( \mbf{v}_{2} + \mbf{v}_{3} \right), \mbf{v}_{3} \right\rangle 
\rtimes (D_{3}^{+} \times (-1)).$$

We write $\Gamma$ in place of $\Gamma_{7}$. Let $P$ denote the polyhedron
in Figure \ref{pictureofgamma7}.  The sides $S_{1}$, $S_{2}$, $S_{3}$, $S_{4}$, and $S_{5}$ are contained in the planes
$x-2y+z=1$, $x+y+z=3/2$, $x=z$, $-2x+y+z=0$, and $x+y+z=-3/2$, respectively.  The outward-pointing unit
normal vectors $N_{1}$, $N_{2}$, $N_{3}$, $N_{4}$, and $N_{5}$ are
$$ \frac{1}{\sqrt{6}} \left( \mathbf{x} - 2\mathbf{y} + \mathbf{z} \right),
\frac{1}{\sqrt{3}} \left( \mathbf{x} + \mathbf{y} + \mathbf{z} \right),
\frac{1}{\sqrt{2}} \left( \mathbf{x} - \mathbf{z} \right),
\frac{1}{\sqrt{6}} \left(-2\mathbf{x} + \mathbf{y} + \mathbf{z} \right),
\frac{-1}{\sqrt{3}} \left( \mathbf{x} + \mathbf{y} + \mathbf{z} \right),$$
respectively.


\begin{figure}[!h]
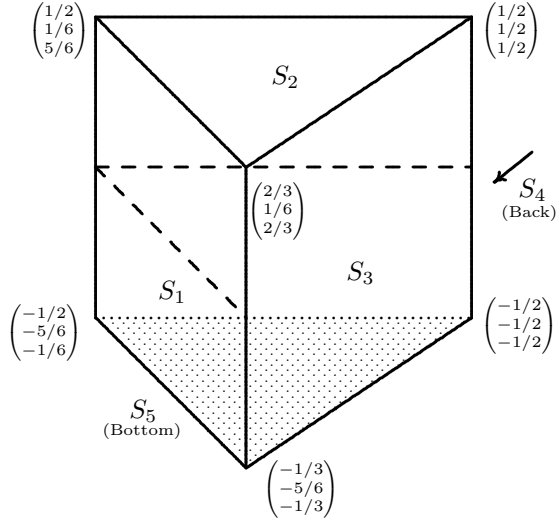
 
    
\begin{center}\hfil
\hbox{
\vbox{
  \beginpicture
\setcoordinatesystem units <1cm,1cm> point at -1  4.2
\setplotarea x from -.5 to 7, y from -2.5 to 4.5
\linethickness=.7pt
%
 \def\myarrow{\arrow <4pt> [.2, 1]} 
%
\setplotsymbol ({\circle*{.4}})
\plotsymbolspacing=.3pt        
\plot 0 0   0 4   5 4   5 0   2 -2  0 0 /
\plot 0 4   2 2  5 4 /
\plot 2 2   2 -2 /
\setdots
\setdots <3pt>   
\plot 0 0   5 0 /

\setdashes
\plot 0 2  5 2 /
\plot 0 2  2 0 /
\setsolid
\put {\tiny $\left(\begin{matrix}-1/2 \\ -5/6 \\ -1/6\end{matrix}\right)$} [r] at -.2 -.2
\put {\tiny $\left(\begin{matrix}-1/3\\-5/6\\-1/3\end{matrix}\right)$} 
[tl] at 2.2 -1.9
\put {\tiny $\left(\begin{matrix}-1/2 \\ -1/2 \\ -1/2\end{matrix}\right)$} [l]
at 5.1 -.1
\put {\tiny $\left(\begin{matrix}1/2 \\ 1/6 \\ 5/6\end{matrix}\right)$}
[tr] at -.1 4.2
\put {\tiny $\left(\begin{matrix}2/3 \\ 1/6 \\ 2/3\end{matrix}\right)$}
[t] at 2.4 1.8
\put {\tiny $\left(\begin{matrix}1/2 \\ 1/2 \\ 1/2\end{matrix}\right)$} [tl]
at 5.1 4.2

\put {$S_1$} [b] at 1 .2

\put {$S_2$} at 2.5 3.2

\put {$S_3$} at 3.5 .6

\put {$S_4$} [l] at 5.6 1.7
\myarrow from  5.8 2.2  to  5.3 1.8 
\put {\tiny (Back)} [t] at 5.75 1.5

\put {$S_5$} [tr] at .8 -1.1
\put {\tiny (Bottom)} [t] at .6 -1.4
\setshadegrid span <2pt>
\hshade -2 2 2   0 0 5 /  
\setsolid
\endpicture}
}
\hfil

\end{center}


\caption{This is a fundamental domain for the action of $\Gamma_{7}$ on ${\mathbb R}^3$. The dashed line
  indicate axes of rotation.  Note that the top and bottom of $P$ are
  triangles having angles measuring $\pi/3, \pi/2$, and $\pi/6$, where the
  right angles are in the foreground.}
\label{pictureofgamma7}
\end{figure}

The dihedral angles between vertical sides are
$$ \theta(S_{1},S_{4}) = \pi/3, \quad  \theta(S_{3},S_{4}) = \pi/6, \quad \theta(S_{1},S_{3}) = \pi/2,$$
and all other angles between sides have measure $\pi/2$.

We consider the set $\Phi = \{ \phi_{S_{1}}, \phi_{S_{2}}, \phi_{S_{3}}, \phi_{S_{4}}, \phi_{S_{5}} \}$,
defined respectively as follows:
$$ \left( \begin{smallmatrix} 1/3 \\ -2/3 \\ 1/3 \end{smallmatrix} \right) 
+ \left( \begin{smallmatrix} 0 & 0 & -1 \\ 0 & -1 & 0 \\ -1 & 0 & 0 \end{smallmatrix} \right), 
\left(\begin{smallmatrix} 1 \\ 1 \\ 1 \end{smallmatrix} \right),
\left( \begin{smallmatrix} 0 & 0 & 1 \\
0 & 1 & 0 \\ 1 & 0 & 0 \end{smallmatrix} \right), \left( \begin{smallmatrix} -1 & 0 & 0 \\ 0 & 0 & -1 \\
0 & -1 & 0 \end{smallmatrix} \right),
\left(\begin{smallmatrix} -1 \\ -1 \\ -1 \end{smallmatrix} \right).
$$
It is not difficult to check that $\Phi$ is a subproper side-pairing.  We leave the checking to the reader.
Note that every ridge cycle
has one or two elements.  All of the dihedral angle sums of the ridge cycles will be submultiplies of $\pi$,
with the exception of a cycle $[x]$, where $x$ is a point (other than the midpoint)
 on the ridge between $S_{1}$ and $S_{4}$.  
For such a cycle, the dihedral angle sum will be $2\pi/3$.  This causes no problem,
since the cycle $[x]$ is cyclic.

It follows from Theorem \ref{poincare}
that $P$ is an exact convex compact fundamental polyhedron for the action of $\langle \Phi \rangle$
on $\mathbb{R}^{3}$.  We claim that $\Gamma \subseteq \langle \Phi \rangle$,
the reverse inclusion being clear.  We sketch the proof.  First, 
$D_{3}^{+} \times (-1) = \langle \phi_{S_{3}}, \phi_{S_{4}} \rangle$.  The group $\langle \phi_{S_{1}},
\phi_{S_{3}}, \phi_{S_{4}} \rangle$ is the group generated by $D_{3}^{+} \times (-1)$ and 
$\frac{1}{3}(\mathbf{v}_{2} + \mathbf{v}_{3}) = \frac{1}{3}( \mathbf{x} - 2\mathbf{y} + \mathbf{z} )$.  It is an exercise to show that the latter
group contains $\mathbf{v}_{3} = -\mathbf{y} + \mathbf{z}$ as well.  The equality $\Gamma = \langle \Phi \rangle$ is now
clear (since $\phi_{S_{2}} = \mathbf{v}_{1}$), 
so that $P$ is an exact convex compact fundamental polyhedron for the action of $\Gamma$ on 
$\mathbb{R}^{3}$.
  
\begin{theorem} \label{theorem:g7}
Let $\widehat{\mathcal{C}}$ denote the $\Gamma$-equivariant cell structure on $\mathbb{R}^{3}$
determined by the standard cellulation $\mathcal{C}$ of $P$.  The quotient $\Gamma \backslash \mathbb{R}^{3}$
is obtained from $P$ by identifying the two halves of each of the sides $S_{1}$ and $S_{4}$, and 
by identifying $S_{2}$ with $S_{5}$ in the obvious way (see Figure \ref{pictureofgamma7}).  The
set consisting of the vertices $(0,0,0)$ and $(1/2,1/2,1/2)$ maps injectively into the
quotient.  The stabilizers of these vertices are determined by the following equalities:
\begin{eqnarray*}
\pi(\Gamma_{(0,0,0)}) & = & D_{3}^{+} \times (-1). \\
\pi(\Gamma_{(1/2,1/2,1/2)}) & = & D_{3}^{+} \times (-1).
\end{eqnarray*}
All of the other stabilizers of cells in the quotient are negligible.
\end{theorem}

\begin{proof}
The description of the quotient follows from the same argument that we have used in the other subsections.

We note that there are $9$ vertices in the standard cellulation of $P$, including the endpoints of the
dashed lines from Figure \ref{pictureofgamma7}.  These endpoints are: 
$$ (0,0,0), \quad (0, -1/3, 1/3), \quad (1/6, -1/3, 1/6).$$
In the quotient, there are only $6$ vertices, however, and we may restrict our attention to the vertices from
the top half of Figure \ref{pictureofgamma7}.

We note that the stabilizers of the vertices $(1/2,1/6,5/6)$ and $(0,-1/3,1/3)$ are negligible by 
Proposition \ref{negproposition}.
The matrix
$$ \left( \begin{smallmatrix} 0 & 1 & 0 \\ 0 & 0 & 1 \\ 1 & 0 & 0 \end{smallmatrix} \right) $$
is in neither $\pi(\Gamma_{(2/3,1/6,2/3)})$ nor $\pi(\Gamma_{(1/6,-1/3,1/6)})$.  As a result, the latter groups
have index $3$ (at least) in $D_{3}^{+} \times (-1)$, so they must be negligible.    
It is clear that $\pi(\Gamma_{(0,0,0)}) = D_{3}^{+} \times (-1)$, and the equality $\pi(\Gamma_{(1/2,1/2,1/2)}) =
D_{3}^{+} \times (-1)$ follows easily from Lemma \ref{pointstab}(1) and the fact that $h \cdot (1/2,1/2,1/2) =
\pm (1/2,1/2,1/2)$, for all $h \in D_{3}^{+} \times (-1)$.  (Note that the vertex $(-1/2,-1/2,-1/2)$ will be identified
with $(1/2,1/2,1/2)$ in the quotient.)  

Now we go on to consider edge stabilizers.  Only two of these might fail to be negligible:  the stabilizers of edges connecting
$(0,0,0)$ to $(1/2,1/2,1/2)$ and $(-1/2,-1/2,-1/2)$, respectively.  Both of the latter edges are identified in the quotient,
so we need only consider the edge between $(0,0,0)$ and $(1/2,1/2,1/2)$.  The stabilizer of this edge is negligible,
since the stabilizer of the point $(1/4,1/4,1/4)$ is negligible by Proposition \ref{negproposition}.
\end{proof}

\section{The homology groups $H_{\ast}^{\g}(E_{\fin}(\g); \mathbb{KZ}^{-\infty})$} \label{section:contributionoffinites}

In this section, we compute the homology groups  $H_{\ast}^{\g}(E_{\fin}(\g); \mathbb{KZ}^{-\infty})$, for all $73$ split three-dimensional crystallographic groups. 
In order to do this, we recall that Quinn
\cite{Qu82} established the existence of a spectral sequence that
converges to this homology group, with $E^2$-terms given by:
$$E^2_{p,q}=H_p(\g \backslash E_{\fin}(\g)\; ;\{Wh_q(\g_{\sigma})\})
\Longrightarrow H_{p+q}^{\Gamma}(E_{\fin}(\Gamma);\mathbb{KZ}^{-\infty}).$$ The chain complex that gives the homology
of $\g \backslash E_{\fin}(\g)$ with local coefficients $\{Wh_q(
\g_{\sigma})\}$ has the form
\[0 \rightarrow
 \bigoplus_{{\sigma}^{3}}^{}Wh_{q}(
\g_{{\sigma}^{3}}) \rightarrow  \bigoplus_{{\sigma}^{2}}^{}Wh_{q}(
\g_{{\sigma}^{2}}) \rightarrow \bigoplus_{{\sigma}^1}^{}Wh_q(\g_{{\sigma}^1}) \rightarrow
\bigoplus_{{\sigma}^0}^{}Wh_q( \g_{{\sigma}^0}) \rightarrow 0,
\]
where ${\sigma}^i$ denotes the cells in dimension $i$, and the sum
is over all $i$-dimensional cells in $\g \backslash E_{\fin}(\g)$. The $p^{th}$
homology group of this complex will give us the entries for the
$E^2_{p,q}$-term of the spectral sequence. Let us recall that
\[
Wh_q(F)=
\begin{cases}
Wh(F), & q=1 \\
\tilde {K}_0(\mathbb Z F), & q=0 \\
K_q(\mathbb Z F), & q \leq -1.
\end{cases}
\]
Observe that for the groups we are interested in it is particularly
easy to obtain a model for $E_{\fin}(\Gamma)$: indeed, it is
well known that, for a lattice in $\iso(\mathbb R^n)$, the
$\Gamma$-space $\mathbb R^n$ is a model for $E_{\fin}$.  In our specific
situation, we obtain models for $E_{\fin}(\Gamma)$ having very
explicit fundamental domains, namely the fundamental polyhedra given in Section 6.

\begin{remark} \label{remark:computefromgraph}
In our models for $E_{\mathcal{FIN}}(\Gamma)$, 
the $3$-dimensional cells will have trivial stabilizers, and there will be a number of
orbits of $2$-dimensional cells, each of which will have stabilizer $1$ or $\z /2$.  Note that since $Wh_q(1)$ and $Wh_q(\z /2)$ vanish for all $q\leq1$, this in particular implies that \emph{there will never be any contribution to the $E^2$-terms from the $3$-dimensional and
$2$-dimensional cells}.  In other words, $E^2_{p,q}=0$ except possibly for $p=0,1$. 
\end{remark}

\subsection{The algebraic $K$-theory of cell stabilizers in  $E_{\fin}(\Gamma)$} \label{subsection:Ktheoryofstabilizers}

In this section, we want to find the algebraic $K$-theory of the finite subgroups that occur as cell stabilizers for the $\Gamma$-action on $\mathbb R^3$, where $\Gamma$ ranges over all $73$ split crystallographic groups.  Recall from the discussion in Section \ref{section:3Dpg} that these cell stabilizers  are (up to conjugacy) precisely the point groups listed in Tables \ref{orientationpreservingpointgroups} and \ref{otherpointgroups}, and in Theorem 
\ref{classinv}. Up to isomorphism these groups are:  1, $\mathbb Z/n$, $D_n$, $\mathbb{Z}/n \times \mathbb{Z}/2$,
$D_n \times \mathbb Z/2$, $A_4$, $S_4$, $A_4 \times \mathbb Z/2$, and $S_4 \times \mathbb Z/2$, where  $n=2,3,4,6$.

\begin{table}[!h]
\renewcommand{\arraystretch}{1.6}
\begin{equation*}
\begin{array}{|c|c|c|c|}
\hline F \in \fin & Wh_q \neq 0, \;q\leq-1 & \tilde{K}_0 \neq 0 & Wh \neq 0 \\ \hline \hline 
\mathbb Z/6 & K_{-1} \cong \mathbb Z & &  \\ \hline 
\mathbb Z/4 \times \mathbb Z/2&  &   \mathbb Z/2 & \\ \hline
\mathbb Z/6 \times \mathbb Z/2& K_{-1} \cong \mathbb Z^3 &   (\mathbb Z/2)^2 & \\ \hline
D_2 \times \mathbb Z/2 & & \mathbb Z/2 & \\ \hline
D_6 & K_{-1} \cong \mathbb Z & & \\ \hline 
D_4 \times \mathbb Z/2 & & \mathbb Z/4 & \\ \hline 
D_6 \times \mathbb Z/2 & K_{-1} \cong \mathbb Z^3 & (\mathbb Z/2)^2& \\ \hline
A_4 \times \mathbb Z/2 &  K_{-1} \cong \mathbb Z &  \mathbb Z/2&\\ \hline
S_4 \times \mathbb Z/2 &  K_{-1} \cong \mathbb Z & \mathbb Z/4 & \\ \hline
\end{array}
\end{equation*}

\caption{Lower algebraic $K$-theory
of cell stabilizers in $E_{\fin}$.}
\label{table:tableKtheoryEfin}
\end{table}

In the spectral sequence computing the homology
$H_*^\Gamma(E_{\fin}(\Gamma);\mathbb{KZ}^{-\infty})$, the
$E^2$-term is computed from the algebraic $K$-groups of the various
cell stabilizers, so we need the algebraic $K$-groups for all of the
finite groups appearing in the list above.  For the convenience of
the reader, we provide in Table \ref{table:tableKtheoryEfin} the results of the
computations of these groups.  Table \ref{table:tableKtheoryEfin} provides a list of all the
\emph{non-trivial} $K$-groups that occur amongst the finite groups we
are considering.

We will now justify the results summarized in Table \ref{table:tableKtheoryEfin}. It is well known 
that if $G$ is a finite group, then $K_{q}(\mathbb ZG)$ is trivial for all $q \leq -2$  (see \cite{C80a}), so we will focus only on the $K$-groups  $K_{-1}$,  $\tilde{K}_0$, and $Wh$. For all but \emph{three} of the finite groups in our list,  the lower algebraic $K$-theory is well known; the reference \cite{LO09} collects references for most of the known results. Bass \cite[pg. 695]{Bas68}
computed $K_{-1}(\mathbb{Z}[\mathbb{Z}/6])$, Reiner and Ullom \cite[Theorem 2.2]{RU74} computed $\widetilde{K}_{0}(\mathbb{Z}[\mathbb{Z}/6])$, and Cohen \cite{Co73} computed $Wh(\mathbb{Z}/6)$.

We compute the lower algebraic $K$-theory of the groups $\mathbb Z/4 \times \mathbb Z/2$, $\mathbb Z/6 \times \mathbb Z/2$, and $A_4 \times \mathbb Z/2$ in Theorems \ref{subsubsection:Z4Z2},
\ref{subsubsection:Z6Z2}, and \ref{subsubsection:A4Z2}, respectively.


\begin{theorem}  \label{subsubsection:Z4Z2}
\[
Wh_q(\mathbb{Z}/4 \times \mathbb{Z}/2) \cong
\begin{cases}
0 & q=1 \\
\mathbb{Z}/2  & q=0 \\
0 & q \leq -1
\end{cases}
\]
\end{theorem}

\begin{proof}
Bass \cite[Theorem 10.6, pg. 695]{Bas68} proves that  $K_{-1}(\mathbb ZG)$ vanishes  if $G$ is a finite abelian group of prime power order, so $K_{-1}(\mathbb Z[\mathbb Z/4 \times \mathbb Z/2]) \cong 0$, and Alves and Ontaneda \cite[Section 5.1]{AO06} show that $Wh(\mathbb Z/4 \times \mathbb Z/2) \cong 0.$
To compute $\tilde{K}_0(\mathbb Z[\mathbb Z/4 \times \mathbb
Z/2])$, consider the following Cartesian square
\[
\xymatrix@C=20pt@R=30pt{
\mathbb Z[\mathbb Z/2][\mathbb Z/4] \ar[d] \ar[r] & \mathbb Z[\mathbb Z/4]
     \ar[d]^{\varphi_1} \\
\mathbb Z[\mathbb Z/4] \ar[r]_{\varphi_2} & \mathbb F_2[\mathbb Z/4]}
\]
where $\varphi_{i}$ is reduction mod 2 for $i=1,2$.  We let $u(\,)$ denote the group of units, and  set
\[
u^{*}(\mathbb Z[\mathbb Z/4])= \varphi_i\{u(\mathbb Z[\mathbb Z/4])\} \subset u(\mathbb F_2[\mathbb Z/4]), \hspace{1cm} i=1,2.
\] 

Let $G=\mathbb Z/4 \times \mathbb Z/2$.  Since $\mathbb QG \cong \mathbb Q[\mathbb Z/4] \otimes_{\mathbb Q} \mathbb Q[\mathbb Z/2]$, and each simple component of  $\mathbb Q[\mathbb Z/4]$, and of  $\mathbb Q[\mathbb Z/2]$, is a full matriz algebra over $\mathbb Q$, then the same is true for $\mathbb QG$, and consequently $\mathbb QG$ satisfies the Eichler condition (i.e., no simple component of $\mathbb QG$ is a totally definite quaternion algebra).   By  \cite[Theorem 1.9]{RU74}, there is an exact sequence
\[ 
1 \rightarrow u^{*}(\mathbb Z[\mathbb Z/4]) \rightarrow u(\mathbb F_2[\mathbb Z/4]) \rightarrow \tilde{K}_0(\mathbb
Z[\mathbb Z/2][\mathbb Z/4])  \rightarrow \tilde{K}_0(\mathbb
Z[\mathbb Z/4]) \oplus \tilde{K}_0(\mathbb Z[\mathbb Z/4]) \rightarrow 0.
\]
Since  $\tilde{K}_0(\mathbb Z[\mathbb Z/4])=0$ (\cite[Theorem 2.8]{RU74}),
\[
1 \rightarrow u^{*}(\mathbb Z[\mathbb Z/4]) \rightarrow u(\mathbb F_2[\mathbb Z/4]) \rightarrow \tilde{K}_0(\mathbb
Z[\mathbb Z/2][\mathbb Z/4])  \rightarrow 0.
\]
Consequently,
$$\tilde{K}_0(\mathbb
Z[\mathbb Z/2][\mathbb Z/4]) \cong  u(\mathbb F_2[\mathbb Z/4])/ u^{*}(\mathbb Z[\mathbb Z/4]).$$  

Next, let $\mathbb Z/4 = \langle \sigma \rangle$. A direct calculation shows that  $u(\mathbb F_2[\mathbb Z/4]) = \langle \sigma \rangle \times \langle \sigma + \sigma^2 + \sigma^3 \rangle \cong \mathbb Z/4 \times \mathbb Z/2$. The equality $u^{*}(\mathbb Z[\mathbb Z/4])= \langle \sigma \rangle$ follows from the fact that 
$\mathbb{Z}[\mathbb{Z}/4]$ has only trivial units (since $Wh(\mathbb{Z}/4) = 0$). Therefore, $\tilde{K}_0(\mathbb Z[\mathbb Z/4 \times \mathbb
Z/2]) =\tilde{K}_0(\mathbb Z[\mathbb Z/2][\mathbb Z/4])=  \langle \sigma + \sigma^2 + \sigma^3 \rangle \cong \mathbb Z/2$. 
\end{proof}


\begin{theorem} \label{subsubsection:Z6Z2}
\[
Wh_q(\mathbb{Z}/6 \times \mathbb{Z}/2) \cong
\begin{cases}
0 & q=1 \\
(\mathbb{Z}/2)^{2}  & q=0 \\
\mathbb{Z}^{3} & q = -1 \\
0 & q < -1
\end{cases}
\]
\end{theorem}

\begin{proof}
Alves and Ontaneda \cite[Section 5.1]{AO06} show that $Wh(\mathbb Z/6 \times \mathbb Z/2)=0.$
To compute $\tilde{K}_0(\mathbb Z[\mathbb Z/6 \times \mathbb Z/2])$, consider as before the following Cartesian square:
\[
\xymatrix@C=20pt@R=30pt{
\mathbb Z[\mathbb Z/2][\mathbb Z/6] \ar[d] \ar[r] & \mathbb Z[\mathbb Z/6]
     \ar[d]^{\varphi_1} \\
\mathbb Z[\mathbb Z/6] \ar[r]_{\varphi_2} & \mathbb F_2[\mathbb Z/6]}
\]
where $\varphi_{i}$ is reduction mod 2 for $i=1,2$.  Denote by $u(\,)$ the group of units and  set
\[
u^*(\mathbb Z[\mathbb Z/6])= \varphi_i\{u(\mathbb Z[\mathbb Z/6])\} \subset u(\mathbb F_2[\mathbb Z/6]), \hspace{1cm} i=1,2.
\]

Let $G=\mathbb Z/6 \times \mathbb Z/2$.  Since the algebra $\mathbb QG$ is commutative, no simple component of $\mathbb QG$ is a totally definite quaternion algebra, and therefore $\mathbb QG$ satisfies the Eichler condition (see the proof of Theorem \ref{subsubsection:Z4Z2}) and, as before, by  
\cite[Theorem 1.9]{RU74}, there is an exact sequence
\[
1 \rightarrow u^{*}(\mathbb Z[\mathbb Z/6]) \rightarrow u(\mathbb F_2[\mathbb Z/6]) \rightarrow \tilde{K}_0(\mathbb
Z[\mathbb Z/2][\mathbb Z/6])  \rightarrow 2\tilde{K}_0(\mathbb Z[\mathbb Z/6]) \rightarrow 0.
\]
Reiner and Ullom  \cite[Theorem 2.2]{RU74} prove that  $\tilde{K}_0(\mathbb Z[\mathbb Z/6])=0$. Therefore
the exact sequence above yields the exact sequence
\[
 1 \rightarrow u^{*}(\mathbb Z[\mathbb Z/6]) \rightarrow u(\mathbb F_2[\mathbb Z/6]) \rightarrow \tilde{K}_0(\mathbb
Z[\mathbb Z/2][\mathbb Z/6])   \rightarrow 0.
\]
  Consequently,
$$\tilde{K}_0(\mathbb
Z[\mathbb Z/2][\mathbb Z/6]) \cong  u(\mathbb F_2[\mathbb Z/6])/ u^{*}(\mathbb Z[\mathbb Z/6]).$$  

To compute  $u(\mathbb F_2[\mathbb Z/6])$, let $R=\mathbb F_2[\mathbb Z/6]$. Since $R$ is a finite ring, the canonical group homomorphism $u(R)=GL_{1}(R) \rightarrow K_1(R)$ is surjective (see \cite[Theorem 4.2(b)]{Bas64} with kernel $V(R)$ generated by the set of all  $V(x, y)= (1+xy)(1+yx)^{-1}$, with $x, y \in R$ and $1 +xy$ invertible (see \cite[Theorem 3.6(b)]{Va70}). Since   
$R$ is a finite commutative ring, $V(R)=\{1\}$, and so $K_{1}(R) \cong u(R)$. Now using \cite[Theorem 4]{Ma06}, we have that $K_1(R)=K_{1}(\mathbb F_2[\mathbb Z/6])\cong K_{1}(\mathbb F_2[\mathbb Z/2 \times \mathbb Z/3]) \cong (\mathbb Z/2)^c \times K_1(\mathbb F_2[\mathbb Z/3])$, where $c$ is the number of conjugacy classes in $\mathbb Z/3$.  A direct calculation shows  that $K_1(\mathbb F_2[\mathbb Z/3]) \cong \mathbb Z/3$ and $c=3$; then it follows $K_{1}(\mathbb F_2[\mathbb Z/6])\cong (\mathbb Z/2)^3 \times \mathbb Z/3$.
(One can also show by a direct calculation that
$u(\mathbb{F}_{2}[\mathbb{Z}/6]) = \langle \sigma^{3}, 1+\sigma+\sigma^{4},
 \sigma^{2} + \sigma^{3} + \sigma^{5} \rangle
\times \langle \sigma^{2} \rangle \cong (\mathbb{Z}/2)^{3} \times \mathbb{Z}/3$, where
$|\sigma|=6$.) 

Next, to compute  $u(\mathbb Z[\mathbb Z/6])$, let $R=\mathbb Z[\mathbb Z/6]$.  Since $R$ is a commutative ring, $K_1(R) \cong u(R) \oplus SK_1(R)$ (see \cite[page 27]{Mi71}).  Bass, Milnor, and Serre in \cite[Proposition 4.14]{BMS67} show that $SK_1(\mathbb ZH)=1$ if $H$ is cyclic, so it follows that $K_1(R) \cong u(R)$.  For $H=\mathbb Z/6$,  it is known that $K_1(\mathbb ZH) \cong \mathbb Z/2 \times H^{ab} \cong \mathbb Z/2 \times \mathbb Z/6 $ (see \cite{O89}), therefore  $u(\mathbb Z[\mathbb Z/6]) \cong \mathbb Z/2 \times \mathbb Z/6$. This implies that $u(\mathbb Z[\mathbb Z/6])$ 
consists of only the trivial units. Therefore  $u^{*}(\mathbb Z[\mathbb Z/6]) \cong \mathbb Z/6.$ Consequently
$$\tilde{K}_0(\mathbb
Z[\mathbb Z/2][\mathbb Z/6]) \cong  u(\mathbb F_2[\mathbb Z/6])/ u^{*}(\mathbb Z[\mathbb Z/6]) \cong (\mathbb Z/2)^2.$$

Next, we  show $K_{-1}(\mathbb Z[\mathbb Z/2][\mathbb Z/6]) \cong \mathbb Z^3$. Carter \cite{C80a} proved
\[
K_{-1}(\mathbb ZG) \cong \mathbb Z^r \oplus (\mZ/2)^s
\]
where
\begin{equation}
r=1-r_{\mathbb Q} + \sum_{p\,|\, |G|} (r_{\mathbb Q_{p}} -r_{\mathbb F_{p}})
\end{equation}
where $p$ is prime and $s$ is the number of simple components $A$ of $\mathbb QG$ with even Schur index but 
with $A_P$ of odd Schur index for each prime ideal $P$ of the center of $A$ that divides $|G|$ (see \cite{LMO10}).

We first recall that the group algebra $\mQ[\mathbb Z/6]$ decomposes into simple components as follows:
$$\mQ[\mathbb Z/6] \cong \mQ ^2 \oplus \mQ(\zeta_6)^2.$$

Since $\mQ [\mathbb Z/6\times \mZ/2] \cong \mQ[\mathbb Z/6] \oplus \mQ[\mathbb Z/6]$, we see that the
Schur indices of all the simple components in the Wedderburn decomposition of 
$\mQ [\mathbb Z/6 \times \mZ/2]$ are equal to 1, so $s=0$. Carter's formula (above) now tells us that $K_{-1}(\mZ [\mathbb Z/6\times \mZ/2])$ is torsion-free, 
and, from equation (2), the rank is given by
\begin{equation}
r=1-r_{\mathbb Q} + (r_{\mathbb Q_{2}} -r_{\mathbb F_{2}}) + (r_{\mathbb Q_{3}} -r_{\mathbb F_{3}}).
\end{equation}

We now proceed to compute the various terms appearing in the above expression.

Recall that for $F$ a field of characteristic $0$, $r_F$ just counts the number of 
simple components in the Wedderburn decomposition of the group algebra $F[\mathbb Z/6\times \mZ/2]$. 
From the discussion in the previous paragraph, we have that
$$\mQ [\mathbb Z/6 \times \mZ/2] \cong   \mQ ^4 \oplus \mQ(\zeta_6)^4.$$
yielding $r_{\mQ}=8$. Now by tensoring the above splitting with $\mQ _p$ with $p=2$ and 3, we obtain:
$$\mQ_p [\mathbb Z/6\times \mZ/2] \cong \mQ_p ^4 \oplus (\mQ_p \otimes _{\mQ} \mQ)(\zeta_6)^4 \cong  \mQ_p ^4 \oplus  \mQ_p(\zeta_6)^4,$$
consequently for each of the  primes $p=2,3$, we obtain that $r_{\mQ_2}=r_{\mQ_3}=8$. 

Next, we consider the situation over the finite fields $\mF_2, \mF_3$. We first recall that the 
integer $r_{\mF _p}$ counts the number of $\mF_p$-conjugacy classes
of $p$-regular elements  in $G$ (an element $x \in G$ is called \emph{$p$-regular} if $p$ does not divide the order of $x$). The $\mF _p$-conjugacy class
of an element $x \in G$ is the union of ordinary conjugacy classes of certain specific powers of $x$, where
the powers (of $x$) are calculated from the Galois  extension $\mF_p(\zeta _{m})$ where $m$ is the least common multiple of the orders  of $p$-regular elements. Note that since the fields $\mathbb F_p$, $\mathbb F_p(\zeta_m)$ are finite, and $\at(\mathbb F_p(\zeta_m)/\mathbb F_p)$ is  cyclic, 
generated by the $p$-power map (since $|\mF_p|=p$),  then 
$\gal(\mF_p(\zeta _{m})/\mF_p)=T_m =\langle \bar{p} \rangle \leq (\mathbb{Z}/m)^{\times}$ (viewed as elements of $(\mathbb{Z}/m)^{\times}$).  
We refer the reader to \cite[Section 3.1]{LMO10} for a more
complete discussion of these points.

For $p=2$, we note that an element in $\mathbb Z/6 \times \mZ/2 \cong \langle t \rangle \times \langle \sigma \rangle$ is $2$-regular precisely if it has
order $1$ or $3$.
There is a single conjugacy class of elements of order one (consisting of the identity element), and the elements of
order $3$ form \emph{two} conjugacy classes inside $\mathbb Z/6\times \mathbb Z/2$; representatives for these two conjugacy
classes are given by $x=(t^2, 1)$, and by $x^2=(t^4, 1)$. Note that there will be either one or two $\mF_2$-conjugacy
classes of elements of order $3$.  To  determine the specific powers of $x$, recall that  the powers of $x$ are given by considering the Galois group of the extension 
$\mF_2(\zeta _{3})$, viewed as elements of $(\mZ /3)^{\times}$. Since the Galois
group is generated by the 2-power map (i.e by squaring), we see that the Galois group is cyclic of order 2, given by the
residue classes $\{\bar 1, \bar 2\} \subset (\mZ /3)^{\times}$. In particular, since $\bar 2$
lies in the Galois group, we see that $x$ and $x^2$ lie in the same $\mF _2$-conjugacy class, 
implying that there is a \emph{unique} $\mF_2$-conjugacy class of elements of order 3. We conclude
that there are two $\mF_2$-conjugacy classes of $2$-regular elements, giving $r_{\mF_2}=2$.

For $p=3$, the $3$-regular elements in $\mathbb Z/6 \times \mathbb Z/2$ have order either $1$ or $2$. The elements of
order $2$ form \emph{three} conjugacy classes inside $\mathbb Z/6\times \mathbb Z/2$; representatives for these three conjugacy
classes are given by $(t^3, 1)$,  $(t^3, \sigma)$ and by $(1, \sigma)$. To determined the specific powers of these elements, recall that the powers are given by considering the Galois group of the extension $\gal\big(\mF_3(\zeta _{2})/\mF_3\big)$, viewed as elements of $(\mZ /2)^{\times}$. Since  the Galois group is generated by the 3-power map, we see that $\gal\big(\mF_3(\zeta _{2})/\mF_3\big)=T_2 =\langle \bar{3} \rangle = \{\bar 1 \} \subset
 (\mathbb{Z}/2)^{\times}$.  We  conclude that for 3-regular elements of order 2, we clearly have \emph{three} distinct (ordinary) conjugacy classes of elements of order 2; each of these ordinary conjugacy classes is also an $\mF _3$-conjugacy class. Also we clearly have a unique $\mF _3$-conjugacy class of elements of order one. We conclude that overall there are four $\mF _3$-conjugacy classes of $3$-regular elements, giving $r_{\mF_3}=4$.

We  end by substituting our calculations into the expression given in equation (3) for the rank
of $K_{-1}(\mZ [\mathbb Z/6 \times \mathbb Z/2])$, obtaining: $$r= 1 - 8 + (8 - 2) + (8 - 4)= 3.$$ Therefore $K_{-1}(\mZ [\mathbb Z/6 \times \mathbb Z/2])\cong \mathbb Z^3$ as claimed.
\end{proof}



\begin{theorem} \label{subsubsection:A4Z2}
\[
Wh_q(A_{4} \times \mathbb{Z}/2) \cong
\begin{cases}
0 & q=1 \\
\mathbb{Z}/2  & q=0 \\
\mathbb{Z} & q = -1 \\
0 & q < -1
\end{cases}
\]
\end{theorem}

\begin{proof}
Alves and Ontaneda in \cite[Lemma 5.4]{AO06} show that $Wh(A_4 \times \mathbb Z/2)=0.$

First, we show $\tilde{K}_0(\mathbb Z[A_4 \times \mathbb Z/2]) \cong \mathbb Z/2$.
To see this, let $H$ be a subgroup of  a group $G$. For any
locally free $\mathbb ZG$-module $M$ its restriction to $H$ (denoted
by $M_H$) is a locally free $\mathbb ZH$-module. The mapping defined
by $\lbrack M \rbrack \to \lbrack M_H \rbrack$ gives a homomorphism
of $\tilde{K}_0(\mathbb ZG) \to \tilde{K}_0(\mathbb ZH)$.

A group $H$ is {\it hyper-elementary} if $H$ is a semidirect product
$N \rtimes P$ of a cyclic normal subgroup $N$ and a subgroup $P$ of
prime power order, where $(|N|, |P|)=1$. Let $\mathcal H(G)$ consist of one
representative from each conjugacy class of 
hyper-elementary subgroups of $G$. We shall
need the following result  presented by  Reiner and Ullom in  \cite[Thm. 3.1]{RU74}: for every 
finite group $G$, the map
\begin{equation}
\tilde{K}_0(\mathbb ZG) \longrightarrow \prod_{H \in \mathcal H(G)} \tilde{K}_0(\mathbb ZH)
\end{equation} is a monomorphism. 
Observe that  all the proper subgroups of the alternating group $A_4$ are hyper-elementary, therefore $\mathcal H(A_4) =\{\,\mathbb Z/2,  \mathbb{Z}/3, D_2\,\}$.  Also, note that the hyper-elementary
subgroups of $G \times \mathbb Z/2$ are of the form $H$ or $H \times
\mathbb Z/2$ for $H \in \mathcal H(G)$.  In particular, the
hyper-elementary subgroups of $A_4 \times \mathbb Z/2$
are all isomorphic to one of:  $\mathbb Z/2$, $\mathbb Z/3$, $D_2$, $\mathbb Z/3 \times \mathbb Z/2$ and $D_2 \times \mathbb Z/2$. By the results given in Table \ref{table:tableKtheoryEfin} we have $\tilde{K}_0(\mathbb ZH)=0$, for all $H \in \mathcal H(A_4 \times
\mathbb Z/2)$ except for $H=D_2 \times \mathbb Z/2$, where $\tilde{K}_0(\mathbb ZH)\cong \mathbb Z/2$.   
This implies that the target of the map given in (3) is isomorphic to $\mathbb Z/2$, and injectivity
of the map now gives us an injection $\tilde{K}_0(\mathbb Z[A_4 \times\mathbb Z/2]) \hookrightarrow \mathbb Z/2$. Since it is known that $\tilde{K}_0(\mathbb Z[A_4 \times\mathbb Z/2])$ is non trivial (see \cite[Thm., pg. 161]{EH79}), it follows that $\tilde{K}_0(\mathbb Z[A_4 \times\mathbb Z/2]) \cong \mathbb Z/2$, as claimed.

Next, we show that  $K_{-1}(\mathbb Z[A_4 \times \mathbb Z/2]) \cong \mathbb Z.$  Here once again, we use Carter's formula, given in equation (2).  We start by first recalling (see \cite[pg. 93]{Se77}) that the group algebra $\mQ A_4$ decomposes into simple components as follows:

$$\mQ A_4 \cong \mQ \oplus \mQ(\zeta_3) \oplus M_3(\mQ).$$

Since $\mQ [A_4\times \mZ/2] \cong \mQ A_4 \oplus \mQ A_4$, we see that the
Schur indices of all the simple components in the Wedderburn decomposition of 
$\mQ [A_4\times \mZ/2]$ are equal to 1. Carter's result 
\cite{C80a} now tells us that $K_{-1}(\mZ [A_4\times \mZ/2])$ is torsion-free, 
and from equation (2), the rank is given by
\[
r=1-r_{\mathbb Q} + (r_{\mathbb Q_{2}} -r_{\mathbb F_{2}}) + (r_{\mathbb Q_{3}} -r_{\mathbb F_{3}}).
\]
We now proceed to compute the various terms appearing in the above expression.

For $F$ a field of characteristic $0$, $r_F$ just counts the number of 
simple components in the Wedderburn decomposition of the group algebra $F[A_4\times \mZ/2]$. 
From the discussion in the above paragraph, we have that
$$\mQ [A_4\times \mZ/2] \cong  \mQ ^2  \oplus \mQ(\zeta_3)^2 \oplus  M_3\big(\mQ)^2.$$
yielding $r_{\mQ}=6$. Now by tensoring the above splitting with $\mQ _p$, $p=2$ or $3$, we obtain:
$$\mQ_p [A_4\times \mZ/2] \cong \mQ_p^2  \oplus \mQ_p(\zeta_3) \oplus  M_3\big(\mQ_p)^2 .$$
Therefore for each of the  primes $p=2,3$, we obtain that $r_{\mQ_2}=r_{\mQ_3}=6$.

Next let us consider the situation over the finite fields $\mF_2, \mF_3$. 

For $p=2$, we note that elements in $A_4 \times \mZ /2$ are $2$-regular precisely if they have order $1$ or $3$.
There is a single conjugacy class of elements of order one (the identity element). The elements of
order $3$ form a single conjugacy class inside $A_4\times \mZ/2$.   We conclude
that there are two $\mF_2$-conjugacy classes of $2$-regular elements, giving $r_{\mF_2}=2$.

For $p=3$, the elements in $A_4 \times \mZ /2$ which are $3$-regular have order $1$ or $2$. Here we look at the Galois group 
associated to the field extension $\mF_3 (\zeta _{2})$. Elements in the Galois group are generated
by the third power, giving us that $\gal(\mF_3(\zeta _{2})/\mF_3) = \{\bar 1\} \subset (\mZ /2^{\times})$.  Now we clearly have
a unique $\mF _3$-conjugacy class of elements of order one. For elements of order $2$, there are
{\sl three} distinct (ordinary) conjugacy classes of elements of order two; each of these ordinary conjugacy
classes is also an $\mF _3$-conjugacy class. We conclude that
overall there are four  $\mF _3$-conjugacy classes of $3$-regular elements, giving $r_{\mF_3}=4$.

To conclude, we substitute our calculations into the expression in equation (3) for the rank
of $K_{-1}(\mZ [A_4\times \mZ/2])$, giving us: $$r= 1 - 6 + (6 - 2) + (6 - 4)= 1.$$  Therefore $K_{-1}(\mZ [A_4 \times \mathbb Z/2])\cong \mathbb Z$ as claimed.
\end{proof}

\subsection{The homology of $E_{\fin}(\Gamma)$}

In the remainder of this section, we will compute the generalized homology groups $H_{\ast}^{\Gamma} (E_{\fin}(\Gamma); \mathbb{KZ}^{-\infty})$ for the
$73$ split crystallographic groups $\Gamma$.



\begin{lemma} \label{lemma:computational}
Let $\Gamma$ be a split three-dimensional crystallographic group. Let $\mathcal{V}_{\Gamma}$ consist of a selection
of one non-negligible vertex from each $\Gamma$-orbit of non-negligible vertices in $E_{\fin}(\Gamma)$.  
We have the isomorphisms:
\begin{align*}
H_{-1}^{\Gamma}(E_{\fin}(\Gamma); \mathbb{KZ}^{-\infty}) &\cong E^{2}_{0,-1},\\
H_{0}^{\Gamma}(E_{\fin}(\Gamma); \mathbb{KZ}^{-\infty}) &\cong  \bigoplus_{\sigma^0 \in \mathcal{V}_{\Gamma}}^{}\widetilde{K}_{0}( \mathbb ZF_{\sigma^0}) \oplus E^{2}_{1,-1}, \\
H_{1}^{\Gamma}(E_{\fin}(\Gamma); \mathbb{KZ}^{-\infty}) &\cong 0.
 \end{align*} 
Moreover, if all of the edges in $E_{\fin}(\Gamma)$ have negligible stabilizer groups, then $E^{2}_{1,-1} = 0$ and
$$ E^{2}_{0,-1} \cong \bigoplus_{\sigma^0 \in \mathcal{V}_{\Gamma}}^{}K_{-1}( \mathbb ZF_{\sigma^0}).$$
\end{lemma} 

\begin{proof}
The lemma follows easily from the remarks about the Quinn spectral sequence that were given in the introduction of this section.

We note that the possible contributions to $H_{1}^{\Gamma}(E_{\fin}(\Gamma); \mathbb{KZ}^{-\infty})$ must come from the 
$E^{2}_{0,1}$ and $E^{2}_{1,0}$ terms. It follows from Table \ref{table:tableKtheoryEfin} that $E^{2}_{0,1} \cong 0$ for all split crystallographic
groups $\Gamma$. The $E^{2}_{1,0}$ terms come from $\widetilde{K}_{0}(\mathbb{Z}[F_{\sigma}])$, where the $\sigma$s are edges. The latter
 $K$-groups are always trivial, so $E^{2}_{1,0} \cong 0$. This proves that 
$H_{1}^{\Gamma}(E_{\fin}(\Gamma); \mathbb{KZ}^{-\infty}) \cong 0$.

The isomorphism 
$$ H_{0}^{\Gamma}(E_{\fin}(\Gamma); \mathbb{KZ}^{-\infty}) \cong  \bigoplus_{\sigma^0 \in \mathcal{V}_{\Gamma}}^{}\widetilde{K}_{0}( \mathbb ZF_{\sigma^0}) \oplus E^{2}_{1,-1}$$
follows from the fact that one or the other of the factors
$$E^{2}_{0,0}  \cong \bigoplus_{\sigma^0 \in \mathcal{V}_{\Gamma}}^{}\widetilde{K}_{0}( \mathbb ZF_{\sigma^0}) \quad \text{and} \quad E^{2}_{1,-1}$$ 
is always zero. Indeed, all edges have negligible stabilizers (so in particular $E^{2}_{1,-1} \cong 0$), except in the five cases
that are treated in Examples \ref{example:gamma5}, \ref{example:c6+-}, \ref{example:c6+}, \ref{example:d6+}, and
\ref{example:d''6}. We have non-vanishing $E^{2}_{1,-1}$ in only two of those cases, namely Examples \ref{example:c6+} and \ref{example:d''6};
in both of these cases, the $E^{2}_{0,0}$ term vanishes.
\end{proof}

\begin{remark}
Lemma \ref{lemma:computational} makes the computation of
$H_{*}^{\Gamma_{i}}(E_{\fin}(\Gamma_{i}); \mathbb{KZ}^{-\infty})$ relatively straightforward, for $i \neq 5$. Indeed,
 Theorems \ref{theorem:g1}, \ref{theorem:g2}, \ref{theorem:g3}, \ref{theorem:g4}, \ref{theorem:g6}, and \ref{theorem:g7},
imply that all of the edges in $E_{\fin}(\Gamma_{i})$ ($i \neq 5$) are negligible. We can let $\mathcal{V}_{\Gamma_{i}}$
be the set of vertices listed in the appropriate theorem (above), and then determine the isomorphism type of the group
$H_{*}^{\Gamma_{i}}(E_{\fin}(\Gamma); \mathbb{KZ}^{-\infty})$ from Lemma \ref{lemma:computational} and
Table \ref{table:tableKtheoryEfin}. We will sketch this procedure in Examples \ref{example:gamma1fin}, \ref{example:a4+fin}, and \ref{example:d2+fin}.
\end{remark}

\begin{procedure} \label{procedure:finitepart}
Let $\Gamma$ and $\Gamma'$ be split crystallographic groups, where $\Gamma'$ has finite index in $\Gamma$.
We describe a procedure for computing $H_*^{\Gamma'}(E_{\fin}(\Gamma'); \mathbb{KZ}^{-\infty})$,
assuming that the calculation of $H_*^{\Gamma}(E_{\fin}(\Gamma); \mathbb{KZ}^{-\infty})$ has been done. 
\begin{enumerate}
\item Select one cell from each $\Gamma$-orbit of non-negligible (relative to $\Gamma$) cells in $E_{\mathcal{FIN}}(\Gamma)$; call
the resulting set $\mathcal{C}_{\Gamma}$. 
\item Let $\Gamma'$ act on $E_{\mathcal{FIN}}(\Gamma)$ (which is clearly a model for $E_{\fin}(\Gamma')$). Since the class of negligible groups is closed under passage to subgroups (see Definition \ref{definition:negligible}), the only cells
$\sigma \subseteq E_{\mathcal{FIN}}(\Gamma)$ such that $\Gamma'_{\sigma}$ is non-negligible must be in the set $\Gamma \cdot \mathcal{C}_{\Gamma}$. We can
write $\Gamma = \Gamma' T$, where $T$ is a right transversal for $\Gamma'$ in $\Gamma$. Thus, 
$\Gamma \cdot \mathcal{C}_{\Gamma} = \Gamma' \cdot (T \cdot \mathcal{C}_{\Gamma})$. (Note that $T \cdot \mathcal{C}_{\Gamma}$ is finite when $\mathcal{C}_{\Gamma}$
is.) We choose a single cell from each $\Gamma'$-orbit that meets $T \cdot \mathcal{C}_{\Gamma}$. Call the resulting set $\widehat{\mathcal{C}}_{\Gamma'}$. We note that if a cell $\sigma \subseteq E_{\fin}(\Gamma)$
has the property that $\Gamma'_{\sigma}$ is non-negligible, then $\sigma$ is in 
$\Gamma' \cdot \widehat{\mathcal{C}}_{\Gamma'}$.
\item It is still possible that $\widehat{\mathcal{C}}_{\Gamma'}$ contains cells $\sigma$ such that
$\Gamma'_{\sigma}$ is negligible. We therefore recompute the cell stabilizer $\Gamma'_{\sigma}$ for each $\sigma \in \widehat{\mathcal{C}}_{\Gamma'}$, 
removing $\sigma$ from our list if $\Gamma'_{\sigma}$ is negligible. The result is the desired list of cells $\mathcal{C}_{\Gamma'}$
(which contains a single cell from each $\Gamma'$-orbit of cells $\sigma$ such that $\Gamma'_{\sigma}$ is non-negligible, and no cells from any other orbits).
\item If $\mathcal{C}_{\Gamma'}$ consists entirely of vertices (i.e., if all of the edge stabilizers in $E_{\fin}(\Gamma')$
are negligible), then
we can determine $H_{\ast}^{\Gamma'}(E_{\fin}(\Gamma'); \mathbb{KZ}^{-\infty})$ from Lemma 
\ref{lemma:computational} and Table \ref{table:tableKtheoryEfin}. If there are non-negligible edge stabilizers, then we apply the quotient map $p: E_{\mathcal{FIN}}(\Gamma) \rightarrow \Gamma' \backslash E_{\mathcal{FIN}}(\Gamma)$ to the cells in $\mathcal{C}_{\Gamma'}$. We note that all of the cells in $\mathcal{C}_{\Gamma'}$
are either $0$- or $1$-dimensional, so the image is a graph. 
We  can then use the graph $p(\mathcal{C}_{\Gamma'})$ to compute the $E^{2}_{0,-1}$ and $E^{2}_{1,-1}$ terms. (The
remainder of the calculation is straightforward.)
\end{enumerate}
\end{procedure}

In Subsection \ref{subsection:examplesfin}, we will give examples to illustrate Procedure \ref{procedure:finitepart}, and provide
tables listing cell stabilizer information for all of the split crystallographic groups with non-negligible point groups. Table \ref{thefinitepart}
summarizes all of the non-zero isomorphism types of the groups $H_{*}^{\Gamma}(E_{\fin}(\Gamma); \mathbb{KZ}^{-\infty})$.

\subsection{Calculations of $H^{\Gamma}_{*}(E_{\fin}(\Gamma); \mathbb{KZ}^{-\infty})$}
\label{subsection:examplesfin}

Let us compute some of the groups $H^{\Gamma}_{*}(E_{\fin}(\Gamma); \mathbb{KZ}^{-\infty})$, using Procedure \ref{procedure:finitepart}. 

\begin{example} \label{example:gamma1fin}
Consider first the group $\Gamma_{1}$, which is generated by the standard cubical lattice $L$ and the point group $S_{4}^{+} \times (-1)$. Theorem \ref{theorem:g1} showed that all of the edges in $E_{\fin}(\Gamma_{1})$ are negligible relative to the action of $\Gamma_{1}$. The non-negligible vertex
stabilizers are described in Table \ref{finitesubgroupsingamma1}. Note that the columns of Table \ref{finitesubgroupsingamma1} are labelled by point groups, and the rows 
are labelled by split crystallographic groups (see Remark \ref{remark:labelconvention}
for a guide to our labelling convention). If a $3$-tuple $(a,b,c) \in \mathbb{R}^{3}$ appears in a box in the table, then the corresponding column heading is
$\pi( \Gamma_{(a,b,c)})$, where $\Gamma$ is the crystallographic group labelling the row, and $\pi: \Gamma \rightarrow H$ is the usual 
projection. Thus, for instance, the first entry in Table \ref{finitesubgroupsingamma1}
tells us that 
$\pi( \Gamma_{(0,0,0)})$ and $\pi( \Gamma_{(1/2, 1/2, 1/2)})$
are both $S_{4}^{+} \times (-1)$, where we have written $\Gamma$ in place of $\Gamma_{1}$ (the label of that row).

We easily see from Table \ref{finitesubgroupsingamma1} that there are $4$ vertices in $\mathcal{C}_{\Gamma_{1}}$. Two of the vertices have stabilizer groups
isomorphic to $S_{4} \times \mathbb{Z}/2$, and two have stabilizer groups isomorphic to $D_{4} \times \mathbb{Z}/2$. 
By Lemma \ref{lemma:computational} and Table \ref{table:tableKtheoryEfin}, we get
\begin{align*}
H^{\Gamma_{1}}_{-1}(E_{\fin}(\Gamma_{1}); \mathbb{KZ}^{-\infty}) &\cong \mathbb{Z}^{2}; \\
H^{\Gamma_{1}}_{0}(E_{\fin}(\Gamma_{1}); \mathbb{KZ}^{-\infty}) &\cong (\mathbb{Z}/4)^{4}; \\
\end{align*}
(We note that, here and in all of the other cases, $H^{\Gamma}_{1}(E_{\fin}(\Gamma); \mathbb{KZ}^{-\infty}) \cong 0$.) These calculations are recorded in the first row of Table \ref{thefinitepart}.
\end{example}

\begin{table}[!h]
\renewcommand{\arraystretch}{1.3}\begin{equation*}
\scriptsize\begin{array}{ | c | c | c | c | c | c | c |} \hline
& S^{+}_{4} \times (-1)& D^{+}_{4_{1}} \times (-1)& A^{+}_{4} \times (-1)&  D^{+}_{4} \times (-1) &  D^{+}_{2} \times (-1)&  C^{+}_{4} \times (-1)   \\ \hline \hline
\Gamma_{1} & (0,0,0)  &( \frac{1}{2},  0, 0)  & & (\frac{1}{2}, \frac{1}{2}, 0) & & \\ 
& (\frac{1}{2}, \frac{1}{2}, \frac{1}{2}) & & & &   &  \\ \hline    
 (A^{+}_{4} \times (-1))_{1} & && (\frac{1}{2}, \frac{1}{2}, \frac{1}{2}) & &( \frac{1}{2},  0, 0)    & \\ 
&&& (0,0,0) & & (\frac{1}{2}, \frac{1}{2}, 0) &   \\ \hline
(D^{+}_{4} \times (-1))_{1} & &&& (0,0,0) & (\frac{1}{2},0,0) & \\
& &&& (\frac{1}{2}, \frac{1}{2}, \frac{1}{2}) & (0, \frac{1}{2}, \frac{1}{2}) & \\
& &&& (0,0,\frac{1}{2}) & & \\
 & &&& (\frac{1}{2},\frac{1}{2},0) &  & \\ \hline
(D^{+}_{2} \times (-1))_{1}    & &  & & & (8) &   \\ \hline
 (C^{+}_{4} \times (-1))_{1} & &&&&& (0,0,0)\\
& &&&&& (\frac{1}{2}, \frac{1}{2}, \frac{1}{2}) \\
& &&&&& (0,0,\frac{1}{2}) \\
 & &&&&& (\frac{1}{2},\frac{1}{2},0)  \\ \hline
\end{array}
\end{equation*}
\caption{Cell stabilizers in $E_{\fin}(\Gamma_1)$. The ``$(8)$" is an abbreviation for the collection
of all $3$-tuples $(a,b,c)$ such that $a$, $b$, $c \in \{ 0, 1/2 \}$.}
\label{finitesubgroupsingamma1}
\end{table}


\begin{table}
\renewcommand{\arraystretch}{1.3}
\begin{equation*}
\scriptsize\begin{array}{ | c | c | c | c | c | c  |} \hline
& D^{+}_{4} \times (-1)& D^{+}_{3} \times (-1)&  D^{+}_{2} \times (-1) &  C^{+}_{4} \times (-1)&  C^{+}_{3} \times (-1) \\ \hline \hline
\Gamma_{2} & (\frac{1}{2},0,0)^{\dag} &  (\frac{1}{4}, \frac{1}{4}, \frac{1}{4}) 
&&&\\ \hline    
 (A^{+}_{4} \times (-1))_{2}        & && (\frac{1}{2}, 0, 0) & 
&  (\frac{1}{4}, \frac{1}{4}, \frac{1}{4})  \\ \hline
(D^{+}_{4} \times (-1))_{2} & (0,0,0) && (\frac{1}{2},0,0)&  &\\
& (0,0,\frac{1}{2}) &&&&\\ \hline
(D^{+}_{2} \times (-1))_{2}    & &  & (0,0,0) & &     \\ 
& &  & (\frac{1}{2},0,0) & &     \\ 
 & &  & (0,\frac{1}{2},0) & &    \\ 
 & &  & (0,0,\frac{1}{2}) & &   \\ \hline
(C^{+}_{4} \times (-1))_{2}    & &  & &(0,0,0) &      \\ 
 & &  & & (0,0,\frac{1}{2}) &    \\ \hline
\end{array}
\end{equation*}
\caption{Cell stabilizers in $E_{\fin}(\Gamma_2)$. Note that the dagger ($\dag$) indicates a vertex with the stabilizer group $D_{4_{1}} \times (-1)$. We have also omitted two entries 
for formatting reasons. The first row should have the origin $(0,0,0)$ listed with stabilizer group $S_{4}^{+} \times (-1)$, and the second row should list the origin with the stabilizer group
$A^{+}_{4} \times (-1)$.}
\label{finitesubgroupsingamma2}
\end{table}


\begin{table}
\renewcommand{\arraystretch}{1.3}
\begin{equation*}
\begin{array}{ | c | c | c | c | c  |} \hline
& S^{+}_{4} \times (-1)& A^{+}_{4} \times (-1)&  D^{+}_{2} \times (-1) &  D_{2} \times (-1) \\ \hline \hline
\Gamma_{3} & (0,0,0) &  &&(\frac{1}{4}, \frac{1}{4}, 0) \\ 
& (\frac{1}{2},0,0) &&& \\ \hline    
 (A^{+}_{4} \times (-1))_{3} && (0,0,0) &&\\ 
&& (\frac{1}{2},0,0) &&\\ \hline    
(D^{+}_{2} \times (-1))_{3} &&& (0,0,0) & \\ 
&& &(\frac{1}{2},0,0) &\\ \hline  \hline
\Gamma_{4} &&& (0,0,0) & \\ 
 &&& (0, \frac{1}{2}, \frac{1}{2}) & \\ 
 &&& (0,0,\frac{1}{2}) & \\ 
&&& (0,\frac{1}{2},0) & \\ \hline \hline
\end{array}
\end{equation*}
\caption{Cell stabilizers in $E_{\fin}(\Gamma_3)$ and $E_{\fin}(\Gamma_{4})$.}
\label{finitesubgroupsingamma34}
\end{table}

\begin{table}
\renewcommand{\arraystretch}{1.3}
\begin{equation*}
\begin{array}{ | c | c | c | c | c  | c |} \hline
& D^{+}_{3} \times (-1) &  C^{+}_{3} \times (-1)  \\ \hline \hline
\Gamma_{6} & (0,0,0) &  \\
&  (\frac{5}{6}, \frac{-1}{6}, \frac{-1}{6}) &  \\ \hline
(C^{+}_{3} \times (-1))_{6} && (0,0,0)   \\
&& (\frac{5}{6}, \frac{-1}{6}, \frac{-1}{6})   \\ \hline \hline
\Gamma_{7} & (0,0,0) &   \\
 & (\frac{1}{2}, \frac{1}{2}, \frac{1}{2}) &   \\ \hline
\end{array}
\end{equation*}
\caption{Cell stabilizers in  $E_{\fin}(\Gamma_{6})$, and $E_{\fin}(\Gamma_{7})$.} 
\label{finitesubgroupsingamma67}
\end{table}


\begin{table}
\renewcommand{\arraystretch}{1.3}
\begin{equation*}
\begin{array}{  | c | c | c | c  | c |} \hline
 & D'_{6} & D''_{6} & \widehat{D}'_{6} & C^{+}_{6} \times (-1) \\ \hline \hline
(D'_{6})_{5}   & (0,0,0) & & & \\
 & (\frac{1}{2}, \frac{1}{2}, \frac{1}{2}) & & & \\
 & (\frac{1}{3}, -\frac{2}{3}, \frac{1}{3}) & & & \\
 & (\frac{5}{6}, -\frac{1}{6}, \frac{5}{6}) & & & \\
 & (-\frac{1}{3}, \frac{2}{3},-\frac{1}{3}) & & & \\
 & (-\frac{5}{6}, \frac{1}{6}, -\frac{5}{6}) & & & \\ \hline
(D''_{6})_{5}  & &(0,0,0)  & & \\
 & &(\frac{1}{2}, \frac{1}{2}, \frac{1}{2}) & &  \\
 & &(\frac{1}{4}, \frac{1}{4}, \frac{1}{4})^{\ast} &  & \\
 & &(-\frac{1}{4}, -\frac{1}{4}, -\frac{1}{4})^{\ast} &  & \\ \hline
(\widehat{D}'_{6})_{5}  & & & (0,0,0) & \\
 & & & (\frac{1}{2}, \frac{1}{2}, \frac{1}{2}) & \\ \hline
 (C^{+}_{6} \times (-1))_{5}  & & & & (0,0,0) \\
&&&& (\frac{1}{2}, \frac{1}{2}, \frac{1}{2}) \\ \hline
\end{array}
\end{equation*}
\caption{Cell stabilizers in  $E_{\fin}(\Gamma_5)$ (part I). An asterisk denotes the stabilizer of an edge, 
where the coordinates indicate the edge's midpoint.
For the cell stabilizers in $\Gamma_{5}$ itself, we refer the reader to Theorem \ref{theorem:g5}. Note that this table
and Table \ref{finitesubgroupsingamma567} should be considered two halves of the same table (which was split
in two purely for the sake of better formatting).}
\label{finitesubgroupsingamma5}
\end{table}


\begin{table}
\renewcommand{\arraystretch}{1.3}
\begin{equation*}
\begin{array}{ | c | c | c | c | c  | c |} \hline
&  D^{+}_{6}  & D^{+}_{3} \times (-1) & C^{+}_{6} & C'_{6} &  C^{+}_{3} \times (-1)  \\ \hline \hline
(\widehat{D}'_{6})_{5} & & & &  (\frac{1}{3}, -\frac{2}{3}, \frac{1}{3}) &   \\
& & & &  (\frac{5}{6}, \frac{-1}{6}, \frac{5}{6})  & \\ \hline
(C^{+}_{6} \times (-1))_{5} & & & (\frac{1}{4}, \frac{1}{4}, \frac{1}{4})^{\ast} & (\frac{1}{3}, \frac{-2}{3}, \frac{1}{3})  & \\
& &&& (\frac{5}{6}, \frac{-1}{6}, \frac{5}{6}) & \\ \hline
(D^{+}_{6})_{5} & (0,0,0) & & (\frac{1}{4},\frac{1}{4},\frac{1}{4})^{\ast}  & & \\
& (\frac{1}{2},\frac{1}{2},\frac{1}{2})  & & & & \\ \hline
(D^{+}_{3} \times (-1))_{5} & & (0,0,0) &  & & \\
& & (\frac{1}{2}, \frac{1}{2}, \frac{1}{2})  & & & \\ \hline
(C^{+}_{6})_{5} & & &(0,0,0)   & &\\
& & &(\frac{1}{2}, \frac{1}{2}, \frac{1}{2})  & & \\
& & &(\frac{1}{4}, \frac{1}{4}, \frac{1}{4})^{\ast}   && \\
& & &(\frac{-1}{4}, \frac{-1}{4}, \frac{-1}{4})^{\ast}   & &\\ \hline
(C'_{6})_{5} &  &&& (0,0,0)  & \\
& &&& (\frac{1}{2}, \frac{1}{2}, \frac{1}{2})   & \\
& &&& (\frac{1}{3}, \frac{-2}{3}, \frac{1}{3})   & \\
& &&& (\frac{5}{6}, \frac{-1}{6}, \frac{5}{6})   & \\
& &&& (\frac{-1}{3}, \frac{2}{3},\frac{-1}{3})   & \\
& &&& (\frac{-5}{6}, \frac{1}{6}, \frac{-5}{6})  & \\ \hline
(C^{+}_{3} \times (-1))_{5}  & & & &  (0,0,0) & \\
&&&& (\frac{1}{2}, \frac{1}{2}, \frac{1}{2}) & \\ \hline 
\end{array}
\end{equation*}
\caption{Cell stabilizers in $E_{\fin}(\Gamma_5)$ (part II).
An asterisk indicates an edge stabilizer, where the
coordinates are for the midpoint of the edge.} 
\label{finitesubgroupsingamma567}
\end{table}

\begin{example} \label{example:a4+fin}
Next, we follow Procedure \ref{procedure:finitepart} with
$\Gamma' = (A^{+}_{4} \times (-1))_{1}$, and $\Gamma = \Gamma_{1}$. A right transversal $T$ for $\Gamma'$ in
$\Gamma$ is as follows:
$$ \left\{ \left(\begin{smallmatrix} 1 & 0 & 0 \\ 0 & 1 & 0 \\ 0 & 0 & 1 \end{smallmatrix}\right), 
 \left(\begin{smallmatrix} 0 & 1 & 0 \\ 1 & 0 & 0 \\ 0 & 0 & 1 \end{smallmatrix}\right) \right\}.$$
We can let $\mathcal{C}_{\Gamma}$ be the collection of vertices from the first row of Table \ref{finitesubgroupsingamma1}.
We note that $T \cdot \mathcal{C}_{\Gamma}$ contains five vertices: the four vertices from $\mathcal{C}_{\Gamma}$,
and $(0,1/2,0)$. We now choose a cell from each $\Gamma'$-orbit meeting $T \cdot \mathcal{C}_{\Gamma}$;
we can simply choose $\widehat{\mathcal{C}}_{\Gamma'} = \mathcal{C}_{\Gamma}$, 
since the new vertex $(0,1/2,0)$ is in the $\Gamma'$-orbit 
of $(1/2,0,0) \in \mathcal{C}_{\Gamma}$. The next step is to compute the stabilizer groups $\Gamma'_{v}$
of the vertices $v \in \widehat{\mathcal{C}}_{\Gamma'}$. This is straightforward, and amounts to computing
the intersections of 
$$ S_{4}^{+} \times (-1), \quad S_{4}^{+} \times (-1), \quad D_{4_{1}}^{+} \times (-1), 
\quad \text{ and } \quad D_{4}^{+} \times (-1)$$
with $A_{4}^{+} \times (-1)$ (respectively). As a result, we see two vertices $v$ (namely $(0,0,0)$ and $(1/2,1/2,1/2)$)
such that $\pi(\Gamma'_{v}) = A_{4}^{+} \times (-1)$, and two vertices such that $\pi(\Gamma'_{v}) = D_{2}^{+} \times (-1)$.
All of these vertices are non-negligible, so $\mathcal{C}_{\Gamma'} = \widehat{\mathcal{C}}_{\Gamma'}$.
The four vertices are recorded in the second row of Table \ref{finitesubgroupsingamma1}, in the appropriate columns.
It now follows from Lemma \ref{lemma:computational} and Table \ref{table:tableKtheoryEfin} that
\begin{align*}
H^{\Gamma'}_{-1}(E_{\fin}(\Gamma'); \mathbb{KZ}^{-\infty}) &\cong \mathbb{Z}^{2}; \\
H^{\Gamma'}_{0}(E_{\fin}(\Gamma'); \mathbb{KZ}^{-\infty}) &\cong (\mathbb{Z}/2)^{4}.
\end{align*}
These calculations are recorded in Table \ref{thefinitepart}.
\end{example}

\begin{example} \label{example:d2+fin}
Now we follow the same procedure with $\Gamma = (A^{+}_{4} \times (-1))_{1}$ and 
$\Gamma' = (D^{+}_{2} \times (-1))_{1}$. We can let $T = C_{3}^{+}$. (Recall that
$C_{3}^{+}$ is the group of matrices that cyclically permute the coordinates -- see 
Table \ref{orientationpreservingpointgroups}.) We can let $\mathcal{C}_{\Gamma}$
be the same collection of four vertices from the second row of Table \ref{finitesubgroupsingamma1}. Applying
$T$, we find
$$ T \cdot \mathcal{C}_{\Gamma} = \{ (a,b,c) \in \mathbb{R}^{3} \mid a, b, c \in \{0, 1/2 \} \}.$$
Next, we must choose one cell from each $\Gamma'$-orbit that meets $T \cdot \mathcal{C}_{\Gamma}$. In fact,
all of the elements of $T \cdot \mathcal{C}_{\Gamma}$ are easily seen to be in distinct $\Gamma'$-orbits. We can
therefore set $\widehat{\mathcal{C}}_{\Gamma'} = T \cdot \mathcal{C}_{\Gamma}$. The next step is to compute
the groups $\pi(\Gamma'_{v})$ for the vertices in $\widehat{\mathcal{C}}_{\Gamma'}$. The best approach here
may be to use Lemma \ref{pointstab}(1). We note that, for $h \in D_{2}^{+} \times (-1)$,
$(a,b,c) - h \cdot (a,b,c) = (a',b',c')$, where $x'$ is either $0$ or $2x$, for $x \in \{ a, b, c \}$.
It follows from Lemma \ref{pointstab}(1) that $\pi(\Gamma'_{v}) = D_{2}^{+} \times (-1)$ for
each $v \in \widehat{\mathcal{C}}_{\Gamma'}$. This implies that each member of the latter set is non-negligible,
so we can let $\mathcal{C}_{\Gamma'} = \widehat{\mathcal{C}}_{\Gamma'}$. It now follows from   
Lemma \ref{lemma:computational} and Table \ref{table:tableKtheoryEfin} that 
\begin{align*}
H^{\Gamma'}_{-1}(E_{\fin}(\Gamma'); \mathbb{KZ}^{-\infty}) &\cong 0; \\
H^{\Gamma'}_{0}(E_{\fin}(\Gamma'); \mathbb{KZ}^{-\infty}) &\cong (\mathbb{Z}/2)^{8},
\end{align*}
as recorded in Table \ref{thefinitepart}.
\end{example}

Examples \ref{example:gamma1fin}, \ref{example:a4+fin}, and \ref{example:d2+fin} illustrate the general pattern in
``easy" cases -- those in which the edge stabilizers are negligible. Note that, given $\Gamma$ and $\Gamma'$,
a choice of a right transversal for $\Gamma'$ in $\Gamma$ is made during the application
of Procedure \ref{procedure:finitepart}. Of course, in a given
case, several choices of transversal are possible, and these could easily give us vertices (or edges) that are different
from the ones recorded in Tables \ref{finitesubgroupsingamma1}, 
\ref{finitesubgroupsingamma2}, \ref{finitesubgroupsingamma34},  
\ref{finitesubgroupsingamma67}, \ref{finitesubgroupsingamma5}, and \ref{finitesubgroupsingamma567}. (In fact,
we choose orbit representatives, too, and this could also give us different cells.) The calculation
of homology is unaffected, however, since the resulting cells are necessarily in the same $\Gamma'$-orbits no matter which
transversal is selected. In practice, we have always chosen the groups $\Gamma$ and $\Gamma'$ in such a way that
$[\Gamma: \Gamma'] \leq 3$. We also favored certain transversals, such as groups of permutation matrices (as in the
above examples), or the group generated by the antipodal map (when applicable).

Note also that most of the split crystallographic groups in Table \ref{splitcrystallographicgroups} have negligible point
groups, which directly implies that the homology groups in question
are $0$. This observation reduces the number of split crystallographic groups that must be considered.

There are five hard cases. We consider these next.

\begin{example}[The case of $\Gamma_{5}$] \label{example:gamma5} We would like to 
compute the generalized homology groups
$H_{*}^{\Gamma_{5}}(E_{\fin}(\Gamma_{5}); \mathbb{KZ}^{-\infty})$. We work with the vertices and edges from the statement of Theorem \ref{theorem:g5}. There is only one non-negligible
edge, which connects $(0,0,0)$ to $(1/2,1/2,1/2)$.
The only part of the calculation that is not completely straightforward
is the calculation of the $E^{2}_{0,-1}$ and $E^{2}_{1,-1}$
terms; we compute these from the complex described in Figure \ref{graphforgamma5}. 
\begin{figure}[!h]
\vbox{\hfil \hspace{-3.5cm}\beginpicture
\setcoordinatesystem units <.7cm,.7cm> point at -.2 1
\setplotarea x from -.2 to 5, y from -2 to 1
\setlinear
\linethickness=.8pt
\putrule from 0 0 to 4 0
\put {$\bullet$} at  0  0
\put {$\bullet$} at  4 0
\put {$D_6^{+} \times (-1)$} [t] at   0 -.2
\put {$D_6^{''}$} [t] at   2 0.8
\put {$D_6^{+} \times (-1)$} [t] at   4 -.2
\endpicture}\hfil
\vspace{-.8 cm}
\caption{The non-negligible edge in $\Gamma_{5} \backslash E_{\fin}(\Gamma_{5})$. In addition to the cells pictured, there are $2$ isolated vertices with stabilizer subgroup $D_2 \times \mathbb Z/2$ and $2$ other isolated vertices with stabilizer subgroup $D_6$.}
\label{graphforgamma5}
\end{figure}

We compute the homology of the chain complex:
\[
0\rightarrow \bigoplus_{{\sigma}^1 \in \mathcal{C}_{\Gamma_{5}}}^{} K_{-1}(\mathbb{Z}\g_{{\sigma}^1}) \rightarrow
\bigoplus_{{\sigma}^0 \in \mathcal{C}_{\Gamma_{5}}}^{}K_{-1}(\mathbb{Z}\g_{{\sigma}^0}) \rightarrow 0,
\]
By Figure \ref{graphforgamma5} and Table \ref{table:tableKtheoryEfin}, the latter chain complex becomes:
$$0\rightarrow K_{-1}(\mathbb ZD_6) \xrightarrow{\rho} (K_{-1}(\mathbb Z[D_6 \times \mathbb Z/2])  \oplus  K_{-1}(\mathbb Z[D_6 \times \mathbb Z/2])) \oplus 2K_{-1}(\mathbb ZD_6) \rightarrow 0.$$

The morphism $\rho$ is determined by the map $K_{-1}(\z D_6)\rightarrow K_{-1}(\z [D_6\times \z/2]).$ We get $K_{-1}(\z D_6)\cong \z$ and $K_{-1}(\z
[D_6\times \z/2]) \cong \z ^3$ by Table \ref{table:tableKtheoryEfin}.  We claim that the map
induced by the natural inclusion $D_6\hookrightarrow D_6\times \z
/2$ is injective, and the quotient group is isomorphic to $\z ^2$.
In order to see this, we merely note that there is a retraction from
$D_6 \times \z/2$ to the subgroup $D_6$, and hence we must have that
$K_{-1}(\z D_6)\cong \z$ is a summand inside $K_{-1}(\z [D_6\times
\z/2])\cong \z ^3$. This immediately gives  the following isomorphisms:
\begin{align*}
E^{2}_{1,-1} &\cong 0; \\
E^{2}_{0,-1} &\cong \mathbb{Z}^{7}.
\end{align*}
By Lemma \ref{lemma:computational}, it is now enough to record $\widetilde{K}_{0}(\mathbb{Z}[F_{\sigma_{0}}])$, where $F_{\sigma_{0}}$ ranges over
the vertex stabilizers from the statement of Theorem \ref{theorem:g5}. We have $\widetilde{K}_{0}(\mathbb{Z}[D_{2} \times \mathbb{Z}/2]) \cong \mathbb{Z}/2$
and $\widetilde{K}_{0}(\mathbb{Z}[D_{6} \times \mathbb{Z}/2]) \cong (\mathbb{Z}/2)^{2}$ (see Table \ref{table:tableKtheoryEfin}). There are two vertices
having each of the latter groups as stabilizers, contributing (in all) a factor of $(\mathbb{Z}/2)^{6}$ to the reduced $K_{0}$. This calculation is recorded 
 in Table \ref{thefinitepart}.
\end{example}

\begin{example}[The case of $(C_{6}^{+} \times (-1))_{5}$] \label{example:c6+-} We now consider the group $\Gamma = 
(C_{6}^{+} \times (-1))_{5}$. 
We will assume that Procedure \ref{procedure:finitepart} has been followed up to (4). There are a total of $5$ cells
in $C_{\Gamma}$, as recorded in Tables \ref{finitesubgroupsingamma5} and \ref{finitesubgroupsingamma567}. 
If we apply the quotient map $p$ to the cells $\mathcal{C}_{\Gamma}$, we get the complex described in Figure \ref{figure:c6+-}.

\begin{figure}[!h]
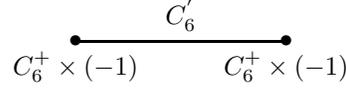

\vbox{\hfil \hspace{-3.5cm} \beginpicture
\setcoordinatesystem units <.7cm,.7cm> point at -.2 1
\setplotarea x from -.2 to 5, y from -2 to 1
\setlinear
\linethickness=.8pt
\putrule from 0 0 to 4 0
\put {$\bullet$} at  0  0
\put {$\bullet$} at  4 0
\put {$C_6^{+} \times (-1)$} [t] at   0 -.2
\put {$C_6^{'}$} [t] at   2 0.8
\put {$C_6^{+} \times (-1)$} [t] at   4 -.2
\endpicture}\hfil
\vspace{-.8 cm}
\caption{The graph $p(\mathcal{C}_{\Gamma})$, for $\Gamma = (C_{6}^{+} \times (-1))_{5}$. There are two
isolated vertices (not pictured), both of which have stabilizer groups isomorphic to $\mathbb{Z}/6$.}
\label{figure:c6+-}
\end{figure}

As before, we want to compute the $E^{2}_{0,-1}$ and $E^{2}_{1,-1}$ terms. The chain complex 
\[
0\rightarrow \bigoplus_{{\sigma}^1 \in \mathcal{C}_{\Gamma}}^{}K_{-1}(\mathbb{Z}\g_{{\sigma}^1}) \rightarrow
\bigoplus_{{\sigma}^0 \in \mathcal{C}_{\Gamma}}^{}K_{-1}( \mathbb{Z}\g_{{\sigma}^0}) \rightarrow 0
\]
becomes
$$0\rightarrow K_{-1}(\mathbb Z[\mathbb Z/6]) \xrightarrow{\rho} 2 K_{-1}(\mathbb Z[\mathbb Z/6 \times \mathbb Z/2]) \oplus 2K_{-1}(\mathbb Z[\mathbb Z/6]) \rightarrow 0.$$
The argument follows the pattern from Example \ref{example:gamma5}.
Note that the morphism $\rho$ is determined by the map $K_{-1}(\z[\mathbb Z/6])\rightarrow K_{-1}(\z [\mathbb Z/6\times \mathbb Z/2]).$ 
We get $K_{-1}(\z[\mathbb Z/6])\cong \z$ and $K_{-1}(\z
[\mathbb Z/6\times \mathbb Z/2]) \cong \z ^3$ by Table \ref{table:tableKtheoryEfin}.  We claim that the map
induced by the natural inclusion $\mathbb Z/6\hookrightarrow \mathbb Z/6\times \mathbb Z/2$ is injective, and the quotient group is isomorphic to $\z^2$.
There is a retraction from
$\mathbb Z/6 \times \mathbb Z/2$ to the subgroup $\mathbb Z/6$, and so
$K_{-1}(\z[\mathbb Z/6])\cong \z$ is a summand inside $K_{-1}(\z [\mathbb Z/6\times
\mathbb Z/2])\cong \z ^3$. It follows easily that
\begin{align*}
E^{2}_{1,-1} &\cong 0; \\
E^{2}_{0,-1} &\cong \mathbb{Z}^{7}.
\end{align*}
By Lemma \ref{lemma:computational}, it is now enough to record $\widetilde{K}_{0}(\mathbb{Z}[F_{\sigma_{0}}])$, where $F_{\sigma_{0}}$ ranges over
the vertex stabilizers. We have 
$\widetilde{K}_{0}(\mathbb{Z}[\mathbb{Z}_{6} \times \mathbb{Z}/2]) \cong (\mathbb{Z}/2)^{2}$ (see Table \ref{table:tableKtheoryEfin}). There are two vertices
having the latter group as the stabilizer, contributing (in all) a factor of $(\mathbb{Z}/2)^{4}$ to the reduced $K_{0}$. (There are no other contributions
to the reduced $K_{0}$ from vertices -- see Table \ref{table:tableKtheoryEfin}.) This calculation is recorded 
 in Table \ref{thefinitepart}.
\end{example}

\begin{example}[The case of $(C_{6}^{+})_{5}$] \label{example:c6+} We
set $\Gamma = (C_{6}^{+})_{5}$. The quotient $p(\mathcal{C}_{\Gamma})$ appears in Figure \ref{figure:c6+}. 

\begin{figure}[!h]
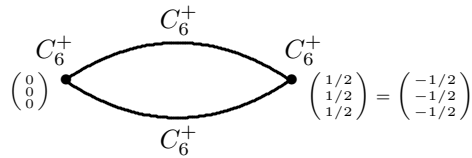
   
\vbox{\hspace{-1cm}\beginpicture
\setcoordinatesystem units <1cm,1cm> point at -1 1.5
\setplotarea x from -2 to 5, y from -1.5 to 1.5
\linethickness=.7pt
\def\myarrow{\arrow <4pt> [.2, 1]} 
\setplotsymbol ({\circle*{.4}})
\plotsymbolspacing=.3pt        
\put {\tiny $\left(\begin{smallmatrix}0 \\ 0 \\ 0\end{smallmatrix}\right)$} [tr] at
-.2 .1
\put {\tiny $\left(\begin{smallmatrix}1/2 \\ 1/2 \\ 1/2\end{smallmatrix}\right)=\left(\begin{smallmatrix}-1/2 \\ -1/2 \\ -1/2\end{smallmatrix}\right)$} [tl]
at 3.2 .1
\put {$\bullet$} at  0 0
\put {$\bullet$} at  3 0 
\put {$C^+_6$} [br] at  .1 .2
\put {$C^+_6$} [bl] at  2.9 .2
\put {$C^+_6$} [b] at  1.5 .6
\put {$C^+_6$} [t] at  1.5 -.6
\setquadratic
\plot  
0 0    1.5 .5  3 0 
/
\plot   
0 0  1.5 -.5  3 0
/
\endpicture} \hfil
\vspace{-.6 cm}
\caption{The graph $p(\mathcal{C}_{\Gamma})$, for $\Gamma = (C_{6}^{+})_{5}$. In this case, there are no isolated vertices.}
\label{figure:c6+}
\end{figure}

The chain complex for computing $E^{2}_{0,-1}$ and $E^{2}_{1,-1}$ is as follows:
$$0\rightarrow K_{-1}(\mathbb Z[\mathbb Z/6]) \oplus  K_{-1}(\mathbb Z[\mathbb Z/6])  \xrightarrow{\rho} K_{-1}(\mathbb Z[\mathbb Z/6])  \oplus  K_{-1}(\mathbb Z[\mathbb Z/6]) \rightarrow 0.$$
Since  $K_{-1}(\mathbb Z[\mathbb Z/6]) \cong \mathbb Z$, the latter complex amounts to the following:
$$0\rightarrow \mathbb Z \oplus \mathbb Z  \xrightarrow{\rho} \mathbb Z  \oplus \mathbb Z \rightarrow 0,$$
where the map $\rho$ sends $(0,1) \mapsto (1,-1)$ and $(1,0) \mapsto (-1, 1)$. It follows that $\ker(\rho) = \langle (1,1) \rangle \cong \mathbb Z$ and  $\im(\rho) =\langle(1, -1) \rangle$. Therefore
\[
E^2_{0,-1} \cong \mathbb Z, \; \text{and} \; \; E^2_{1,-1} \cong \mathbb Z.
\]
Since $\widetilde{K}_{0}(\mathbb{Z}[\mathbb{Z}/6]) \cong 0$, this and Lemma \ref{lemma:computational} directly gives us
the calculation that is recorded in Table \ref{thefinitepart}.
\end{example}

\begin{example}[The case of $(D_{6}^{+})_{5}$]  \label{example:d6+} In this case, the graph $p(\mathcal{C}_{\Gamma})$ is simply an edge --
see Figure \ref{figure:d6+}.

\begin{figure}[!h]   
\vbox{\hfil \hspace{-3.5cm}\beginpicture
\setcoordinatesystem units <1cm,1cm> point at -1 1
\setplotarea x from -2 to 4, y from -1 to 1
\linethickness=.7pt
\def\myarrow{\arrow <4pt> [.2, 1]} 
\setplotsymbol ({\circle*{.4}})
\plotsymbolspacing=.3pt        
\plot 0 0  3 0 /
\put {$\bullet$} at  0 0
\put {$\bullet$} at  3 0 
\put {$D^+_6$} [b] at  0 .2
\put {$D^+_6$} [b] at  3 .2
\put {$C^+_6$} [t] at  1.5 -.2
\put {\tiny $\left(\begin{smallmatrix}0 \\ 0 \\ 0\end{smallmatrix}\right)$} [t] at
0 -.1
\put {\tiny $\left(\begin{smallmatrix}1/2 \\ 1/2 \\ 1/2\end{smallmatrix}\right)$} [t]
at 3 -.1
\endpicture}\hfil
\vspace{-.3cm}
\caption{The graph $p(\mathcal{C}_{\Gamma})$ for $\Gamma = (D_{6}^{+})_{5}$. There are no isolated vertices.}
\label{figure:d6+}
\end{figure}

The chain complex for computing $E^{2}_{p,-1}$ is as follows:
$$0\rightarrow K_{-1}(\mathbb Z[\mathbb Z/6]) \xrightarrow{\rho} K_{-1}(\mathbb ZD_6)  \oplus  K_{-1}(\mathbb ZD_6)\rightarrow 0.$$
Since  $K_{-1}(\mathbb Z[\mathbb Z/6]) \cong K_{-1}(\mathbb Z[D_6]) \cong \mathbb Z$, the complex can be written
$$0\rightarrow \mathbb Z   \xrightarrow{\rho} \mathbb Z  \oplus \mathbb Z \rightarrow 0,$$
where the map $\rho$ sends $1 \mapsto (1,-1)$. It follows that $\ker(\rho)$ is trivial,  and  $\im(\rho)= \langle (1, -1) \rangle  \cong \mathbb Z$. Therefore
\[
E^2_{0,-1} \cong \mathbb Z, \; \text{and} \; \; E^2_{1,-1} \cong 0.
\]
Since $\widetilde{K}_{0}(\mathbb{Z}F) \cong 0$ for $F \in \{ \mathbb{Z}/6, D_{6} \}$ (see Table \ref{table:tableKtheoryEfin}),
we immediately get the calculation that is recorded in Table \ref{thefinitepart}.
\end{example}

\begin{example}[The case of $(D''_{6})_{5}$] \label{example:d''6} This case follows the exact pattern of Example \ref{example:c6+}. 
The quotient $p(\mathcal{C}_{\Gamma})$ is pictured in Figure \ref{figure:d''6}.

\begin{figure}[!h]   
\vbox{\hspace{1cm}\beginpicture
\setcoordinatesystem units <1cm,1cm> point at -1 1.5
\setplotarea x from -2 to 5, y from -1.5 to 1.5
\linethickness=.7pt
\def\myarrow{\arrow <4pt> [.2, 1]} 
\setplotsymbol ({\circle*{.4}})
\plotsymbolspacing=.3pt        
\put {\tiny $\left(\begin{smallmatrix}0 \\ 0 \\ 0\end{smallmatrix}\right)$} [tr] at
-.2 .1
\put {\tiny $\left(\begin{smallmatrix}1/2 \\ 1/2 \\ 1/2\end{smallmatrix}\right)=\left(\begin{smallmatrix}-1/2 \\ -1/2 \\ -1/2\end{smallmatrix}\right)$} [tl]
at 3.2 .1
\put {$\bullet$} at  0 0
\put {$\bullet$} at  3 0 
\put {$D^{''}_6$} [br] at  .1 .2
\put {$D^{''}_6$} [bl] at  2.9 .2
\put {$D^{''}_6$} [b] at  1.5 .6
\put {$D^{''}_6$} [t] at  1.5 -.6
\setquadratic
\plot  
0 0    1.5 .5  3 0 
/
\plot   
0 0  1.5 -.5  3 0
/
\endpicture}\hfil
\vspace{-.5cm}
\caption{The graph $p(\mathcal{C}_{\Gamma})$, for $\Gamma = (D''_{6} )_{5}$. There are no isolated vertices.}
\label{figure:d''6}
\end{figure}

The complex for computing $E^{2}_{p,-1}$ is:
$$0\rightarrow K_{-1}(\mathbb Z[D_6]) \oplus  K_{-1}(\mathbb Z[D_6])  \xrightarrow{\rho} K_{-1}(\mathbb Z[D_6])  \oplus  K_{-1}(\mathbb Z[D_6]) \rightarrow 0.$$
Since  $K_{-1}(\mathbb ZD_6) \cong \mathbb Z$, we get
$$0\rightarrow \mathbb Z \oplus \mathbb Z  \xrightarrow{\rho} \mathbb Z  \oplus \mathbb Z \rightarrow 0,$$
where the map $\rho$ sends $(0,1) \mapsto (1,-1)$ and $(1,0) \mapsto (-1, 1)$. It follows that $\ker(\rho) = \langle (1,1) \rangle \cong \mathbb Z$ and  $\im(\rho) =\langle(1, -1) \rangle \cong \mathbb Z$. Therefore
\[
E^2_{0,-1} \cong \mathbb Z, \; \text{and} \; \; E^2_{1,-1} \cong \mathbb Z.
\]
Since $\widetilde{K}_{0}(\mathbb{Z}D_{6}) \cong 0$, these terms represent the only contributions to $K$-theory. See
Table \ref{thefinitepart}.
\end{example}

\begin{table}[!h]
\footnotesize
\renewcommand{\arraystretch}{1.8}
\[
\begin{array}{|c|c|c|}
\hline \g & K_{-1} \neq 0 & \tilde{K}_0 \neq 0
\\ \hline \hline
\Gamma_1 & \mathbb Z^2 & (\mathbb Z/4)^4
\\ \hline
(A_4^+ \times (-1))_1 & \mathbb Z^2 & (\mathbb Z/2)^4   \\ \hline
(D^{+}_4 \times (-1))_1& & (\mathbb Z/2)^2 \oplus (\mathbb Z/4)^4  \\ \hline
(D_2^{+} \times (-1))_1 &  & (\mathbb Z/2)^8  \\ \hline
(C^{+}_{4} \times (-1))_1 & &(\mathbb Z/2)^4   \\ \hline \hline
\Gamma_2 & \mathbb Z^2 &(\mathbb Z/4)^2  \\ \hline 
(A_4^+ \times (-1))_2 & \mathbb Z^2 & (\mathbb Z/2)^2  \\ \hline
(D^{+}_4 \times (-1))_2 &  & \mathbb Z/2 \oplus (\mathbb Z/4)^2  \\ \hline
(D_2^{+} \times (-1))_2 & & (\mathbb Z/2)^4   \\ \hline
(C^{+}_{4} \times (-1))_2 & & (\mathbb Z/2)^2  \\ \hline \hline
\Gamma_3 & \mathbb Z^2 & \mathbb Z/2 \oplus (\mathbb Z/4)^2  \\ \hline
(A_4^+ \times (-1))_3 & \mathbb Z^2 & (\mathbb Z/2)^2 \\ \hline
(D_2^{+} \times (-1))_3  & &(\mathbb Z/2)^2  \\ \hline \hline
\Gamma_4 & & (\mathbb Z/2)^4   \\ \hline \hline
\Gamma_5 & \mathbb Z^7 & (\mathbb Z/2)^{6} \\ \hline
(D^+_6)_5 & \mathbb Z &   \\ \hline
(C'_6)_5 & \mathbb Z^6 &  \\ \hline
(C^+_6 \times (-1))_5 & \mathbb Z^7 & (\mathbb Z/2)^4  \\ \hline
(D'_6)_5 & \mathbb Z^6&      \\ \hline
(C^+_6)_5 &\mathbb Z & \mathbb Z    \\ \hline
(D_3^{+} \times (-1))_5 & \mathbb Z^2 &  \\ \hline
(\widehat{D}'_{6})_5 & \mathbb Z^4 &  \\ \hline
(C^+_3 \times (-1))_5 & \mathbb Z^2 &  \\ \hline
(D''_6)_5  &\mathbb Z & \mathbb Z   \\ \hline \hline
\Gamma_6 & \mathbb Z^2&  \\ \hline
(C^+_3 \times (-1))_6 & \mathbb Z^2 &  \\ \hline \hline
\Gamma_7 & \mathbb Z^2 &  \\ \hline
\end{array}
\]
\caption{The homology groups $H_{*}^{\Gamma}(E_{\fin}(\Gamma); \mathbb{KZ}^{-\infty})$.}
\label{thefinitepart}
\end{table}


\section{Fundamental domains for the actions of the $\Gamma_{i}$ on the spaces of planes $\base$}\label{section:actionsonplanes}

Theorem \ref{theorem:splitting2} showed that the lower algebraic $K$-theory of any crystallographic group can be computed in two pieces. In Section
\ref{section:contributionoffinites}, we completed the first half of the computation for the $73$ split three-dimensional crystallographic groups. The results
obtained are summarized in Table \ref{thefinitepart}. 

Next, according to Theorem \ref{theorem:splitting2}, we must calculate the equivariant homology groups $H_n^{\Gamma_{\widehat{\ell}}}(E_{\fin}(\Gamma_{\widehat{\ell}}) \rightarrow  \ast;\;  \mathbb{KZ}^{-\infty})$, where $\widehat{\ell} \in \mathcal{T}''$. 
In the current section, we will determine which groups $H_n^{\Gamma_{\widehat{\ell}}}(E_{\fin}(\Gamma_{\widehat{\ell}}) \rightarrow  \ast;\;  \mathbb{KZ}^{-\infty})$ contribute to the lower algebraic $K$-theory of the seven maximal crystallographic groups 
$\Gamma_{i}$, ($i=1, \ldots, 7$). We will describe fundamental domains for the actions of the groups $\Gamma_{i}$ on their associated spaces of
planes $\base$. Our arguments here are generally analogous to the ones from Section \ref{section:fundamentaldomains}, and the organization of this
section is similar.

\subsection{Negligible Line Stabilizer Groups}

As in Section \ref{section:fundamentaldomains}, the notion of a negligible stabilizer group will be very useful. We will want to adapt the old definition
(Definitions \ref{definition:negligible1} and \ref{definition:negligible}) to our needs in the current section. 

The main result of this subsection is Proposition \ref{proposition:finiteT''}, which greatly
simplifies the remainder of the computation of the lower algebraic $K$-groups of the split crystallographic groups.

\begin{definition}
Let $\Gamma$ be a split crystallographic group. Let $\ell \subseteq \mathbb{R}^{3}$
be a line. We let 
$\overline{\Gamma}_{\ell} = \{ \gamma \in \Gamma \mid \gamma_{\mid \ell} = \mathrm{id}_{\ell} \}$.
We will sometimes call this group the \emph{strict stabilizer group of $\ell$}, to distinguish it from the stabilizer group 
$\Gamma_{\ell} = \{ \gamma \in \Gamma \mid \gamma \cdot \ell = \ell \}$.

If $G$ is an infinite virtually cyclic subgroup of $\Gamma$, then $G = \Gamma_{\ell}$ for some line
$\ell \subseteq \mathbb{R}^{3}$. We then let $\overline{G}$ denote $\overline{\Gamma}_{\ell}$.
\end{definition}

\begin{definition} \label{definition:VCnegligible}
An infinite virtually cyclic subgroup $G \leq \Gamma$ is called \emph{negligible}
if the finite subgroup $F$ from Remark \ref{remark:VCgroups} has square-free order.
  We will also say that $\overline{G}$ is negligible,
and that  $\sigma$ is negligible, if $G$ is the stabilizer of $\sigma$.
\end{definition}

\begin{proposition} \label{proposition:squarefreenegligible}
If $G$ is infinite virtually cyclic and $G \cong F \rtimes_{\alpha} \mathbb{Z}$ or $G \cong G_{1} \ast_{F} G_{2}$, where $F$ has square-free order, then
$NK_{q}(\mathbb{Z}F, \alpha)$ or $NK_{q}(\mathbb{Z}F; \mathbb{Z}[G_{1} - F], \mathbb{Z}[G_{2}-F])$ (respectively) is trivial for all $q \leq 1$.
\end{proposition}

\begin{proof}
If $G$ has the first form, then results of \cite{Ha87} and \cite{J-PR09} show that $NK_{q}(\mathbb{Z}F, \alpha)$ is trivial for $q \leq 1$.
If $G$ has the second form, then the canonical index two subgroup of $G$ has the form $F \rtimes_{\alpha} \mathbb{Z}$, so the corresponding
group $NK_{q}(\mathbb{Z}F, \alpha)$ is trivial for $q \leq 1$, as above. Results of Lafont and Ortiz \cite{LO08} show that
when $NK_{q}(\mathbb{Z}F, \alpha)$ is trivial, so is $NK_{q}(\mathbb{Z}F; \mathbb{Z}[G_{1} - F], \mathbb{Z}[G_{2}-F])$.
\end{proof}

\begin{remark} \label{remark:clarifynegligible}
If an infinite virtually cyclic group $G$ is \emph{negligible} in the sense of Definition \ref{definition:negligible1}, then it is also
\emph{negligible} in the sense of Definition \ref{definition:VCnegligible}. Note that the two notions are not the same, however. For instance, $\mathbb{Z}/6 \times D_{\infty}$ is \emph{negligible} in the sense of Definition \ref{definition:VCnegligible}, but \emph{non-negligible}
in the sense of Definition \ref{definition:negligible1}. In fact, when we express $\mathbb{Z}/6 \times D_{\infty}$ as an infinite virtually cyclic
group of type II, we have
$$ \mathbb{Z}/6 \times D_{\infty} \cong (\mathbb{Z}/6 \times \mathbb{Z}/2) \ast_{\mathbb{Z}/6} (\mathbb{Z}/6 \times \mathbb{Z}/2),$$
so $\mathbb{Z}/6 \times D_{\infty}$ is negligible in the sense of Definition \ref{definition:VCnegligible} (since $6$ is square-free). But, since 
$\mathbb{Z}/6 \times \mathbb{Z}/2$ is not isomorphic to a subgroup of $S_{4}$, then it follows that 
$\mathbb{Z}/6 \times D_{\infty}$ is not negligible in the sense
of Definition \ref{definition:negligible1}.
\end{remark}

\begin{remark} \label{remark:indexingsetnegligible}
Remark \ref{remark:clarifynegligible} showed that the class of negligible groups in the sense of Definition \ref{definition:VCnegligible}
is strictly wider than the class of negligible groups in the sense of Definition \ref{definition:negligible1}.
Thus, the class of non-negligible groups in the sense of Definition \ref{definition:negligible1} is strictly wider than the class of 
non-negligible groups in the sense of Definition \ref{definition:VCnegligible}. 

Recall that, in 
Theorem \ref{theorem:splitting2}, we considered an indexing set $\mathcal{T}''$, where $\mathcal{T}''$ consisted of a choice of vertex
$v$ from each orbit of non-negligible (in the sense of Definition \ref{definition:negligible1}) vertices in $\base$. In practice, it will be easier to consider
the smaller indexing set $\widehat{\mathcal{T}}''$, where $\widehat{\mathcal{T}}''$ consists of a choice of vertex $v$ from each orbit of non-negligible (in the sense of Definition \ref{definition:VCnegligible}) vertices in $\base$. Note that the indexing set $\widehat{\mathcal{T}}''$ leaves out only vertices
$v \in \base$ with stabilizers such as $\mathbb{Z}/6 \times D_{\infty}$, which make no contribution to Nils by Proposition \ref{proposition:squarefreenegligible}. 
It follows that we can use the 
smaller indexing set $\widehat{\mathcal{T}}''$ in Theorem
\ref{theorem:splitting2}.

In practice, we will use the indexing set $\widehat{\mathcal{T}}''$, but write $\mathcal{T}''$. This means that, from now on,
we will simply read
Theorem \ref{theorem:splitting2} with the current definition (Definition \ref{definition:VCnegligible}) of negligible in mind.

\end{remark}

\begin{proposition} \label{proposition:negligibletest}
Let $\Gamma = \langle H, L \rangle$ be a split crystallographic group. Let $\ell \subseteq \mathbb{R}^{3}$
be a line; let $r(\alpha) = t + \alpha v$ ($t, v \in \mathbb{R}^{3}$; $\alpha \in \mathbb{R}^{3}$) be a
parametrization of $\ell$. If $H_{v} = \{ h \in H \mid h \cdot v = v \}$ has square-free order, then $\Gamma_{\ell} = 
\{ \gamma \in \Gamma \mid \gamma \cdot \ell = \ell \}$ is negligible.
\end{proposition}

\begin{proof}
Let $\gamma \in \overline{\Gamma}_{\ell}$; we write $\gamma = v' + h$, where $h$ is the linear part of the isometry $\gamma$ (i.e., $h = \pi(\gamma)$)
and $v' \in L$. Now if $\gamma \cdot r(\alpha) = r(\alpha)$ for all $\alpha \in \mathbb{R}$, we clearly must have $h \cdot v = v$. Thus, $\pi(\gamma) \in H_{v}$.
Since $\pi: \overline{\Gamma}_{\ell} \rightarrow H$ is injective by Lemma \ref{pointstab}(2) and $\pi(\overline{\Gamma}_{\ell}) \subseteq H_{v}$, it follows that $\overline{\Gamma}_{\ell}$ has square-free order, making $\Gamma_{\ell}$ negligible. 
\end{proof}

\begin{definition}
If $H$ is a point group, then a vector $v \in \mathbb{R}^{3}$ is a \emph{pole vector} if there is some orientation-preserving
$h \in H - \{ 1 \}$ such that $h \cdot v = v$.
\end{definition}

\begin{corollary} \label{corollary:pole}
If the line $\ell \subseteq \mathbb{R}^{3}$ can be parametrized as $r(\alpha) = t + \alpha v$ where $v$ is 
not a pole vector of $H$, then the stabilizer group $\Gamma_{\ell}$ is negligible.
\end{corollary}

\begin{proof}
We note that if $v$ is not a pole vector of $H$, then the only orientation-preserving $h \in H_{v}$ is $h = 1_{H}$.
The orientation-preserving subgroup of $H_{v}$ has index at most two in $H_{v}$, so $|H_{v}| \leq 2$.
Thus, $H_{v}$ has square-free order, and Proposition \ref{proposition:negligibletest} applies.
\end{proof} 

\begin{proposition} \label{proposition:finiteT''}
Let $\Gamma$ be one of the groups $\Gamma_{i}$, ($i=1, \ldots, 7$). We let $H$ denote the point group of $\Gamma$.
\begin{enumerate}
\item If $H = S_{4}^{+} \times (-1)$, then the vertex $\widehat{\ell} \in \base$ is negligible unless $\widehat{\ell}$
admits a parametrization of the form $r(\alpha) =t + \alpha v$ ($t, v \in \mathbb{R}^{3}$, $\alpha \in \mathbb{R}$), where
$v = (1,0,0)$ or $(1,1,0)$ (or some image of the latter vectors under the action of $H$).
\item If $H = D_{2}^{+} \times (-1)$, then the vertex $\widehat{\ell} \in \base$ is negligible unless $\widehat{\ell}$
admits a parametrization of the form $r(\alpha) =t + \alpha v$ ($t, v \in \mathbb{R}^{3}$, $\alpha \in \mathbb{R}$), where
$v = (1,0,0)$, $(0,1,0)$, or $(0,0,1)$.
\item If $H = D_{6}^{+} \times (-1)$, then the vertex $\widehat{\ell} \in \base$ is negligible unless $\widehat{\ell}$
admits a parametrization of the form $r(\alpha) =t + \alpha v$ ($t, v \in \mathbb{R}^{3}$, $\alpha \in \mathbb{R}$), where
$v = (1,1,1)$ , $(1,-1,0)$, or $(1,-2,1)$ (or some image of the latter vectors under the action of $H$).
\item If $H= D_{3}^{+} \times (-1)$, then all of the vertices $\widehat{\ell} \in \base$ are negligible. In particular, the complete lower algebraic
K-groups of $\Gamma_{i}$ for $i=6,7$ appear in Table \ref{thefinitepart}.
\end{enumerate}
\end{proposition}

\begin{proof}
The proofs of all four parts are similar. Proposition \ref{OS} showed that there are only two or three orbits of (unit) pole vectors for any point group $H$.
Proposition \ref{OS} also shows that there are exactly three orbits of unit pole vectors for all four of the groups in the statement of our proposition. For each
of the four groups above, we can find the three distinct orbits of unit pole vectors by inspection. For instance, if $H  = D_{6}^{+} \times (-1)$, we note that
the vectors $(1,1,1)$, $(1,-1,0)$, and $(1,-2,1)$ are all pole vectors of $H$ since the orientation-preserving isometries
$$ \left( \begin{smallmatrix} 0 & 1 & 0 \\ 0 & 0 & 1 \\ 1 & 0 & 0 \end{smallmatrix} \right),\quad 
\left( \begin{smallmatrix} 0 & -1 & 0 \\ -1 & 0 & 0 \\ 0 & 0 & -1 \end{smallmatrix} \right),  \quad
\frac{1}{3}\left( \begin{smallmatrix} -2 & -2 & 1 \\ -2 & 1 & -2 \\ 1 & -2 & -2 \end{smallmatrix} \right)$$
from $H$ fix the given vectors (respectively). Next, it is straightforward to check that the given three vectors are in separate orbits with respect to $H$ (even after normalization),
so the normalized vectors represent the three orbits of unit pole vectors for $H$. The rest of (3) now follows easily from Corollary \ref{corollary:pole}. 
Case (2) is easier.

The proof of (1) begins in the same way. It is straightforward to check that the vectors
$(1,0,0)$, $(1,1,0)$, and $(1,1,1)$  are pole vectors, and that they represent distinct $H$-orbits (even after normalization). 
It follows directly from Corollary \ref{corollary:pole}
that the vertex $\widehat{\ell} \in \base$ is negligible unless it admits a parametrization $r(\alpha) = t + \alpha v$ with one of the latter vectors playing the role of $v$.
We can furthermore argue that $\widehat{\ell}$ is negligible when $v = (1,1,1)$. One checks that the orbit of $(1,1,1)$ contains eight vectors,
which implies that $|H_{v}| = 6$. Thus, $\Gamma_{\widehat{\ell}}$ is negligible by Proposition \ref{proposition:negligibletest}. Statement (1) follows.

If $H = D_{3}^{+} \times (-1)$, then the vectors $(1,1,1)$, $(1,-1,0)$, and $(0,-1,1)$ determine the three distinct orbits of unit pole vectors. In this 
case, one checks that the stabilizer groups $H_{v}$ have orders $6$, $2$, and $2$ (respectively), so the desired conclusion follows
from Proposition \ref{proposition:negligibletest}.
\end{proof}

\subsection{The finiteness of the indexing set $\mathcal{T}''$} \label{subsection:finiteT''}

In this subsection, we will explicitly identify the indexing set $\mathcal{T}''$ from Theorem \ref{theorem:splitting2} for the crystallographic
groups $\Gamma_{i}$ ($i=1, \ldots, 5$). (Proposition \ref{proposition:finiteT''} has already shown that we can take $\mathcal{T}'' = \emptyset$
when $i=6$ or $7$.) It will follow easily that $\mathcal{T}''$ is always finite for any $3$-dimensional crystallographic group.
We also record here the strict line stabilizers for the groups $\Gamma_{i}$ ($i = 1, \ldots, 5$). We will need a lemma.

\begin{lemma} \label{lemma:computess}
Let $\Gamma = \Gamma_{i}$ ($i=1, \ldots, 5$), and let $\ell \subseteq \mathbb{R}^{3}$ be a line. The group $\Gamma(\ell)$ acts isometrically on
the space of lines $\mathbb{R}^{2}_{\ell}$. We let  
$\psi: \Gamma(\ell) \rightarrow \mathrm{Isom}(\mathbb{R}^{2}_{\ell})$ denote the associated homomorphism, and let $G$ denote the image of $\psi$.
\begin{enumerate}
\item If $\ell' \in \mathbb{R}^{2}_{\ell}$, then $\psi_{\mid \overline{\Gamma}_{\ell'}}: \overline{\Gamma}_{\ell'} \rightarrow G_{\ell'}$ is
injective. In particular, $|\overline{\Gamma}_{\ell'}|$ divides $|G_{\ell'}|$.
\item Let $g \in G_{\ell'}$, and let $\gamma \in \psi^{-1}(g)$. We have $g \in \psi(\overline{\Gamma}_{\ell'})$ if and only if there is an element 
$k \in \mathrm{Ker} \, \psi$ such that $\gamma_{\mid \ell'} = k_{\mid \ell'}$.
\end{enumerate}
\end{lemma}

\begin{proof}
The statement that $\Gamma(\ell)$ acts isometrically on $\mathbb{R}^{2}_{\ell}$ is straightforward.
\begin{enumerate}
\item We claim that $\psi_{\mid \overline{\Gamma}_{\ell'}}: \overline{\Gamma}_{\ell'} \rightarrow \mathrm{Isom}(\mathbb{R}^{2}_{\ell})$ is injective. 
Suppose that $\gamma \in \mathrm{Ker}(\psi_{\mid \overline{\Gamma}_{\ell'}})$. Since $\gamma \in \overline{\Gamma}_{\ell'}$, $\gamma$ acts trivially on
the line $\ell' \subseteq \mathbb{R}^{3}$, and, since $\gamma \in \mathrm{Ker}(\psi)$, $\gamma$ leaves every line parallel to $\ell'$ invariant
as a set. Let $\ell''$ be parallel to $\ell'$, and choose distinct points $x_{1}, x_{2} \in \ell'$. There are points $y_{1}, y_{2} \in \ell''$ that are the unique closest points
to $x_{1}$ and $x_{2}$ (respectively) among all points on $\ell''$. Since $\ell''$ is invariant under $\gamma$, and $x_{1}$ and $x_{2}$ are fixed, it must be
that $y_{1}$ and $y_{2}$ are fixed as well. Thus, all of $\ell''$ is fixed. This shows that every line parallel to $\ell'$ is fixed by $\gamma$, so $\gamma = 1$. This
proves the claim.

We can therefore identify $\overline{\Gamma}_{\ell'}$ with a subgroup of $G_{\ell'}$ and apply Lagrange's Theorem.
\item We assume $g \in \psi(\overline{\Gamma}_{\ell'})$ and $\psi(\gamma) = g$. Choose $\widehat{\gamma} \in \overline{\Gamma}_{\ell'}$ such that 
$\psi(\widehat{\gamma}) = g$. We can set $k = \widehat{\gamma}^{-1} \gamma$, which has the required properties. Conversely, if $\gamma_{\mid \ell'} = 
k_{\mid \ell'}$, then $k^{-1} \gamma \in \overline{\Gamma}_{\ell'}$, and $\psi(k^{-1} \gamma) = g$, as required. 
\end{enumerate}
\end{proof}

\begin{theorem} \label{theorem:gamma1T}
For the group $\Gamma = \Gamma_{1}$, we can choose
$$ \mathcal{T}'' = \left\{
\left( \begin{smallmatrix} 0 \\ 0 \\  \alpha \end{smallmatrix} \right), 
\left( \begin{smallmatrix} 1/2 \\ 1/2 \\  \alpha \end{smallmatrix} \right), 
\left( \begin{smallmatrix} 1/2 \\ 0 \\  \alpha \end{smallmatrix} \right), 
\left( \begin{smallmatrix} \alpha \\ \alpha \\  0 \end{smallmatrix} \right), 
\left( \begin{smallmatrix} \alpha \\ \alpha \\  1/2 \end{smallmatrix} \right) \right\},$$
 where we have expressed the vertices (i.e., lines) in $\base$ in parametric form.

For the vertices $v \in \mathcal{T}''$, the strict stabilizer groups $\overline{\Gamma}_{v}$ satisfy
$$ \pi( \overline{\Gamma}_{v}) = D_{4}'', \quad D_{4}'', \quad D_{2}', \quad \langle A, B \rangle, \text { and } \quad \langle A, B \rangle,$$
respectively, where 
$$ A = \left( \begin{smallmatrix} 0 & 1 & 0 \\ 1 & 0 & 0 \\ 0 & 0 & 1 \end{smallmatrix} \right) \text{ and } \quad B = \left( \begin{smallmatrix} 1 & 0 & 0 \\
0 & 1 & 0 \\ 0 & 0 & -1 \end{smallmatrix} \right).$$
Note that $\langle A, B \rangle \cong D_{2}$.
\end{theorem}

\begin{proof}
The point group $H$ of $\Gamma_{1}$ is $S_{4}^{+} \times (-1)$, so Proposition \ref{proposition:finiteT''}(1) implies that the only non-negligible point
stabilizers from $\base$ have parametrizations of the form $r(\alpha) = t + \alpha v$, where $v = (1,0,0)$ or $(1,1,0)$ (or the image of one of these vectors
under the action of $H$).

\begin{figure}[!h]
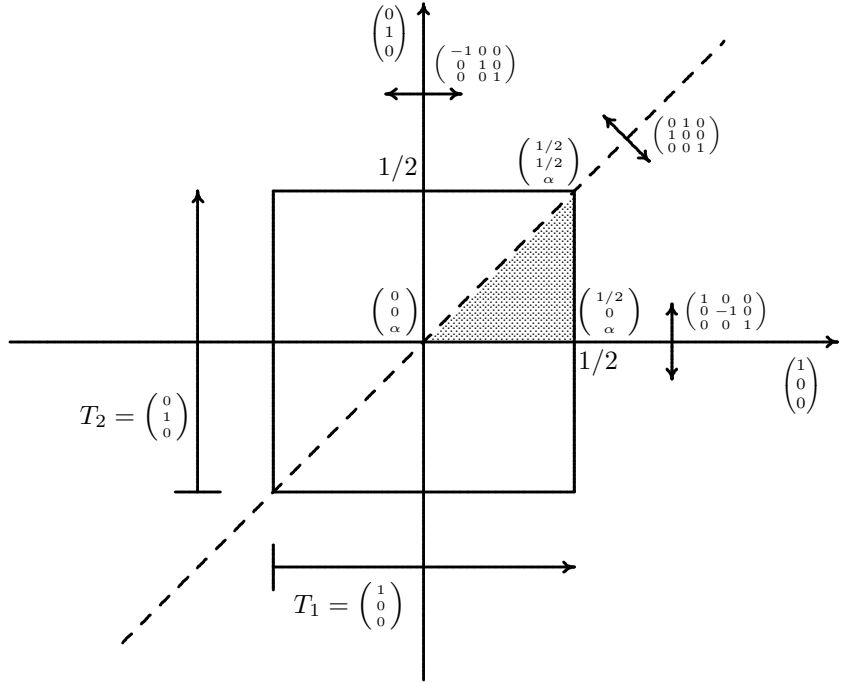

\begin{center}
\hbox{
\vbox{\beginpicture
\setcoordinatesystem units <1cm,1cm> point at -6 4.5
\setplotarea x from -6 to 6, y from -4.5 to 4.5

\linethickness=.7pt
%

 \def\myarrow{\arrow <4pt> [.2, 1]} 
%
\setplotsymbol ({\circle*{.4}})
\plotsymbolspacing=.3pt        
\myarrow from -5.5 0 to 5.5 0     
\myarrow from   0 -4.5  to  0 4.5 

\plot -2 2   2 2   2 -2  -2 -2  -2 2 /  
\myarrow from -3 -2 to -3 2    
\plot -3.3 -2   -2.7 -2 /          
\put {$T_2 = \tiny \left(\begin{smallmatrix} 0 \\ 1 \\ 0 \end{smallmatrix} \right)$} [r] at -3.1 -1       
\myarrow from  -2 -3 to 2 -3      
\plot -2 -3.3  -2 -2.7 /            
\put {$T_1  = \tiny\left(\begin{smallmatrix} 1 \\ 0 \\ 0 \end{smallmatrix} \right)$} [t] at -1 -3.2       
\myarrow from  3.3 0 to 3.3 .5   
\myarrow from  3.3 0 to 3.3 -.5

\myarrow from  0 3.3 to  .5 3.3
\myarrow from  0 3.3 to  -.5 3.3

\myarrow from  2.7 2.7 to 2.4 3
\myarrow from  2.7 2.7 to 3 2.4

\put{$ \tiny \left( \begin{smallmatrix} 0 \\ 0 \\ \alpha \end{smallmatrix} \right)$} [r] at -.05 .4
\put{$ \tiny \left( \begin{smallmatrix} 1/2 \\ 0 \\ \alpha \end{smallmatrix} \right)$} [l] at 2 .4
\put{$ \tiny \left( \begin{smallmatrix} 1/2 \\ 1/2 \\ \alpha \end{smallmatrix} \right)$} [r] at 2.1 2.4

\setdots <2pt>
\setdashes
\plot -4 -4   4 4 /
\setsolid 
\put {\tiny $\left(\begin{matrix} 0 \\ 1 \\ 0\end{matrix}\right)$} [r] at
-.2 4.1
\put {\tiny $\left(\begin{matrix} 1 \\ 0 \\ 0\end{matrix}\right)$} [t] at
5 -.2
\put {$1/2$} [t] at 2.3 -.1 
\put {$1/2$} [r] at -.1  2.3

\put {\tiny $\left(\begin{smallmatrix} -1 & 0 & 0 \\ 0 & 1 & 0 \\ 0 & 0 & 1 \end{smallmatrix} \right)$} [l] at .1  3.7
\put {\tiny $\left(\begin{smallmatrix} 0 & 1 & 0 \\ 1 & 0 & 0 \\ 0 & 0 & 1 \end{smallmatrix} \right)$} [tl] at 3 3  
\put {\tiny $\left(\begin{smallmatrix} 1 & 0 & 0 \\ 0 & -1 & 0 \\ 0 & 0 & 1 \end{smallmatrix} \right)$} [b] at 4 .15

\setshadegrid span <1pt>
\hshade 0 0 2   2 2 2 /  
\endpicture}
}
\end{center}
\caption{The image of $\Gamma(0,0,\alpha)$ under the homomorphism
$\psi: \Gamma(0,0,\alpha) \rightarrow \mathbb{R}^{2}_{(0,0,\alpha)}$. A fundamental domain for the image of $\psi$ is shaded.}
\label{figure:plane1,1} 
\end{figure}

We consider the first case. Let $r(\alpha) = t + \alpha v$, where $v=(1,0,0)$ (or one of its images under the action of $H$). It is clear that any such 
parametrized line $r(\alpha)$ can be moved to the plane $\mathbb{R}^{2}_{(0,0,\alpha)}$ by an element of $\Gamma$. Thus, we can find a complete 
set of $\Gamma$-orbit representatives with the given type (i.e., with tangent vector $v = (0,0,1)$) inside $\mathbb{R}^{2}_{(0,0,\alpha)}$. We consider the action of the group 
$$\Gamma(0,0,\alpha) = \{ \gamma \in \Gamma \mid \gamma (0,0,\alpha) \text{ is parallel to } (0,0,\alpha) \}$$
on the plane $\mathbb{R}^{2}_{(0,0,\alpha)}$. (In words, $\Gamma(0,0,\alpha)$ takes lines parallel to the $z$-axis to other lines parallel to the $z$-axis,
possibly reversing the directions of the lines.)
It is straightforward to check that $\Gamma(0,0,\alpha) = (D_{4}^{+} \times (-1))_{1}$. 
We can identify $\mathbb{R}^{2}_{(0,0,\alpha)}$ with the $xy$-plane $P(z=0)$ 
in $\mathbb{R}^{3}$. With respect to this identification, the point group $D_{4}^{+} \times (-1)$
acts by restricting its usual action on $\mathbb{R}^{3}$ to $P(z=0)$, and the cubical lattice acts on $P(z=0)$ by ignoring the third coordinate (
so the translation $(1,0,0) \in L$ acts as $(1,0)$ on $P(z=0)$, and $(0,0,1)$ acts trivially). This action yields a homomorphism
$\psi: \Gamma(0,0,\alpha) \rightarrow \mathrm{Isom}(\mathbb{R}^{2})$. The kernel of $\psi$ is
$$ \left\langle \left( \begin{smallmatrix} 0 \\ 0 \\ 1 \end{smallmatrix} \right), \left( \begin{smallmatrix} 1 & 0 & 0 \\ 0 & 1 & 0 \\ 0 & 0 & -1 \end{smallmatrix} \right)
\right\rangle.$$
The image $\psi(\Gamma(0,0,\alpha)) = G$ 
is described in Figure \ref{figure:plane1,1}.  The translations $T_{1}$ and $T_{2}$ (pictured) generate the whole group of
translations in the image. Double-tailed arrows in the figure denote reflections, both in 
$\mathbb{R}^{2}$ and in $\mathbb{R}^{3}$. Using Theorem \ref{poincare}, it is straightforward to verify that the shaded triangle $P$ is an
exact convex compact fundamental polyhedron for the image of $\psi$. The side-pairings are all reflections (in fact, the image of $\psi$ is a Coxeter group), 
so in particular all of the vertices in the 
triangle are in separate orbits. 

Now we compute the strict stabilizers of the vertices. (Corollary \ref{corollary:cellulated} 
has
already shown that the edges and $2$-cells have negligible stabilizers.) The stabilizers 
 $G_{(0,0,\alpha)}$ and $G_{(1/2,1/2,\alpha)}$ are easily seen to have order $8$, and 
the stabilizer $G_{(1/2,0,\alpha)}$ has order $4$. It follows from 
Lemma \ref{lemma:computess}(1) that the orders of the 
strict stabilizers
$\overline{\Gamma}_{\widehat{\ell}}$ divide $8$, $8$, and $4$ (respectively), where
$\widehat{\ell}$ ranges over the given lines. Consider the line $(0,0,\alpha)$.
It is easy to see that the isometries
$$\left(\begin{smallmatrix} 0 &1 & 0 \\ 1&0&0 \\ 0&0&1 \end{smallmatrix} \right),
\quad \left(\begin{smallmatrix} 1&0&0 \\ 0&-1&0 \\ 0&0&1\end{smallmatrix}\right)$$
(pictured in Figure \ref{figure:plane1,1}) are in the strict stabilizer 
$\overline{\Gamma}_{(0,0,\alpha)}$, and they clearly generate the group $D_{4}''$.
It follows that $\overline{\Gamma}_{(0,0,\alpha)} = D_{4}''$. The other two strict stabilizers
associated to this plane can similarly be computed by inspection
of Figure \ref{figure:plane1,1}.

\begin{figure}[!h]
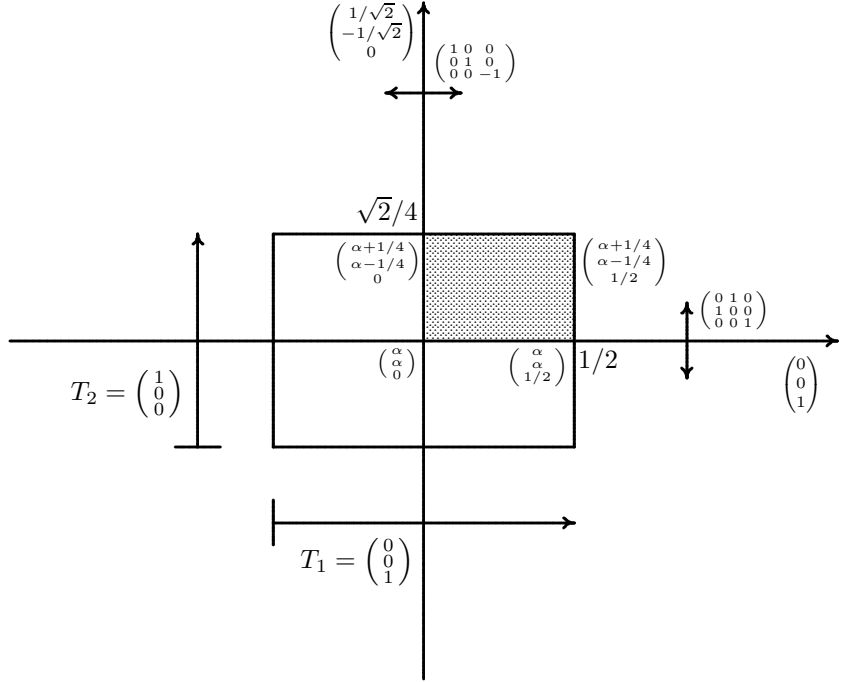

\begin{center}
\hbox{
\vbox{\beginpicture
\setcoordinatesystem units <1cm,1cm> point at -6 4.5
\setplotarea x from -6 to 6, y from -4.5 to 4.5
\linethickness=.7pt

 \def\myarrow{\arrow <4pt> [.2, 1]} 
%
\setplotsymbol ({\circle*{.4}})
\plotsymbolspacing=.3pt        
\myarrow from -5.5 0 to 5.5 0     
\myarrow from   0 -4.5  to  0 4.5 

\plot -2 1.42   2 1.42   2 -1.42  -2 -1.42  -2 1.422 /  
\myarrow from -3 -1.42 to -3 1.42    
\plot -3.3 -1.42   -2.7 -1.42 /          
\put {$T_2 = \left( \begin{smallmatrix} 1 \\ 0 \\ 0 \end{smallmatrix} \right)$} [r] at -3.2 -.7       
\myarrow from  -2 -2.42 to 2 -2.42      
\plot -2 -2.72  -2 -2.12 /            
\put {$T_1 = \left( \begin{smallmatrix} 0 \\ 0 \\ 1 \end{smallmatrix} \right)$} [t] at -.9 -2.62       

\put{\tiny $\left( \begin{smallmatrix} \alpha \\ \alpha \\ 0 \end{smallmatrix} \right)$} [r] at -.1 -.3
\put{\tiny $\left( \begin{smallmatrix} \alpha+ 1/4 \\ \alpha - 1/4 \\ 0 \end{smallmatrix} \right)$} [r] at -.05 1.05
\put{\tiny $\left( \begin{smallmatrix} \alpha \\ \alpha \\ 1/2 \end{smallmatrix} \right)$} [r] at 1.9 -.35
\put{\tiny $\left( \begin{smallmatrix} \alpha+1/4 \\ \alpha-1/4 \\ 1/2 \end{smallmatrix} \right)$} [l] at 2.05 1.05

\myarrow from  3.5 0 to 3.5 .5   
\myarrow from  3.5 0 to 3.5 -.5

\myarrow from  0 3.3 to  .5 3.3
\myarrow from  0 3.3 to  -.5 3.3

\setdots <2pt>
\put {\tiny $\left(\begin{matrix} 1/\sqrt{2} \\ -1/\sqrt{2} \\ 0\end{matrix}\right)$} [r] at
-.1 4.1
\put {\tiny $\left(\begin{matrix} 0 \\ 0 \\ 1\end{matrix}\right)$} [t] at
5 -.2
\put {$1/2$} [t] at 2.3 -.1 
\put {$\sqrt{2}/4$} [r] at -.1  1.72

\put {\tiny $\left(\begin{smallmatrix} 1 & 0 & 0 \\ 0 & 1 & 0 \\ 0 & 0 & -1 \end{smallmatrix} \right)$} [l] at .1  3.7
\put {\tiny $\left(\begin{smallmatrix} 0 & 1 & 0 \\ 1 & 0 & 0 \\ 0 & 0 & 1 \end{smallmatrix} \right)$} [b] at 4.1 .15

\setshadegrid span <1pt>
\vshade 0 0 1.42   2 0 1.42 /  
\endpicture}
}
\end{center}
\caption{A picture of the plane $\mathbb{R}^{2}_{(\alpha,\alpha,0)}$. The action
of $\Gamma(\alpha,\alpha,0)$ is indicated, and a fundamental domain is shaded.}
\label{figure:plane1,2}
\end{figure}

Next we must consider the plane $\mathbb{R}^{2}_{(\alpha,\alpha,0)}$, which 
contains orbit representatives of all of the other relevant vertices, by 
Proposition \ref{proposition:finiteT''}. Note that in Figure \ref{figure:plane1,2} we have identified the plane $\mathbb{R}^{2}_{(\alpha,\alpha,0)}$ with the plane in $\mathbb{R}^{3}$ through
the origin and perpendicular to the vector $(1,1,0)$, namely
$P(x+y=0)$. The group 
$\Gamma(\alpha,\alpha,0)$ is generated by the lattice $\mathbb{Z}^{3} \subseteq
\mathbb{R}^{3}$ and the point group that is generated by the reflections from
Figure \ref{figure:plane1,2} and the antipodal map. The action of $\Gamma(\alpha,\alpha,0)$
  is determined as before: translations act like their 
projections into the plane $P(x+y=0)$, and the point group acts by restriction. We have labelled the axes with a convenient orthogonal basis for $P(x+y=0)$. The shaded rectangle
is a fundamental domain; the image group $G$ is generated by reflections in the
sides of this rectangle.

We compute the strict stabilizers of the lines that form the corners of
the rectangle. It is clear from Lemma \ref{lemma:computess}(1) that the orders
of these strict stabilizer groups divide $4$. One checks directly that
the reflections labelling the double-tailed arrows from Figure \ref{figure:plane1,2}
are both in the strict stabilizer of $(\alpha, \alpha,0)$, and they clearly generate
a group of order $4$. It follows that $\overline{\Gamma}_{(\alpha,\alpha,0)}= \langle A, B \rangle$. One can also check that $\pi(\overline{\Gamma}_{(\alpha,\alpha,1/2)}) =
\langle A, B \rangle$ in roughly the same way: the isometry $A$ is in 
$\overline{\Gamma}_{(\alpha,\alpha,1/2)}$, and the isometry $T_{1} + B$
is also in $\overline{\Gamma}_{(\alpha,\alpha,1/2)}$, and these two isometries generate
a group of order $4$, which must therefore be all of 
$\overline{\Gamma}_{(\alpha,\alpha,1/2)}$.

Now we will show that the top corners of the rectangle are negligible. Consider the 
vertex $(\alpha + 1/4, \alpha - 1/4, 0)$. We claim that the reflection $g$ across
the top edge of the rectangle is not in the image
of $\psi_{\mid}: \overline{\Gamma}_{(\alpha+1/4, \alpha-1/4,0)} \rightarrow 
\mathrm{Isom}(\mathbb{R}^{2})$. The proof uses Lemma \ref{lemma:computess}(2).
We note that $\gamma = T_{2} + A$ is an isometry in $\Gamma(\alpha,\alpha,0)$
that leaves the given line  $(\alpha + 1/4, \alpha - 1/4, 0)$ invariant. Indeed, if we write the line in the form $r(\alpha)$,
then $\gamma \cdot r(\alpha)= r(\alpha + 1/2)$. Note that $g = \psi(\gamma)$. By Lemma \ref{lemma:computess}(2),
$g  \in \psi(\overline{\Gamma}_{(\alpha +1/4, \alpha -1/4,0)})$ if and only
if there is some $k \in \mathrm{Ker} \, \psi$ such that $k$ and $\gamma$ agree on
$r(\alpha)$. However, we can explicitly identify the kernel of
$\psi: \Gamma(\alpha,\alpha,0) \rightarrow \mathrm{Isom}(\mathbb{R}^{2})$ as follows:
$$\mathrm{Ker} \, \psi = \left\langle \left( \begin{smallmatrix} 0&-1&0 \\-1&0&0\\0&0&1
\end{smallmatrix}\right), \left(\begin{smallmatrix} 1\\1\\0 \end{smallmatrix}\right) \right\rangle.$$
Note that the first generator $\gamma_{1}$ acts by 
$\gamma_{1} \cdot r(\alpha) = r(-\alpha)$, and the second generator $\gamma_{2}$
acts by $\gamma_{2} \cdot r(\alpha) = r(\alpha + 1)$. It follows that there is no $k$
in the kernel such that $k \cdot r(\alpha) = r(\alpha +1/2)$. This proves the claim. It follows
that $\psi(\overline{\Gamma}_{(\alpha+1/4,\alpha-1/4,0)})$ is proper subgroup
of $G_{(\alpha+1/4,\alpha-1/4,0)}$, which has order $4$. It follows that $\overline{\Gamma}_{(\alpha+1/4,\alpha-1/4,0)}$ has order at most $2$, and is therefore
negligible. The proof that $\overline{\Gamma}_{(\alpha+1/4,\alpha-1/4,1/2)}$ is negligible
follows a similar pattern.
\end{proof}

\begin{remark} \label{remark:generalplanes}
Theorem \ref{theorem:gamma1T} establishes the general pattern of all five basic cases:
\begin{itemize}
\item It suffices to consider two or three planes $\mathbb{R}^{2}_{\ell}$ (depending on the point group, as from Proposition \ref{proposition:finiteT''});
all of the required $\Gamma$-orbit representatives occur as vertices in these planes.  
\item Given a split crystallographic group $\Gamma = \Gamma_{i}$ ($i=1, \ldots, 5$) and one of the above planes $\mathbb{R}^{2}_{\ell}$, there is a
natural homomorphism $\psi: \Gamma(\ell) \rightarrow \mathrm{Isom}(\mathbb{R}^{2}_{\ell})$ and a natural identification of $\mathbb{R}^{2}_{\ell}$
with a $2$-dimensional subspace of $\mathbb{R}^{3}$. Specifically, choosing $\ell$ to be a $1$-dimensional subspace of $\mathbb{R}^{3}$ (as we may), we can 
identify $\mathbb{R}^{2}_{\ell}$ with the perpendicular $2$-dimensional subspace $S$. The group $\Gamma(\ell)$, which is necessarily a split crystallographic group,
acts on $S$. An element of the point group of $\Gamma(\ell)$ acts by restriction to $S$, and an element of the lattice acts by its projection into $S$.
\item In every case, we will specify generators for the image $G = \psi(\Gamma(\ell))$ in accompanying figures, and label 
the coordinate axes with convenient unit vectors. The group $G$ is often, but not always, a Coxeter group. 
Double-tailed arrows in the figures will always denote reflections
in the plane $S$ that preserve the normal direction (as isometries of $\mathbb{R}^{3}$).
 The action of $G$ on $S$ will always have an
exact convex compact fundamental polyhedron $P$. It is generally straightforward to check that the shaded region in the figures
is an exact convex compact fundamental polyhedron. The proof will be omitted. Having $P$ lets us determine orbit information:
every point (line) in the given plane is in some translate of $P$, and the intersections of $\Gamma(\ell)$-orbits with $P$ have the form
$P \cap [x]$, where $x \in P$ and $[x]$ denotes the cycle of $x$ (see Subsection \ref{specialcase}). By Corollary \ref{corollary:cellulated}, only vertices 
in the cellulation determined by $P$ 
(see Theorem \ref{cells}) can be non-negligible.
\item We compute the strict stabilizer groups using the two parts of Lemma \ref{lemma:computess}.
\end{itemize}
This general outline will be assumed in what follows.
\end{remark}

\begin{theorem} \label{theorem:gamma2T}
For the group $\Gamma = \Gamma_{2}$, we can choose
$$ \mathcal{T}'' = \left\{
\left( \begin{smallmatrix} 0 \\ 0 \\  \alpha \end{smallmatrix} \right), 
\left( \begin{smallmatrix} 1/2 \\ 0 \\  \alpha \end{smallmatrix} \right), 
\left( \begin{smallmatrix} \alpha \\ \alpha \\  0 \end{smallmatrix} \right)
 \right\},$$
 where we have expressed the vertices (i.e., lines) in $\base$ in parametric form.

For the vertices $v \in \mathcal{T}''$, the strict stabilizer groups $\overline{\Gamma}_{v}$ satisfy
$$ \pi( \overline{\Gamma}_{v}) = D_{4}'', \quad D_{2}',\text{ and } \quad \langle A, B \rangle,$$
respectively, where 
$$ A = \left( \begin{smallmatrix} 0 & 1 & 0 \\ 1 & 0 & 0 \\ 0 & 0 & 1 \end{smallmatrix} \right) \text{ and } \quad B = \left( \begin{smallmatrix} 1 & 0 & 0 \\
0 & 1 & 0 \\ 0 & 0 & -1 \end{smallmatrix} \right).$$
\end{theorem}

\begin{proof}
We will consider the same planes $\mathbb{R}^{2}_{\ell}$ as in the proof of Theorem \ref{theorem:gamma1T}, 
although they will have different actions.
First, we consider the plane $\mathbb{R}^{2}_{(0,0,\alpha)}$, pictured in Figure \ref{figure:plane2,1}.

\begin{figure}[!h]
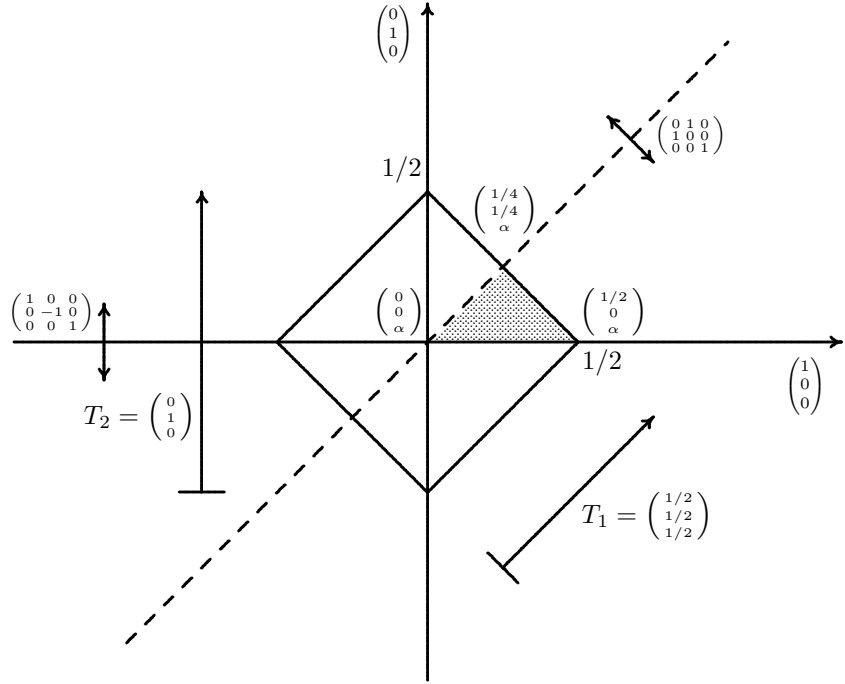

\begin{center}
\hbox{
\vbox{\beginpicture
\setcoordinatesystem units <1cm,1cm> point at -6 4.5
\setplotarea x from -6 to 6, y from -4.5 to 4.5
\linethickness=.7pt
%

 \def\myarrow{\arrow <4pt> [.2, 1]} 
%
\setplotsymbol ({\circle*{.4}})
\plotsymbolspacing=.3pt        
\myarrow from -5.5 0 to 5.5 0     
\myarrow from   0 -4.5  to  0 4.5 

\plot -2 0   0 2   2 0  0 -2  -2 0 /  
\myarrow from -3 -2 to -3 2    
\plot -3.3 -2   -2.7 -2 /          
\put {$T_2 = \tiny \left(\begin{smallmatrix} 0 \\ 1 \\ 0 \end{smallmatrix} \right)$} [r] at -3.1 -1       
\myarrow from  1 -3 to 3 -1 
\plot .8 -2.8  1.2 -3.2 /            
\put {$T_1  = \tiny\left(\begin{smallmatrix} 1/2 \\ 1/2 \\ 1/2 \end{smallmatrix} \right)$} [t] at 2.9 -2       
\myarrow from  -4.3 0 to -4.3 .5   
\myarrow from  -4.3 0 to -4.3 -.5


\myarrow from  2.7 2.7 to 2.4 3
\myarrow from  2.7 2.7 to 3 2.4

\put{$ \tiny \left( \begin{smallmatrix} 0 \\ 0 \\ \alpha \end{smallmatrix} \right)$} [r] at -.05 .4
\put{$ \tiny \left( \begin{smallmatrix} 1/2 \\ 0 \\ \alpha \end{smallmatrix} \right)$} [l] at 2 .4
\put{$ \tiny \left( \begin{smallmatrix} 1/4 \\ 1/4 \\ \alpha \end{smallmatrix} \right)$} [b] at 1 1.45

\setdots <2pt>
\setdashes
\plot -4 -4   4 4 /
\setsolid 
\put {\tiny $\left(\begin{matrix} 0 \\ 1 \\ 0\end{matrix}\right)$} [r] at
-.2 4.1
\put {\tiny $\left(\begin{matrix} 1 \\ 0 \\ 0\end{matrix}\right)$} [t] at
5 -.2
\put {$1/2$} [t] at 2.3 -.1 
\put {$1/2$} [r] at -.1  2.3

\put {\tiny $\left(\begin{smallmatrix} 0 & 1 & 0 \\ 1 & 0 & 0 \\ 0 & 0 & 1 \end{smallmatrix} \right)$} [tl] at 3 3  
\put {\tiny $\left(\begin{smallmatrix} 1 & 0 & 0 \\ 0 & -1 & 0 \\ 0 & 0 & 1 \end{smallmatrix} \right)$} [r] at -4.45 .4

\setshadegrid span <1pt>
\hshade 0 0 2   1 1 1 /  
\endpicture}
}
\end{center}
\caption{This picture describes the action of $\Gamma(0,0,\alpha)$ on $\mathbb{R}^{2}_{(0,0,\alpha)}$,
where $\Gamma = \Gamma_{2}$. The point group of $\Gamma(0,0,\alpha)$ is $D_{4}^{+} \times (-1)$.}
\label{figure:plane2,1}
\end{figure}

It is not difficult to check that the shaded triangle is an exact convex compact fundamental polyhedron for
the action of $G$, which is generated by reflections in the sides of the triangle. It is straightforward
to check that the stabilizers $G_{(0,0,\alpha)}$, $G_{(1/2,0,\alpha)}$, and $G_{(1/4,1/4,\alpha)}$ have orders
$8$, $8$, and $4$, respectively. (This is because the images of the fundamental polyhedron cannot overlap 
in their interiors -- see Definition \ref{definition:fundamentaldomain}.) By Lemma \ref{lemma:computess}(1), the map 
$\psi_{\mid}: \overline{\Gamma}_{\ell'} \rightarrow G_{\ell'}$ is injective for each of the lines in question. It 
is straightforward to check, exactly as in the proof of Theorem \ref{theorem:gamma1T}, that
$\psi_{\mid}: \overline{\Gamma}_{(0,0,\alpha)} \rightarrow G_{(0,0,\alpha)}$ is in fact an isomorphism, and
that $\overline{\Gamma}_{(0,0,\alpha)} = D_{4}''$. 

Now we consider the strict stabilizer group $\overline{\Gamma}_{(1/2,0,\alpha)}$. It is straightforward to check
that 
$$ \left( \begin{smallmatrix} 1 & 0 & 0 \\ 0 & -1 & 0 \\ 0 & 0 & 1 \end{smallmatrix}\right), \quad
\left( \begin{smallmatrix} 1 \\ 0 \\ 0 \end{smallmatrix} \right) +
\left( \begin{smallmatrix} -1 & 0 & 0 \\ 0 & 1 & 0 \\ 0 & 0 & 1 \end{smallmatrix}\right),$$
are both in $\overline{\Gamma}_{(1/2,0,\alpha)}$. It follows directly that
$D_{2}' \subseteq \pi(\overline{\Gamma}_{(1/2,0,\alpha)})$. We claim that $\psi(\overline{\Gamma}_{(1/2,0,\alpha)})$
is a proper subgroup of $G_{(1/2,0,\alpha)}$, so that $D_{2}' = \pi(\overline{\Gamma}_{(1/2,0,\alpha)})$.
We prove the claim. Consider the isometry
$$ \gamma = \left( \begin{smallmatrix} 1/2 \\ 1/2 \\ 1/2 \end{smallmatrix} \right) + 
\left( \begin{smallmatrix} 0 & -1 & 0 \\ -1 & 0 & 0 \\ 0 & 0 & 1 \end{smallmatrix}\right).$$
(Note that the linear part is reflection in the line $y=-x$ in the plane from Figure \ref{figure:plane2,1}.)
It is easy to check that $g = \psi(\gamma)  \in G_{(1/2,0,\alpha)}$. By Lemma \ref{lemma:computess}(2),
$g \in \psi(\overline{\Gamma}_{(1/2,0,\alpha)})$ if and only if the action of $\gamma$ on $(1/2,0,\alpha)$
agrees with the action of some $k$ from the kernel of 
$\psi: \Gamma(0,0,\alpha) \rightarrow \mathrm{Isom}(\mathbb{R}^{2})$. The latter kernel is
$$ \left\langle \left( \begin{smallmatrix} 1 & 0 & 0 \\ 0 & 1 & 0 \\ 0 & 0 & -1 \end{smallmatrix} \right),
\left( \begin{smallmatrix} 0 \\ 0 \\ 1 \end{smallmatrix} \right) \right\rangle.$$
We have $\gamma \cdot r(\alpha) = r(\alpha + 1/2)$ (where $r(\alpha) = (1/2,0,\alpha)$), however, 
and there is no $k \in \mathrm{Ker} \, \psi$
that acts this way on $r(\alpha)$. This proves the claim.

One can show that $\psi(\overline{\Gamma}_{(1/4,1/4,\alpha)})$ is a proper subgroup of $G_{(1/4,1/4,\alpha)}$
in an analogous way, using the same $\gamma$ as above. One then concludes that 
$|\overline{\Gamma}_{(1/4,1/4,\alpha)}| \leq 2$, so $\overline{\Gamma}_{(1/4,1/4,\alpha)}$ is negligible.

\begin{figure}[!h]
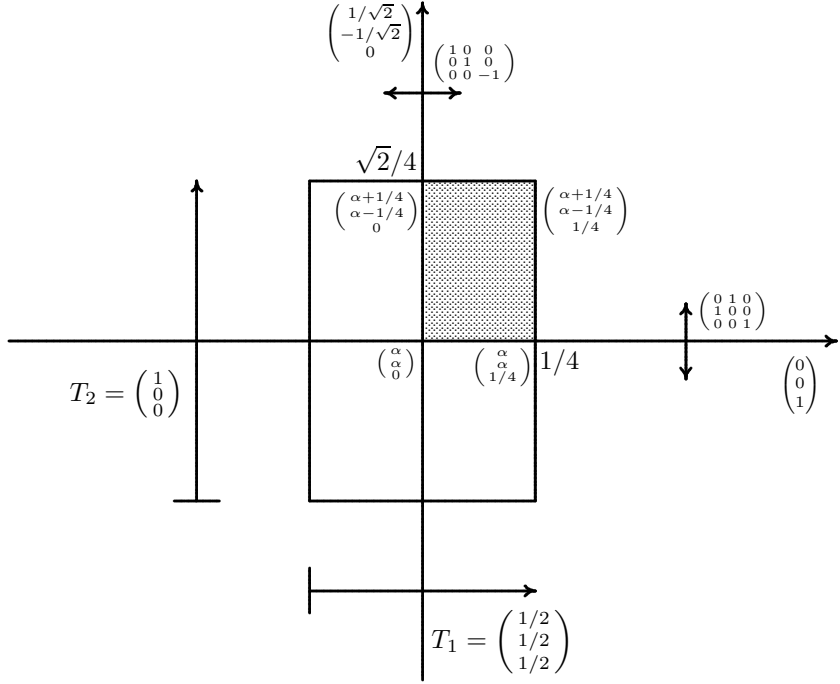

\begin{center}
\hbox{
\vbox{\beginpicture
\setcoordinatesystem units <1cm,1cm> point at -6 4.5
\setplotarea x from -6 to 6, y from -4.5 to 4.5
\linethickness=.7pt

 \def\myarrow{\arrow <4pt> [.2, 1]} 
%
\setplotsymbol ({\circle*{.4}})
\plotsymbolspacing=.3pt        
\myarrow from -5.5 0 to 5.5 0     
\myarrow from   0 -4.5  to  0 4.5 

\plot -1.5 2.13   1.5 2.13  1.5 -2.13  -1.5 -2.13  -1.5 2.13 /  
\myarrow from -3 -2.13 to -3 2.13    
\plot -3.3 -2.13   -2.7 -2.13 /          
\put {$T_2 = \left( \begin{smallmatrix} 1 \\ 0 \\ 0 \end{smallmatrix} \right)$} [r] at -3.2 -.7       
\myarrow from  -1.5 -3.32 to 1.5 -3.32      
\plot -1.5 -3.02  -1.5 -3.62 /            
\put {$T_1 = \left( \begin{smallmatrix} 1/2 \\ 1/2 \\ 1/2 \end{smallmatrix} \right)$} [l] at .1 -4      

\put{\tiny $\left( \begin{smallmatrix} \alpha \\ \alpha \\ 0 \end{smallmatrix} \right)$} [r] at -.1 -.3
\put{\tiny $\left( \begin{smallmatrix} \alpha+ 1/4 \\ \alpha - 1/4 \\ 0 \end{smallmatrix} \right)$} [r] at -.05 1.73
\put{\tiny $\left( \begin{smallmatrix} \alpha \\ \alpha \\ 1/4 \end{smallmatrix} \right)$} [r] at 1.45 -.35
\put{\tiny $\left( \begin{smallmatrix} \alpha+1/4 \\ \alpha-1/4 \\ 1/4 \end{smallmatrix} \right)$} [l] at 1.55 1.73

\myarrow from  3.5 0 to 3.5 .5   
\myarrow from  3.5 0 to 3.5 -.5

\myarrow from  0 3.3 to  .5 3.3
\myarrow from  0 3.3 to  -.5 3.3

\setdots <2pt>
\put {\tiny $\left(\begin{matrix} 1/\sqrt{2} \\ -1/\sqrt{2} \\ 0\end{matrix}\right)$} [r] at
-.1 4.1
\put {\tiny $\left(\begin{matrix} 0 \\ 0 \\ 1\end{matrix}\right)$} [t] at
5 -.2
\put {$1/4$} [t] at 1.8 -.1 
\put {$\sqrt{2}/4$} [r] at -.1  2.4

\put {\tiny $\left(\begin{smallmatrix} 1 & 0 & 0 \\ 0 & 1 & 0 \\ 0 & 0 & -1 \end{smallmatrix} \right)$} [l] at .1  3.7
\put {\tiny $\left(\begin{smallmatrix} 0 & 1 & 0 \\ 1 & 0 & 0 \\ 0 & 0 & 1 \end{smallmatrix} \right)$} [b] at 4.1 .15

\setshadegrid span <1pt>
\vshade 0 0 2.13   1.5 0 2.13 /  
\endpicture}
}
\end{center}
\caption{This picture describes the action of $\Gamma(\alpha,\alpha,0)$ on $\mathbb{R}^{2}_{(\alpha,\alpha,0)}$.
The point group is generated by the given reflections and the antipodal map.}
\label{figure:plane2,2}
\end{figure}

Next, we consider the action of $\Gamma(\alpha,\alpha,0)$ on $\mathbb{R}^{2}_{(\alpha,\alpha,0)}$, as
described in Figure \ref{figure:plane2,2}. The image group $G$ acts by reflections on the sides of the 
shaded rectangle, which is an exact convex compact fundamental polyhedron. (In fact, $G$ is again a Coxeter group.)
It is easy to see that the four vertices forming the corners of the rectangle have stabilizers of order $4$ relative
to $G$. 

The reflections in Figure \ref{figure:plane2,2} are both easily seen to be elements of the group
$\overline{\Gamma}_{(\alpha,\alpha,0)}$. Since these reflections generate a group of order $4$, it follows
from Lemma \ref{lemma:computess}(1) 
that this is the entire strict stabilizer of $(\alpha,\alpha,0)$ (as stated in the theorem). 

All of the other corners have negligible stabilizers. We will prove this for the corner $(\alpha,\alpha,1/4)$, the
other cases being similar. Consider the isometry 
$\gamma = T_{1} + R_{z}$, where $T_{1}$ is pictured in Figure \ref{figure:plane2,2} and $R_{z}$
is multiplication by $-1$ in the $z$-coordinate; i.e., $R_{z}$ labels the horizontal double-tailed arrow from
Figure \ref{figure:plane2,2}. Setting $r(\alpha) = (\alpha, \alpha, 1/4)$, one easily checks that 
$\gamma \cdot r(\alpha) = r(\alpha + 1/2)$. By Lemma \ref{lemma:computess}(2), 
$g = \psi(\gamma) \in \psi(\overline{\Gamma}_{(\alpha,\alpha,1/4)})$ if and only if there is some
$k$ in the kernel 
$$ \left\langle \left( \begin{smallmatrix} 1 \\ 1 \\ 0 \end{smallmatrix} \right), 
\left(\begin{smallmatrix} 0& -1 & 0 \\ -1 & 0 & 0 \\ 0 & 0 & 1 \end{smallmatrix}\right) \right\rangle$$
such that $k \cdot r(\alpha) = r(\alpha + 1/2)$. The generators of the kernel send $r(\alpha)$ to
$r(\alpha + 1)$ and $r(-\alpha)$ (respectively). It follows that the required $k$ does not exist. Thus
$g$ is in $G_{(\alpha,\alpha,1/4)}$ but not in $\psi(\overline{\Gamma}_{(\alpha,\alpha,1/4)})$,
so the latter group is a proper subgroup of the former. It follows that
$|\overline{\Gamma}_{(\alpha,\alpha,1/4)}| \leq 2$, proving that
$\overline{\Gamma}_{(\alpha,\alpha,1/4)}$ is negligible.
 \end{proof}

\begin{theorem} \label{theorem:gamma3T}
For the group $\Gamma = \Gamma_{3}$, we can choose
$$ \mathcal{T}'' = \left\{
\left( \begin{smallmatrix} 0 \\ 0 \\  \alpha \end{smallmatrix} \right), 
\left( \begin{smallmatrix} 1/4 \\ 1/4 \\  \alpha \end{smallmatrix} \right),  
\left( \begin{smallmatrix} \alpha \\ \alpha \\  0 \end{smallmatrix} \right),  
\left( \begin{smallmatrix} \alpha \\ \alpha \\  1/2 \end{smallmatrix} \right) \right\},$$
 where we have expressed the vertices (i.e., lines) in $\base$ in parametric form.

For the vertices $v \in \mathcal{T}''$, the strict stabilizer groups $\overline{\Gamma}_{v}$ satisfy
$$ \pi( \overline{\Gamma}_{v}) = D_{4}'', \quad
\langle A, C \rangle, \quad \langle A, B \rangle, \text { and } \quad \langle A, B \rangle,$$
respectively, where 
$$ A = \left( \begin{smallmatrix} 0 & 1 & 0 \\ 1 & 0 & 0 \\ 0 & 0 & 1 \end{smallmatrix} \right),  \quad B = \left( \begin{smallmatrix} 1 & 0 & 0 \\
0 & 1 & 0 \\ 0 & 0 & -1 \end{smallmatrix} \right), \text{ and } \quad 
C = \left( \begin{smallmatrix} -1 & 0 & 0 \\ 0 & -1 & 0 \\ 0 & 0 & 1 \end{smallmatrix} \right).$$
\end{theorem}

\begin{proof}
We first consider the action on $\mathbb{R}^{2}_{(0,0,\alpha)}$. The image group $G$ is a Coxeter group, generated
by reflections in the sides of the triangle pictured in Figure \ref{figure:plane3,1}. We note first that it is straightforward
to verify that $\overline{\Gamma}_{(0,0,\alpha)} = D_{4}''$. The argument follows the same lines as in the proofs
of the previous
theorems. 

We concentrate on the remaining two corners of the triangle. First consider the line $(1/4,0,\alpha)$. It
is clear that $G_{(1/4,0,\alpha)}$ has order $4$. We claim that $\overline{\Gamma}_{(1/4,0,\alpha)}$ is
negligible. This will follow from Lemma \ref{lemma:computess}(1) and Proposition \ref{proposition:squarefreenegligible}
once we show that $\psi(\overline{\Gamma}_{(1/4,0,\alpha)})$ is a proper subgroup of $G_{(1/4,0,\alpha)}$. Consider
the element 
$$ \gamma = T_{1} + \left( \begin{smallmatrix} -1 & 0 & 0 \\ 0 & 1 & 0 \\ 0 & 0 & 1 \end{smallmatrix} \right).$$
It is easy to check that $\psi(\gamma) \in G_{(1/4,0,\alpha)}$. By Lemma \ref{lemma:computess}(2),
$g = \psi(\gamma)$ is in $\psi(\overline{\Gamma}_{(1/4,0,\alpha)})$ if and only if there is some element $k$
of the kernel of $\psi: \Gamma(0,0,\alpha) \rightarrow \mathrm{Isom}(\mathbb{R}^{2})$ that agrees with
$\gamma$ on $(1/4,0,\alpha)$. We note that the kernel is 
$$ \left\langle \left( \begin{smallmatrix} 0 \\ 0 \\ 1 \end{smallmatrix} \right), 
\left( \begin{smallmatrix} 1 & 0 & 0 \\ 0 & 1 & 0 \\ 0 & 0 & -1 \end{smallmatrix} \right) \right\rangle.$$
The above generators act on $r(\alpha) = (1/4,0,\alpha)$ by sending $r(\alpha)$ to 
$r(\alpha + 1)$ and $r(-\alpha)$, respectively. On the other hand, $\gamma \cdot r(\alpha) = r(\alpha + 1/2)$.
It follows that $g = \psi(\gamma) \notin \psi(\overline{\Gamma}_{(1/4,0,\alpha)})$. This proves the claim.

\begin{figure}[!h]
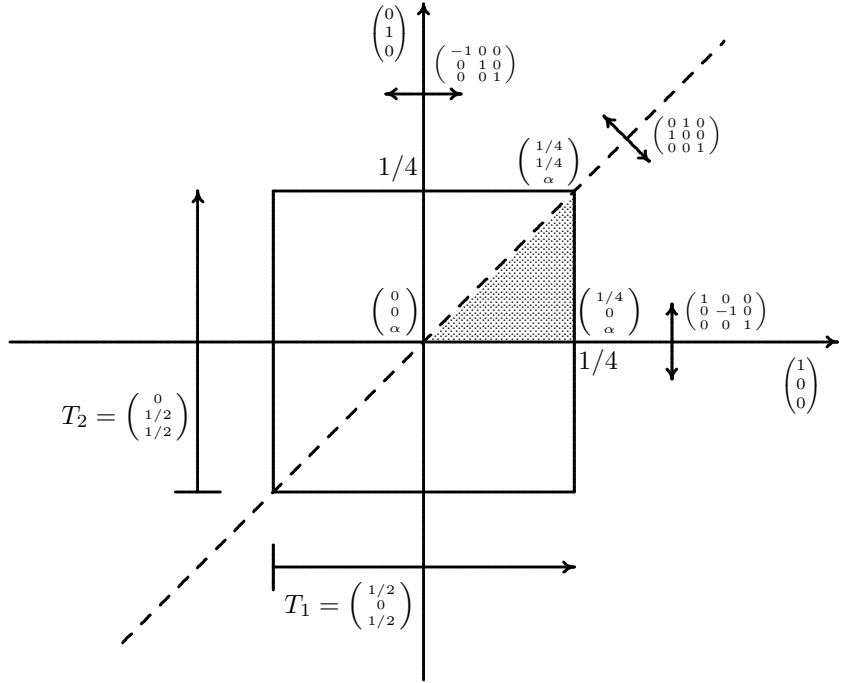

\begin{center}
\hbox{
\vbox{\beginpicture
\setcoordinatesystem units <1cm,1cm> point at -6 4.5
\setplotarea x from -6 to 6, y from -4.5 to 4.5
\linethickness=.7pt
%

 \def\myarrow{\arrow <4pt> [.2, 1]} 
%
\setplotsymbol ({\circle*{.4}})
\plotsymbolspacing=.3pt        
\myarrow from -5.5 0 to 5.5 0     
\myarrow from   0 -4.5  to  0 4.5 

\plot -2 2   2 2   2 -2  -2 -2  -2 2 /  
\myarrow from -3 -2 to -3 2    
\plot -3.3 -2   -2.7 -2 /          
\put {$T_2 = \tiny \left(\begin{smallmatrix} 0 \\ 1/2 \\ 1/2 \end{smallmatrix} \right)$} [r] at -3.1 -1       
\myarrow from  -2 -3 to 2 -3      
\plot -2 -3.3  -2 -2.7 /            
\put {$T_1  = \tiny\left(\begin{smallmatrix} 1/2 \\ 0 \\ 1/2 \end{smallmatrix} \right)$} [t] at -1 -3.2       
\myarrow from  3.3 0 to 3.3 .5   
\myarrow from  3.3 0 to 3.3 -.5

\myarrow from  0 3.3 to  .5 3.3
\myarrow from  0 3.3 to  -.5 3.3

\myarrow from  2.7 2.7 to 2.4 3
\myarrow from  2.7 2.7 to 3 2.4

\put{$ \tiny \left( \begin{smallmatrix} 0 \\ 0 \\ \alpha \end{smallmatrix} \right)$} [r] at -.05 .4
\put{$ \tiny \left( \begin{smallmatrix} 1/4 \\ 0 \\ \alpha \end{smallmatrix} \right)$} [l] at 2 .4
\put{$ \tiny \left( \begin{smallmatrix} 1/4 \\ 1/4 \\ \alpha \end{smallmatrix} \right)$} [r] at 2.1 2.4

\setdots <2pt>
\setdashes
\plot -4 -4   4 4 /
\setsolid 
\put {\tiny $\left(\begin{matrix} 0 \\ 1 \\ 0\end{matrix}\right)$} [r] at
-.2 4.1
\put {\tiny $\left(\begin{matrix} 1 \\ 0 \\ 0\end{matrix}\right)$} [t] at
5 -.2
\put {$1/4$} [t] at 2.3 -.1 
\put {$1/4$} [r] at -.1  2.3

\put {\tiny $\left(\begin{smallmatrix} -1 & 0 & 0 \\ 0 & 1 & 0 \\ 0 & 0 & 1 \end{smallmatrix} \right)$} [l] at .1  3.7
\put {\tiny $\left(\begin{smallmatrix} 0 & 1 & 0 \\ 1 & 0 & 0 \\ 0 & 0 & 1 \end{smallmatrix} \right)$} [tl] at 3 3  
\put {\tiny $\left(\begin{smallmatrix} 1 & 0 & 0 \\ 0 & -1 & 0 \\ 0 & 0 & 1 \end{smallmatrix} \right)$} [b] at 4 .15

\setshadegrid span <1pt>
\hshade 0 0 2   2 2 2 /  
\endpicture}
}
\end{center}
\caption{This picture describes the action of $\Gamma(0,0,\alpha)$ on the plane $\mathbb{R}^{2}_{(0,0,\alpha)}$,
where $\Gamma = \Gamma_{3}$.}
\label{figure:plane3,1}
\end{figure}

Next, we consider the line $(1/4,1/4,\alpha)$. Essentially the same reasoning as above, using the same $\gamma$,
shows that $\psi(\overline{\Gamma}_{(1/4,1/4,\alpha)})$ is a proper subgroup of $G_{(1/4,1/4,\alpha)}$. 
We note that $G_{(1/4,1/4,\alpha)}$ has order $8$, so that $\overline{\Gamma}_{(1/4,1/4,\alpha)}$
has order at most $4$. It is straightforward to check that
$$ \left\langle \left( \begin{smallmatrix} 0 & 1 & 0 \\ 1 & 0 & 0 \\ 0 & 0 & 1 \end{smallmatrix} \right),
\left( \begin{smallmatrix} 1/2 \\ 1/2 \\ 0 \end{smallmatrix} \right) + 
\left( \begin{smallmatrix} -1 & 0 & 0 \\ 0 & -1 & 0 \\ 0 & 0 & 1 \end{smallmatrix} \right) \right\rangle$$
is a group of order $4$ that is a subgroup of $\overline{\Gamma}_{(1/4,1/4,\alpha)}$, so equality must hold.
It easily follows that $\pi(\overline{\Gamma}_{(1/4,1/4,\alpha)}) = \langle A, C \rangle$, as in the statement
of the theorem.

\begin{figure}[!h]
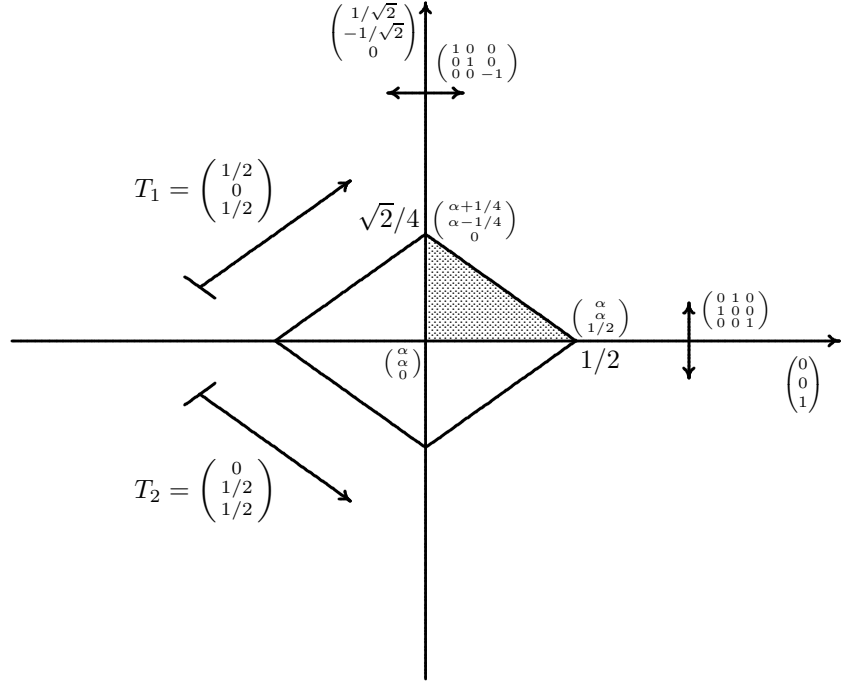

\begin{center}
\hbox{
\vbox{\beginpicture
\setcoordinatesystem units <1cm,1cm> point at -6 4.5
\setplotarea x from -6 to 6, y from -4.5 to 4.5
\linethickness=.7pt

 \def\myarrow{\arrow <4pt> [.2, 1]} 
%
\setplotsymbol ({\circle*{.4}})
\plotsymbolspacing=.3pt        
\myarrow from -5.5 0 to 5.5 0     
\myarrow from   0 -4.5  to  0 4.5 

\plot -2 0   0 1.42   2 0  0 -1.42  -2 0 /  
\myarrow from -3 .71 to -1 2.13 
\plot -3.2 .852   -2.8 .568 /          
\put {$T_2 = \left( \begin{smallmatrix} 0 \\ 1/2 \\ 1/2 \end{smallmatrix} \right)$} [r] at -2 -2       
\myarrow from  -3 -.71 to -1 -2.13     
\plot -3.2 -.852  -2.8 -.568 /            
\put {$T_1 = \left( \begin{smallmatrix} 1/2 \\ 0 \\ 1/2 \end{smallmatrix} \right)$} [r] at -2 2      

\put{\tiny $\left( \begin{smallmatrix} \alpha \\ \alpha \\ 0 \end{smallmatrix} \right)$} [r] at -.05 -.3
\put{\tiny $\left( \begin{smallmatrix} \alpha+ 1/4 \\ \alpha - 1/4 \\ 0 \end{smallmatrix} \right)$} [l] at .05 1.6
\put{\tiny $\left( \begin{smallmatrix} \alpha \\ \alpha \\ 1/2 \end{smallmatrix} \right)$} [l] at 1.9 .3

\myarrow from  3.5 0 to 3.5 .5   
\myarrow from  3.5 0 to 3.5 -.5

\myarrow from  0 3.3 to  .5 3.3
\myarrow from  0 3.3 to  -.5 3.3

\setdots <2pt>
\put {\tiny $\left(\begin{matrix} 1/\sqrt{2} \\ -1/\sqrt{2} \\ 0\end{matrix}\right)$} [r] at
-.1 4.1
\put {\tiny $\left(\begin{matrix} 0 \\ 0 \\ 1\end{matrix}\right)$} [t] at
5 -.2
\put {$1/2$} [t] at 2.3 -.1 
\put {$\sqrt{2}/4$} [r] at -.1  1.63

\put {\tiny $\left(\begin{smallmatrix} 1 & 0 & 0 \\ 0 & 1 & 0 \\ 0 & 0 & -1 \end{smallmatrix} \right)$} [l] at .1  3.7
\put {\tiny $\left(\begin{smallmatrix} 0 & 1 & 0 \\ 1 & 0 & 0 \\ 0 & 0 & 1 \end{smallmatrix} \right)$} [b] at 4.1 .15

\setshadegrid span <1pt>
\vshade 0 0 1.42   2 0 0 /  
\endpicture}
}
\end{center}
\caption{This picture describes the action of $\Gamma(\alpha,\alpha,0)$ on $\mathbb{R}^{2}_{(\alpha,\alpha,0)}$,
where $\Gamma = \Gamma_{3}$.}
\label{figure:plane3,2}
\end{figure}

Next we must consider the plane $\mathbb{R}^{2}_{(\alpha,\alpha,0)}$. The action is described in 
Figure \ref{figure:plane3,2}. We note that the image group $G$ does not act as a Coxeter group in this case:
the side-pairing of the slanted side of the triangle $P$ with itself is a rotation of $180$ degrees about the midpoint
of that side. It follows that the top vertex and the right-hand vertex are in the same orbit under the
action of $\Gamma_{(\alpha,\alpha,0)}$. As a result, we only need to consider the vertices $(\alpha,\alpha,0)$
and $(\alpha,\alpha,1/2)$. It is easy to check that $\overline{\Gamma}_{(\alpha,\alpha,0)} = \langle A, B \rangle$,
so we will concentrate on the vertex $(\alpha,\alpha,1/2)$.

We claim that $|G_{(\alpha,\alpha,1/2)}| < 8$. To prove the claim, we note that $|G_{(\alpha,\alpha,1/2)}|$
must be less than or equal to the number of translates of the fundamental domain 
that touch $(\alpha,\alpha,1/2)$. There are $8$ such translates, which proves that $|G_{(\alpha,\alpha,1/2)}| \leq 8$.
Next, we consider the side-pairing $\phi$ 
of the slanted side of $P$ with itself, which is rotation by $180$ degrees
about the midpoint. Note that $\phi(P)$ touches $(\alpha,\alpha,1/2)$, although $\phi \notin G_{(\alpha,\alpha,1/2)}$.
There cannot be $g \in G_{(\alpha,\alpha,1/2)}$ such that $g(P) = \phi(P)$, for, by Definition \ref{definition:fundamentaldomain}, this would imply that $g = \phi$, a contradiction. Thus, the
elements of $G_{(\alpha,\alpha,1/2)}$ must move $P$ to some subset of the remaining $7$ translates of $P$
that touch $(\alpha,\alpha,1/2)$. Since different elements of $G_{(\alpha,\alpha,1/2)}$ must move $P$ to different
elements of the tesselation $\{ gP \mid g \in G \}$ (again by Definition \ref{definition:fundamentaldomain}), the
claim follows. 

Next, we directly check that 
$$ \left\langle  \left( \begin{smallmatrix} 0 & 1 & 0 \\ 1 & 0 & 0 \\ 0 & 0 & 1 \end{smallmatrix} \right),
\left( \begin{smallmatrix} 0 \\ 0 \\ 1 \end{smallmatrix} \right) +  
\left( \begin{smallmatrix} 1 & 0 & 0 \\ 0 & 1 & 0 \\ 0 & 0 & -1 \end{smallmatrix} \right) \right\rangle
\leq \overline{\Gamma}_{(\alpha,\alpha,1/2)}.$$
It is easy to see that the above isometries generate a group of order $4$. In view of the inequality
$|G_{(\alpha,\alpha,1/2)}| < 8$ and Lemma \ref{lemma:computess}(1), the above containment must be
an equality.

(Technically, Theorem \ref{cells} tells us that the midpoint of the slanted side in Figure \ref{figure:plane3,2}
should be a vertex $v$ in the cellulation of the plane, so we should consider its strict stabilizer as well. However,
it is clear that $|G_{v}| = 2$, so the strict stabilizer is necessarily negligible, by Lemma \ref{lemma:computess}(1).) 
\end{proof}

\begin{figure}[!h]
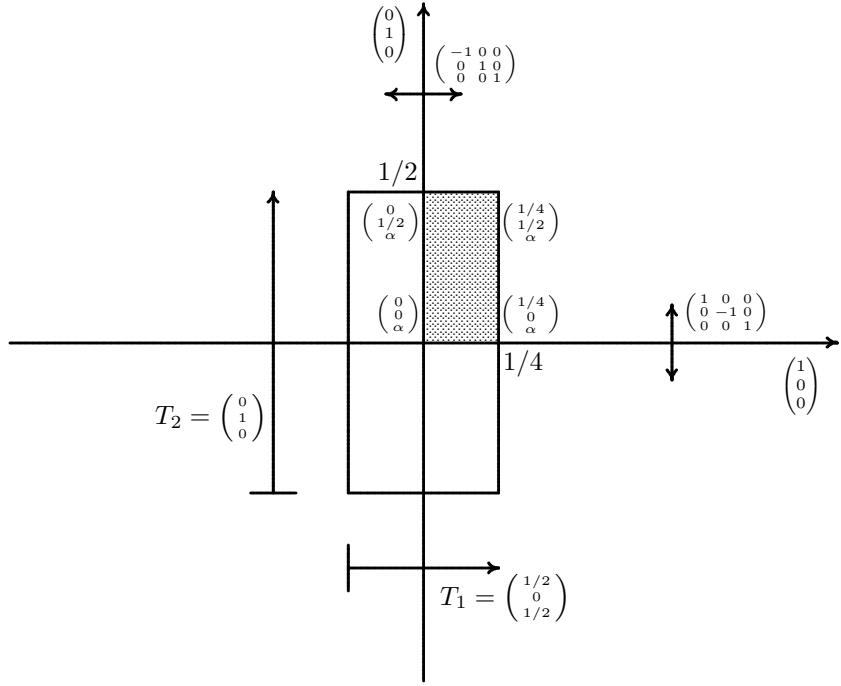

\begin{center}
\hbox{
\vbox{\beginpicture
\setcoordinatesystem units <1cm,1cm> point at -6 4.5
\setplotarea x from -6 to 6, y from -4.5 to 4.5
\linethickness=.7pt
%

 \def\myarrow{\arrow <4pt> [.2, 1]} 
%
\setplotsymbol ({\circle*{.4}})
\plotsymbolspacing=.3pt        
\myarrow from -5.5 0 to 5.5 0     
\myarrow from   0 -4.5  to  0 4.5 

\plot -1 2   1 2   1 -2  -1 -2  -1 2 /  
\myarrow from -2 -2 to -2 2    
\plot -2.3 -2   -1.7 -2 /          
\put {$T_2 = \tiny \left(\begin{smallmatrix} 0 \\ 1 \\ 0 \end{smallmatrix} \right)$} [r] at -2.1 -1       
\myarrow from  -1 -3 to 1 -3      
\plot -1 -3.3  -1 -2.7 /            
\put {$T_1  = \tiny \left(\begin{smallmatrix} 1/2 \\ 0 \\ 1/2 \end{smallmatrix} \right)$} [l] at .2 -3.4       
\myarrow from  3.3 0 to 3.3 .5   
\myarrow from  3.3 0 to 3.3 -.5

\myarrow from  0 3.3 to  .5 3.3
\myarrow from  0 3.3 to  -.5 3.3


\put {\tiny $\left( \begin{smallmatrix} 0 \\ 0 \\ \alpha \end{smallmatrix} \right)$} [r] at -.05 .37
\put {\tiny $\left( \begin{smallmatrix} 1/4 \\ 0 \\ \alpha \end{smallmatrix} \right)$} [l] at 1 .37
\put {\tiny $\left( \begin{smallmatrix} 1/4 \\ 1/2 \\ \alpha \end{smallmatrix} \right)$} [l] at 1 1.6
\put {\tiny $\left(\begin{smallmatrix} 0 \\ 1/2 \\ \alpha \end{smallmatrix} \right)$} [r] at -.05 1.6

\put {\tiny $\left(\begin{matrix} 0 \\ 1 \\ 0 \end{matrix}\right)$} [r] at -.2 4.1
\put {\tiny $\left(\begin{matrix} 1 \\ 0 \\ 0 \end{matrix}\right)$} [t] at 5 -.2
\put {$1/4$} [t] at 1.3 -.1 
\put {$1/2$} [r] at -.1  2.25

\put {\tiny $\left(\begin{smallmatrix} -1 & 0 & 0 \\ 0 & 1 & 0 \\ 0 & 0 & 1 \end{smallmatrix} \right)$} [l] at .1  3.7
\put {\tiny $\left(\begin{smallmatrix} 1 & 0 & 0 \\ 0 & -1 & 0 \\ 0 & 0 & 1 \end{smallmatrix} \right)$} [b] at 4 .15

\setshadegrid span <1pt>
\hshade 0 0 1   2 0 1 /  
\endpicture}
}
\end{center}
\caption{This picture describes the action of $\Gamma(0,0,\alpha)$ on the plane $\mathbb{R}^{2}_{(0,0,\alpha)}$,
where $\Gamma = \Gamma_{4}$.}
\label{figure:plane4,1}
\end{figure}

\begin{theorem} \label{theorem:gamma4T}
For the group $\Gamma = \Gamma_{4}$, we can choose
$$ \mathcal{T}'' = \left\{
\left( \begin{smallmatrix} 0 \\ 0 \\  \alpha \end{smallmatrix} \right), 
\left( \begin{smallmatrix} 0 \\ 1/2 \\  \alpha \end{smallmatrix} \right),
\left( \begin{smallmatrix} \alpha \\ 0 \\  0 \end{smallmatrix} \right), 
\left( \begin{smallmatrix} \alpha \\ 1/2 \\  0 \end{smallmatrix} \right), 
\left( \begin{smallmatrix} 0 \\ \alpha \\ 0 \end{smallmatrix} \right), 
\left( \begin{smallmatrix} 0 \\ \alpha \\ 1/2 \end{smallmatrix} \right) \right\},$$
 where we have expressed the vertices (i.e., lines) in $\base$ in parametric form.

For the vertices $v \in \mathcal{T}''$, the strict stabilizer groups $\overline{\Gamma}_{v}$ are
$$ \pi( \overline{\Gamma}_{v}) = D_{2}', \quad  D_{2}', \quad D_{2_{1}}', \quad  D_{2_{1}}', \quad
D_{2_{2}}',\text{ and }  \quad D_{2_{2}}'.$$
\end{theorem}

\begin{figure}[!p]
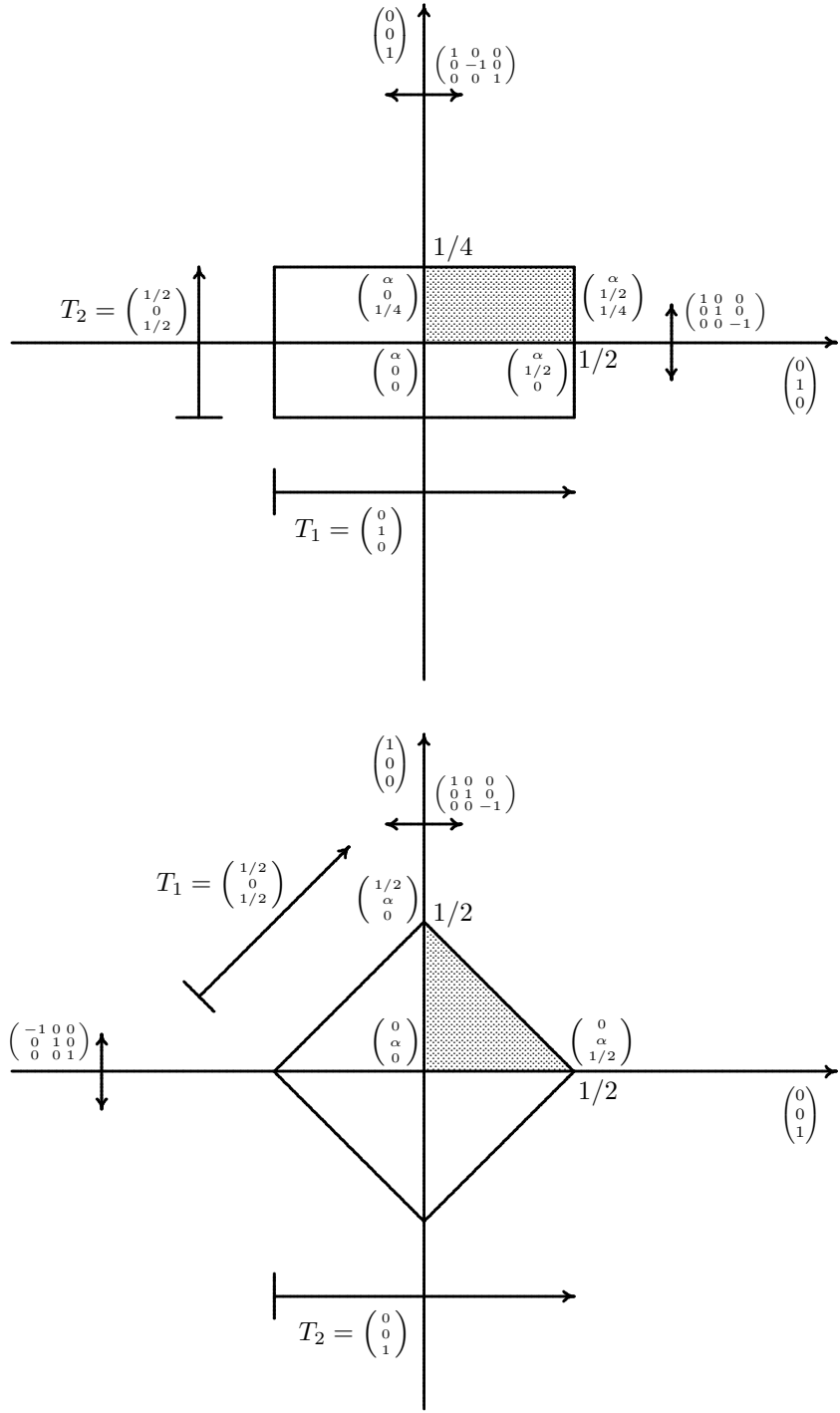

\begin{center}
\hbox{
\vbox{\beginpicture
\setcoordinatesystem units <1cm,1cm> point at -6 4.5
\setplotarea x from -6 to 6, y from -4.5 to 4.5
\linethickness=.7pt
%

 \def\myarrow{\arrow <4pt> [.2, 1]} 
%
\setplotsymbol ({\circle*{.4}})
\plotsymbolspacing=.3pt        
\myarrow from -5.5 0 to 5.5 0     
\myarrow from   0 -4.5  to  0 4.5 

\plot -2 1   2 1   2 -1  -2 -1  -2 1 /  
\myarrow from -3 -1 to -3 1    
\plot -3.3 -1   -2.7 -1 /          
\put {$T_2 = \tiny \left(\begin{smallmatrix} 1/2 \\ 0 \\  1/2 \end{smallmatrix} \right)$} [b] at -4 .1       
\myarrow from  -2 -2 to 2 -2      
\plot -2 -2.3  -2 -1.7 /            
\put {$T_1  = \tiny\left(\begin{smallmatrix} 0 \\ 1 \\ 0 \end{smallmatrix} \right)$} [t] at -1 -2.2       
\myarrow from  3.3 0 to 3.3 .5   
\myarrow from  3.3 0 to 3.3 -.5

\myarrow from  0 3.3 to  .5 3.3
\myarrow from  0 3.3 to  -.5 3.3

\put{$ \tiny \left( \begin{smallmatrix} \alpha \\ 0 \\ 0 \end{smallmatrix} \right)$} [r] at -.05 -.4
\put{$ \tiny \left( \begin{smallmatrix} \alpha \\ 1/2 \\ 0 \end{smallmatrix} \right)$} [r] at 1.95 -.4
\put{$ \tiny \left( \begin{smallmatrix} \alpha \\ 1/2 \\ 1/4 \end{smallmatrix} \right)$} [l] at 2.05 .6
\put{$ \tiny \left( \begin{smallmatrix} \alpha \\ 0 \\ 1/4 \end{smallmatrix} \right)$} [r] at -.05 .6

\put {\tiny $\left(\begin{matrix} 0 \\ 0 \\ 1\end{matrix}\right)$} [r] at
-.2 4.1
\put {\tiny $\left(\begin{matrix} 0 \\ 1 \\ 0\end{matrix}\right)$} [t] at
5 -.2
\put {$1/2$} [t] at 2.3 -.08 
\put {$1/4$} [l] at .1  1.25

\put {\tiny $\left(\begin{smallmatrix} 1 & 0 & 0 \\ 0 & -1 & 0 \\ 0 & 0 & 1 \end{smallmatrix} \right)$} [l] at .1  3.7

\put {\tiny $\left(\begin{smallmatrix} 1 & 0 & 0 \\ 0 & 1 & 0 \\ 0 & 0 & -1 \end{smallmatrix} \right)$} [b] at 4 .15

\setshadegrid span <1pt>
\hshade 0 0 2   1 0 2 /  
\endpicture}
}

\vspace{20pt}

\hbox{
\vbox{\beginpicture
\setcoordinatesystem units <1cm,1cm> point at -6 4.5
\setplotarea x from -6 to 6, y from -4.5 to 4.5
\linethickness=.7pt
%

 \def\myarrow{\arrow <4pt> [.2, 1]} 
%
\setplotsymbol ({\circle*{.4}})
\plotsymbolspacing=.3pt        
\myarrow from -5.5 0 to 5.5 0     
\myarrow from   0 -4.5  to  0 4.5 

\plot -2 0   0 2   2 0  0 -2  -2 0 /  
\myarrow from -2 -3 to 2 -3    
\plot -2 -2.7  -2 -3.3    /          
\put {$T_2 = \tiny \left(\begin{smallmatrix} 0 \\ 0 \\ 1 \end{smallmatrix} \right)$} [r] at -.2 -3.5      
\myarrow from  -3 1 to -1 3  
\plot -3.2 1.2  -2.8 .8  /            
\put {$T_1  = \tiny\left(\begin{smallmatrix} 1/2 \\ 0 \\ 1/2 \end{smallmatrix} \right)$} [r] at -1.85 2.5       
\myarrow from  -4.3 0 to -4.3 .5   
\myarrow from  -4.3 0 to -4.3 -.5

\myarrow from  0 3.3 to  .5 3.3
\myarrow from  0 3.3 to  -.5 3.3


\put{$ \tiny \left( \begin{smallmatrix} 0 \\ \alpha \\ 0 \end{smallmatrix} \right)$} [r] at -.05 .4
\put{$ \tiny \left( \begin{smallmatrix} 0 \\ \alpha \\ 1/2 \end{smallmatrix} \right)$} [l] at 1.9 .4
\put{$ \tiny \left( \begin{smallmatrix} 1/2 \\ \alpha \\ 0 \end{smallmatrix} \right)$} [r] at -.05 2.3

\put {\tiny $\left(\begin{matrix} 1 \\ 0 \\ 0\end{matrix}\right)$} [r] at
-.2 4.1
\put {\tiny $\left(\begin{matrix} 0 \\ 0 \\ 1 \end{matrix}\right)$} [t] at
5 -.2
\put {$1/2$} [t] at 2.3 -.1 
\put {$1/2$} [l] at .1  2.1

\put {\tiny $\left(\begin{smallmatrix} 1 & 0 & 0 \\ 0 & 1 & 0 \\ 0 & 0 & -1 \end{smallmatrix} \right)$} [l] at .1  3.7
\put {\tiny $\left(\begin{smallmatrix} -1 & 0 & 0 \\ 0 & 1 & 0 \\ 0 & 0 & 1 \end{smallmatrix} \right)$} [r] at -4.45 .4

\setshadegrid span <1pt>
\hshade 0 0 2   2 0 0 /  
\endpicture}
}
\end{center}
\caption{The planes $\mathbb{R}^{2}_{(\alpha,0,0)}$ and $\mathbb{R}^{2}_{(0,\alpha,0)}$ appear on the
top and bottom (respectively). The acting group is $\Gamma(\ell)$ (for appropriate $\ell$) and 
$\Gamma = \Gamma_{4}$.}
\label{figure:plane4,23}
\end{figure}

\begin{proof}
We must consider the planes $\mathbb{R}^{2}_{(0,0,\alpha)}$,
$\mathbb{R}^{2}_{(\alpha,0,0)}$, and $\mathbb{R}^{2}_{(0,\alpha,0)}$. These three cases pose no
problems that we haven't encountered before, so we will give an explicit argument for the plane
$\mathbb{R}^{2}_{(0,0,\alpha)}$, and leave the remaining cases to the reader.

We note  that the image group $G$ is generated by reflections in the sides of the shaded rectangle, and is in
particular a Coxeter group. It follows that the four vertices are all in separate orbits modulo the action
of $\Gamma(0,0,\alpha)$. It is also easy to see that $|G_{\ell}| = 4$ for the lines $\ell$ in question. It is
completely straightforward to verify that $D_{2}' \leq \overline{\Gamma}_{(0,0,\alpha)}$, so equality must
hold by Lemma \ref{lemma:computess}(1). 

Next we consider the line $(0,1/2,\alpha)$. We easily check that 
$$ \left\langle \left( \begin{smallmatrix} -1 & 0 & 0 \\ 0 & 1 & 0 \\ 0 & 0 & 1 \end{smallmatrix} \right),
\left( \begin{smallmatrix} 0 \\ 1 \\ 0 \end{smallmatrix} \right) +
\left( \begin{smallmatrix} 1 & 0 & 0 \\ 0 & -1 & 0 \\ 0 & 0 & 1 \end{smallmatrix} \right) \right\rangle
\leq \overline{\Gamma}_{(0,1/2,\alpha)}.$$
Since the isometries above generate a group of order $4$, and $|\overline{\Gamma}_{(0,1/2,\alpha)}| \leq
|G_{(0,1/2,\alpha)}| = 4$ by Lemma \ref{lemma:computess}(1), equality is forced. 
The equality $\pi(\overline{\Gamma}_{(0,1/2,\alpha)}) = D_{2}'$ follows directly.

We claim that the remaining two vertices in this plane are negligible. We prove this fact for the vertex (line)
$(1/4,0,\alpha)$, the other case being similar.
Consider the isometry
$$ \gamma = \left( \begin{smallmatrix} 1/2 \\ 0 \\ 1/2 \end{smallmatrix} \right) +
\left( \begin{smallmatrix} -1 & 0 & 0 \\ 0 & 1 & 0 \\ 0 & 0 & 1 \end{smallmatrix} \right).$$
Clearly $g = \psi(\gamma) \in G_{(1/4,0,\alpha)}$. It is sufficient to show that
$g \notin \psi(\overline{\Gamma}_{(1/4,0,\alpha)})$; we will establish this by proving that there is no 
$k \in \mathrm{Ker} \, \psi$ such that $k_{\mid (1/4,0,\alpha)} = \gamma_{\mid (1/4,0,\alpha)}$. The
kernel is equal to 
$$ \left\langle \left( \begin{smallmatrix} 0 \\ 0 \\ 1 \end{smallmatrix} \right),
\left( \begin{smallmatrix} 1 & 0 & 0 \\ 0 & 1 & 0 \\ 0 & 0 & -1 \end{smallmatrix} \right) \right\rangle.$$
These generators act on $r(\alpha) = (1/4,0,\alpha)$ by sending $r(\alpha)$ to
$r(\alpha + 1)$ and $r(-\alpha)$ (respectively). The element $\gamma$ acts by the rule
$\gamma \cdot r(\alpha) = r(\alpha + 1/2)$. It follows that the required $k$ does not exist, so 
$\overline{\Gamma}_{(1/4,0,\alpha)}$ is negligible.
\end{proof}

\begin{theorem} \label{theorem:gamma5T}
For the group $\Gamma = \Gamma_{5}$, we can choose
$$ \mathcal{T}'' = \left\{
\left( \begin{smallmatrix} \alpha \\ \alpha \\ \alpha \end{smallmatrix} \right),
\left( \begin{smallmatrix} \alpha + 1/2 \\ \alpha - 1/2 \\  \alpha \end{smallmatrix} \right),
\left( \begin{smallmatrix} \alpha \\ -\alpha \\  0 \end{smallmatrix} \right), 
\left( \begin{smallmatrix} \alpha + 1/2 \\ -\alpha + 1/2 \\  1/2 \end{smallmatrix} \right),
\left( \begin{smallmatrix} \alpha \\ -2\alpha \\ \alpha \end{smallmatrix} \right), 
\left( \begin{smallmatrix} \alpha + 1/2 \\ -2\alpha + 1/2 \\ \alpha + 1/2 \end{smallmatrix} \right) \right\},$$
 where we have expressed the vertices (i.e., lines) in $\base$ in parametric form.

For the vertices $v \in \mathcal{T}''$, the strict stabilizer groups $\overline{\Gamma}_{v}$ satisfy
$$ \pi( \overline{\Gamma}_{v}) = D_{6}'', \quad  \langle A, D \rangle, \quad \langle D, E \rangle, 
\quad  \langle D, E \rangle, \quad \langle E, F \rangle,\text{ and }  \quad \langle E, F \rangle,$$
respectively, where 
$$ A = \left( \begin{smallmatrix} 0 & 1 & 0 \\ 1 & 0 & 0 \\ 0 & 0 & 1 \end{smallmatrix} \right),  \quad 
D = \frac{1}{3}\left( \begin{smallmatrix} 2 & -1 & 2 \\ -1 & 2 & 2 \\ 2 & 2 & -1 \end{smallmatrix} \right),  \quad 
E = \frac{1}{3}\left( \begin{smallmatrix} 1 & -2 & -2 \\ -2 & 1 & -2 \\ -2 & -2 & 1 \end{smallmatrix}\right), 
\text{ and } \quad
F = \left( \begin{smallmatrix} 0 & 0 & 1 \\ 0 & 1 & 0 \\ 1 & 0 & 0 \end{smallmatrix} \right)
 .$$
 The groups $\langle A, D \rangle$, $\langle D, E \rangle$, and $\langle E, F \rangle$ are isomorphic to $D_{2}$.
\end{theorem}

\begin{proof}
We will consider only the planes $\mathbb{R}^{2}_{(\alpha,\alpha,\alpha)}$ and 
$\mathbb{R}^{2}_{(\alpha,-\alpha,0)}$, leaving $\mathbb{R}^{2}_{(\alpha,-2\alpha,\alpha)}$ 
to the reader.

\begin{figure}[!p]
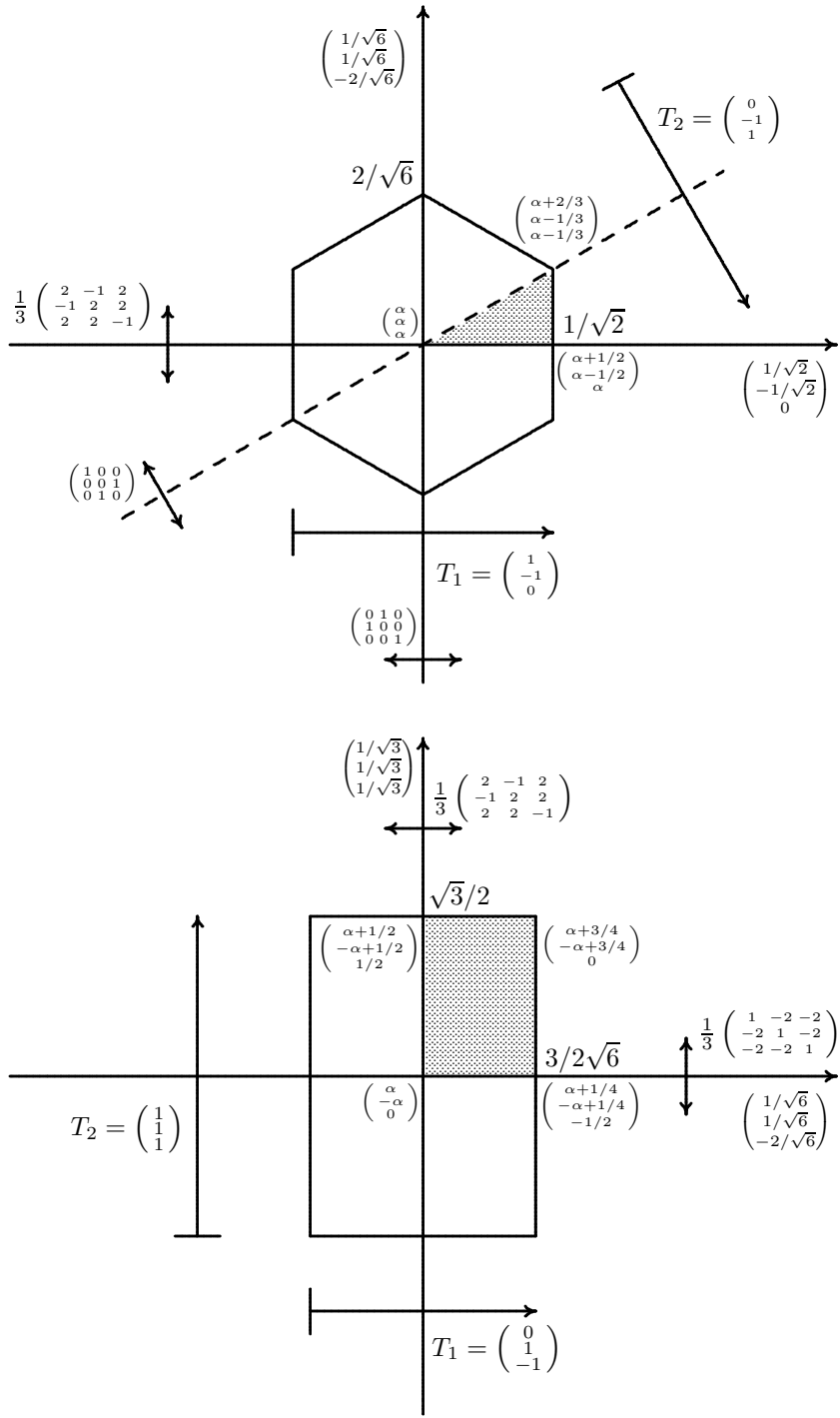

\begin{center}
\hbox{
\vbox{\beginpicture
\setcoordinatesystem units <1cm,1cm> point at -6 4.5
\setplotarea x from -6 to 6, y from -4.5 to 4.5
\linethickness=.7pt
%

 \def\myarrow{\arrow <4pt> [.2, 1]} 
\setplotsymbol ({\circle*{.4}})
\plotsymbolspacing=.3pt        
\myarrow from -5.5 0 to 5.5 0     
\myarrow from  0 -4.5  to  0 4.5 

\plot -1.732 1  0 2   1.732 1   1.732 -1   0 -2  -1.732 -1  -1.732 1 /  
\myarrow from -1.732 -2.5 to 1.732 -2.5
\plot -1.732 -2.2   -1.732 -2.8 /          
\put {$T_1 =  \tiny \left( \begin{smallmatrix} 1 \\ -1 \\ 0 \end{smallmatrix} \right)$} [t] at 1 -2.75      
\myarrow from  2.598 3.5   to  4.330 .5     
\plot 2.398 3.3946 2.798 3.6054 / 
\put {$T_2 = \tiny \left( \begin{smallmatrix} 0 \\ -1 \\ 1 \end{smallmatrix} \right)$} [l] at 3.1 3      
\myarrow from  -3.4 0 to -3.4 .5   
\myarrow from  -3.4 0 to -3.4 -.5

\myarrow from -3.464 -2 to -3.714 -1.567
\myarrow from -3.464 -2 to -3.214 -2.433

\myarrow from 0 -4.2 to .5 -4.2
\myarrow from 0 -4.2 to -.5 -4.2

\setdots <2pt>
\setdashes
\plot  -4 -2.31  4 2.31 /
\setsolid 
\put {\tiny $\left(\begin{matrix} 1/\sqrt{6} \\ 1/\sqrt{6} \\ -2/\sqrt{6}\end{matrix}\right)$} [r] at
-.2 3.8
\put {\tiny $\left(\begin{matrix} 1/\sqrt{2}  \\ -1/\sqrt{2} \\ 0 \end{matrix}\right)$} [t] at
4.8 -.2
\put {$1/\sqrt{2}$} [b] at 2.25 .07
\put {$2/\sqrt{6}$} [r] at -.15  2.25

\put { $\frac{1}{3} \tiny \left( \begin{smallmatrix}  2 & -1 & 2 \\ -1 & 2 & 2 \\ 2 & 2 & -1 \end{smallmatrix} \right)$} [r] at -3.6 .5
\put { \tiny $\left( \begin{smallmatrix} 1 & 0 & 0 \\ 0 & 0 & 1 \\ 0 & 1 & 0 \end{smallmatrix} \right)$} [b] at -4.35 -2.1
\put { \tiny $\left( \begin{smallmatrix} 0 & 1 & 0 \\ 1 & 0 & 0 \\ 0 & 0 & 1 \end{smallmatrix} \right)$} [t] at -.6 -3.5

\put { \tiny $\left( \begin{smallmatrix} \alpha \\ \alpha \\ \alpha \end{smallmatrix} \right)$} [r] at -.05 .3
\put { \tiny $\left( \begin{smallmatrix} \alpha + 1/2 \\ \alpha - 1/2 \\ \alpha \end{smallmatrix} \right)$} [t] at 2.25 -.1
\put { \tiny $\left( \begin{smallmatrix} \alpha + 2/3 \\ \alpha - 1/3 \\ \alpha - 1/3 \end{smallmatrix} \right)$} [b] at 1.7 1.35

\setshadegrid span <1pt>
\hshade 0 0 1.732   1 1.732 1.732 /  
\endpicture}
}

\vspace{20pt}


\hbox{
\vbox{\beginpicture
\setcoordinatesystem units <1cm,1cm> point at -6 4.5
\setplotarea x from -6 to 6, y from -4.5 to 4.5
\linethickness=.7pt

 \def\myarrow{\arrow <4pt> [.2, 1]} 
%
\setplotsymbol ({\circle*{.4}})
\plotsymbolspacing=.3pt        
\myarrow from -5.5 0 to 5.5 0     
\myarrow from   0 -4.5  to  0 4.5 

\plot -1.5 2.13   1.5 2.13  1.5 -2.13  -1.5 -2.13  -1.5 2.13 /  
\myarrow from -3 -2.13 to -3 2.13  
\plot -3.3 -2.13   -2.7 -2.13 /          
\put {$T_2 = \left( \begin{smallmatrix} 1 \\ 1 \\ 1 \end{smallmatrix} \right)$} [r] at -3.2 -.7       
\myarrow from  -1.5 -3.13 to 1.5 -3.13     
\plot -1.5 -2.83  -1.5 -3.43 /            
\put {$T_1 = \left( \begin{smallmatrix} 0 \\ 1 \\ -1 \end{smallmatrix} \right)$} [l] at .1 -3.63     

\put{\tiny $\left( \begin{smallmatrix}\alpha \\ -\alpha \\ 0 \end{smallmatrix} \right)$} [r] at -.05 -.35
\put{\tiny $\left( \begin{smallmatrix}\alpha+ 1/2 \\ -\alpha + 1/2 \\ 1/2 \end{smallmatrix} \right)$} [r] at -.05 1.7
\put{\tiny $\left( \begin{smallmatrix}\alpha + 1/4 \\ -\alpha + 1/4 \\ -1/2 \end{smallmatrix} \right)$} [r] at 2.9 -.4
\put{\tiny $\left( \begin{smallmatrix}\alpha+3/4 \\ -\alpha+3/4 \\ 0 \end{smallmatrix} \right)$} [r] at 2.9 1.75

\myarrow from  3.5 0 to 3.5 .5   
\myarrow from  3.5 0 to 3.5 -.5

\myarrow from  0 3.3 to  .5 3.3
\myarrow from  0 3.3 to  -.5 3.3

\setdots <2pt>
\put {\tiny $\left(\begin{matrix} 1/\sqrt{3} \\ 1/\sqrt{3} \\ 1/\sqrt{3} \end{matrix}\right)$} [r] at
-.1 4.1
\put {\tiny $\left(\begin{matrix} 1/\sqrt{6} \\ 1/\sqrt{6} \\ -2/\sqrt{6} \end{matrix}\right)$} [t] at
4.8 -.2
\put {$3/2\sqrt{6}$} [b] at 2.1 .05
\put {$\sqrt{3}/2$} [l] at .05 2.4

\put {$\frac{1}{3} \tiny \left(\begin{smallmatrix} 2 & -1 & 2 \\ -1 & 2 & 2 \\ 2 & 2 & -1 \end{smallmatrix} \right)$} [l] at .1  3.7
\put {$\frac{1}{3} \tiny \left(\begin{smallmatrix} 1 & -2 & -2 \\ -2 & 1 & -2 \\ -2 & -2 & 1 \end{smallmatrix} \right)$} [b] at 4.6 .25

\setshadegrid span <1pt>
\vshade 0 0 2.13   1.5 0 2.13 /  
\endpicture}
}
\end{center}
\caption{These are the planes $\mathbb{R}^{2}_{(\alpha,\alpha,\alpha)}$ and $\mathbb{R}^{2}_{(\alpha,-\alpha,0)}$
(respectively, from the top), with the associated actions from $\Gamma(\ell)$, where $\Gamma = \Gamma_{5}$.}
\label{figure:plane5,12} 
\end{figure}

\begin{figure}[!t]
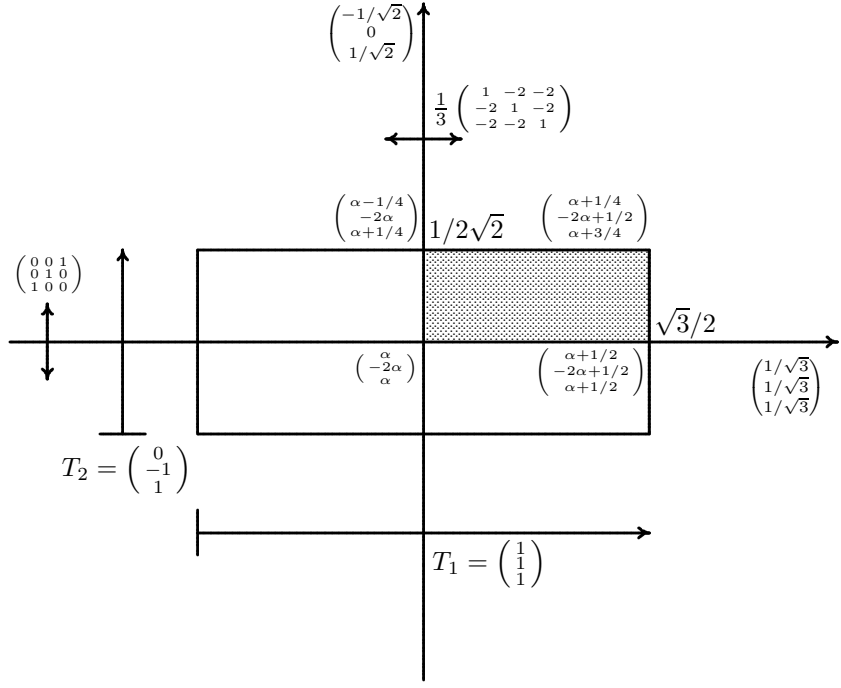

\begin{center}
\hbox{
\vbox{\beginpicture
\setcoordinatesystem units <1cm,1cm> point at -6 4.5
\setplotarea x from -6 to 6, y from -4.5 to 4.5
\linethickness=.7pt

 \def\myarrow{\arrow <4pt> [.2, 1]} 
%
\setplotsymbol ({\circle*{.4}})
\plotsymbolspacing=.3pt        
\myarrow from -5.5 0 to 5.5 0     
\myarrow from   0 -4.5  to  0 4.5 

\plot -3 1.224   3 1.224  3 -1.224  -3 -1.224  -3 1.224 /  
\myarrow from -4 -1.224 to -4 1.224  
\plot -4.3 -1.224  -3.7 -1.224 /          
\put {$T_2 = \left( \begin{smallmatrix} 0 \\ -1 \\ 1 \end{smallmatrix} \right)$} [r] at -3.1 -1.7      
\myarrow from  -3 -2.54 to 3 -2.54     
\plot -3 -2.84  -3 -2.24 /            
\put {$T_1 = \left( \begin{smallmatrix} 1 \\ 1 \\ 1 \end{smallmatrix} \right)$} [l] at .1 -2.95    

\put{\tiny $\left( \begin{smallmatrix}\alpha \\ -2\alpha \\ \alpha \end{smallmatrix} \right)$} [r] at -.1 -.35
\put{\tiny $\left( \begin{smallmatrix}\alpha- 1/4 \\ -2\alpha \\ \alpha + 1/4 \end{smallmatrix} \right)$} [r] at -.05 1.65
\put{\tiny $\left(\begin{smallmatrix}\alpha + 1/2 \\ -2\alpha + 1/2 \\ \alpha + 1/2 \end{smallmatrix}\right)$} [r] at 2.95 -.4
\put{\tiny $\left( \begin{smallmatrix}\alpha+1/4 \\ -2\alpha+1/2 \\ \alpha + 3/4 \end{smallmatrix} \right)$} [r] at 3 1.65

\myarrow from  -5 0 to -5 .5   
\myarrow from  -5 0 to -5 -.5

\myarrow from  0 2.7 to  .5 2.7
\myarrow from  0 2.7 to  -.5 2.7

\setdots <2pt>
\put {\tiny $\left(\begin{matrix} -1/\sqrt{2} \\ 0 \\ 1/\sqrt{2} \end{matrix}\right)$} [r] at
-.1 4.1
\put {\tiny $\left(\begin{matrix} 1/\sqrt{3} \\ 1/\sqrt{3} \\ 1/\sqrt{3} \end{matrix}\right)$} [t] at
4.8 -.2
\put {$\sqrt{3}/2$} [l] at 3.05 .25
\put {$1/2\sqrt{2}$} [l] at .05 1.5

\put {\tiny $\left(\begin{smallmatrix} 0 & 0 & 1 \\ 0 & 1 & 0 \\ 1 & 0 & 0 \end{smallmatrix} \right)$} [t] at -5 1.15
\put {$\frac{1}{3} \tiny \left(\begin{smallmatrix} 1 & -2 & -2 \\ -2 & 1 & -2 \\ -2 & -2 & 1 \end{smallmatrix} \right)$} [l] at .1 3.1

\setshadegrid span <1pt>
\vshade 0 0 1.224   3 0 1.224 /  
\endpicture}
}
\end{center}
\caption{This is the plane $\mathbb{R}^{2}_{(\alpha,-2\alpha,\alpha)}$ with its associated action
by $\Gamma(\alpha,-2\alpha,\alpha)$, where $\Gamma = \Gamma_{5}$.}
\label{figure:plane5,3}
\end{figure}

We note that $\mathbb{R}^{2}_{(\alpha,\alpha,\alpha)}$ is pictured at the top of Figure
\ref{figure:plane5,12} with its associated action by $\Gamma(\alpha,\alpha,\alpha)$. The image group $G$
is generated by reflections in the sides of the shaded triangle (whose interior angles are $\pi/6$, $\pi/3$, and $\pi/2$,
reading clockwise from the origin). In particular, $G$ is a Coxeter group, and all three vertices are distinct in the
quotient. It also follows that the stabilizers $G_{\ell}$ of the vertices have orders $12$, $6$, and $4$ (respectively, reading clockwise
from the origin). It follows from Lemma \ref{lemma:computess}(1) that the top corner of the fundamental domain 
has a negligible strict stabilizer (since the order divides $6$). It is also straightforward to check that 
$D_{6}'' \leq \overline{\Gamma}_{(\alpha,\alpha,\alpha)}$, and an application of Lemma \ref{lemma:computess}(1)
forces this inclusion to be an equality by counting. Next, we consider the vertex $(\alpha+1/2,\alpha-1/2,\alpha)$.
It is straightforward to check that $D$ and $T_{1}  + A$ are both in 
$\overline{\Gamma}_{(\alpha+1/2,\alpha-1/2,\alpha)}$, and that these isometries generate a group of order $4$.
It follows that  $\overline{\Gamma}_{(\alpha+1/2,\alpha-1/2,\alpha)} = \langle D, T_{1} + A \rangle$, so
that $\pi(\overline{\Gamma}_{(\alpha+1/2,\alpha-1/2,\alpha)}) = \langle D, A \rangle$, as claimed.

Finally, we consider the plane $\mathbb{R}^{2}_{(\alpha,-\alpha,0)}$, which is pictured at the bottom of
Figure \ref{figure:plane5,12}. We note that the image group $G$ is a Coxeter group, generated by reflections
in the sides of the shaded rectangle. It follows that all four corners are distinct in the quotient, and that the
stabilizers $G_{\ell}$ have order $4$, for each corner $\ell$ of the rectangle.

We first consider the two vertices along the vertical axis. Note first that $\langle D, E \rangle \leq 
\overline{\Gamma}_{(\alpha,-\alpha,0)}$, and equality is forced by Lemma \ref{lemma:computess}(1) since
the former group has order $4$. Similarly, one can show that
$\langle D, T_{2} + E \rangle \leq \overline{\Gamma}_{(\alpha+1/2,-\alpha+1/2,1/2)}$, and equality is again forced
by the same argument.

We consider the vertex $r(\alpha) = (\alpha + 1/4, -\alpha + 1/4, -1/2)$. We set $\gamma = T_{1} + D$, and note
that $\gamma \cdot r(\alpha) = r(\alpha - 1/2)$. It will follow that $g = \psi(\gamma) \notin 
\psi(\overline{\Gamma}_{(\alpha +1/4, -\alpha + 1/4, -1/2)})$ if there is no element of the kernel that acts
on $r(\alpha)$ in the same way. The kernel is the group
$$ \left\langle \left( \begin{smallmatrix} 0 & 1 & 0 \\ 1 & 0 & 0 \\ 0 & 0 & 1 \end{smallmatrix} \right),
\left( \begin{smallmatrix} 1 \\ -1 \\ 0 \end{smallmatrix} \right) \right\rangle;$$
these generators send $r(\alpha)$ to $r(-\alpha)$ and $r(\alpha + 1)$ (respectively). Therefore,
$$\overline{\Gamma}_{(\alpha+1/4, -\alpha +1/4, -1/2)}$$ is negligible. One argues that the remaining vertex is
negligible in the same way. 
\end{proof}

 \section{Maximal infinite virtually cyclic groups and cokernels of the relative assembly maps for $\vc_{\infty}$}\label{section:cokernels}
  
In this section, we will compute the contribution of the infinite virtually cyclic groups to the lower algebraic $K$-theory of the split $3$-dimensional crystallographic
groups. There are three steps. First, in Subsection \ref{subsection:passingtosubgroups}, 
we must determine the (non-negligible) strict stabilizers of lines $\ell$ relative to all $73$ split $3$-dimensional crystallographic groups. 
This amounts to finding the indexing set $\mathcal{T}''$ from the statement of Theorem \ref{theorem:splitting2}. In Section 
\ref{section:actionsonplanes}, we determined the indexing sets $\mathcal{T}''$ for
 the groups $\Gamma_{i}$ ($i = 1, \ldots, 5$). We will give
a procedure  (Procedure \ref{procedure:nilpart1}) that describes, for given split crystallographic groups $\Gamma$, $\Gamma'$ 
such that
$[\Gamma : \Gamma'] < \infty$, a suitable choice of $\mathcal{T}''$ for the group $\Gamma'$,
assuming that a choice of $\mathcal{T}''$ for $\Gamma$ is known to us. The procedure
is analogous to the one that we have already used in 
Section \ref{section:contributionoffinites} (see Procedure \ref{procedure:finitepart}), 
when we were determining vertex stabilizers in $E_{\fin}(\Gamma)$. We will also identify the strict stabilizer groups for all of the lines in 
question.

The next step is to rebuild the line stabilizer groups from their strict stabilizers. This is done in Subsection \ref{subsection:vcsubgroups} using 
Procedure \ref{procedure:computethestabilizer}. The final step is to determine the remaining factor from Theorem \ref{theorem:splitting2}:
$$\bigoplus_{ \widehat{\ell}  \in \mathcal{T}''}  H_n^{\Gamma_{\widehat{\ell}}}(E_{\fin}(\Gamma_{\widehat{\ell}}) \rightarrow  \ast;\;  \mathbb{KZ}^{-\infty}).$$
This will be done in Subsection \ref{subsection:cokernels}. The nontrivial cokernels are summarized in Table \ref{table:cokernels}.

\subsection{Passing to Subgroups} \label{subsection:passingtosubgroups}

In this subsection, we will list, for each split crystallographic group $\Gamma$,
a complete set $\mathcal{T}''$ of orbit representatives  of lines $\ell$ with non-negligible strict stabilizer 
groups. We will also describe the isomorphism types of the latter groups.
In Subsection \ref{subsection:finiteT''}, we saw such
lists for the groups $\Gamma_{i}$ ($i=1, \ldots, 5$). Our method of computing the strict 
stabilizer groups of lines for the remaining groups 
(Procedure \ref{procedure:nilpart1} below) involves passing to finite index
subgroups and generally follows the pattern of
Procedure \ref{procedure:finitepart}. In Subsection \ref{subsection:vcsubgroups}, we
will describe how to compute the stabilizer group of a line from its strict stabilizer.

\begin{procedure} \label{procedure:nilpart1}
Let $\Gamma = \langle L, H \rangle$ and $\Gamma' = \langle L', H' \rangle$ be split crystallographic groups, where $\Gamma'$ has finite index in $\Gamma$. Let $\mathcal{L}(\Gamma)$ denote a set of lines
$\ell \subseteq \mathbb{R}^{3}$ that contains exactly one line $\ell$ from each $\Gamma$-orbit
of lines $\widehat{\ell}$ with non-negligible strict stabilizer $\overline{\Gamma}_{\widehat{\ell}}$. 
We describe a procedure that computes an analogous set $\mathcal{L}(\Gamma')$ for
$\Gamma'$.
\begin{enumerate}
\item Let $T$ be a finite right transversal for $\Gamma'$ in $\Gamma$. First choose
one line from each $\Gamma'$-orbit of (unparametrized) lines in $T \cdot \mathcal{L}(\Gamma)$. Denote the
resulting set $\mathcal{L}'(\Gamma')$.
\item For each $\ell \in \mathcal{L}'(\Gamma')$, compute the strict stabilizer 
$\overline{\Gamma}'_{\ell}$, which can be done as follows. Assume that $\ell$ admits
the parametrization $r(\alpha) = t + \alpha v$. We have the equality $\pi(\overline{\Gamma}'_{\ell}) =
\{ h' \in H'_{v} \mid  t - h'(t) \in L' \}$, where $\pi: \Gamma' \rightarrow H'$ is the usual
projection to the point group. We recall that $\pi: \overline{\Gamma}'_{\ell} \rightarrow \pi(\overline{\Gamma}'_{\ell})$ is an isomorphism. 
\item We eliminate a line $\ell$ from the list $\mathcal{L}'(\Gamma')$ if its strict stabilizer
is \emph{negligible}. The resulting list is $\mathcal{L}(\Gamma')$.
\end{enumerate}
\end{procedure}

Next, we illustrate this procedure in a few representative examples, and summarize the results in Tables \ref{table:ssgamma1}, \ref{table:ssgamma2}, \ref{table:ssgamma3}, \ref{table:ssgamma4}, and \ref{table:ssgamma5}. We note that the matrices $A$, $B$, and $C$ from those tables are the same as the ones from Theorem \ref{theorem:gamma2T},
and $D$, $E$, and $F$ are the same as the ones from Theorem \ref{theorem:gamma5T}.

\begin{example} \label{sspass1}
We first let $\Gamma = \Gamma_{1}$ and $\Gamma' = 
\langle L, D_{4}^{+} \times (-1) \rangle$, where
$L$ is the standard cubical lattice. It is not difficult to see that $T=C_{3}^{+}$ is a 
right transversal for $\Gamma'$ in $\Gamma$. Recall that a set 
$\mathcal{L}(\Gamma)$ was computed in Theorem \ref{theorem:gamma1T};
the results of that computation are summarized in the top row of Table \ref{table:ssgamma1}. Following
Procedure \ref{procedure:nilpart1}, we apply $T$ to the five parametrized lines
in the top row of Table \ref{table:ssgamma1}. This results in a list of $15$ lines (which are obtained by cyclically permuting the entries of the original $5$, by the description of
$C_{3}^{+}$ in Theorem \ref{class+}). Since the permutation matrix which swaps the
first two coordinates is in $D_{4}^{+} \times (-1)$, it is easy to check that we can eliminate
$4$ lines from our list, since these $4$ are in the same $\Gamma'$-orbits as some 
of the other $11$. One possible choice of the remaining $11$ lines is as follows:
$$ \left(\begin{smallmatrix} \ast \\ \ast \\ \alpha \end{smallmatrix} \right), 
\left(\begin{smallmatrix} \alpha \\ \ast \\ \ast \end{smallmatrix} \right),
\left(\begin{smallmatrix} \alpha \\ \alpha \\ \ast \end{smallmatrix} \right),
\left(\begin{smallmatrix} \ast \\ \alpha \\ \alpha \end{smallmatrix} \right),
\left(\begin{smallmatrix} \alpha \\ 1/2 \\ 0 \end{smallmatrix} \right),
\left(\begin{smallmatrix}  1/2 \\ 0 \\ \alpha \end{smallmatrix} \right),
\left(\begin{smallmatrix}  0 \\ \alpha \\ 1/2 \end{smallmatrix} \right),
$$
where $\ast \in \{ 0, 1/2 \}$ (and must be one or the other, not both within one vector).
It is now straightforward to check that no two of the above lines are in the same 
$\Gamma'$-orbit. We may therefore let the above collection be $\mathcal{L}'(\Gamma')$.

Next, we compute the strict stabilizers of each of the new lines. Set 
$H = D_{4}^{+} \times (-1)$. For the vectors $v = \mathbf{x}, \mathbf{y}, \mathbf{z},
\mathbf{x} + \mathbf{y}, \mathbf{y} + \mathbf{z}$, we have 
$$H_{v} = D'_{2_{1}}, \quad D'_{2_{2}}, \quad D''_{4}, \quad \langle A, B \rangle, 
\quad \left\langle
\left(\begin{smallmatrix} -1 & 0 & 0 \\ 0 & 1 & 0 \\ 0 & 0 & 1 \end{smallmatrix}\right) 
\right\rangle,$$
respectively. The last group has order $2$, and it follows that each line
of the form $(\ast, \alpha, \alpha)$ is negligible. We note that the strict stabilizer $\overline{\Gamma}'_{(1/2,0,\alpha)}$ has order at most $4$, since
$|\overline{\Gamma}'_{(1/2,0,\alpha)}|$ divides $|H_{v}| = 8$, and $A \not \in \pi(\overline{\Gamma}'_{(1/2,0,\alpha)})$. One can check that
$D'_{2} \leq \pi(\overline{\Gamma}'_{(1/2,0,\alpha)})$, so equality must hold. It is straightforward to check that
each of the remaining parametrized lines $r(\alpha) = t + \alpha v$ satisfies 
$\pi(\overline{\Gamma}'_{\ell}) = H_{v}$. 

It now follows that the remaining $9$ lines form $\mathcal{L}(\Gamma')$. The strict 
stabilizer groups of these lines (as computed above) are described in the relevant line
of Table \ref{table:ssgamma1}.
\end{example}

\begin{table}
\renewcommand{\arraystretch}{1.3}
\begin{equation*}
\footnotesize\begin{array}{ | c | c | c | c | c | c | c |c |} \hline
H & D^{''}_4 & D^{'}_2 & D^{'}_{2_1} &  D^{'}_{2_2} & C^{+}_4 & \langle A, B \rangle & \langle A, C \rangle \\ \hline \hline
S^{+}_{4} \times (-1) & (0,  0, \alpha)  &( \frac{1}{2},  0, \alpha)  & & & & (\alpha, \alpha, 0) &  \\ 
& (\frac{1}{2}, \frac{1}{2}, \alpha) & & & &  & (\alpha, \alpha, \frac{1}{2}) &  \\ \hline
S^{+}_{4} & & & & &  (0, 0,\alpha)& & \\ 
& & & & & (\frac{1}{2}, \frac{1}{2}, \alpha) & &  \\ \hline
S'_{4} & & & & & & &(0,0, \alpha)  \\
& & & & & & &  (\frac{1}{2}, \frac{1}{2}, \alpha)  \\ \hline
A^{+}_{4} \times (-1)& & (0, 0,\alpha) & & & & &  \\
 & & (\frac{1}{2}, 0, \alpha) & & & & &  \\ 
 & & (0, \frac{1}{2}, \alpha) & & & & &  \\ 
 & & (\frac{1}{2},  \frac{1}{2}, \alpha) & & & & &  \\  \hline
D^{+}_{4} \times (-1) & (0, 0, \alpha) &  (\frac{1}{2},  0, \alpha)&  (\alpha, 0, 0)&   (0, \alpha, \frac{1}{2}) & & (\alpha, \alpha, 0) &\\
& (\frac{1}{2},  \frac{1}{2}, \alpha) & & (\alpha,  \frac{1}{2}, 0)& & &(\alpha, \alpha,  \frac{1}{2}) &\\
& & & (\alpha,  \frac{1}{2},  \frac{1}{2}) & & & &\\ \hline
D^{+}_{4} & & & & & (0, 0, \alpha) & &  \\
& & & & & (\frac{1}{2},  \frac{1}{2}, \alpha) & &  \\ \hline
C^{+}_{4} \times (-1) & & & & & (0, 0, \alpha) & &  \\
& & & & & (\frac{1}{2},  \frac{1}{2}, \alpha) & &   \\ \hline
C^{+}_{4}  & & & & & (0, 0, \alpha) & &  \\
& & & & & (\frac{1}{2},  \frac{1}{2}, \alpha) & & \\ \hline
D^{+}_{2} \times (-1) & &  (0, 0, \alpha)& (\alpha, 0,0) &(0, \alpha, 0) & & & \\ 
& &  (0,  \frac{1}{2}, \alpha)& (\alpha,  \frac{1}{2}, 0) &(0, \alpha,  \frac{1}{2}) & & & \\ 
& &  ( \frac{1}{2},  \frac{1}{2}, \alpha)& (\alpha, 0,  \frac{1}{2}) &( \frac{1}{2}, \alpha, 0) & & & \\ 
& &  ( \frac{1}{2},  0, \alpha)& (\alpha, \frac{1}{2} ,  \frac{1}{2}) &( \frac{1}{2}, \alpha, \frac{1}{2}) & & & \\  \hline
D'_{2} & &  (0, 0, \alpha)& & & & &  \\ 
& &  (\frac{1}{2}, 0, \alpha)& & & & &  \\ 
& &  (0, \frac{1}{2}, \alpha)& & & & &  \\ 
& &  (\frac{1}{2}, \frac{1}{2}, \alpha)& & & & &  \\ \hline
D'_{4}& & & & & & & (0, 0, \alpha)\\
& & & & & & & (\frac{1}{2}, \frac{1}{2}, \alpha) \\ \hline
\widehat{D}'_{4} & & (0, 0, \alpha)& & & & & \\ 
& &(\frac{1}{2}, \frac{1}{2}, \alpha)  & & & & & \\
& &(\frac{1}{2}, 0, \alpha) & & & & & \\ \hline
D''_{4} & (0,0,\alpha)& ( \frac{1}{2}, 0, \alpha)  & & & & & \\ 
& (\frac{1}{2}, \frac{1}{2}, \alpha) & & & & & & \\  \hline
 \end{array}
 \end{equation*}
\caption{The entries in the table are parametrized lines; if a line $\ell$
appears in the row and column labelled (respectively) $H_{1}$ and $H_{2}$, then $\pi(\overline{\Gamma}_{\ell}) = H_{2}$, where $\Gamma = \langle L, H_{1} \rangle$ and $L$
is the standard cubical lattice.}
\label{table:ssgamma1}
\end{table}

\begin{table}
\footnotesize
\renewcommand{\arraystretch}{1.3}
\begin{equation*}
\begin{array}{ | c | c | c | c | c | c | c |c |} \hline
H & D^{''}_4 & D^{'}_2 & D^{'}_{2_1} &  D^{'}_{2_2} & C^{+}_4 & \langle A, B \rangle & \langle A, C \rangle\\ \hline \hline
S^{+}_{4} \times (-1) & (0,  0, \alpha)  &( \frac{1}{2},  0, \alpha)  & & & & (\alpha, \alpha, 0) & \\ \hline
S'_{4} & & & & & & &(0,0, \alpha)  \\ \hline
A^{+}_{4} \times (-1)& & (0, 0,\alpha) & & &  &  & \\
& & (\frac{1}{2}, 0, \alpha) & & & & &  \\ \hline
D''_{4} & (0,0,\alpha)& ( \frac{1}{2}, 0, \alpha)  & & & & & \\  \hline
S^{+}_{4} & & & & & (0, 0, \alpha) & &  \\ \hline
D^{+}_{4} \times (-1) & (0, 0, \alpha) &  (\frac{1}{2},  0, \alpha)& &   (0, \alpha, \frac{1}{2}) & & (\alpha, \alpha, 0) & \\
&& & & (0, \alpha, 0) & & & \\ \hline
D^{+}_{4} & & & & & (0, 0, \alpha) & &  \\ \hline
C^{+}_{4} \times (-1) & & & & & (0, 0, \alpha) & &  \\ \hline
C^{+}_{4}  & & & & & (0, 0, \alpha) & &  \\ \hline
D^{+}_{2} \times (-1) & &  (0, 0, \alpha)& (\alpha, 0,0) &(0, \alpha, 0) & & & \\ 
& &  ( \frac{1}{2}, 0, \alpha)& (\alpha,  \frac{1}{2}, 0) &(0, \alpha,  \frac{1}{2}) & & &\\ \hline
D'_{4}& & & & & & & (0, 0, \alpha) \\ \hline
\widehat{D}'_{4} & & (0, 0, \alpha)& & & & & \\ 
& &(\frac{1}{2}, 0, \alpha)  & & & & & \\ \hline
D'_{2} & &  (0, 0, \alpha)& & & & & \\ 
& &  (\frac{1}{2}, 0, \alpha)& & & & &\\ \hline
\end{array}
\end{equation*}
\caption{The entries in the table are parametrized lines; if a line $\ell$
appears in the row and column labelled (respectively) $H_{1}$ and $H_{2}$, then $\pi(\overline{\Gamma}_{\ell}) = H_{2}$, where $\Gamma = \langle L, H_{1} \rangle$ and $L$
is the lattice $\langle \frac{1}{2}(\mathbf{x} + \mathbf{y} + \mathbf{z}), \mathbf{y},
\mathbf{z} \rangle$.}
\label{table:ssgamma2}
\end{table}

\begin{table}
\footnotesize
\renewcommand{\arraystretch}{1.3}
\begin{equation*}
\begin{array}{ | c | c | c | c | c | c | c |c |} \hline
H & D^{''}_4 & D^{'}_2 & D^{'}_{2_1} &  D^{'}_{2_2} & C^{+}_4 & \langle A, B \rangle & \langle A, C \rangle \\ \hline \hline
S^{+}_{4} \times (-1) & (0,  0, \alpha)  &  & & & & (\alpha, \alpha, 0) & (\frac{1}{4}, \frac{1}{4}, \alpha) \\ 
& & & & &  & (\alpha, \alpha, \frac{1}{2}) &  \\ \hline
S^{+}_{4} & & & & &  (0, 0,\alpha)& & \\ \hline
S'_{4} & & & & & & &(0,0, \alpha)  \\
& & & & & & &  (\frac{1}{4}, \frac{1}{4}, \alpha)  \\ \hline
A^{+}_{4} \times (-1)& & (0, 0,\alpha) & & & & &  \\ \hline
D'_{2} & &  (0, 0, \alpha)& & & & &  \\  \hline
D^{+}_{2} \times (-1) & &  (0, 0, \alpha)& (\alpha, 0,0) &(0, \alpha, 0) & & & \\ \hline
\end{array}
\end{equation*}
\caption{The entries in the table are parametrized lines; if a line $\ell$
appears in the row and column labelled (respectively) $H_{1}$ and $H_{2}$, then $\pi(\overline{\Gamma}_{\ell}) = H_{2}$, where $\Gamma = \langle L, H_{1} \rangle$ and $L$
is the lattice $\langle \frac{1}{2}(\mathbf{x} + \mathbf{y}),  \frac{1}{2}(\mathbf{x} + \mathbf{z}),
\frac{1}{2}(\mathbf{y} + \mathbf{z}) \rangle$.}
\label{table:ssgamma3}
\end{table}

\begin{table}
\footnotesize
\renewcommand{\arraystretch}{1.3}
\begin{equation*}
\begin{array}{ | c | c | c | c |} \hline
H & D^{'}_2 & D^{'}_{2_1} &  D^{'}_{2_2}  \\ \hline \hline
D^{+}_{2} \times (-1) &   (0, 0, \alpha)& (\alpha, 0,0) &(0, \alpha, 0)  \\ 
&    (0, \frac{1}{2}, \alpha)& (\alpha, \frac{1}{2},0) &(0, \alpha, \frac{1}{2})  \\ \hline
D'_{2} &   (0, 0, \alpha)& &  \\  
& (0,\frac{1}{2},\alpha) &  & \\ \hline 
D'_{2_{2}}  & & & (0, \alpha, 0)  \\
& & &  (0, \alpha, \frac{1}{2})  \\ \hline 
\end{array}
\end{equation*}
\caption{The entries in the table are parametrized lines; if a line $\ell$
appears in the row and column labelled (respectively) $H_{1}$ and $H_{2}$, then $\pi(\overline{\Gamma}_{\ell}) = H_{2}$, where $\Gamma = \langle L, H_{1} \rangle$ and $L$
is the lattice $\langle \frac{1}{2}(\mathbf{x} + \mathbf{z}),  \mathbf{y},
 \mathbf{z} \rangle$.}
\label{table:ssgamma4}
\end{table}

\begin{table}
\footnotesize
\renewcommand{\arraystretch}{1.3}
\[
\begin{array}{ | c | c | c | c | c |} \hline
H & D^{''}_6 & \langle A, D \rangle & \langle D, E \rangle & \langle E, F \rangle\\ \hline \hline
D^{+}_{6} \times (-1) & (\alpha, \alpha, \alpha)&  (\alpha +\frac{1}{2}, \alpha -\frac{1}{2}, \alpha) & (\alpha, -\alpha, 0) &(\alpha, -2\alpha, \alpha)\\ 
& & & (\alpha+\frac{1}{2}, -\alpha+\frac{1}{2}, \frac{1}{2}) &(\alpha+\frac{1}{2}, -2\alpha+\frac{1}{2}, \alpha+ \frac{1}{2}) \\ \hline
D^{''}_6& (\alpha, \alpha, \alpha) & (\alpha+ \frac{1}{2}, \alpha-\frac{1}{2}, \alpha) & & \\ \hline
D^{'}_6&  & & & (\alpha, -2\alpha, \alpha) \\ 
& &  & & (\alpha +\frac{1}{2}, -2\alpha +\frac{1}{2}, \alpha+\frac{1}{2}) \\ \hline
\widehat{D}^{'}_6&  & &  (\alpha, -\alpha, 0) &\\ 
& &  &  (\alpha +\frac{1}{2}, -\alpha +\frac{1}{2}, \frac{1}{2}) & \\ \hline
\end{array}
\]
\caption{The entries in the table are parametrized lines; if a line $\ell$
appears in the row and column labelled (respectively) $H_{1}$ and $H_{2}$, then $\pi(\overline{\Gamma}_{\ell}) = H_{2}$, where $\Gamma = \langle L, H_{1} \rangle$ and $L$
is the lattice $\langle \mathbf{v}_{1}, \mathbf{v}_{2}, \mathbf{v}_{3} \rangle$.}
\label{table:ssgamma5}
\end{table}

\begin{example} \label{sspass2}
Now we let $\Gamma = \langle L, D_{4}^{+} \times (-1)\rangle$ and 
$\Gamma' = \langle L, D'_{4} \rangle$, where $L$ is still the standard cubical lattice. 
We can take the group generated by the antipodal
map as $T$. Applying the antipodal map to the nine lines of $\mathcal{L}(\Gamma)$
from Example \ref{sspass1}, 
we get six new ones (three of the lines are invariant under the antipodal map). It is
straightforward to check that these six new lines are all in $\Gamma'$-orbits of the
original nine. For instance, $(-\alpha, -1/2, 0)$ (one of the new lines) is clearly
in the $\Gamma'$-orbit of $(-\alpha, 1/2, 0) = (\alpha, 1/2, 0) \in \mathcal{L}(\Gamma)$.
The five remaining lines are redundant by similar (easy) calculations. It follows that we can 
take $\mathcal{L}'(\Gamma') = \mathcal{L}(\Gamma)$.

Next, we compute the strict stabilizer $\overline{\Gamma}'_{\ell}$ of each 
$\ell \in \mathcal{L}'(\Gamma')$. Recall that each element $M \in D'_{4}$ can be factored
as $M= SP$, where $P$ is a permutation matrix that fixes the last coordinate and $S$
is a diagonal matrix with an even number of $-1$s on the diagonal, the remaining diagonal
entries being $1$. We let $H = D'_{4}$. We need to find the stabilizer groups $H_{v}$
for $v = \mathbf{x}, \mathbf{y}, \mathbf{z}, \mathbf{x} + \mathbf{y}$; by the above description
of $D'_{4}$, these groups are
$$\left\langle \left(\begin{smallmatrix} 1 & 0 & 0 \\ 0 & -1 & 0 \\ 0 & 0 & -1 \end{smallmatrix}
\right) \right\rangle, \quad 
\left\langle \left(\begin{smallmatrix} -1 & 0 & 0 \\ 0 & 1 & 0 \\ 0 & 0 & -1 \end{smallmatrix}
\right) \right\rangle, \quad
\left\langle \left(\begin{smallmatrix} 0 & 1 & 0 \\ 1 & 0 & 0 \\ 0 & 0 & 1 \end{smallmatrix}
\right),  \left(\begin{smallmatrix} -1 & 0 & 0 \\ 0 & -1 & 0 \\ 0 & 0 & 1 \end{smallmatrix}
\right)\right\rangle, \quad
\left\langle \left(\begin{smallmatrix} 0 & 1 & 0 \\ 1 & 0 & 0 \\ 0 & 0 & 1 \end{smallmatrix}
\right) \right\rangle,$$
respectively. It follows directly that all of the lines except possibly
$$(0, 0, \alpha), \quad (1/2, 1/2, \alpha), \quad (1/2, 0, \alpha)$$
are negligible. It is straightforward to check that the first two lines $\ell$
satisfy $\pi(\overline{\Gamma}'_{\ell}) = H_{\mathbf{z}}$. For the final line $\ell$,
we have $|\pi(\overline{\Gamma}'_{\ell})| = 2$, so the strict stabilizer of $\ell$ is 
negligible.

This implies that $\mathcal{L}(\Gamma') = \{(0,0,\alpha), (1/2,1/2,\alpha) \}$. The strict
stabilizer groups of the latter lines are recorded in Table \ref{table:ssgamma1}.
\end{example}

\begin{example} \label{sspass3}
Now we let $\Gamma = \langle L, D_{4}^{+} \times (-1)\rangle$ and 
$\Gamma' = \langle L, \widehat{D}'_{4} \rangle$, where $L$ is the standard cubical lattice. 
We can again choose the group generated by the antipodal map as $T$. Exactly the
same calculation as in Example \ref{sspass2} shows that we arrive at the same
set of nine lines $\mathcal{L'}(\Gamma')$.

Set $H = \widehat{D}'_{4}$.  Recall that
$$\widehat{D}'_{4} = \left\langle \left(\begin{smallmatrix} 0&1&0 \\1&0&0\\0&0&-1 \end{smallmatrix}\right), \left(\begin{smallmatrix} 1&0&0 \\0&-1&0\\0&0&1 \end{smallmatrix}\right) \right\rangle.$$
One can check (by listing all of the matrices in $\widehat{D}'_{4}$), that $D'_{2} \leq \widehat{D}'_{4}$, and that the remaining matrices can all be expressed in the form
$M = SP$, where $P$ is the permutation matrix that interchanges the first two coordinates,
and $S$ is a diagonal matrix with $1$s and $-1$s down the diagonal, but with a $-1$
in the bottom corner. With this description, it is easy to check that
 the stabilizer groups $H_{v}$
for $v = \mathbf{x}, \mathbf{y}, \mathbf{z}, \mathbf{x} + \mathbf{y}$ satisfy
$$H_{v} = \left\langle \left(\begin{smallmatrix} 1&0&0\\0&-1&0\\0&0&1 \end{smallmatrix}
\right) \right\rangle, \quad 
\left\langle \left(\begin{smallmatrix} -1&0&0\\0&1&0\\0&0&1 \end{smallmatrix}
\right) \right\rangle, \quad 
D'_{2}, \quad
\left\langle \left(\begin{smallmatrix} 0&1&0\\1&0&0\\0&0&-1 \end{smallmatrix}
\right) \right\rangle,
$$
respectively. It follows directly that a line $\ell \in \mathcal{L}'(\Gamma')$ has 
a negligible strict stabilizer group unless it has $\mathbf{z}$ as a tangent vector.
Thus, 
$$\mathcal{L}(\Gamma') = \{ (0,0,\alpha), (1/2,1/2,\alpha), (1/2,0,\alpha)\}.$$
Finally, an easy check shows that $\pi(\overline{\Gamma}'_{\ell}) = D'_{2}$ for each 
of the latter three lines, confirming the entries in the relevant row of
Table \ref{table:ssgamma1}.
\end{example}

\subsection{Reconstructing $\Gamma_{\ell}$ from $\overline{\Gamma}_{\ell}$} \label{subsection:vcsubgroups}

Let $\Gamma = \langle L, H \rangle$.
Assume that $\ell$ is one of the parametrized lines from Table \ref{table:ssgamma1},
\ref{table:ssgamma2}, \ref{table:ssgamma3}, \ref{table:ssgamma4}, or \ref{table:ssgamma5}.
We now describe a procedure for computing the
stabilizer group $\Gamma_{\ell}$ of $\ell$ from its strict stabilizer 
$\overline{\Gamma}_{\ell}$ (which we computed in 
Subsection \ref{subsection:passingtosubgroups}). 

\begin{procedure} \label{procedure:computethestabilizer}
We may assume that $\ell$ has a
parametrization of the form $r(\alpha) = \frac{1}{2}t + \alpha v$, where $t, v \in L$,
$t \perp v$, and $\beta v \notin L$ for $\beta \in (0,1)$, since
all of the parametrized lines in our tables have this form. 

\begin{enumerate}
\item Find the smallest positive 
constant $C$ such that there is $\gamma_{T} \in \Gamma$ with
the property that $\gamma_{T} \cdot r(\alpha) = r(\alpha + C)$. (The isometry $\gamma_{T}$ is the translation
that acts on $\ell$ with minimal translation length.) 
A smallest $C$ always exists, and we can find both $C$ and 
$\gamma_{T}$ as follows:
\begin{enumerate}
\item Compute the set
$$ A_{v, L} = \left\{ \frac {v_{1} \cdot v} {v \cdot v} \mid v_{1} \in L \right\}.$$
This set is a finitely generated subgroup of the rational numbers, and
therefore cyclic. Let $q > 0$ be a generator. Note that $1 \in A_{v, L}$, so $q$ is the 
reciprocal of a positive integer.
\item Determine the smallest number $C$ 
in the sequence $q$, $2q$, $3q$, $\ldots$, such that there is $h \in H$ satisfying: $h(v) = v$ and $\frac{1}{2}(t-h(t)) + Cv \in L$. 
\item We set 
$$\gamma_{T} = \frac{1}{2}(t-h(t)) + Cv + h.$$
\end{enumerate}
We note that this entire step becomes trivial to perform
 if $t=0$ or if $A_{v, L}$ is the set of integers,
for then $C=1$ is forced and we can take $\gamma_{T} = v$ (i.e., the translation by $v$). In either of these cases, we will
call this step \emph{easy}.
\item Determine whether some isometry $\gamma_{R} \in \Gamma$ acts as a reflection
on the parametrized line $r(\alpha)$; that is, determine whether there is some 
isometry $\gamma_{R} \in \Gamma$ and $D \in \mathbb{R}$ such that
$\gamma_{R} \cdot r(\alpha) = r(D - \alpha)$. 
 This amounts to doing the following. Fix a basis $b_{1}, b_{2}, b_{3}$ for $L$ as a free
abelian group. For each $h \in H$ such that $h(v) = -v$, determine whether the
equation
$$ c_{1}b_{1} + c_{2}b_{2} + c_{3}b_{3} = \frac{1}{2}(t - h(t)) + Dv$$
has a solution for integers $c_{i}$ ($i =1, 2, 3$) and $D \in \mathbb{R}$. If the solution exists
for some $h \in H$,  then we
set 
$$ \gamma_{R} = \left( \frac{1}{2}(t -h(t)) + Dv \right) + h.$$

We note that this step becomes trivial to perform if there is no $h \in H$ such that 
$h(v) = -v$ (in which case there can exist no $\gamma_{R}$), or if there is $h \in H$
such that $h(v) = -v$ and $h(t) = \pm t$ (in which case we can set $D=0$ and let
$\gamma_{R} = h$ or $t + h$, respectively). Note, in particular, that the latter conditions are always satisfied if
$(-1) \in H$. We call all of these cases \emph{easy}.

\item There are two possible outcomes:
\begin{enumerate}
\item If there is no reflection $\gamma_{R}$ as above, then 
$\Gamma_{\ell} = \langle \overline{\Gamma}_{\ell}, \gamma_{T} \rangle$, and we have the
isomorphism
$$ \Gamma_{\ell} \cong \pi(\overline{\Gamma}_{\ell}) \rtimes_{\phi} \mathbb{Z},$$
where the action is conjugation by $\pi(\gamma_{T})$ and $\pi$ is the usual projection
into the point group $H$.
\item If there is a reflection $\gamma_{R}$, then we get a free product with amalgamation:
$$\Gamma_{\ell} = \langle \overline{\Gamma}_{\ell}, \gamma_{R} \rangle 
\ast_{\overline{\Gamma}_{\ell}} \langle \overline{\Gamma}_{\ell}, 
\gamma_{T}\gamma_{R} \rangle. $$
\end{enumerate}
We can determine the abstract isomorphism type of $\Gamma_{\ell}$ by applying 
the projection $\pi$ to the factors and the amalgamated subgroup.

\end{enumerate}
\end{procedure}

\begin{lemma} \label{lemma:valid}
Procedure \ref{procedure:computethestabilizer} is valid, and its steps can
be performed algorithmically.
\end{lemma}

\begin{proof}
We assume that $\Gamma = \langle L, H \rangle$, $\ell \subseteq \mathbb{R}^{3}$, and that $\ell$
admits a parametrization $r(\alpha) = \frac{1}{2}t + \alpha v$, where $t,v \in L$, $t \perp v$, and $\beta v \notin L$
for $\beta \in (0,1)$. 

Consider step (1) from Procedure \ref{procedure:computethestabilizer}. The set $A_{v,L}$ is an additive subgroup
of $\mathbb{R}$ because of the bilinearity of the dot product. It is a set of rational numbers since each possible
lattice $L$ is contained in $\mathbb{Q}^{3}$, and so each member of $A_{v,L}$ is a quotient of two rational numbers.
It also follows from bilinearity that $A_{v,L}$ is generated by 
$$ \frac{b_{1} \cdot v}{v \cdot v}, \quad \frac{b_{2} \cdot v}{v \cdot v}, \quad \frac{b_{3} \cdot v}{v \cdot v},$$
if $\{ b_{1}, b_{2}, b_{3} \}$ generates $L$ as an abelian group. It follows that $A_{v,L}$ is a finitely generated 
additive subgroup of $\mathbb{Q}$, and therefore cyclic. It is clear that $1 \in A_{v,L}$, so the generator is the
reciprocal of some positive integer, as claimed. We note that it is straightforward to find the generator $q$
algorithmically, given 
the above generating set. 

It is already clear that Step 1(b) can be performed algorithmically; the procedure is guaranteed to stop since 
$1 \in A_{v,L}$ and the given equations can always be satisfied with $h=1$ and $C=1$.
Step 1(c) poses no problems.

We need to show that if $\gamma \in \Gamma$ acts on $\ell$ by the rule $\gamma \cdot r(\alpha) = r(\alpha + C')$,
then $C' \in A_{v,L}$. (It is easy to show that $\gamma_{T}$, as defined in Procedure \ref{procedure:computethestabilizer}, satisfies $\gamma_{T} \cdot r(\alpha) = r(\alpha + C)$.) Assume that
$\gamma \cdot r(\alpha) = r(\alpha + C')$; we let $\gamma = \widehat{t} + h$, where $\widehat{t} \in L$ and
$h \in H$. Note that we must have $h(v) = v$ if the line $\ell$ is to be acted on by a translation.
\begin{align*}
\gamma \cdot r(\alpha) &= (\widehat{t} + h) (\frac{1}{2}t + \alpha v) \\
&= \widehat{t} + \frac{1}{2} h(t) + \alpha v.
\end{align*}
Setting the latter expression equal to $r(\alpha + C')$, we get
\begin{align*}
\widehat{t} + \frac{1}{2} h(t) + \alpha v = \frac{1}{2}t + C'v + \alpha v  \quad &\Rightarrow \quad 
\widehat{t} = \frac{1}{2}( t - h(t)) + C'v.
\end{align*}
Taking the dot product with $v$ on both sides gives the equality $\widehat{t} \cdot v = C' (v \cdot v)$, since
both $t$ and $h(t)$ are perpendicular to $v$. It follows that $C' \in A_{v,L}$, as claimed. Therefore, Step 1 is valid,
and can be performed algorithmically.

Next we consider Step 2. It is straightforward to check that the isometry $\gamma_{R}$ (as defined in
Procedure \ref{procedure:computethestabilizer}) acts on $\ell$ by the rule
$\gamma_{R} \cdot r(\alpha) = r(D - \alpha)$, and that $\gamma_{R} \in \Gamma$ if 
$\frac{1}{2}(t - h(t)) + Dv \in L$. It is now enough to solve
the equation
$$ c_{1}b_{1} + c_{2}b_{2} + c_{3}b_{3} = \frac{1}{2}(t - h(t)) + Dv,$$
where the $c_{i}$ ($i=1,2,3$) are integers and $D \in \mathbb{R}$, by an algorithm (or to show, by the same algorithm, that no solution exists). This is also straightforward:
we represent the right side of the equation as an $\mathbb{R}$-linear combination of the basis elements $b_{1}$,
$b_{2}$, $b_{3}$, treating $D$ as an independent variable. That is,
$$ \frac{1}{2}(t - h(t)) + Dv = C_{1}(D) b_{1} + C_{2}(D) b_{2} + C_{3}(D) b_{3},$$
where $C_{i}(D)$ is a linear function of $D$, for $i=1,2,3$. The problem of solving the original equation amounts
to the problem of solving the system of congruences $C_{i}(D) = 0$ mod $\mathbb{Z}$, which can clearly be done
algorithmically. This shows that Step 2 is valid, and can be done algorithmically.

Finally, we note that the validity of Step 3 follows from the Bass-Serre theory of groups acting on trees \cite{Se80}.
\end{proof}

We will now show how to apply Procedure \ref{procedure:computethestabilizer} in a few representative cases.
The results of all of the calculations are summarized in Tables \ref{table:vc1}, \ref{table:vc2}, \ref{table:vc3},
\ref{table:vc4}, and \ref{table:vc5}. 

\begin{example}
Let us first suppose that $L$ is the standard integral lattice, and $H \leq S^{+}_{4} \times (-1)$. Thus, 
$\Gamma = \langle L, H \rangle$
is one of the subgroups $\Gamma \leq \Gamma_{1}$, whose strict line stabilizers are recorded in 
Table \ref{table:ssgamma1}. Note that almost all of the parametrized lines in Table \ref{table:ssgamma1} 
have the property that $v \cdot v = 1$. (The four exceptions occur in the column headed $\langle A, B \rangle$.) In
all such cases, $A_{v,L}$ must be $\mathbb{Z}$, since the dot product $v_{1} \cdot v$ is an integer when $v_{1},
v \in \mathbb{Z}^{3} = L$. It follows that Step 1 of Procedure \ref{procedure:computethestabilizer} is easy
in all of these cases, and we can take $\gamma_{T} = v$.  

This leaves four more cases to be considered. The remaining lines have the
form $(\alpha, \alpha, 0)$ or $(\alpha, \alpha, 1/2)$. We note that Step 1 is still easy for the line $(\alpha, \alpha, 0)$
(with respect to both crystallographic groups, since $t = 0$), so we can again take $\gamma_{T} = v = (1,1,0)$. 

We consider the remaining lines. Suppose first that $H = S_{4}^{+} \times (-1)$ and 
$\ell$ has the parametrization $r(\alpha) = (\alpha, \alpha, 1/2)$. 
It follows that $t = (0,0,1)$ and $v=(1,1,0)$. An easy check
shows that $A_{v,L} = \frac{1}{2}\mathbb{Z}$, so that $q = 1/2$ is a generator. 
We must determine whether there is $h \in H$ such that
$h(v) = v$ and $\frac{1}{2}(t - h(t)) + \frac{1}{2}v \in L$. There are four signed permutation matrices
$h \in H$ such that $h(v) = v$:
$$ \left( \begin{smallmatrix} 1 & 0 & 0 \\ 0 & 1 & 0 \\ 0 & 0 & 1 \end{smallmatrix} \right), \quad
\left( \begin{smallmatrix} 0 & 1 & 0 \\ 1 & 0 & 0 \\ 0 & 0 & 1 \end{smallmatrix} \right), \quad
 \left( \begin{smallmatrix} 1 & 0 & 0 \\ 0 & 1 & 0 \\ 0 & 0 & -1 \end{smallmatrix} \right), \quad
\left( \begin{smallmatrix} 0 & 1 & 0 \\ 1 & 0 & 0 \\ 0 & 0 & -1 \end{smallmatrix} \right).$$
All of these matrices $h$ have the property that $\frac{1}{2} (t - h(t)) \in L$. In particular, we can have
$\frac{1}{2}(t - h(t)) + \frac{1}{2}v \in L$ if and only if $\frac{1}{2}v \in L$. Since the latter inclusion is 
false, there is no $h$ with the required properties.

\begin{table} 
\renewcommand{\arraystretch}{1.5}
\[
\begin{array}{|c|c|}
\hline H &  \vc_{\infty}  \\ \hline \hline
S^{+}_{4} \times (-1)  & D_4\times D_\infty  \;(\text{twice}),\;  D_2\times D_\infty
 \; (\text{three times}) \\  \hline
S^{+}_{4} &   D_4 *_{C_4} D_4 \; (\text{twice}) \\ \hline
S'_{4} & D_4 *_{D_2}D_4 \;(\text{twice}) \\ \hline
A^{+}_{4} \times (-1)& D_2\times D_\infty (\text{four times})\\ \hline
D''_{4} & D_4 \times \mathbb Z  \;(\text{twice}), \;  D_2 \times \mathbb Z  \\  \hline
D^{+}_{4} \times (-1) &  D_4\times D_\infty \;(\text{twice}),  D_2\times D_\infty \;(\text{seven times}) \\ \hline
D^{+}_{4} &  D_4 *_{C_4} D_4  \; (\text{twice}) \\ \hline
D^{+}_{2} \times (-1)& D_2\times D_\infty \;(\text{twelve times}) \\ \hline
C^{+}_{4} \times (-1) &  C_4 \times D_\infty \; (\text{twice})\\ \hline
C^{+}_{4} & C_{4} \times \mathbb{Z} \; (\text{twice}) \\ \hline
D'_{2} &  D_2 \times \mathbb Z \;(\text{four times}) \\ \hline
D'_{4}&   D_4 *_{D_2} D_4  \;(\text{twice}) \\ \hline
\widehat{D}'_{4}  & D_4 *_{D_2} D_4  \;(\text{twice}), \; D_2 \times \mathbb Z    \\ \hline
\end{array}
\]
\caption{The structure of $\vc_{\infty}$ subgroups of $\left\langle \langle \mathbf{x}, \mathbf{y}, \mathbf{z} \rangle,  H \right\rangle$.}
\label{table:vc1}
\end{table}

\begin{table}
\renewcommand{\arraystretch}{1.5}
\[
\begin{array}{|c|c|}
\hline H &  \vc_{\infty}  \\ \hline \hline
S^{+}_{4} \times (-1)  & D_4\times D_\infty ,\;  D_2\times D_\infty, (D_2 \times \mathbb Z/2) *_{D_2} D_4 \\  \hline
S^{+}_{4} &   D_4 *_{C_4} D_4 \\ \hline
S^{'}_{4} &   D_{4} \ast_{D_{2}} D_{4} \\ \hline
A^{+}_{4} \times (-1)& D_2\times D_\infty (\text{twice})\\ \hline
D''_{4} & D_4 \times \mathbb Z, \;  D_2 \rtimes \mathbb Z  \\  \hline
D^{+}_{4} \times (-1)  & D_4\times D_\infty ,\;  D_2\times D_\infty \;(\text{three times}), \; (D_2 \times \mathbb Z/2) *_{D_2} D_4 \\  \hline
D^{+}_{4} &  D_{4} \ast_{C_{4}} D_{4} \\ \hline
D^{+}_{2} \times (-1)& D_2\times D_\infty \;(\text{six times}) \\ \hline
C^{+}_{4} \times (-1) &  C_4 \times D_\infty \\ \hline
C^{+}_{4} & C_4 \times \mathbb Z \\ \hline
D'_{2} &  D_2 \times \mathbb Z \;(\text{twice}) \\ \hline
D'_{4}&   D_4 *_{D_2} D_4  \\ \hline
\widehat{D}'_{4}  & D_4 *_{D_2} D_4  \;(\text{twice})   \\ \hline
\end{array}
\]
\caption{The structure of $\vc_{\infty}$ subgroups of $\left\langle \langle \frac{1}{2}(\mathbf{x}+ \mathbf{y}+ \mathbf{z}), \mathbf{y}, \mathbf{z} \rangle, H \right\rangle$.}
\label{table:vc2}
\end{table}

\begin{table}
\renewcommand{\arraystretch}{1.5}
\[
\begin{array}{|c|c|}
\hline H &  \vc_{\infty}  \\ \hline
S^{+}_{4} \times (-1)  & D_4\times D_\infty ,\;  D_2\times D_\infty (\text{two times}), (D_2 \times \mathbb Z/2) *_{D_2} D_4 \\  \hline
S^{+}_{4} &   D_4 *_{C_4} D_4 \\ \hline
S^{'}_{4} & D_4 *_{D_2} D_4 \;(\text{twice}) \\ \hline
A^{+}_{4} \times (-1)& D_2\times D_\infty \\ \hline
D'_{2} &  D_2 \times \mathbb Z  \\ \hline
D^{+}_{2} \times (-1)& D_2\times D_\infty \;(\text{three times}) \\ \hline
\end{array}
\]
\caption{The structure of $\vc_{\infty}$ subgroups 
of $\left\langle \langle \frac{1}{2}(\mathbf{x}+ \mathbf{y}), (\mathbf{x} +\mathbf{z}), (\mathbf{y}+ \mathbf{z}) \rangle, H \right\rangle$.}
\label{table:vc3}
 \end{table}

\begin{table}
\renewcommand{\arraystretch}{1.5}
\[
\begin{array}{|c|c|}
\hline H &  \vc_{\infty}  \\ \hline \hline
D^{+}_{2} \times (-1)& D_2\times D_\infty \;(\text{six times}) \\ \hline
D'_{2} &  D_2 \times \mathbb Z  \;(\text{twice}) \\ \hline
D'_{2_{2}} &  D_2 \times \mathbb Z  \;(\text{twice}) \\ \hline
\end{array}
\]
\caption{The structure of $\vc_{\infty}$ subgroups of $\left\langle \langle \frac{1}{2}( \mathbf{x}+ \mathbf{z}), \mathbf{y}, \mathbf{z}\rangle, H \right\rangle$.}
\label{table:vc4}
\end{table}

\begin{table}
\renewcommand{\arraystretch}{1.5}
\[
\begin{array}{|c|c|}
\hline H &  \vc_{\infty}  \\ \hline \hline
D^{+}_{6} \times (-1)& D_2\times D_\infty \;(\text{five times}), \;\; D_6 \times D_\infty\\ \hline
D'_6 & D_2 \times \mathbb Z \;(\text{twice}) \\ \hline
\widehat{D}'_{6} &  D_2 \times \mathbb Z  \;(\text{twice}) \\ \hline
D^{''}_{6} &  D_2 \times \mathbb Z, \;\; D_6 \times \mathbb Z \\ \hline
\end{array}
\]
\caption{The structure of $\vc_{\infty}$ subgroups of $\left\langle \langle \mathbf{v}_{1}, \mathbf{v}_{2}, \mathbf{v}_{3} \rangle, H \right\rangle$.}
\label{table:vc5}
\end{table}

Now we must consider the next largest element of $A_{v,L}$, namely $1$. We ask for an $h \in H$ that satisfies the
conditions $h(v) = v$ and $\frac{1}{2}(t - h(t)) + v \in L$. It is clear that we can set $h = 1$, and we then 
set $\gamma_{T} = v$. This is the minimal translation to act on the line $(\alpha,\alpha, 1/2)$.

The case in which $H = D_{4}^{+} \times (-1)$ and $r(\alpha) = (\alpha, \alpha, 1/2)$ is similar, and we again
have $\gamma_{T} = v$. It follows that the minimal translation to act on $\ell$ is always $\gamma_{T} = v$, for 
all of the lines in Table \ref{table:ssgamma1}. 

Next we consider Step 2. Recall that this step is always easy if $(-1) \in H$ (see Procedure \ref{procedure:computethestabilizer}), and
we can set $\gamma_{R} = (-1)$ or $t + (-1)$. For instance, suppose
that $H = S_{4}^{+} \times (-1)$. Using the fact that $\pi(\gamma_{T}) = 1$ and $\pi(\gamma_{R}) = (-1)$ for
each of the five parametrized lines $r(\alpha)$ in the corresponding row of Table \ref{table:ssgamma1},
Step 3 of Procedure \ref{procedure:computethestabilizer} shows that the line stabilizers are
$$ (D_{4} \times \mathbb{Z}/2) \ast_{D_{4}} (D_{4} \times \mathbb{Z}/2)  \quad
\text{ and } \quad (D_{2} \times \mathbb{Z}/2) \ast_{D_{2}} (D_{2} \times \mathbb{Z}/2),$$
where the first group appears twice, and the second appears three times. 
(Here we record only the isomorphism types of the stabilizer groups.) The latter groups are isomorphic to
$D_{4} \times D_{\infty}$ and $D_{2} \times D_{\infty}$, respectively. This calculation is recorded in
the first line of Table \ref{table:vc1}.

In fact, Step 2 is often quite easy in practice when $L = \mathbb{Z}^{3}$, even if $(-1) \notin H$. 
For instance, assume that $H = S_{4}^{+}$. We note that for 
$$ h = \left( \begin{smallmatrix} 1 & 0 & 0 \\ 0 & -1 & 0 \\ 0 & 0 & -1 \end{smallmatrix} \right),$$
the isometries $\gamma_{1} = h$ and $\gamma_{2} = (0, 1, 0) + h$ act on the lines
$(0,0,\alpha)$ and $(1/2, 1/2, \alpha)$ (respectively) as reflections $\gamma_{R}$ with $D=0$. By our results 
from Step 1, we therefore have $\pi(\gamma_{T}) = 1$ and $\pi(\gamma_{R}) = h$ for both of these lines. It follows
from Step 3 and Table \ref{table:ssgamma1} that both line stabilizers have the form
$$ \langle C_{4}^{+}, h \rangle \ast_{C_{4}^{+}} \langle C_{4}^{+}, h \rangle.$$
Since $\langle C_{4}^{+}, h \rangle = D_{4}^{+}$, the above amalgam reduces to the one listed in the
second row of Table \ref{table:vc1}. (Note that this case is not ``easy" in the sense described in Procedure \ref{procedure:computethestabilizer}.)

In a few cases, Step 2 is ``easy'' because there is no $h \in H$ that reverses the
direction of $v$. This is the case for $H = D'_{2}$ and $H = D''_{4}$ in Table \ref{table:ssgamma1}, for example. 
\end{example}

\begin{example} \label{example:funny}
We will now run
Procedure \ref{procedure:computethestabilizer} with the group 
$\Gamma_{2}$ and the line $(1/2, 0, \alpha)$ as input (see Table \ref{table:ssgamma2}). (Thus,
$t = (1,0,0)$ and $v=(0,0,1)$.)

An easy calculation establishes that $A_{v,L} = \frac{1}{2} \mathbb{Z}$. We must determine whether
there is $h \in H$ such that $h(v) = v$ and $\frac{1}{2}(t - h(t)) + \frac{1}{2}v \in L$. One can check that these
conditions are satisfied by
$$ h = \left( \begin{smallmatrix} 0 & 1 & 0 \\ 1 & 0 & 0 \\ 0 & 0 & 1 \end{smallmatrix} \right).$$
It follows that a minimal translation $\gamma_{T}$ satisfies $\gamma_{T} \cdot r(\alpha) = r(\alpha + 1/2)$,
and that $\gamma_{T}$ can be chosen so that $\pi(\gamma_{T}) = h$. 

We turn to Step 2. This step is easy because $(-1) \in H$. We can take $\gamma_{R} = t + (-1)$ (and $D=0$). 

It follows that the stabilizer is isomorphic to 
$$ \left\langle D'_{2}, (-1) \right\rangle \ast_{D'_{2}} \left\langle D'_{2}, -h \right\rangle \cong
(D_{2} \times \mathbb{Z}/2) \ast_{D_{2}} D_{4}.$$
This stabilizer is recorded in the first line of Table \ref{table:vc2}.
\end{example}

\begin{example} \label{example:last}
Consider $\Gamma_{3}$ and the parametrized line $(1/4,1/4,\alpha)$ (see Table \ref{table:ssgamma3}).
Thus $t = (1/2,1/2,0)$ and $v=(0,0,1)$.  

It is straightforward to check that $A_{v,L} = \frac{1}{2} \mathbb{Z}$. We therefore begin by checking for
$h \in H$ that satisfy $h(v) = v$ and $\frac{1}{2}(t - h(t)) + \frac{1}{2}v \in L$. Note that
$D''_{4}$
is the stabilizer group $H_{v}$. Setting 
$$h = \left( \begin{smallmatrix} 1 & 0 & 0 \\ 0 & -1 & 0 \\ 0 & 0 & 1 \end{smallmatrix} \right),$$
we get a solution. It follows that there is a minimal translation 
$\gamma_{T}$ such that $\pi(\gamma_{T}) = h$. 

Step 2 is easy -- we simply set $\gamma_{R} = t + (-1)$ (and $D=0$). 

It follows that the stabilizer group has the form
$$ \left\langle A, C, (-1) \right\rangle \ast_{\langle A, C \rangle} \left\langle A, C, -h \right\rangle
\cong (D_{2} \times \mathbb{Z}/2) \ast_{D_{2}} D_{4}.$$
This calculation contributes the asymmetric amalgam in the first line of Table \ref{table:vc3}.
\end{example}

\subsection{Cokernels of relative assembly maps} \label{subsection:cokernels}

In view of Theorem \ref{theorem:splitting2}, we will need for our computations the cokernels of the relative assembly maps for the various maximal
infinite virtually cyclic subgroups  that were listed in Section \ref{subsection:vcsubgroups}. The following list contains
\emph{all} of the \emph{non-negligible}  maximal infinite virtually cyclic groups that
appear as subgroups of the 73 split $3$-dimensional crystallographic groups:

\vspace{10pt}

\noindent {\bf Maximal infinite virtually cyclic subgroups:} $C_4 \times \mathbb Z$, $C_4 \rtimes \mathbb Z$, $D_2 \times \mathbb Z$, $D_2 \rtimes \mathbb Z$,  $D_4 \times \mathbb Z$, $D_{6} \times \mathbb{Z}$, $D_4 \ast_{C_4} D_4 $, $D_4 \ast_{D_2} D_4$,   $(D_2 \times
\mathbb Z/2) \ast_{D_2} D_4$, $C_4 \times D_\infty$,  and $D_n \times D_{\infty}$ for $n=2, 4$, and $6.$

\vspace{10pt}

We first note that for the  groups $D_2 \times \mathbb Z$,  $(D_2 \times \mathbb Z/2) \ast_{D_2} D_4$,  $D_4 \ast_{D_2} D_4$,  and $D_n \times D_{\infty}$ for $n=2, 4$, the cokernels  have already  been computed  by Lafont and the second author in   \cite[Section 4]{LO07} and 
\cite[Sections 6.2, 6.3, 6.4]{LO09}.  The remaining  groups in our list will be discussed in
the following subsections. 

Observe that by a result of Farrell and Jones \cite{FJ95},
the cokernels of the relative assembly maps $ H_n^{\Gamma_{\widehat{\ell}}}(E_{\fin}(\Gamma_{\widehat{\ell}}) \rightarrow  \ast)$ 
are automatically trivial for $n< -1$ (in fact, both the source and target groups
vanish in this case).  In the same paper, they establish that, for the case
$n=-1$, these cokernels are finitely generated, which, by results of Farrell
\cite{F77}, Ramos \cite{Ra}, and Grunewald \cite{G07}, implies that the
cokernel is actually trivial.  In particular, we only need to
focus on the cases $n=0$ and $n=1$.  These cokernels are precisely
the elusive Bass, Farrell, and Waldhausen Nil-groups.  

It follows from additional results of Farrell and Jones \cite{FJ95} that the groups $K_{-1}(\mathbb{Z}\Gamma_{\widehat{\ell}})$ are
generated by the images of the groups $K_{-1}(\mathbb{Z}F)$, where $F$ runs over finite subgroups of $\Gamma_{\widehat{\ell}}$, and
the maps in question are induced by inclusion.

We are able to identify all of these cokernels explicitly, with the exceptions of $NK_{1}(\mathbb{Z}[D_{4}])$ and $NK_{1}(\mathbb{Z}[D_{6}])$.
 (see Subsections 
\ref{subsection:KtheoryC4Z} and \ref{loweralgebriacktheoryofD6timesDinfty}). The first group is known to be an infinite torsion group of exponent $2$ or $4$ \cite{We09}. 
We summarize the known 
non-trivial cokernels in Table \ref{table:cokernels}.

\renewcommand{\arraystretch}{1.6}
\begin{table}[!h]
\[
\begin{array}{|c|c|c|}
\hline V \in \vc_{\infty} &  H_0^{V}(E_{\fin}(V)\rightarrow
*) \neq 0 & H_1^{V}(E_{\fin}(V)\rightarrow
*)\neq 0 \\ \hline
C_4 \times \mathbb Z &\bigoplus_{\infty} \mathbb Z/2 &\bigoplus_{\infty}\mathbb Z/2 \\
\hline D_2 \times \mathbb Z & \bigoplus_{\infty} \mathbb Z/2 & \bigoplus_{\infty} \mathbb Z/2 \\
\hline D_2 \rtimes \mathbb Z &  \bigoplus_{\infty} \mathbb Z/2& \bigoplus_{\infty} \mathbb Z/2 \\
\hline D_4 \times \mathbb Z & \bigoplus_{\infty} \mathbb Z/2 \oplus \bigoplus_{\infty} \mathbb Z/4 & 2NK_1(\mathbb {Z}D_4) \\
\hline D_6 \times \mathbb{Z} & \bigoplus_{\infty} \mathbb{Z}/2 & 2NK_1(\mathbb ZD_6) \\
\hline C_4 \times D_{\infty} &  \bigoplus_{\infty} \mathbb Z/2 &  \bigoplus_{\infty} \mathbb Z/2 \\
\hline D_2 \times D_{\infty} &\bigoplus_{\infty} \mathbb Z/2 &\bigoplus_{\infty}\mathbb Z/2 \\
\hline D_4 \ast_{D_2} D_4 &\bigoplus_{\infty} \mathbb Z/2 &\bigoplus_{\infty} \mathbb Z/2 \\
\hline D_4 \ast_{C_4} D_4 &\bigoplus_{\infty} \mathbb Z/2 &\bigoplus_{\infty} \mathbb Z/2 \\
\hline  (D_2\times \mathbb Z/2)*_{D_2}D_4 & \bigoplus_{\infty} \mathbb Z/2 & \bigoplus_{\infty} \mathbb Z/2 \\
\hline D_4 \times D_{\infty} & \bigoplus_{\infty} \mathbb Z/2 \oplus \bigoplus_{\infty} \mathbb Z/4 &  NK_1(\mathbb{Z}D_4)   \\
\hline D_6 \times D_{\infty} & \bigoplus_{\infty} \mathbb{Z}/2 & NK_1(\mathbb ZD_6)\\ \hline
\end{array}
\]
\caption{ Cokernels of relative assembly maps for maximal infinite $V \in \vc_{\infty}$.}
\label{table:cokernels}
\end{table}

\subsubsection{{\bf The Lower algebraic $K$-theory of $C_4 \times \mathbb Z$, $D_4 \times \mathbb Z$, and $D_{6} \times \mathbb{Z}$}} \label{subsection:KtheoryC4Z}
The  Bass-Heller-Swan decomposition yields the following isomorphism for any group $F$ and $q \leq 1$:
\[
Wh_q(F \times \mathbb Z) \cong Wh_{q-1}(F) \oplus Wh_{q}(F) \oplus NK_q(\mathbb ZF) \oplus NK_q(\mathbb ZF).
\]
Also, as mentioned earlier, $K_{q}(\mathbb{Z} V)$ is zero for $q <-1$ (see
\cite{FJ95}). 

Let us consider first the cases of $Wh_{q}(V)$ for $V=C_4 \times \mathbb Z$ and  $D_4 \times \mathbb Z$.
Since $Wh_q(C_4) \cong 0$ for $q \leq 1$ (Section \ref{section:contributionoffinites}, 
Table \ref{table:tableKtheoryEfin}) and $NK_q(\mathbb ZC_4) \cong \bigoplus_{\infty} \mathbb Z/2$ for  $q= 0, 1$ (see \cite{We09}), it follows that
\[
Wh_q(C_4 \times \mathbb Z) \cong
\begin{cases}
\bigoplus_{\infty} \mathbb Z/2 & q=1 \\
\bigoplus_{\infty} \mathbb Z/2 & q=0 \\
0 & q \leq -1,
\end{cases} 
\]
and the cokernels of the relative assembly maps  are both isomorphic to $\bigoplus_{\infty} \mathbb Z/2.$ 

Since $Wh_q(D_4) \cong 0$ for $q \leq 1$, it follows that
\[
Wh_q(D_4 \times \mathbb Z) \cong 
\begin{cases}
2NK_1(\mathbb ZD_4) & q=1 \\
\bigoplus_{\infty} \mathbb Z/2  \oplus \bigoplus_{\infty} \mathbb Z/4 & q=0 \\
0 & q \leq -1,
\end{cases}
\]
since the Bass Nil-group  $NK_0(\mathbb ZD_4) \cong \, \bigoplus_{\infty} \mathbb Z/2  \oplus \bigoplus_{\infty} \mathbb Z/4$ was computed by Weibel in \cite{We09}. He also showed that $NK_1(\mathbb ZD_4)$ is a countably infinite torsion group of exponent 2 or 4.

Since $K_{-1}(\mathbb ZD_6)\cong \mathbb Z$ and $Wh_q(D_6) \cong 0$ for $q=0, 1$, it follows that 
\[
Wh_q(D_6 \times \mathbb Z) \cong
\begin{cases}
2NK_1(\mathbb ZD_6) & q=1 \\
2NK_0(\mathbb ZD_6)  & q=0 \\
\mathbb Z & q= -1\\
0 & q <-1.
\end{cases}
\]

\begin{proposition}
$NK_0(\mathbb{Z}D_{6}) \cong \bigoplus_{\infty} \mathbb{Z}/2$.
\end{proposition}

\begin{proof}
We first show that $NK_0(\mathbb Z[D_6]) \cong NK_1(\mathbb F_2[D_3])$. Write $\mathbb Z[D_6]=\mathbb Z[D_3 \times \mathbb Z/2]$ as $\mathbb Z[D_3][\mathbb Z/2]$, and let $\mathbb Z/2=\langle t \rangle$.   Consider the following Cartesian square:
\[
\xymatrix@C=20pt@R=30pt{
\mathbb Z[D_3][\mathbb Z/2] \ar[d]_{t \mapsto -1} \ar[r]^{\hspace{.2cm}t \mapsto 1} & \mathbb Z[D_3]
     \ar[d]^{\text{mod}\; 2} \\
\mathbb Z[D_3] \ar[r]^{\text{mod}\; 2} & \mathbb F_2[D_3]}
\]      
Applying the $NK$-functor to this Cartesian square yields the Mayer-Vietoris sequence (see \cite[Theorem 6.4]{Mi71})
\begin{equation*}
\begin{split}
NK_2(\mathbb F_2[D_3]) &\rightarrow NK_1(\mathbb Z[D_6]) \rightarrow NK_1(\mathbb Z[D_3]) \oplus NK_1(\mathbb Z[D_3]) \rightarrow NK_1(\mathbb F_2[D_3])\\ &\rightarrow NK_0(\mathbb Z[D_6]) \rightarrow 
 NK_0(\mathbb Z[D_3]) \oplus NK_0(\mathbb Z[D_3]) \rightarrow \cdots.
\end{split}
\end{equation*}
Since $NK_i(\mathbb Z[D_3])=0$, for $i \leq 1$ (see \cite{Ha87}), we obtain the desired isomorphism 
$NK_0(\mathbb Z[D_6])\cong NK_{1}(\mathbb F_2[D_3])$, as claimed. 

Next, we claim $NK_1(\mathbb F_2[D_3])\cong \bigoplus_{\infty} \mathbb Z/2$. Note that the ring $\mathbb F_2[D_3] \cong M_{2}(\mathbb F_2) \times \mathbb F_2[C_2]$ (see \cite[Example 2]{Ma07}). Consider the following cartesian square
\[
\xymatrix@C=20pt@R=30pt{
\mathbb F_2[D_3] \ar[d] \ar[r] & \mathbb F_2[C_2]
     \ar[d] \\
M_2(\mathbb F_2) \ar[r] & 0}
\]      
Applying the $NK$-functor to this Cartesian square yields the following isomorphism
\[
0 \rightarrow  NK_1(\mathbb F_2[D_3])  \rightarrow  NK_1(M_2(\mathbb F_2)) \oplus NK_1(\mathbb F_2[C_2]) \rightarrow 0. 
\]
Since $M_2(\mathbb F_2)$ is a regular ring (see \cite[Proposition 2.3]{We09}), $NK_i(M_2(\mathbb F_2)) \cong 0$ for all $i \in \mathbb{Z}$. It follows that  $NK_1(\mathbb F_2[D_3]) \cong  NK_1(\mathbb F_2[C_2]) \cong \bigoplus_{\infty} \mathbb Z/2$. For the latter isomorphism we refer the reader to \cite{LO09}.

This shows that $NK_0(\mathbb ZD_6) \cong NK_1(\mathbb F_2[D_3]) \cong \bigoplus_{\infty} \mathbb Z/2$.
\end{proof}

We summarize our calculations as follows:

\[
Wh_q(D_6 \times \mathbb Z) \cong
\begin{cases}
2NK_1(\mathbb ZD_6) & q=1 \\
\bigoplus_{\infty} \mathbb{Z}/2 & q=0 \\
\mathbb Z & q= -1\\
0 & q <-1.
\end{cases}
\]

\begin{remark}
Consider the first Cartesian square from the above proof. The head of the associated Mayer-Vietoris sequence gives the following epimorphism,
 since $NK_1(\mathbb ZD_3)$ vanishes:
$$NK_2(\mathbb F_2[D_3])  \rightarrow NK_1(\mathbb ZD_6)\rightarrow  0.$$ 
The second Cartesian square gives us an epimorphism
$$ NK_{2}(\mathbb{F}_{2}[D_{3}]) \rightarrow NK_{2}(\mathbb{F}_{2}[\mathbb{C}_{2}]) \oplus NK_{2}(M_{2}(\mathbb{F}_{2})) \rightarrow 0.$$
As mentioned earlier, $NK_{2}(M_{2}(\mathbb{F}_{2}))$ is trivial. It follows that we have an epimorphism
$$ NK_{2}(\mathbb{F}_{2}[D_{3}]) \rightarrow NK_{2}(\mathbb{F}_{2}[\mathbb{C}_{2}]) \rightarrow 0.$$
This implies that $NK_{2}(\mathbb{F}_{2}[D_{3}])$ is non-trivial, and therefore must be an infinitely generated torsion group (see \cite{F77}).
It remains to study the first epimorphism above. We would like to show that this map is an isomorphism, but we have no proof.
\end{remark}

\subsubsection{{\bf The Lower algebraic $K$-theory of $D_2 \rtimes \mathbb Z$}} 

Let $V= D_2 \rtimes \mathbb Z$, where a generator of $\mathbb Z$ acts by an automorphism $\alpha$ of order 2.
Since $K_{-1}(\mathbb ZF) \cong0$ for all $F \leq D_{2}$ (Section \ref{section:contributionoffinites}, 
Table \ref{table:tableKtheoryEfin}), we have $K_{-1}(\mathbb
ZV)\cong 0$.

Farrell and Hsiang in \cite{FH68} show that the group $Wh_q(F \rtimes_{\alpha} \mathbb Z)$ can be expressed in the following form:
\[
Wh_q(F \rtimes_{\alpha} \mathbb Z) \cong C  \oplus NK_q(\mathbb ZF, \alpha) \oplus NK_q(\mathbb ZF, \alpha^{-1}),
\]
where $C$ is a suitable quotient (determined by the automorphism $\alpha$) of   the  $K$-groups $Wh_{q-1}(F) \oplus Wh_{q}(F)$.  Since $Wh_q(D_2) \cong 0$ for $q \leq 1$,  it follows that  $Wh_q(D_2 \rtimes \mathbb Z) \cong NK_q(\mathbb ZF, \alpha) \oplus NK_q(\mathbb ZF, \alpha^{-1})$ for $q=0,1$. Farrell and Hsiang also show that  $NK_q(\mathbb ZF, \alpha) \cong NK_q(\mathbb ZF, \alpha^{-1})$, therefore   $Wh_q(D_2 \rtimes \mathbb Z) \cong NK_q(\mathbb ZF, \alpha) \oplus NK_q(\mathbb ZF, \alpha).$

The group $D_2 \rtimes \mathbb Z$ is the canonical index two subgroup of the group $(D_2\times \mathbb Z/2)*_{D_2}D_4$, so by  \cite{DKR11} (see also \cite{DQR11}), we know that the Waldhausen  Nil-groups $NK_q(\mathbb ZD_2; \mathbb Z[(D_2 \times \mathbb Z/2)-D_2], \mathbb Z[D_4-D_2])$ are isomorphic to the corresponding Farrell
Nil-groups for the canonical index two subgroup $D_2 \rtimes \mathbb Z \trianglelefteq  (D_2 \times \mathbb Z/2) \ast_{D_2} D_4.$ In \cite[Section 6.2]{LO09}, Lafont and the second author showed that $NK_q(\mathbb ZD_2; \mathbb Z[(D_2 \times \mathbb Z/2)-D_2], \mathbb Z[D_4-D_2]) \cong \bigoplus_{\infty} \mathbb Z/2$;  it follows that
\[
NK_q(\mathbb ZF, \alpha) \cong \bigoplus_{\infty} \mathbb Z/2, \hspace{.5cm} q=0,1.
\]
Therefore, we have that
\[
Wh_q(D_2 \rtimes \mathbb Z) \cong
\begin{cases}
\bigoplus_{\infty} \mathbb Z/2  & q=1 \\
\bigoplus_{\infty} \mathbb Z/2   & q=0 \\
0 & q \leq -1,
\end{cases}
\]
 and the cokernels of the relative assembly maps are both isomorphic to $\bigoplus_{\infty} \mathbb Z/2.$ 

\subsubsection{{\bf The Lower algebraic $K$-theory of $D_4 \ast_{C_4} D_4$}}
Since  $K_{-1}(\mathbb ZF) \cong 0$ for all $F \leq D_{4}$ (Section \ref{section:contributionoffinites}, 
Table \ref{table:tableKtheoryEfin}), we see that, for $V=D_4 \ast_{C_4} D_4$, we have that $K_{-1}(\mathbb
ZV)\cong 0$.

 For the remaining $K$-groups, using \cite[Lemma 3.8]{CP02}, we have that $\widetilde {K}_0(\mathbb ZV) \cong 
NK_0(\mathbb ZC_4; A_1, A_2)$, where $A_i=\mathbb Z[(D_4-C_4]$ is the $\mathbb ZC_4$-bimodule generated by $D_4 - C_4$, for $i=1,2$.  Similarly, we
have that $Wh(V) \cong NK_1(\mathbb ZC_4; A_1, A_2)$, where $A_1$,
$A_2$ are the bi-modules defined above.
 
Now by \cite{DKR11} (see also \cite{DQR11}), we know that the Waldhausen  Nil-groups $NK_q(\mathbb ZC_4; A_1, A_2)$ are isomorphic to the corresponding Farrell
Nil-group for the canonical index two subgroup $C_4 \times
\mathbb Z \, \trianglelefteq \, D_4 \ast_{C_4} D_4$.  Note that, in this
case, the Farrell Nil-group is untwisted, and hence is just the Bass
Nil-group $NK_q(\mathbb ZC_4) \cong \bigoplus_{\infty} \mathbb Z/2$ ($q=0,1$), so the  lower algebraic $K$-theory of $D_4  \ast_{C_4} D_4$ is given by:

\[
Wh_q(D_4 \ast_{C_4} D_4) \cong
\begin{cases}
\bigoplus_{\infty} \mathbb Z/2 & q=1 \\
\bigoplus_{\infty} \mathbb Z/2 & q=0 \\
0 & q \leq -1.
\end{cases}
\]
The cokernels of the relative assembly maps are both
isomorphic to $\bigoplus _{\infty} \mathbb Z/2$.

\subsubsection{{\bf The Lower algebraic $K$-theory of $C_4 \times D_{\infty}$}}

First, note that $C_4 \times D_{\infty} \cong (C_4 \times \mathbb Z/2) \ast_{C_4}
(C_4 \times \mathbb Z/2)$.   Since $K_{-1}(\mathbb ZF) \cong 0$ for all $F \leq C_{4} \times \mathbb{Z}/2$ (Section \ref{section:contributionoffinites}, 
Table \ref{table:tableKtheoryEfin}), we see that
for $V=(C_4 \times \mathbb Z/2) \ast_{C_4} (C_4 \times \mathbb Z/2)$, we have $K_{-1}(\mathbb
ZV) \cong 0$.

For the remaining $K$-groups, we use \cite[Lemma 3.8]{CP02}.   Since  $\widetilde{K}_0(\mathbb
ZC_4) \cong 0$ and $\widetilde{K}_0(\mathbb Z[C_4 \times \mathbb Z/2])\cong \mathbb Z/2$ (Section \ref{section:contributionoffinites}, 
Table \ref{table:tableKtheoryEfin}), it follows that $\widetilde {K}_0(\mathbb
ZV) \cong (\mathbb Z/2)^2 \oplus NK_0(\mathbb ZC_4; B_1, B_2)$, where $B_i=\mathbb Z[(C_4 \times \mathbb Z/2)
- C_4]$ is the $\mathbb ZC_4$-bimodule generated by $(C_4 \times \mathbb Z/2) - C_4$
for $i=1,2$. Since $Wh(C_4) \cong  Wh(C_4 \times \mathbb Z/2) \cong 0$ (see Section 7, Table 4), it follows that $Wh(V) \cong  NK_1(\mathbb ZC_4; B_1, B_2)$, with $B_1$ and $B_2$ as before. The Nil-groups appearing in these computations are the Waldhausen Nil-groups.

Now by \cite{DKR11} (see also \cite{DQR11}), we know that the Waldhausen  Nil-groups $NK_q(\mathbb ZC_4; B_1, B_2)$ are isomorphic to the corresponding Farrell
Nil-group for the canonical index two subgroup $C_4 \times \mathbb Z \trianglelefteq  C_4 \times D_\infty$.  Note that, in this
case, the Farrell Nil-group is untwisted, and hence is just the Bass
Nil-group $NK_q(\mathbb ZC_4) \cong \bigoplus_{\infty} \mathbb Z/2$ ($q=0,1$). We summarize
our computations as follows:
\[
Wh_q(C_4\times D_\infty) \cong 
\begin{cases}
\bigoplus_{\infty} \mathbb Z/2 & q=1 \\
 (\mathbb Z/2)^2 \oplus \bigoplus_{\infty} \mathbb Z/2& q=0 \\
0 & q \leq -1.
\end{cases}
\]
The cokernels of the relative assembly maps are both
isomorphic to $\bigoplus _{\infty} \mathbb Z/2$.

\subsubsection{{\bf The Lower algebraic $K$-theory of $D_6 \times D_{\infty}$}\label{loweralgebriacktheoryofD6timesDinfty}}

First, note that $D_6\times D_{\infty} \cong (D_{6} \times \mathbb Z/2) \ast_{D_6}
(D_{6} \times \mathbb Z/2)$.   Since $K_{-1}(\mathbb ZD_6) \cong \mathbb Z$, and
$K_{-1}(\mathbb ZD_6 \times \mathbb Z/2) \cong \mathbb Z^3$ (Section \ref{section:contributionoffinites}, 
Table \ref{table:tableKtheoryEfin}), we see that
for $V=(D_6 \times \mathbb Z/2) \ast_{D_6} (D_6 \times \mathbb Z/2)$, we have $K_{-1}(\mathbb
ZV) \cong \mathbb Z^5$ (see also Example \ref{example:gamma5}).

For the remaining $K$-groups, we  use \cite[Lemma 3.8]{CP02}.  Since  $\widetilde{K}_0(\mathbb
ZD_6) \cong 0$   and $\widetilde{K}_0(\mathbb Z[D_6 \times \mathbb Z/2])  \cong (\mathbb Z/2)^2$, we have that $\widetilde {K}_0(\mathbb
ZV) \cong (\mathbb Z/2)^4 \oplus NK_0(\mathbb ZD_6; C_1, C_2)$, where $C_i=\mathbb Z[(D_6 \times \mathbb Z/2)
- D_6]$ is the $\mathbb ZD_6$-bimodule generated by $(D_6 \times \mathbb Z/2) - D_6$
for $i=1,2$.  Since $Wh(D_6)$ and $Wh(D_6 \times \mathbb Z/2)$ are both trivial, it follows that $Wh(V) \cong  NK_1(\mathbb ZD_6; C_1, C_2)$, with $C_1$ and $C_2$ as before.  The Nil-groups
appearing in these computations are the Waldhausen Nil-groups.

Now by \cite{DKR11} (see also \cite{DQR11}), we know that the Waldhausen  Nil- groups $NK_q(\mathbb ZD_6; C_1, C_2)$ are isomorphic to the corresponding Farrell
Nil-group for the canonical index two subgroup $D_6 \times
\mathbb Z \trianglelefteq D_6 \times D_\infty$.  Note that, in this
case, the Farrell Nil-group is untwisted, and hence is just the Bass
Nil-group $NK_q(\mathbb ZD_6)$. We summarize
our computations as follows:
\[
Wh_q(D_6\times D_\infty)=
\begin{cases}
NK_1(\mathbb ZD_6) & q=1 \\
(\mathbb{Z}/2)^{4} \oplus \bigoplus_{\infty} \mathbb{Z}/2 & q=0 \\
\mathbb Z^5 & q= -1\\
0 & q \leq -2.
\end{cases}
\]
Here the cokernels are simply the direct sums of two copies of the Bass Nil-groups of $D_{6}$. We do not know the isomorphism types of
these groups.  

\section{Summary} \label{section:summary}

We can now compute the lower algebraic $K$-theory of the $73$ split crystallographic groups. 
Recall that Theorem \ref{theorem:splitting2} tells us that, for all such groups $\Gamma$, we have an isomorphism
\[
K_n(\mathbb Z\Gamma) \cong 
H_{\ast}^{\g}(E_{\fin}(\g); \mathbb{KZ}^{-\infty}) \oplus  \bigoplus_{ \widehat{\ell} \in \mathcal{T}''}  H_n^{\Gamma_{\widehat{\ell}}}(E_{\fin}(\Gamma_{\widehat{\ell}}) \rightarrow  \ast;\;  \mathbb{KZ}^{-\infty}).
\]
 For all 73 of our groups, we have
 
 \begin{itemize}
 \item explicitly computed in Section \ref{section:contributionoffinites} the homology groups
 \[
 H_{\ast}^{\g}(E_{\fin}(\g); \mathbb{KZ}^{-\infty}),
 \]
 and summarized the results in Table \ref{thefinitepart};
 
 \item described in Subsection \ref{subsection:passingtosubgroups} a suitable indexing set $\mathcal{T}''$ (as summarized in Tables
 \ref{table:ssgamma1}, \ref{table:ssgamma2}, \ref{table:ssgamma3}, \ref{table:ssgamma4}, and \ref{table:ssgamma5});   

\item explicitly computed in Subsection \ref{subsection:vcsubgroups} the isomorphism types of the groups associated to the indexing set $\mathcal{T}''$ (as 
summarized in Tables \ref{table:vc1}, \ref{table:vc2}, \ref{table:vc3}, \ref{table:vc4}, and \ref{table:vc5});

\item explicitly computed in Subsection \ref{subsection:cokernels} (see
Table \ref{table:cokernels}) the cokernels 
$$H_n^{\Gamma_{\widehat{\ell}}}(E_{\fin}(\Gamma_{\widehat{\ell}}) \rightarrow  \ast;\;  \mathbb{KZ}^{-\infty})$$
for all of the infinite virtually cyclic subgroups that
occur as stabilizers of lines $\widehat{\ell}$ from $\mathcal{T}''$, except for the groups $NK_{1}(\mathbb{Z}D_{n})$ ($n=4,6$). (We note
that Weibel \cite{We09} proved that $NK_{1}(\mathbb{Z}D_{4})$ is a countably infinite  torsion  group of exponent $2$ or $4$.)
\end{itemize}

This information makes it straightforward to apply Theorem \ref{theorem:splitting2} to a given split $3$-dimensional crystallographic group $\Gamma$, 
yielding an explicit calculation of $K_{-1}(\mathbb{Z} \Gamma)$, $\widetilde K_0(\mathbb{Z} \Gamma)$, and $Wh (\Gamma)$. We have summarized these calculations
in Tables \ref{LoweralgebraicKtheoryofsplitthreedimensionalcrystallographicgroups1} and \ref{LoweralgebraicKtheoryofsplitthreedimensionalcrystallographicgroups2}.
We have entered only 
the non-zero terms in these Tables; all of the blank squares represent
entries where the corresponding group vanishes, and if $\Gamma$ does not appear in either table, then $Wh_n(\mathbb Z\Gamma)=0$ for all $n \leq 1$.
 
We conclude with a pair of examples that illustrate how the pieces of the calculation fit together.

  
 \begin{table}
\renewcommand{\arraystretch}{1.8}
\[
\footnotesize \begin{array}{|c|c|c|c|}
\hline \g & K_{-1} \neq 0 & \tilde{K}_0 \neq 0 & Wh \neq 0 
\\ \hline \hline
\Gamma_1 & \mathbb Z^2 & (\mathbb Z/4)^4 \oplus \bigoplus_{\infty} \mathbb Z/2 \oplus  \bigoplus_{\infty} \mathbb Z/4&  \bigoplus_{\infty} \mathbb Z/2 \oplus 2NK_1(\mathbb ZD_4)
\\ \hline
(S_4^+)_1 & &\bigoplus_{\infty} \mathbb Z/2 & \bigoplus_{\infty} \mathbb Z/2 \\ \hline
(S'_4)_1 & &\bigoplus_{\infty} \mathbb Z/2 &\bigoplus_{\infty} \mathbb Z/2 \\ \hline
(A_4^+ \times (-1))_1 & \mathbb Z^2 & (\mathbb Z/2)^4 \oplus\bigoplus_{\infty} \mathbb Z/2 &\bigoplus_{\infty} \mathbb Z/2 \\ \hline
(D''_4)_1 & & \bigoplus_{\infty} \mathbb Z/2 \oplus  \bigoplus_{\infty} \mathbb Z/4 & \bigoplus_{\infty} \mathbb Z/2  \oplus 4NK_1(\mathbb ZD_4) \\ \hline
(D^{+}_4 \times (-1))_1& & (\mathbb Z/2)^2 \oplus (\mathbb Z/4)^4 \oplus \bigoplus_{\infty} \mathbb Z/2 \oplus  \bigoplus_{\infty} \mathbb Z/4 &   \bigoplus_{\infty} \mathbb Z/2 \oplus 2NK_1(\mathbb ZD_4) \\ \hline
(D_4^+)_1 &  &  \bigoplus_{\infty} \mathbb Z/2 &  \bigoplus_{\infty} \mathbb Z/2 \\ \hline
(D_2^{+} \times (-1))_1 &  & (\mathbb Z/2)^8 \oplus \bigoplus_{\infty} \mathbb Z/2 & \bigoplus_{\infty} \mathbb Z/2 \\ \hline
(C^{+}_{4} \times (-1))_1 & &(\mathbb Z/2)^4 \oplus   \bigoplus_{\infty} \mathbb Z/2  &   \bigoplus_{\infty} \mathbb Z/2 \\ \hline
(C^{+}_{4})_1 &  & \bigoplus_{\infty} \mathbb{Z}/2 & \bigoplus_{\infty} \mathbb{Z}/2 \\ \hline
(D'_2)_1 &  &   \bigoplus_{\infty} \mathbb Z/2  &   \bigoplus_{\infty} \mathbb Z/2  \\ \hline
(D'_{4} )_1&  &    \bigoplus_{\infty} \mathbb Z/2&   \bigoplus_{\infty} \mathbb Z/2 \\ \hline
(\widehat{D}'_{4})_1&  &   \bigoplus_{\infty} \mathbb Z/2 &   \bigoplus_{\infty} \mathbb Z/2 \\ \hline \hline
\Gamma_2 & \mathbb Z^2 &(\mathbb Z/4)^2 \oplus \bigoplus_{\infty} \mathbb Z/2  \oplus \bigoplus_{\infty} \mathbb Z/4 & \bigoplus_{\infty} \mathbb Z/2   \oplus NK_1(\mathbb ZD_4) \\ \hline
(S_4^+)_2 & & \bigoplus_{\infty} \mathbb Z/2  &  \bigoplus_{\infty} \mathbb Z/2  \\ \hline
(S'_4)_2& & \bigoplus_{\infty} \mathbb Z/2  &  \bigoplus_{\infty} \mathbb Z/2  \\ \hline
(A_4^+ \times (-1))_2 & \mathbb Z^2 & (\mathbb Z/2)^2 \oplus  \bigoplus_{\infty} \mathbb Z/2 &  \bigoplus_{\infty} \mathbb Z/2 \\ \hline
(D''_4)_2 &  &  \bigoplus_{\infty} \mathbb Z/2  \oplus \bigoplus_{\infty} \mathbb Z/4 &  \bigoplus_{\infty} \mathbb Z/2   \oplus 2NK_1(\mathbb ZD_4)  \\ \hline
(D^{+}_4 \times (-1))_2 &  & (\mathbb Z/2) \oplus (\mathbb Z/4)^2 \oplus \bigoplus_{\infty} \mathbb Z/2  \oplus \bigoplus_{\infty} \mathbb Z/4 & \bigoplus_{\infty} \mathbb Z/2  \oplus  NK_1(\mathbb ZD_4)\\ \hline
(D^{+}_4)_2  &&  \bigoplus_{\infty} \mathbb Z/2  &  \bigoplus_{\infty} \mathbb Z/2  \\ \hline
(D_2^{+} \times (-1))_2 & & (\mathbb Z/2)^4 \oplus \bigoplus_{\infty} \mathbb Z/2 &   \bigoplus_{\infty} \mathbb Z/2 \\ \hline
(C^{+}_{4} \times (-1))_2 & & (\mathbb Z/2)^2 \oplus  \bigoplus_{\infty} \mathbb Z/2 &  \bigoplus_{\infty} \mathbb Z/2  \\ \hline
(C^{+}_{4})_2 & &  \bigoplus_{\infty} \mathbb Z/2  &   \bigoplus_{\infty} \mathbb Z/2 \\ \hline
(D'_2)_2 &  &   \bigoplus_{\infty} \mathbb Z/2  &   \bigoplus_{\infty} \mathbb Z/2  \\ \hline
(D'_{4} )_2&  &    \bigoplus_{\infty} \mathbb Z/2&   \bigoplus_{\infty} \mathbb Z/2 \\ \hline
(\widehat{D}'_{4})_2&  &   \bigoplus_{\infty} \mathbb Z/2 &   \bigoplus_{\infty} \mathbb Z/2 \\ \hline 
\end{array} \]
\caption{Lower algebraic $K$-theory of split three-dimensional crystallographic groups (part I).}
\label{LoweralgebraicKtheoryofsplitthreedimensionalcrystallographicgroups1}
\end{table}

\begin{table}
\renewcommand{\arraystretch}{1.8}
\[
\footnotesize \begin{array}{|c|c|c|c|}
\hline \g & K_{-1} \neq 0 & \tilde{K}_0 \neq 0 & Wh \neq 0 
\\ \hline \hline
\Gamma_3 & \mathbb Z^2 & (\mathbb Z/2) \oplus (\mathbb Z/4)^2 \oplus \bigoplus_{\infty} \mathbb Z/2 \oplus  \bigoplus_{\infty} \mathbb Z/4&  \bigoplus_{\infty} \mathbb Z/2 \oplus NK_1(\mathbb ZD_4) \\ \hline
(S_4^+)_3 & &\bigoplus_{\infty} \mathbb Z/2 & \bigoplus_{\infty} \mathbb Z/2 \\ \hline
(S'_4)_3 & &\bigoplus_{\infty} \mathbb Z/2 &\bigoplus_{\infty} \mathbb Z/2 \\ \hline
(A_4^+ \times (-1))_3 & \mathbb Z^2 & (\mathbb Z/2)^2 \oplus\bigoplus_{\infty} \mathbb Z/2 & \bigoplus_{\infty} \mathbb Z/2 \\ \hline
(D'_2)_3 &  & \bigoplus_{\infty} \mathbb Z/2 &   \bigoplus_{\infty} \mathbb Z/2  \\ \hline
(D_2^{+} \times (-1))_3  & &(\mathbb Z/2)^2 \bigoplus_{\infty} \mathbb Z/2 &   \bigoplus_{\infty} \mathbb Z/2 \\ \hline \hline
\Gamma_4 & & (\mathbb Z/2)^4 \oplus  \bigoplus_{\infty} \mathbb Z/2 &  \bigoplus_{\infty} \mathbb Z/2 \\ \hline
(D'_2)_4 &  & \bigoplus_{\infty} \mathbb Z/2 &   \bigoplus_{\infty} \mathbb Z/2  \\ \hline
(D'_{2_{2}})_4&  &   \bigoplus_{\infty} \mathbb Z/2 &   \bigoplus_{\infty} \mathbb Z/2  \\ \hline \hline
\Gamma_5 & \mathbb Z^7 & (\mathbb Z/2)^6 \oplus   \bigoplus_{\infty} \mathbb{Z}/2 &  \bigoplus_{\infty} \mathbb Z/2 \oplus NK_1(\mathbb ZD_6) \\ \hline
(D^+_6)_5 & \mathbb Z & &  \\ \hline
(C'_6)_5 & \mathbb Z^6 & & \\ \hline
(C^+_6 \times (-1))_5 & \mathbb Z^7 & (\mathbb Z/2)^4 & \\ \hline
(D'_6)_5 & \mathbb Z^6&    \bigoplus_{\infty} \mathbb Z/2 &   \bigoplus_{\infty} \mathbb Z/2  \\ \hline
(C^+_6)_5 &\mathbb Z & \mathbb Z &   \\ \hline
(D_3^{+} \times (-1))_5 & \mathbb Z^2 & & \\ \hline
(\widehat{D}'_{6})_5 & \mathbb Z^4 & \bigoplus_{\infty} \mathbb{Z}/2 & \bigoplus_{\infty} \mathbb{Z}/2 \\ \hline
(C^+_3 \times (-1))_5 & \mathbb Z^2 & & \\ \hline
(D''_6)_5  &\mathbb Z & \mathbb Z  \oplus   \bigoplus_{\infty} \mathbb{Z}/2 &  \bigoplus_{\infty} \mathbb Z/2 \oplus 2NK_1(\mathbb ZD_6) \\ \hline \hline
\Gamma_6 & \mathbb Z^2& & \\ \hline
(C^+_3 \times (-1))_6 & \mathbb Z^2 & & \\ \hline \hline
\Gamma_7 & \mathbb Z^2 & & \\ \hline
\end{array}
\]
\caption{Lower algebraic $K$-theory of split three-dimensional crystallographic groups (part II).}
\label{LoweralgebraicKtheoryofsplitthreedimensionalcrystallographicgroups2}
\end{table}


\begin{example} We compute the lower algebraic $K$-theory of $(A_{4}^{+} \times (-1))_{1} = \langle L, A_{4}^{+} \times (-1) \rangle$, where $L$ is the standard cubical lattice. We will
write $\Gamma$ in place of $(A_{4}^{+} \times (-1))_{1}$ and $H$ in place of $A_{4}^{+} \times (-1)$.
Note that 
\begin{align*}
H_{-1}^{\g}(E_{\fin}(\g); \mathbb{KZ}^{-\infty}) &\cong \mathbb{Z}^{2}; \\
H_{0}^{\g}(E_{\fin}(\g); \mathbb{KZ}^{-\infty}) &\cong (\mathbb{Z}/2)^{4}; \\
H_{1}^{\g}(E_{\fin}(\g); \mathbb{KZ}^{-\infty}) &\cong 0,
\end{align*}
by Table \ref{thefinitepart}. (We recall that the group $Wh(G)$ is always trivial if $G$ is a finite subgroup of a $3$-dimensional
crystallographic group -- see Table \ref{table:tableKtheoryEfin}.) The next step is to compute the second summand from the
formula in Theorem \ref{theorem:splitting2} (as reproduced at the beginning of this section).  We found in Subsection \ref{subsection:passingtosubgroups} (see Table \ref{table:ssgamma1}) that the parametrized lines
$$ \left( \begin{smallmatrix} 0 \\ 0 \\ \alpha \end{smallmatrix} \right), \left( \begin{smallmatrix} 1/2 \\ 0 \\ \alpha \end{smallmatrix} \right), \left( \begin{smallmatrix} 0 \\ 1/2 \\ \alpha \end{smallmatrix} \right), \left( \begin{smallmatrix} 1/2 \\ 1/2 \\ \alpha \end{smallmatrix} \right)$$
represent a choice of indexing set $\mathcal{T}''$ for the latter summand. The strict stabilizers $\overline{\Gamma}_{\widehat{\ell}}$
 of these lines $\widehat{\ell}$ (i.e., the subgroups that
fix the lines pointwise) satisfy $\pi(\overline{\Gamma}_{\widehat{\ell}}) = D'_{2}$, where $\pi: \Gamma \rightarrow H$ is the
canonical projection into the point group. We can apply Procedure \ref{procedure:computethestabilizer} to conclude that
the stabilizers $\Gamma_{\widehat{\ell}}$ of these lines are all isomorphic to $D_{2} \times D_{\infty}$ (see Table \ref{table:vc1}). 
It follows that 
$$\bigoplus_{ \widehat{\ell} \in \mathcal{T}''}  H_n^{\Gamma_{\widehat{\ell}}}(E_{\fin}(\Gamma_{\widehat{\ell}}) \rightarrow  \ast;\;  \mathbb{KZ}^{-\infty}) \cong \left( \bigoplus_{\infty} \mathbb{Z}/2 \right)^{4} \cong \bigoplus_{\infty} \mathbb{Z}/2 $$
for $n=0$ and $1$. (We recall that \cite{FJ95} prove that the term in question is always trivial when $n=-1$.) See Table \ref{table:cokernels}. It now follows directly from Theorem \ref{theorem:splitting2} that the lower algebraic $K$-groups of
$\Gamma$ are as described in Table \ref{LoweralgebraicKtheoryofsplitthreedimensionalcrystallographicgroups1}.
\end{example}

\begin{example} We compute the lower algebraic $K$-theory of $\Gamma_{5} = \langle L, D_{6}^{+} \times (-1) \rangle$, 
where $L$ is the lattice $\langle \mathbf{v}_{1}, \mathbf{v}_{2}, \mathbf{v}_{3} \rangle$. We will
write $\Gamma$ in place of $\Gamma_{5}$ and $H$ in place of $D_{6}^{+} \times (-1)$.
Note that 
\begin{align*}
H_{-1}^{\g}(E_{\fin}(\g); \mathbb{KZ}^{-\infty}) &\cong \mathbb{Z}^{7}; \\
H_{0}^{\g}(E_{\fin}(\g); \mathbb{KZ}^{-\infty}) &\cong (\mathbb{Z}/2)^{6}; \\
H_{1}^{\g}(E_{\fin}(\g); \mathbb{KZ}^{-\infty}) &\cong 0,
\end{align*}
by Table \ref{thefinitepart}. Now we compute the second summand from the
formula in Theorem \ref{theorem:splitting2}. Theorem \ref{theorem:gamma5T} showed that the parametrized lines
$$\left( \begin{smallmatrix} \alpha \\ \alpha \\ \alpha \end{smallmatrix} \right),
\left( \begin{smallmatrix} \alpha + 1/2 \\ \alpha - 1/2 \\  \alpha \end{smallmatrix} \right),
\left( \begin{smallmatrix} \alpha \\ -\alpha \\  0 \end{smallmatrix} \right), 
\left( \begin{smallmatrix} \alpha + 1/2 \\ -\alpha + 1/2 \\  1/2 \end{smallmatrix} \right),
\left( \begin{smallmatrix} \alpha \\ -2\alpha \\ \alpha \end{smallmatrix} \right), 
\left( \begin{smallmatrix} \alpha + 1/2 \\ -2\alpha + 1/2 \\ \alpha + 1/2 \end{smallmatrix} \right)$$
represent a choice of indexing set $\mathcal{T}''$ for the latter summand (the same fact is recorded in Table \ref{table:ssgamma5}). 
The strict stabilizers $\overline{\Gamma}_{\widehat{\ell}}$
 of these lines $\widehat{\ell}$  satisfy 
$$\pi(\overline{\Gamma}_{\widehat{\ell}}) = D''_{6}, \, \, \langle A, D \rangle, \, \, \langle D, E \rangle, \, \, \langle D, E \rangle, 
\, \, \langle E, F \rangle, \, \, \langle E, F \rangle,$$ 
respectively, where $\pi: \Gamma \rightarrow H$ is the
canonical projection into the point group, and the matrices $A$, $D$, $E$, and $F$ are identified in Theorem \ref{theorem:gamma5T}. We can apply Procedure \ref{procedure:computethestabilizer} to conclude that
the stabilizers $\Gamma_{\widehat{\ell}}$ of these lines are either $D_{6} \times D_{\infty}$ or $D_{2} \times D_{\infty}$, where the
first group occurs once and the latter group occurs five times (see Table \ref{table:vc5}). 
It follows that 
\begin{align*}
\bigoplus_{ \widehat{\ell} \in \mathcal{T}''}  H_n^{\Gamma_{\widehat{\ell}}}(E_{\fin}(\Gamma_{\widehat{\ell}}) \rightarrow  \ast;\;  \mathbb{KZ}^{-\infty}) &\cong NK_{n}(\mathbb{Z}D_{6}) \oplus \left( \bigoplus_{\infty} \mathbb{Z}/2 \right)^{5}   \\
&\cong NK_{n}(\mathbb{Z}D_{6}) \oplus \bigoplus_{\infty} \mathbb{Z}/2  
\end{align*}
for $n=0$ and $1$. (The given term is trivial when $n=-1$.) See Table \ref{table:cokernels}. It now follows directly from Theorem \ref{theorem:splitting2} that the lower algebraic $K$-groups of
$\Gamma$ are as described in Table \ref{LoweralgebraicKtheoryofsplitthreedimensionalcrystallographicgroups2}.
\end{example}

\end{document}